\documentclass[reqno]{amsart}

\author{Kenji Fukaya}
\address{SSimons Center for Geometry and Physics, State University of New York, Stony Brook, NY 11794-3636, USA \&
Center for Geometry and Physics, Institute for
Basic Sciences (IBS), Pohang, Gyungbuk 790-784} \email{kfukaya@scgp.stonybrook.edu}
\author{Yong-Geun Oh}
\address{Center for Geometry and Physics, Institute for
Basic Sciences (IBS), Pohang, Gyungbuk 790-784 \& 
Department of Mathematics, POSTECH, Pohang, Korea } \email{yongoh1@ibs.re.kr}
\author{Hiroshi Ohta}
\address{Graduate School of Mathematics,
Nagoya University, Nagoya, 464-8602, Japan } \email{ohta@math.nagoya-u.ac.jp}
\author{Kaoru Ono}
\address{Research Institute for Mathematical Sciences, Kyoto University, Kyoto 606-8502, Japan}
\email{ono@kurims.kyoto-u.ac.jp}

\input epsf

\title[Anti-symplectic involution and Floer cohomology]
{Anti-symplectic involution and
Floer cohomology}

\usepackage[all]{xy}
\usepackage{amsmath, amssymb}
\def\E{\ifmmode{\mathbb E}\else{$\mathbb E$}\fi} 
\def\N{\ifmmode{\mathbb N}\else{$\mathbb N$}\fi} 
\def\R{\ifmmode{\mathbb R}\else{$\mathbb R$}\fi} 
\def\Q{\ifmmode{\mathbb Q}\else{$\mathbb Q$}\fi} 
\def\C{\ifmmode{\mathbb C}\else{$\mathbb C$}\fi} 
\def\H{\ifmmode{\mathbb H}\else{$\mathbb H$}\fi} 
\def\Z{\ifmmode{\mathbb Z}\else{$\mathbb Z$}\fi} 
\def\P{\ifmmode{\mathbb P}\else{$\mathbb P$}\fi} 
\def\T{\ifmmode{\mathbb T}\else{$\mathbb T$}\fi} 
\def\SS{\ifmmode{\mathbb S}\else{$\mathbb S$}\fi} 
\def\DD{\ifmmode{\mathbb D}\else{$\mathbb D$}\fi} 
\def\K{\ifmmode{\mathbb K}\else{$\mathbb K$}\fi}

\newcommand{\del}{\partial}

\newcommand{\ben}{\begin{enumerate}}
\newcommand{\een}{\end{enumerate}}
\newcommand{\be}{\begin{equation}}
\newcommand{\ee}{\end{equation}}
\newcommand{\bea}{\begin{eqnarray}}
\newcommand{\eea}{\end{eqnarray}}
\newcommand{\beastar}{\begin{eqnarray*}}
\newcommand{\eeastar}{\end{eqnarray*}}
\newcommand{\bc}{\begin{center}}
\newcommand{\ec}{\end{center}}

\newtheorem{thm}{Theorem}[section]
\newtheorem{cor}[thm]{Corollary}
\newtheorem{lem}[thm]{Lemma}
\newtheorem{sublem}[thm]{Sublemma}
\newtheorem{prop}[thm]{Proposition}

\theoremstyle{definition}
\newtheorem{defn}[thm]{Definition}
\newtheorem{rem}[thm]{Remark}

\newtheorem{exm}[thm]{Example}
\newtheorem{conds}[thm]{Condition}
\newtheorem{prob}[thm]{Problem}

\newtheorem{lem-def}[thm]{Lemma-Definition}

\newtheorem*{thm*}{Theorem}

\numberwithin{equation}{section}

\def\R{{\mathbb R}}

\def\E{{\mathbb E}}
\def\Z{{\mathbb Z}}
\def\C{{\mathbb C}}
\def\R{{\mathbb R}}
\def\P{{\mathbb P}}

\def\N{{\mathbb N}}

\def\11{{\mathbb I}}

\def\delbar{{\overline \partial}}

\def\C{\mathbb{C}}
\def\Z{\mathbb{Z}}

\def\T{\mathbb{T}}

\def\Q{\mathbb{Q}}

\def\E{\ifmmode{\mathbb E}\else{$\mathbb E$}\fi} 
\def\N{\ifmmode{\mathbb N}\else{$\mathbb N$}\fi} 
\def\R{\ifmmode{\mathbb R}\else{$\mathbb R$}\fi} 
\def\Q{\ifmmode{\mathbb Q}\else{$\mathbb Q$}\fi} 
\def\C{\ifmmode{\mathbb C}\else{$\mathbb C$}\fi} 
\def\H{\ifmmode{\mathbb H}\else{$\mathbb H$}\fi} 
\def\Z{\ifmmode{\mathbb Z}\else{$\mathbb Z$}\fi} 
\def\P{\ifmmode{\mathbb P}\else{$\mathbb P$}\fi} 
\def\SS{\ifmmode{\mathbb S}\else{$\mathbb S$}\fi} 
\def\DD{\ifmmode{\mathbb D}\else{$\mathbb D$}\fi} 

\def\R{{\mathbb R}}

\def\E{{\mathbb E}}
\def\Z{{\mathbb Z}}
\def\C{{\mathbb C}}
\def\R{{\mathbb R}}

\def\N{{\mathbb N}}

\def\delbar{{\overline \partial}}

\def\CJ{{\mathcal J}}

\def\CM{{\mathcal M}}

\def\CP{{\mathcal P}}

\def\CP{{\mathcal P}}

\def\darr#1{\raise1.5ex\hbox{$\leftrightarrow$}
\mkern-16.5mu #1}

\def\roughly#1{\raise.3ex\hbox{$#1$\kern-.75em
\lower1ex\hbox{$\sim$}}}

\def\opname#1{\mathop{\kern0pt{\rm #1}}\nolimits}

\def\dim{\opname{dim}}

\begin{document}
\quad \vskip1.375truein

\def\mq{\mathfrak{q}}
\def\mp{\mathfrak{p}}
\def\mH{\mathfrak{H}}
\def\mh{\mathfrak{h}}
\def\ma{\mathfrak{a}}
\def\ms{\mathfrak{s}}
\def\mm{\mathfrak{m}}
\def\mn{\mathfrak{n}}
\def\mz{\mathfrak{z}}
\def\mw{\mathfrak{w}}
\def\Hoch{{\tt Hoch}}
\def\mt{\mathfrak{t}}
\def\ml{\mathfrak{l}}
\def\mT{\mathfrak{T}}
\def\mL{\mathfrak{L}}
\def\mg{\mathfrak{g}}
\def\md{\mathfrak{d}}
\def\mr{\mathfrak{r}}

\begin{abstract}
The main purpose of the present paper is a study of orientations of the
moduli spaces of pseudo-holomorphic discs with boundary lying on
a \emph{real} Lagrangian submanifold, i.e., the fixed point set of an
anti-symplectic involution $\tau$ on a symplectic manifold.
We introduce the notion of $\tau$-relative spin structure for an anti-symplectic involution $\tau$,
and study how the orientations on the moduli space behave under the involution $\tau$.
We also apply this to the study of Lagrangian Floer theory of real Lagrangian submanifolds.
In particular, we study unobstructedness of the $\tau$-fixed point set
of symplectic manifolds and in particular prove its unobstructedness in the case of  Calabi-Yau manifolds.
And we also do explicit calculation of Floer cohomology of $\R P^{2n+1}$
over $\Lambda_{0,{\rm nov}}^{\Z}$ which provides an example whose Floer cohomology is not
isomorphic to its classical cohomology.
We study Floer cohomology of the diagonal of the square of a symplectic
manifold, which leads to a rigorous construction of the quantum Massey product of
symplectic manifold in complete generality.
\end{abstract}

\date{Nov. 25, 2015}
\keywords{anti-symplectic involution, orientation, Floer cohomology, unobstructed
Lagrangian submanifolds, quantum cohomology}

\maketitle

\tableofcontents
\section{Introduction and statement of results}
\label{sec:introduction}

An {\it anti-symplectic involution} $\tau$ on
a symplectic manifold $(M,\omega)$ is an involution
on $M$ which satisfies $\tau^{\ast} \omega = -\omega$.
Two prototypes of anti-symplectic involutions are the complex
conjugation of a complex projective space with respect to the
Fubini-Study metric and the canonical reflection along the zero
section on the cotangent bundle. (See also \cite{CMS} for a construction of
an interesting class of anti-symplectic involutions on Lagrangian torus
fibrations.) The fixed point set of $\tau$, if it is non-empty, gives
an example of Lagrangian submanifolds. For instance, the set
of real points of a complex projective manifold defined over $\R$ belongs to this class.
In this paper we study Lagrangian intersection Floer theory for
the fixed point set of an anti-symplectic involution.
\par
Let $(M,\omega)$ be a compact, or more generally tame,
$2n$-dimensional symplectic manifold and $L$ an oriented closed Lagrangian submanifold of $M$.
It is well-known by now that the Floer cohomology of a
Lagrangian submanifold $L$ can not be defined in general.
The phenomenon of bubbling-off holomorphic discs is the main source of troubles in defining Floer cohomology of Lagrangian submanifolds.
In our books \cite{fooobook1} and \cite{fooobook2}, we developed general theory
of obstructions and deformations of Lagrangian intersection Floer cohomology based on the theory of filtered $A_{\infty}$ algebras
which we associate to each Lagrangian submanifold.
However it is generally very hard to formulate the criterion for unobstructedness to defining Floer cohomology let alone
to calculate Floer cohomology for a given Lagrangian submanifold.
In this regard, Lagrangian torus fibers in toric manifolds
provide good test cases for these problems, which we
studied in \cite{foootoric1}, \cite{foootoric2} in detail.
For this class of Lagrangian submanifolds,
we can do many explicit calculations of various notions and invariants
that are introduced in the books \cite{fooobook1} and \cite{fooobook2}.

Another important class of Lagrangian submanifolds is that of
the fixed point set of an anti-symplectic involution.
Actually, the set of real points in Calabi-Yau manifolds
plays an important role of homological mirror symmetry conjecture.
(See \cite{Wal}, \cite{PSW} and \cite {fukaya;counting}.
See also \cite{Wel} for related topics of real points.)
The purpose of the present paper is to study Floer cohomology of this class of Lagrangian submanifolds.
For example, we prove unobstructedness for such Lagrangian submanifolds
in Calabi-Yau manifolds and also provide some other examples of explicit calculations of Floer cohomology.
The main ingredient of this paper is a careful study of
orientations of the moduli spaces of pseudo-holomorphic discs.
\par
Take an $\omega$-compatible almost complex structure $J$ on $(M,\omega)$.
We consider
moduli space $\CM(J;\beta)$ of $J$-holomorphic stable maps from bordered Riemann surface
$(\Sigma, \partial \Sigma)$ of genus $0$ to $(M,L)$
which represents a class $\beta \in \Pi(L)=\pi_2(M,L)/\sim$: $\beta \sim \beta' \in \pi_2(M,L)$ if and only if
$
\omega(\beta)=\omega(\beta')$
and
$
\mu_L(\beta) =\mu_L(\beta')$.
Here $\mu_L : \pi_2(M,L) \to \Z$ is the Maslov index homomorphism. The values of $\mu_L$ are even integers if $L$ is oriented.
When the domain $\Sigma$ is a $2$-disc $D^2$,
we denote by $\CM ^{\text{\rm reg}}(J;\beta)$
the subset of $\CM(J;\beta)$ consisting of \emph{smooth} maps, that is, psedo-holomorphic maps from disc
without disc or sphere bubbles.
The moduli space $\CM(J;\beta)$  has a Kuranishi structure, see Proposition 7.1.1 \cite{fooobook2}.
However it in not orientable in the sense of Kuranishi
structure in general.
In Chapter 8 \cite{fooobook2} we introduce the notion of relative spin structure on $L\subset M$ and its stable conjugacy class,
and prove that
if $L$ carries a relative spin structure
$(V,\sigma)$, its stable conjugacy class
$[(V,\sigma)]$ determines an orientation
on the moduli space $\CM(J;\beta)$
(see Sections \ref{sec:pre}, \ref{sec:relspin} for the precise definitions and notations.)
We denote it by $\CM(J;\beta)^{[(V,\sigma)]}$
when we want to specify the stable conjugacy class of the relative spin structure.
If we have a diffeomorphism $f : M \to M$ satisfying $f(L)=L$, we
can define the pull-back $f^{\ast}[(V,\sigma)]$ of the relative spin structure.
(See also Subsection \ref{subsec:relspin}.)

Now we consider the case that $\tau : M \to M$ is an anti-symplectic involution
and
$$L = {\rm Fix} ~\tau.$$
We assume $L$ is nonempty, oriented and relatively spin.
Take an $\omega$-compatible almost complex structure $J$
satisfying $\tau^{\ast} J =-J$. We call such $J$ {\it $\tau$-anti-invariant}.
Then we find that $\tau$ induces a map
$$\tau_{\ast} : {\CM}^{\text{\rm reg}}(J;\beta)\longrightarrow {\CM}^{\text{\rm reg}}(J;\beta)$$
which satisfies $\tau_{\ast} \circ \tau_{\ast} = {\rm Id}$.
(See Definition \ref{def:inducedtau} and Lemma \ref{Lemma38.6}.)
Here we note that $\tau_{\ast}(\beta)=\beta$ in $\Pi(L)$
(see Remark \ref{rem:tau}).
We pick a conjugacy class of relative spin structure $[(V,\sigma)]$ and consider the pull back
$\tau^{\ast}[(V,\sigma)]$.
Then we have an induced map
$$\tau_{\ast} : {\CM}^{\text{\rm reg}}(J;\beta)^{\tau^{\ast}[(V,\sigma)]}
\longrightarrow {\CM}^{\text{\rm reg}}(J;\beta)^{[(V,\sigma)]}.$$
We will prove in Proposition \ref{regtau}  that $\tau_{\ast}$ is induced by an automorphism  of
${\CM}^{\text{\rm reg}}(J;\beta)$ as a space with Kuranishi structure, see Definition \ref{def:auto}.
The definition for an automorphism to be orientation preserving
in the sense of Kuranishi structure is given in Definition \ref{def:oripres}.
The first problem we study is the question whether $\tau_{\ast}$ respects the orientation or not.
The following theorem plays a fundamental role in this paper.

\begin{thm}[Theorem \ref{Proposition38.7}]\label{thm:fund}
Let $L$ be a fixed point set of an anti-symplectic involution $\tau$ on $(M, \omega)$ and
$J$ a $\tau$-anti-invariant almost complex structure
compatible with $\omega$.
Suppose that $L$ is oriented and carries a relative spin structure $(V,\sigma)$.
Then the map
$
\tau_*: {\CM}^{\operatorname{reg}}(J;\beta)^{\tau^{\ast}[(V,\sigma)]} \to
{\CM}^{\operatorname{reg}}(J;\beta)^{[(V,\sigma)]}
$
is orientation preserving if $\mu_L(\beta) \equiv 0
\mod 4$ and is orientation reversing if
$\mu_L(\beta) \equiv 2
\mod 4$.
\end{thm}

\begin{rem}\label{rem;thm1}
If $L$ has a {\it $\tau$-relative spin structure}
(see Definition \ref{Definition44.17}),
then
$${\CM}^{\text{\rm reg}}(J;\beta)^{\tau^{\ast}[(V,\sigma)]}
={\CM}^{\text{\rm reg}}(J;\beta)^{[(V,\sigma)]}$$ as spaces with oriented Kuranishi structures.
Corollary \ref{corProposition38.7} is nothing but this case.
If $L$ is spin, then it is automatically $\tau$-relatively spin
(see Example \ref{Remark44.18}).
Later in
Proposition \ref{Proposition44.19} we show that there is an example of Lagrangian submanifold
$L$ which is relatively spin but not $\tau$-relatively spin.
\end{rem}

Including marked points, we consider the
moduli space $\CM_{k+1, m}(J;\beta)$
of $J$-holomorphic stable maps to $(M,L)$ from a bordered Riemann surface
$(\Sigma, \partial \Sigma)$ in class $\beta \in \Pi(L)$ of genus $0$ with
$(k+1)$ boundary marked points
and $m$ interior marked points.
The anti-symplectic involution $\tau$ also induces
a map $\tau_{\ast}$ on the moduli space of $J$-holomorphic maps with
marked points. See Theorem \ref{Proposition38.11}.
Then we have:

\begin{thm}[Theorem \ref{Proposition38.11}]\label{withmarked}
The induced map
$$\tau_* : \CM_{k+1,m}(J;\beta)^{\tau^{\ast}[(V,\sigma)]} \to
\CM_{k+1,m}(J;\beta)^{[(V,\sigma)]}
$$ is orientation preserving if and only if $\mu_L(\beta)/2 + k + 1 + m$
is even.
\end{thm}

When we construct the filtered $A_{\infty}$ algebra
$(C(L,\Lambda_{0,{\rm nov}}), \mathfrak m)$ associated to a relatively spin Lagrangian submanifold $L$,
we use the component of $\CM_{k+1}(J;\beta)$ consisting of the elements
whose boundary marked points lie in counter-clockwise cyclic
order on $\partial \Sigma$. We also involve interior marked points.
For the case of $(k+1)$ boundary marked points
on $\partial \Sigma$ and $m$ interior marked points in
${\rm Int} ~\Sigma$,
we denote the corresponding component by
$\CM_{k+1, m}^{\rm main}(J;\beta)$ and call it the {\it main component}.
Moreover, we consider the moduli space
$\CM_{k+1,m}^{\rm main}(J;\beta;P_1, \ldots , P_k)$
which is defined by taking a fiber product of $\CM_{k+1,m}^{\rm main}(J;\beta)$
with smooth singular simplices $P_1, \ldots , P_k$ of $L$.
(This is nothing but the main component of (\ref{withP}) with $m=0$.)
A stable conjugacy class of a relative spin structure determines orientations on
these spaces as well.
See Sections \ref{sec:pre} and \ref{sec:relspin} for
the definitions and a precise description of their orientations.
Here we should note that $\tau_{\ast}$ above does {\it not} preserve
the cyclic ordering of boundary marked points and so
it does not preserve the main component.
However, we can define the maps denoted by
$$
\tau_{\ast}^{\rm main} : \CM_{k+1,m}^{\rm main}(J;\beta)^{\tau^*[(V,\sigma)]} \to
\CM_{k+1,m}^{\rm main}(J;\beta)^{[(V,\sigma)]}
$$
and
\begin{equation}\label{eq:taumain}
\tau_*^{\operatorname{main}} :
\CM^{\operatorname{main}}_{k+1,m}(J;\beta;P_1,\dots,P_k)^{\tau^*[(V,\sigma)]}
\to
\CM^{\operatorname{main}}_{k+1,m}(J; \beta ; P_k,\dots,P_1)^{[(V,\sigma)]}
\end{equation}
respectively.
See (\ref{38.13}), (\ref{taumain}) and (\ref{38.16})
for the definitions.
We put
$\deg ' P= \deg P -1$ which is the shifted degree of $P$
as a singular cochain of $L$ (i.e., $\deg P= \dim L-\dim P$.)
Then we show the following:

\begin{thm}[Theorem \ref{Lemma38.17}]\label{withsimplex}
Denote
$$
\epsilon = \frac{\mu_L(\beta)}{2} + k + 1 + m + \sum_{1 \le i < j \le k} \deg'P_i\deg'P_j.
$$
Then the map induced by the involution $\tau$
$$
\tau_*^{\operatorname{main}} :
\CM^{\operatorname{main}}_{k+1,m}(J;\beta;P_1,\dots,P_k)^{\tau^*[(V,\sigma)]}
\longrightarrow
\CM^{\operatorname{main}}_{k+1,m}(J; \beta ; P_k,\dots,P_1)^{[(V,\sigma)]}
$$
is orientation preserving if $\epsilon$ is even, and
orientation reversing if $\epsilon$ is odd.
\end{thm}
See Theorem \ref{Lemma38.17withQ}
for a more general statement involving the fiber product with singular
simplicies $Q_j$ $(j=1,\ldots ,m)$ of $M$.
\par\medskip
These results give rise to some non-trivial applications
to Lagrangian intersection Floer theory for the case $L={\rm Fix}~\tau$.
We briefly describe some consequences in the rest of this section.
\par
In the books \cite{fooobook1} and \cite{fooobook2}, using the moduli spaces
$\CM^{\operatorname{main}}_{k+1}(J;\beta;P_1,\dots,P_k)$,
we construct a filtered $A_{\infty}$ algebra
$(C(L;\Lambda_{0,{\rm nov}}^{\Q}), {\mathfrak m})$ with $\mathfrak m =\{{\mathfrak m_k}\}_{k=0,1,2,\ldots}$
for any relatively spin closed Lagrangian submanifold $L$
of $(M,\omega)$ (see Theorem \ref{thm:Ainfty}) and
developed the obstruction and deformation theory of
Lagrangian intersection Floer cohomology. Here $\Lambda_{0,{\rm nov}}^{\Q}$ is
the universal Novikov ring over $\Q$ (see (\ref{eq:nov})).
In particular, we formulate the unobstructedness to
defining Floer cohomology of $L$ as the existence of
solutions of the Maurer-Cartan equation
for the filtered $A_{\infty}$ algebra (see Definition \ref{def:boundingcochain}).
We denote the set of such solutions by $\CM(L;\Lambda_{0,{\rm nov}}^\Q)$.
By definition, when $\CM(L;\Lambda_{0,{\rm nov}}^{\Q}) \ne \emptyset$,
we can use any element $b \in \CM(L;\Lambda_{0,{\rm nov}}^{\Q})$ to deform the Floer's `boundary' map
and define a deformed Floer cohomology $HF((L,b),(L,b);\Lambda_{0,{\rm nov}}^{\Q})$.
See Subsection \ref{subsec:Ainfty} for a short review of this process.
Now for the case $L={\rm Fix}~\tau$,
Theorem \ref{withsimplex} yields
the following particular property of the filtered
$A_{\infty}$ algebra:

\begin{thm}\label{Theorem34.20} Let $M$ be a compact, or tame, symplectic manifold and
$\tau$ an anti-symplectic involution. If $L = \text{\rm Fix}~\tau$
is non-empty, compact, oriented and $\tau$-relatively spin, then the
filtered $A_{\infty}$ algebra $(C(L;\Lambda_{0,{\rm nov}}^{\Bbb
Q}),\mathfrak m)$ can be chosen so that
\begin{equation}\label{34.21}
\mathfrak m_{k,\beta}(P_1,\dots,P_k)
= (-1)^{\epsilon_1} \mathfrak m_{k,\tau_*\beta}(P_k,\dots,P_1)
\end{equation}
where
$$
\epsilon_1 = \frac{\mu_L(\beta)}{2} + k+ 1 + \sum_{1\le i < j \le k}
\deg'P_i\deg'P_j.
$$
\end{thm}
Using the results from  \cite{fooobook1} and \cite{fooobook2}, we derive that
Theorem \ref{Theorem34.20} implies
unobstructedness
of $L={\rm Fix}~\tau$ in the following cases:

\begin{cor}\label{Corollary34.22}
Let $\tau$ and $L = \text{\rm Fix}~\tau$ be as in Theorem \ref{Theorem34.20}. In addition,
we assume either
\begin{enumerate}
\item
$c_1(TM)\vert_{\pi_2(M)} \equiv 0 \mod 4$,
or
\item $c_1(TM)\vert_{\pi_2(M)} \equiv 0 \mod 2$ and
$i_{\ast} : \pi_1(L) \to \pi_1(M)$ is injective.
(Here $i_{\ast}$ is the natural map induced by inclusion $i : L \to M$.)
\end{enumerate}
Then
$L$ is unobstructed over $\Lambda_{0,{\rm nov}}^{\Bbb Q}$
(i.e., $\CM(L;\Lambda_{0,{\rm nov}}^\Q)\ne \emptyset$)
and so $HF((L,b),(L,b);\Lambda_{0,{\rm nov}}^{\Q})$ is defined
for any $b \in \CM(L;\Lambda_{0,{\rm nov}}^\Q)$.
Moreover, we may choose $b \in \CM(L;\Lambda_{0,{\rm nov}}^\Q)$ so that the map
$$
\aligned
(-1)^{k(\ell+1)} {(\mathfrak m_2)}_* & : HF^k((L,b),(L,b);\Lambda_{0,{\rm nov}}^\Q)
\otimes HF^{\ell}((L,b),(L,b);\Lambda_{0,{\rm nov}}^\Q) \\
& \quad \longrightarrow HF^{k+\ell}((L,b),(L,b);\Lambda_{0,{\rm nov}}^\Q)
\endaligned
$$
induces a graded commutative product.
\end{cor}

\begin{rem} \label{Remark34.24}
By symmetrizing the filtered $A_{\infty}$ structure
$\mathfrak m_k$ of $(C(L,\Lambda_{0,{\rm nov}}^{\Bbb
Q}),\mathfrak m)$,
we obtain a filtered $L_{\infty}$ algebra
$(C(L,\Lambda_{0,{\rm nov}}^{\Bbb Q}),\mathfrak l)
=(C(L,\Lambda_{0,{\rm nov}}^{\Bbb Q}),\{\mathfrak l_k\}_{k=0,1,2,\ldots})$.
See Section A3 \cite{fooobook2} for the definitions of the symmetrization
and of the filtered $L_{\infty}$ structure.
In the situation of Corollary \ref{Corollary34.22},
the same proof shows that we have
$\mathfrak l_k = \overline{\mathfrak l}_k\otimes \Lambda_{0,{\rm nov}}^{\Bbb
Q}$ if $k$ is even. Here $\overline{\mathfrak l}_k$ is the (unfiltered) $L_{\infty}$ structure
obtained by the reduction of the coefficient of
$(C(L,\Lambda_{0,{\rm nov}}^\Q),\mathfrak l)$ to $\Q$.
Note that {\it over} $\R$ we may choose $\overline{\mathfrak l}_k = 0$
for $k\geq 3$ by Theorem X in Chapter 1 \cite{fooobook1}. On the other hand,
Theorem A3.19 \cite{fooobook2} shows that
$\overline{\mathfrak l}_k =0$
for $H(L; \Q)$.
\end{rem}

We note that we do not assert that Floer cohomology $HF((L,b),(L,b);
\Lambda_{0,{\rm nov}}^\Q)$ is isomorphic to $H^*(L;\Bbb Q) \otimes
\Lambda_{0,{\rm nov}}^\Q$, in general.
(Namely, we do not assert $\mathfrak m_1 =
\overline{\mathfrak m}_{1}\otimes \Lambda_{0,{\rm nov}}^{\Q}$.)
Indeed, we show in Subsection \ref{subsec:Appl2} that
for the case $L=\R P^{2n+1}$ in $\C P^{2n+1}$ the Floer cohomology group
is {\it not} isomorphic to the classical cohomology group.
(See Theorem \ref{Theorem44.24}.)

Moreover, if we assume $c_1(TM)\vert_{\pi_2(M)} = 0$ in addition,
we can show the following  non-vanishing theorem of Floer
cohomology:

\begin{cor}\label{TheoremN}
Let $\tau$, $L = \text{\rm Fix}~\tau$ be as in Theorem \ref{Theorem34.20}.
Assume $c_1(TM)\vert_{\pi_2(M)} = 0$. Then
$L$ is unobstructed over $\Lambda_{0,{\rm nov}}^{\Bbb Q}$
and
$$
HF((L,b),(L,b);\Lambda_{{\rm nov}}^{\Q})\ne 0
$$
for any $b \in \CM(L;\Lambda_{0,{\rm nov}}^\Q)$.
In particular, we have
$$
\psi (L) \cap L \ne \emptyset
$$
for any Hamiltonian diffeomorphism $\psi : M \to M$.
\end{cor}

Theorem \ref{Theorem34.20} and Corollaries \ref{Corollary34.22},
\ref{TheoremN}
can be applied to
the real point set $L$ of any Calabi-Yau manifold (defined over
$\Bbb R)$ if it is oriented and $\tau$-relative spin.

\par\medskip
Another application of Theorem \ref{Theorem34.20} and Corollary \ref{Corollary34.22} is a ring isomorphism between
quantum cohomology and Lagrangian Floer cohomology for
the case of the diagonal
of square of a symplectic manifold. Let $(N,\omega_N)$ be a closed symplectic manifold.
We consider the product
$$
(M,\omega_M) = (N\times N, -{\rm pr}_1^*\omega_N + {\rm pr}_2^* \omega_N),
$$
where ${\rm pr}_i$ is the projection to the $i$-th factor.
The involution $\tau : M \to M$ defined by
$\tau(x,y) = (y,x)$ is
anti-symplectic and its fixed point set $L$ is the diagonal
$$
\Delta_N= \{(x,x) \mid x \in N\} \cong N.
$$
As we will see in the proof of Theorem
\ref{Proposition34.25}, the diagonal set is always  unobstructed. Moreover,
we note that the natural map $i_*: H_*(\Delta_N, \Q) \to H_*(N \times N;\Q)$
is injective and so the spectral sequence constructed
in Chapter 6 \cite{fooobook1} collapses at $E_2$-term
by Theorem D (D.3) \cite{fooobook1}, which in turn induces the natural isomorphism
$H(N;\Bbb Q) \otimes \Lambda_{0,{\rm nov}} \cong HF((L,b),(L,b);\Lambda_{0,{\rm nov}}^{\Q})$
for any $b \in\CM(L;\Lambda_{0,{\rm nov}}^{\Q})$.
We prove in the proof of
Theorem \ref{Proposition34.25} in Subsection \ref{subsec:Appl1} that $\mathfrak m_2$ also
derives a graded commutative product
$$
\cup_Q : HF((L,b),(L,b);\Lambda_{0,{\rm nov}}^\Q) \otimes
HF((L,b),(L,b);\Lambda_{0,{\rm nov}}^\Q) \to HF((L,b),(L,b);\Lambda_{0,{\rm nov}}^\Q).
$$
In fact, we can prove that the
following stronger statement.

\begin{thm}\label{Proposition34.25}
Let $(N, \omega_N)$ be a closed symplectic manifold.
\begin{enumerate}
\item
The diagonal set of
$(N\times N, -{\rm pr}_1^* \omega_N + {\rm pr}_2^*\omega_N)$ is unobstructed
over $\Lambda_{0,{\rm nov}}^{\Q}$.
\item
There exists a bounding cochain $b$ such that product $\cup_Q$ coincides with the quantum
cup product on $(N,\omega_N)$ under the natural isomorphism
$$HF((L,b),(L,b);\Lambda_{0,{\rm nov}}^{\Q}) \cong H(N;\Bbb Q) \otimes
\Lambda_{0,{\rm nov}}^{\Q}.$$
\end{enumerate}
\end{thm}
\par
If we use Corollary 3.8.43 \cite{fooobook1},
we can easily find that the diagonal set is
{\it weakly unobstructed} in the sense of Definition 3.6.29 \cite{fooobook1}. See also Remark \ref{Remark44.25}.
We also note that for the case of diagonals, $\mathfrak m_k$ $(k\ge 3)$ define
a quantum (higher) Massey product.
It was discussed formally in \cite{fukaya;mhtpy}.
We have made it rigorous here:
\begin{cor}\label{qMassey}
For any closed symplectic manifold $(N, \omega_N)$,
there exists a filtered $A_{\infty}$ structure
$\mathfrak m_k$
on $H(N;\Lambda_{0,{\rm nov}}^{\Q})=H(N;\Bbb Q) \otimes
\Lambda_{0,{\rm nov}}^{\Q}$ such that

\begin{enumerate}
\item $\mathfrak m_0 = \mathfrak m_1 =0$;
\item $\cup_Q$ defined by (\ref{qproduct}) using
$\mathfrak m_2$ coincides with the quantum cup product;
\item the $\R$-reduction
$(H(N;\Q), \overline{\mathfrak m})\otimes_{\Q} \R$
of the filtered $A_{\infty}$ algebra is homotopy equivalent to the de Rham complex of $N$ as an $A_{\infty}$ algebra,
where $(H(N;\Q), \overline{\mathfrak m})$ is
the reduction of the coefficient of $(H(N;\Lambda_{0,{\rm nov}}^{\Q}),\mathfrak m)$
to $\Q$.
\end{enumerate}
\end{cor}
\par\smallskip
The paper is organized as follows:
In Section \ref{sec:pre}, we briefly recall some basic material
on the moduli space of stable maps from a bordered Riemann surface
of genus $0$.
In Section \ref{sec:relspin},
we also recall from \cite{fooobook2} the notion of relative spin structure, its stable
conjugacy class
and the orientation of the moduli space of pseudo-holomorphic discs.
We describe how the stable conjugacy class of relative spin structure determines an orientation on the moduli space.
We introduce here the notion of $\tau$-relative spin structure
for an anti-symplectic involution $\tau : M \to M$, and
also give some examples which are relatively spin but not
$\tau$-relatively spin Lagrangian submanifolds.
In Section \ref{sec:inducedtau}, we define the map
$\tau_{\ast}$ on the moduli space induced by
$\tau$ and study how the induced map $\tau_{\ast}$
on various moduli spaces
changes or preserves the orientations.
Assuming Theorem \ref{thm:fund} holds, we prove Theorem \ref{withmarked} in this section.
The fundamental theorems Theorem \ref{thm:fund}
and Theorem \ref{withsimplex} are proved in
Section \ref{sec:Proofth}.
Section \ref{sec:Appl} is devoted to
various applications of the results obtained above
to Lagrangian Floer cohomology.
After a short review of the general story of Lagrangian intersection
Floer theory laid out in \cite{fooobook1} and \cite{fooobook2}, we prove
Theorem \ref{Theorem34.20}, Corollary \ref{Corollary34.22},
Corollary \ref{TheoremN} in Subsection \ref{subsec:Appl1}.
Subsection \ref{proof1.9}  is devoted to the proofs of Theorem \ref{Proposition34.25} and Corollary \ref{qMassey}.
In particular, we introduce stable maps with admissible system of circles and study their moduli spaces in Subsection \ref{6.4}.   
In Subsection \ref{subsec:Appl2},
we calculate Floer cohomology of $\R P^{2n+1}$ over $\Lambda_{0,{\rm nov}}^{\Z}$
coefficients by studying orientations in detail.
The calculation shows that the Floer cohomology of $\R P^{2n+1}$
over $\Lambda_{0,{\rm nov}}^{\Z}$ is not isomorphic to
the usual cohomology.
This result contrasts with Oh's earlier calculation \cite{Oh93} of the Floer
cohomology of real projective spaces over $\Z_2$ coefficients,
where the Floer cohomology is isomorphic to the usual cohomology
over $\Z_2$. In the first two subsections of Appendix, we briefly recall form \cite {fooobook2} the definition of orientation on the space with Kuranishi structure
and the notion of group action on a space with Kuranishi structure.
In the third subsection of Appendix, we present how to promote filtered $A_{n,K}$-structures keeping the invariance under the involution.
\par
\medskip
Originally, the content of this paper appeared as a part of Chapter 8
in the preprint version \cite{fooo06} of the books
\cite{fooobook1}, \cite{fooobook2}, and
was intended to be published in a part of the book. However, due to the publisher's page
restriction on the AMS/IP Advanced Math Series, we took out two chapters, Chapter 8 and Chapter 10 from
the preprint \cite{fooo06} and published the book without
those two chapters.
The content of this paper is based on the parts extracted from
Chapter 8 (Floer theory of Lagrangian submanifolds over $\Z$)
and Chapter 9 (Orientation) in the preprint \cite{fooo06}.
We also note that this is a part of the paper cited as [FOOO09I] in the books \cite{fooobook1}, \cite{fooobook2}.
The half of the remaining part of Chapter 8 in \cite{fooo06} is published as
\cite{fooointeger}.
\par
The authors would like to thank Cheol-Hyun Cho for some helpful discussion
concerning Theorem \ref{Proposition34.25}. They also thank
Mohammad Tehrani, Aleksey Zinger and also the referee for pointing out an error
in Lemma \ref{oripreversing}, Lemma \ref{Lemma44.28}
in the first draft, respectively, and thank anonymous referees for their careful reading and for their comments
that help to improve presentation of the paper.
\par
\par
Kenji Fukaya is supported partially by JSPS Grant-in-Aid for Scientific Research
No. 23224002 and NSF grant \# 1406423, Yong-Geun
Oh by IBS project \# IBS-R003-D1 in Korea and by US NSF grant \#  0904197, 
Hiroshi Ohta by JSPS Grant-in-Aid
for Scientific Research Nos.19340017, 23340015, 15H02054 and Kaoru Ono by JSPS Grant-in-Aid for
Scientific Research, Nos. 18340014, 21244002, 26247006.

\section{Preliminaries}
\label{sec:pre}

In this section, we prepare some basic notations we use in this paper.
We refer Section 2.1 in \cite {fooobook1} and Section A1 in \cite{fooobook2} for more detailed explanation
of moduli spaces and the notion of Kuranishi structure, respectively.
Let $L$ be an oriented compact Lagrangian submanifold
of $(M,\omega)$. Take an $\omega$-compatible almost complex structure $J$ on $M$.
We recall Definition 2.4.17 in \cite{fooobook1} where we introduce
the relation $\sim$ in $\pi_2(M,L)$:
We define $\beta \sim \beta'$ in $\pi_2(M,L)$ if and only if
$
\omega(\beta)=\omega(\beta')$
and
$
\mu_L(\beta) =\mu_L(\beta').
$
We denote the quotient group by
\begin{equation}\label{eq:Pi}
\Pi(L) =\pi_2(M,L)/\sim.
\end{equation}
This is an abelian group.
Let $\beta \in \Pi(L)$.
{\it A stable map from a bordered Riemann surface of genus zero with
$(k+1)$ boundary marked points and $m$ interior marked points}
is a pair $((\Sigma, \vec{z}, {\vec {z}}^+), w)=((\Sigma, z_0, \dots , z_k, z_1^{+}, \dots, z_{m}^+), w)$
such that $(\Sigma, \vec{z}, \vec{z}^+)$ is a bordered
semi-stable curve of genus zero with $(k+1)$ boundary
marked points and $m$ interior marked points
and $w:(\Sigma, \partial \Sigma) \to (M,L)$ is
a $J$-holomorphic map such that its automorphism group, i.e.,
the set of biholomorphic maps $\psi : \Sigma \to \Sigma$
satisfying $\psi(z_i)=z_i, \psi(\vec{z_i}\,^+) =\vec{z_i}\,^+$
and $w\circ \psi =w$ is finite.
We say that $((\Sigma, \vec{z}, {\vec {z}}\,^+), w)$ is isomorphic to
$((\Sigma', \vec{z}\,', {\vec {z}}\,^{+\prime}), w')$,
if there exists a bi-holomorphic map
$\psi : \Sigma \to \Sigma'$ satisfying
$\psi(z_i)=z_i', \psi(\vec{z_i}\,^{+}) =\vec{z_i}\,^{+\prime}$
and $w'\circ \psi =w$. We denote by $\CM_{k+1, m}(J;\beta)$
the set of the isomorphism classes of stable maps in class $\beta \in \Pi(L)$ from a bordered
Riemann surface of genus zero with $(k+1)$ boundary
marked points and $m$ interior marked points.
When the domain curve $\Sigma$ is a smooth 2-disc $D^2$, we
denote the corresponding subset by
$\CM^{\rm reg}_{k+1,m}(J;\beta)$.
We note that $\CM_{k+1, m}(J;\beta)$ is a compactification of
$\CM^{\rm reg}_{k+1,m}(J;\beta)$.
The virtual real dimension is
$$
\dim_{\R} \CM_{k+1, m}(J;\beta) =
n + \mu_L(\beta)+k+1 +2m -3,
$$
where $n=\dim L$ and $\mu_L(\beta)$ is the Maslov index which
is an even integer for an oriented Lagrangian submanifold $L$.
When we do not consider interior marked points,
we denote them by $\CM_{k+1}(J;\beta)$,
$\CM_{k+1}^{\rm reg}(J;\beta)$, and
when we do not consider any marked points, we simply denote
them by $\CM(J;\beta)$, $\CM^{\rm reg}(J;\beta)$ respectively.
Furthermore, we define a component $ \CM^{\rm main}_{k+1,m}(J;\beta) $ of $\CM_{k+1,m}(J;\beta)$
by
$$
\aligned
 \CM^{\rm main}_{k+1,m}(J;\beta)
=\{
((& \Sigma, \vec{z}, {\vec {z}}\,^+), w) \in  \CM_{k+1, m}(J;\beta)
 ~\vert~ \\
& {\text{$(z_0,z_1,\dots , z_k)$
is in counter-clockwise cyclic order on $\partial \Sigma $}}
\},
\endaligned$$
which we call the {\it main component}.
We define $\CM_{k+1,m}^{\rm main, reg}(J;\beta)$,
$\CM_{k+1}^{\rm main}(J;\beta)$ and
$\CM_{k+1}^{\rm main, reg}(J;\beta)$
in a similar manner.
\par
We have a Kuranishi structure on $\CM_{k+1, m}(J;\beta)$
so that the evaluation maps
\begin{equation}\label{eq:ev}
\aligned
ev_i ~ & :~ \CM_{k+1, m}(J;\beta) \longrightarrow
L, \quad i=0,1,\dots ,k\\
ev_j^+ ~ & :~\CM_{k+1, m}(J;\beta) \longrightarrow
M, \quad j=1,\dots ,m
\endaligned
\end{equation}
defined by $ev_i((\Sigma, \vec{z}, {\vec {z}}\,^+), w))=w(z_i)$
and
$ev_j^+((\Sigma, \vec{z}, {\vec {z}}\,^+), w))=w(z_j^+)$
are weakly submersive.
(See Section 5 \cite{FO} and Section A1.1 \cite{fooobook2} for
the definitions of Kuranishi structure
and weakly submersive map.)
Then for given smooth singular simplicies $(f_i : P_i \to L)$ of $L$ and
$(g_j : Q_j \to M)$ of $M$, we can define the fiber product in the sense of Kuranishi structure:
\begin{equation}\label{withsimplicies}
\aligned
& \CM_{k+1,m}(J;\beta; \vec{Q},\vec{P}) \\
:= & \CM_{k+1,m}(J;\beta) {}_{(ev_1^+, \dots, ev_m^+,ev_1,\dots,ev_k)} \times _{g_1 \times \dots \times f_k}\left(\prod_{j=1}^{m} Q_j\times
\prod_{i=1} ^{k}P_i \right).
\endaligned
\end{equation}
See Section A1.2 \cite{fooobook2} for the definition of fiber product
of Kuranishi structures.
We define $\CM_{k+1,m}^{\rm main}(J;\beta; \vec{Q},\vec{P})$
in a similar way.
When we do not consider the interior marked points,
we denote the corresponding moduli spaces by
$\CM_{k+1}(J;\beta; \vec{P})$ and
$\CM_{k+1}^{\rm main}(J;\beta; \vec{P})$, respectively.
Namely,
\begin{equation}\label{withP}
\CM_{k+1}(J;\beta; \vec{P})
:= \CM_{k+1}(J;\beta) {}_{(ev_1,\dots,ev_k)}
\times_{f_1 \times \dots \times f_k}
\left(\prod_{i=1}^{k}P_i \right).
\end{equation}
In Subsection \ref{subsec:orimain}, we describe
the orientations on these spaces precisely.

\section{$\tau$-relative spin structure and orientation}\label{sec:relspin}
\subsection{Review of relative spin structure and orientation}\label{subsec:relspin}
It is known that the moduli space of pseudo-holomorphic discs
with Lagrangian boundary condition is not always orientable.
To discuss orientability and orientation of the moduli space,
we first recall the notion of relative spin structure
and its stable conjugacy class introduced in \cite{fooobook2}
and also briefly review how the stable conjugacy class of relative
spin structure determines an orientation of the moduli space of
pseudo-holomorphic discs with Lagrangian boundary condition.
See Section 8.1 \cite{fooobook2} for more details.
See also V. de Silva's work \cite{Si}.
\begin{defn}\label{def:relspin}
An oriented Lagrangian submanifold $L$ of $M$ is
called {\it relatively spin} if there exists a class
$st \in H^2(M;\Z_2)$ such that $st\vert _L = w_2(TL)$.
\par
A pair of oriented Lagrangian submanifolds $(L^{(1)},L^{(0)})$ is called
{\it relatively spin}, if there exists a class $st\in H^2(M;\Z_2)$ satisfying
$st |_{L^{(i)}}=w_2(TL^{(i)})$ ($i = 0,1$)
simultaneously.
\end{defn}
\begin{rem}\label{rem:pin}
Using the relative pin structure, J. Solomon \cite{Sol} generalized our results
about the orientation problem studied in \cite{fooobook2}
to the case of non-orientable Lagrangian submanifolds.
\end{rem}
Let $L$ be a relatively spin Lagrangian submanifold of $M$.
We fix a triangulation of $M$ such that $L$ is a subcomplex.
A standard obstruction theory yields that we can take an oriented real vector bundle $V$ over the $3$-skeleton
$M_{[3]}$ of $M$ which satisfies $w_2(V)=st$.
Then $w_2(TL\vert_{L_{[2]}} \oplus V\vert_{L_{[2]}})=0$ and so
$TL \oplus V$ carries a spin structure on the $2$-skeleton $L_{[2]}$ of $L$.
\begin{defn}\label{def:relspinstr}
The choice of an orientation on $L$,
a cohomology class
$st\in H^2(M;\Z_2)$,
an oriented real vector bundle
$V$ over the $3$-skeleton $M_{[3]}$ satisfying
$w_2(V)=st$ and a spin structure $\sigma$ on
$(TL \oplus V)\vert_{L_{[2]}}$ is called a {\it relative spin structure} on $L \subset M$.
\par
 A {\it relative spin structure} on the pair
$(L^{(1)},L^{(0)})$ is the choice of orientations on $L^{(i)}$, a cohomology class
$st\in H^2(M;\Z_2)$,
an oriented real vector bundle $V$ over the 3 skeleton $M_{[3]}$ satisfying
$w_2(V)=st$
and spin structures on $(TL^{(i)} \oplus V)|_{L^{(i)}_{[2]}}$
$(i=0,1)$.
\end{defn}
In this paper we fix an orientation on $L$.
If $L$ is spin, we have an associated relative spin structure for each spin structure on $L$ as follows:
Take $st=0$ and $V$ is trivial. Then the spin structure
on $L$ naturally induces the spin structure on $TL \oplus V$.

Definition \ref{def:relspinstr} depends on the choices of $V$ and the triangulation of $M$.
We introduce an equivalence relation called {\it stable conjugacy}
on the set of relative spin structures so that
the stable conjugacy class is independent of such choices.
\begin{defn}\label{def:stablyconj}
We say that two relative spin structures
$(st_i,V_i,\sigma_i)$ $(i=1,2)$ on $L$ are
{\it stably conjugate}, if there exist integers $k_i$
and an orientation preserving bundle isomorphism
$\varphi: V_1 \oplus \R^{k_1} \to V_2 \oplus \R^{k_2}$
such that by $1\oplus \varphi\vert_{L_{[2]}}: (TL \oplus V_1)_{L_{[2]}} \oplus \R^{k_1}
\to (TL \oplus V_2)_{L_{[2]}} \oplus \R^{k_2}$, the spin structure $\sigma_1 \oplus 1$
induces the spin structure $\sigma_2 \oplus 1$.
\end{defn}
Here $\R^{k_i}$ denote trivial vector bundles of rank $k_i$
($i=1,2$).
We note that
in Definition \ref{def:stablyconj}, we still fix a triangulation $\mathfrak T$ of $M$ such that $L$ is a subcomplex.
However, by Proposition 8.1.6 in \cite{fooobook2},
we find
that the stable conjugacy class of relative spin structure
is actually independent of the choice of a triangulation of
$M$ as follows:
We denote by $\operatorname{Spin}(M,L;\mathfrak T)$
the set of all the stable conjugacy classes of relative spin structures on $L\subset M$.
\begin{prop}[Proposition 8.1.6 in \cite{fooobook2}]\label{prop:8.1.6}
$(1)$ There is a simply transitive action of $H^2( M, L; \Z_2)$
on $\operatorname{Spin}(M,L;\mathfrak T)$.
\par
$(2)$ For two triangulations $\mathfrak T$ and $\mathfrak T'$ of $M$ such that $L$ is a subcomplex, there exists a
canonical isomorphism
$\operatorname{Spin}(M,L;\mathfrak T) \cong \operatorname{Spin}(M,L;\mathfrak T')$ compatible with the above action.
\end{prop}
In particular, if a spin structure of $L$ is given, there is a canonical
isomorphism $\operatorname{Spin}(M,L;\mathfrak T) \cong H^2(M,L;\Z_2)$.
Thus, hereafter,
we denote by $\operatorname{Spin}(M,L)$
the set of the stable conjugacy classes of relative spin structures on $L$
without specifying any triangulation of $M$.

Since the class $st$ is determined by $V$, we
simply write the stable conjugacy class of relative spin structure
as $[(V,\sigma)]$ where $\sigma$ is a spin structure on
$(TL \oplus V)\vert_{L_{[2]}}$.

The following theorem is proved in Section 8.1 \cite{fooobook2}.
We denote by $\widetilde{\CM}^{\rm reg}(J;\beta)$
the set of all $J$-holomorphic maps from $(D^2,\partial D^2)$
to $(M,L)$ representing a class $\beta$.
We note that $\CM^{\rm reg}(J;\beta)={\widetilde{\CM}}^{\rm reg}(J;\beta)/PSL(2,\R)$.
\begin{thm}\label{thm:ori}
If $L$ is a relatively spin Lagrangian submanifold,
$\widetilde{\CM}^{\rm reg}(J; \beta)$ is orientable. Furthermore, the choice
of stable conjugacy class of relative spin structure on $L$ determines an orientation on
$\widetilde{\CM}^{\rm reg}(J;\beta)$ canonically for all $\beta \in \pi_2(M,L)$.
\end{thm}
\begin{rem}\label{rem:thmori}
(1) Following Convention 8.2.1 in \cite{fooobook2},
we have an induced orientation on the quotient space.
Thus Theorem \ref{thm:ori} holds for the quotient space
$\CM^{\rm reg}(J;\beta)={\widetilde{\CM}}^{\rm reg}(J;\beta)/PSL(2,\R)$ as well.
Here we use the orientation of $PSL(2,\R)$ as in Convention 8.3.1 in \cite{fooobook2}.
\par
(2) Since $\CM^{\rm reg}(J;\beta)$ is the top dimensional
stratum of $\CM(J;\beta)$, the orientation on
$\CM^{\rm reg}(J;\beta)$ determines one on
$\CM(J;\beta)$.
In this sense, it is enough to consider $\CM^{\rm reg}(J;\beta)$
when we discuss orientation on $\CM(J;\beta)$.
The same remark applies to other moduli spaces including
marked points and fiber products with singular simplices.
\par
(3) The moduli space $\mathcal{M}(J,\beta)$ may not contain a smooth holomorphic disc, i.e.,
$\mathcal{M}^{\operatorname{reg}}(J,\beta) = \emptyset$.
However, the orientation issue
can be discussed as if $\mathcal{M}^{\operatorname{reg}}(J,\beta) \neq \emptyset$.
This is because we consider the orientation of Kuranishi structure, i.e., the orientation of
$\det E^{\ast}\otimes \det TV$, where $V$ is a Kuranishi neighborhood around a point $p=[w:(\Sigma,\partial \Sigma) \to (M,L)]$
in $\mathcal{M}(J,\beta)$ and $E \to V$ is the obstruction bundle.
Even though $p$ is not represented by a bordered stable map with
an irreducible domain, i.e., a disc, $V$ contains a solution of $\overline{\partial} u \equiv 0 \mod E$ for
$u:(D^2, \partial D^2) \to (M,L)$.
The determinant bundle of the linearized $\overline{\partial}$-operators parametrized by $V$
is trivialized around $p$, hence the orientation of the determinant line at $[u]$ determines the one at $p$.
\end{rem}

We recall from Section 8.1 \cite{fooobook2} how each stable conjugacy class of
relative spin structures determines
an orientation on the moduli space of holomorphic discs.
Once we know the orientability of $\widetilde{\CM}^{\rm reg}(J;\beta)$,
it suffices to give an orientation
on the determinant of the tangent space at a point
$w \in \widetilde{\CM}^{\rm reg}(J;\beta)$
for each stable conjugacy class of
relative spin structures.
We consider
the linearized operator
of the pseudo-holomorphic
curve equation
\begin{equation}\label{linearizedeq}
\aligned
D_w\overline{\partial}:\,\,&W^{1,p}(D^2,\partial D^2;w^*TM,\ell^*TL) \to
L^p(D^2;w^*TM\otimes \Lambda^{0,1}_{D^2}).
\endaligned
\end{equation}
Here $\ell = w\vert_{\partial D^2}$
and $p>2$.
Since it has the same symbol as the Dolbeault operator
$$\overline{\partial}_{(w^*TM,\ell^*TL)}:W^{1,p}(D^2,\partial
D^2;w^*TM,\ell^*TL)
\to  L^p(D^2;w^*TM\otimes \Lambda^{0,1}_{D^2}),
$$
we may consider the determinant of the index of
this Dolbeault operator $\overline{\partial}_{(w^*TM,\ell^*TL)}$
instead.
We can deform $w:(D^2, \partial D^2) \to (M,L)$
to $w_0:(D^2, \partial D^2) \to (M_{[2]},L_{[1]})$
by the simplicial approximation theorem.
We put $\ell_0 = w_0 \vert_{\partial D^2}$.

Now pick
$[(V,\sigma)] \in \operatorname{Spin}(M,L)$.
Then it determines the stable homotopy class of
trivialization of $\ell_0^{\ast}(TL \oplus V)$.
The existence of the oriented bundle $V$ on $M_{[3]}$ induces a unique
homotopy class of trivialization of $\ell_0^* V$.
Thus, we have a unique homotopy class of
trivialization of $\ell_0^*TL$.
Using this trivialization and Proposition \ref{prop:8.1.4} below
(applied to the pair of $(E,\lambda)$ with $E=w_0^{\ast}TM, \lambda=\ell_0^{\ast}TL$),
we can assign
an orientation on the determinant of the index
$$
\det \operatorname{Index}~
\overline{\partial}_{(w_0^*TM,{\ell}_0^*TL)})
:=
\det (\operatorname{coker}~
\overline{\partial}_{(w_0^*TM,{\ell}_0^*TL)})^* \otimes \det
\operatorname{ker}~\overline{\partial}_{(w_0^*TM,{\ell}_0^*TL)}.
$$
This process is invariant under
stably conjugate relation of relative spin structures.
Therefore we obtain an orientation on $\widetilde{\CM}^{\rm {reg}}(J;\beta)$
and so on $\CM^{\rm {reg}}(J;\beta)$ for each stable conjugacy class of relative spin structure $[(V,\sigma)]$.
\begin{prop}[Proposition 8.1.4 \cite{fooobook2}]
\label{prop:8.1.4}
Let $E$ be a complex vector bundle over $D^2$
and $\lambda$ a maximally totally real bundle
over $\partial D^2$ with an isomorphism
$$
E\vert_{\partial D^2} \cong
\lambda \otimes \C.
$$
Suppose that $\lambda$ is
trivial. Then each trivialization on $\lambda$ canonically induces
an orientation on
$\operatorname{Index}\overline{\partial}_{(E,\lambda)}$.
Here $\overline{\partial}_{(E,\lambda)}$ is the Dolbeault operator on $(D^2, \partial D^2)$
with coefficient $(E,\lambda)$:
$$
\overline{\partial}_{(E,\lambda)} :
W^{1,p}(D^2,\partial D^2;E,\lambda) \to
L^p(D^2;E\otimes \Lambda^{0,1}_{D^2}).
$$
\end{prop}
\begin{rem}\label{rem:prop8.1.4}
In order to explain some part of the proof of
Theorem \ref{thm:fund} given in Section \ref{sec:Proofth} in a self-contained way,
we briefly recall the outline of the proof of Proposition
\ref{prop:8.1.4}. See Subsection 8.1.1 \cite{fooobook2} for more
detail.
For $0<\epsilon < 1$, we put
$A(\epsilon) = \{ z \in D^2 ~\vert~ 1-\epsilon \le \vert z\vert
\le 1 \}$ and
$C_{1-\epsilon}= \{ z \in D^2 ~\vert~  \vert z\vert =1-\epsilon
\}$.
By pinching the circle $C$ to a point,
we have a union of a $2$-disc $D^2$ and a $2$-sphere $\C P^1$
with the center $O \in D^2$ identified with
a point $p \in \C P^1$.
The resulting space $\Sigma =D^2 \cup \C P^1$ has naturally a structure of a nodal curve where $O=p$ is the nodal point.
Under the situation of Proposition \ref{prop:8.1.4},
the trivial bundle $\lambda \to \partial D^2$ trivially extends
to $A(\epsilon )$ and the complexification of
each trivialization of $\lambda \to \partial D^2$
gives a trivialization on $E\vert_{A(\epsilon )} \to A(\epsilon )$.
Thus the bundle $E\to D^2$ descends to a bundle over
the nodal curve $\Sigma$ together with a maximally totally real bundle over $\partial \Sigma = \partial D^2$. We denote them by $E' \to \Sigma$
and $\lambda' \to \partial \Sigma$ respectively.
We also denote by $W^{1,p}(\C P^1; E'\vert_{\C P^1})$ the space of $W^{1,p}$-sections of $E' \vert_{\C P^1} \to \C P^1$ and
by  $W^{1,p}
(D^2; E'\vert_{D^2}, \lambda')$ the space of
$W^{1,p}$-sections $\xi_{D^2}$ of $E' \vert_{D^2} \to D^2$ satisfying
$\xi_{D^2}(z)\in \lambda'_z$, $z \in \partial D^2 = \partial  \Sigma$.
We consider a map denoted by $\operatorname{diff}$:
\begin{eqnarray*}
&{}& \operatorname{diff}:
W^{1,p}(\C P^1; E'\vert_{\C P^1}) \oplus W^{1,p}
(D^2, \partial D^2; E'\vert_{D^2}, \lambda') \to \C^n; \\
&{}& \hskip1in (\xi_{\C P^1}, \xi_{D^2}) \mapsto
\xi_{\C P^1}(p) - \xi _{D^2}(O).
\end{eqnarray*}
We put $W^{1,p}(E', \lambda'):=
\operatorname{diff}^{-1}(0)$ and consider the index of operator
$$
\overline{\partial}_{(E',\lambda')}~:~
W^{1,p}(E', \lambda')
\to
L^p(\C P^1; E'\vert_{\C P^1} \otimes \Lambda^{0,1}_{\C P^1})
\oplus
L^p(D^2, \partial D^2;E'\vert_{D^2} \otimes \Lambda^{0,1}_{D^2}).
$$
Then the orientation problem for $\operatorname{Index} \overline{\partial}_{(E,\lambda)}$ on $(D^2, \partial D^2)$
is translated into the problem for
$\operatorname{Index} \overline{\partial}_{(E',\lambda')}$
on $(\Sigma, \partial \Sigma)$.
Firstly, we note that the operator
$$
\overline{\partial}_{(E'\vert_{D^2}, \lambda'\vert_{\partial D^2})}
~:~
W^{1,p}(D^2, \partial D^2;E'\vert_{D^2},\lambda') \to L^p(D^2;E'\vert_{D^2}\otimes \Lambda^{0,1}_{D^2})
$$
is surjective. Each trivialization of $\lambda \to \partial D^2$
gives an identification:
\begin{equation}\label{eq:isokernel}
\operatorname{ker} \overline{\partial}_{(E'\vert_{D^2}, \lambda'\vert_{\partial D^2})} \cong
\operatorname{ker} \overline{\partial}_{(D^2 \times \C ^n, \partial D^2 \times \R ^n)} \cong
\R ^n,
\end{equation}
where $\operatorname{ker} \overline{\partial}_{(D^2 \times \C ^n, \partial D^2 \times \R ^n)}$ is the space of solutions
$\xi : D^2 \to \C ^n$ of the Cauchy-Riemann equation with
boundary condition:
$$\overline{\partial} \xi =0, \quad \xi (z) \in \lambda'_z \equiv \R^n, \quad z \in \partial D^2.
$$
Thus the solution must be a real constant vector.
This implies that we have a canonical isomorphism in
(\ref{eq:isokernel}).
Then the argument in Subsection 8.1.1 \cite{fooobook2} shows that
the orientation problem can be reduced to the orientation
on $\operatorname{ker} \overline{\partial}_{(E'\vert_{D^2}, \lambda'\vert_{\partial D^2})}$ and
$\operatorname{Index} \overline{\partial}_{E'\vert_{\C P^1}}$.
The latter one has a complex orientation.
By taking a finite dimensional complex vector space
$W \subset L^p(\C P^1;E'\vert_{\C P^1}\otimes \Lambda^{0,1}_{\C P^1})$
such that
$$L^p(\C P^1;E'\vert_{\C P^1}\otimes \Lambda^{0,1}_{\C P^1})
= \operatorname{Image} \overline{\partial}_{E'\vert_{\C P^1}}
+ W,
$$
a standard argument (see the paragraphs after Remark 8.1.3
in \cite{fooobook2}, for example) shows
that the orientation problem on $\operatorname{Index} \overline{\partial}_{E'\vert_{\C P^1}}$ is further reduced to
one on $\operatorname{ker} \overline{\partial}_{E'\vert_{\C P^1}}$
which is the space of holomorphic sections of $E'\vert_{\C P^1} \to \C P^1$,
denoted by $\operatorname{Hol} (\C P^1; E'\vert_{\C P^1})$.
\end{rem}
\par\smallskip
We next describe how the orientation behaves
under the change of stable conjugacy classes of
relative spin structures.
Proposition \ref{prop:8.1.6} shows that the difference of relative spin structures is measured by
an element $\mathfrak x$ in $H^2(M,L;{\Z}_2)$.
We denote the simply transitive action of $H^2(M,L;{\Z}_2)$ on
$\operatorname{Spin}(M,L)$ by
$$
(\mathfrak x, [(V,\sigma)]) \mapsto \mathfrak x \cdot [(V,\sigma)].
$$
When we change the relative spin structure by $\mathfrak r
\in H^2(M,L;{\Z}_2)$, then
we find that the orientation on
the index of the operator $D_w\overline{\partial}$
in (\ref{linearizedeq}) changes by $(-1)^{{\mathfrak x}[w]}$.
The following result is proved in Proposition 8.1.16 in \cite{fooobook2} and also obtained
by Cho \cite{Cho04}, Solomon \cite{Sol}.
This proposition is used in Subsections
\ref{proof1.9} and \ref{subsec:Appl2}.
\begin{prop}\label{Proposition44.16}
The identity map
$$
\CM(J;\beta)^{[(V,\sigma)]} \longrightarrow \CM(J;\beta)^{\mathfrak x \cdot[(V,\sigma)]}
$$
is orientation preserving if and only if $\mathfrak x [\beta]=0$.
\end{prop}
For a diffeomorphism $\psi : (M,L) \to (M',L')$
satisfying $\psi(L)=L'$, we define the pull-back map
\begin{equation}\label{pullback}
\psi^{\ast} ~:~
\operatorname{Spin}(M',L') \longrightarrow
\operatorname{Spin}(M,L)
\end{equation}
by $\psi^{\ast}[(V',\sigma')] = [(\psi^{\ast}V', \psi^{\ast} \sigma')]$.
That is, we take a triangulation on $M'$ such that
$L'$ is its subcomplex and $\psi : (M,L) \to (M',L')$ is
a simplicial map.
Then $\psi^{\ast}V'$ is a real vector bundle over $M_{[3]}$
and $\sigma'$ induces a spin structure on
$(TL \oplus \psi^{\ast}V')\vert_{L_{[2]}}$.
Then it is easy to see that
$$
\psi_{\ast} ~:~ \CM(\beta;M,L;J)^{\psi^{\ast}[(V',\sigma')]}
\longrightarrow
\CM(\psi_{\ast}\beta;M',L';\psi_{\ast}J)^{[(V',\sigma')]}
$$
is orientation preserving.

\subsection{$\tau$-relative spin structure and an example}
\label{subsec:taurelspin}

After these general results are prepared in the previous subsection,
we focus ourselves on the case $L= {\rm Fix}~\tau$, the fixed point set
of an anti-symplectic involution $\tau$ of $M$.
We define the notion of $\tau$-relative spin structure and
discuss its relationship with the orientation of the moduli space.
Note that for $\tau$ such that $L = \text{Fix} ~\tau$ is relative spin, $\tau$
induces an involution $\tau^*$ on the set of relative spin structures $(V,\sigma)$
by pull-back \eqref{pullback}.
\begin{defn}\label{Definition44.17}
A $\tau$-{\it relative spin structure} on $L$ is a relative spin
structure $(V,\sigma)$ on $L$ such that $\tau^*(V,\sigma)$ is stably
conjugate to $(V,\sigma)$, i.e.,
$\tau^*[(V,\sigma)]=[(V,\sigma)]$ in $\operatorname{Spin}(M,L)$.
We say that $L$ is $\tau$-{\it relatively spin} if it carries a
$\tau$-relative spin structure, i.e., if the involution
$\tau^*: \operatorname{Spin}(M,L) \to \operatorname{Spin}(M,L)$ has a
fixed point.
\end{defn}
\begin{exm}\label{Remark44.18}
If $L$ is spin, then it is $\tau$-relatively spin: Obviously $\tau$
preserves spin structure of $L$ since it is the identity on $L$. And we
may take $V$ needed in the definition of relative spin structure to
be the trivial vector bundle.
\end{exm}
\begin{rem}\label{Remark44.20}
We would like to emphasize that a relative spin structure $(V, \sigma ,st)$ satisfying
$\tau^{\ast} st = st$ is {\it not necessarily} a
$\tau$-relative spin structure in the sense of Definition \ref{Definition44.17}.
See Proposition \ref{Proposition44.19} below.
\end{rem}
Now we give an example of $L = \text{Fix }\tau$ that is relatively spin but {\it
not} $\tau$-relatively spin.

Consider $M = \C P^{n}$ with standard symplectic and complex
structures and $L = \R P^n \subset \C P^n$, the real point set. The real projective space $\R P^n$ is
oriented if and only if $n$ is odd. We take the tautological real
line bundle $\xi$ on $\R P^n$ such that the 1-st Stiefel-Whitney class $w_1(\xi):= x$ 
is a generator of $H^1(\R P^n;\Z_2)$. Then we have
$$
T\R P^n \oplus \R \cong \xi^{\oplus (n+1)}
$$
and so the total Stiefel-Whitney class is given by
$$
w(T\R P^n) = (1+x)^{n+1}.
$$
Therefore we have
\begin{equation}\label{eq:w2}
w_2(\R P^{2n+1})
= \begin{cases}
x^2  &\text{if $n$ is even,}
\\
0  &\text{if $n$ is odd.}
\end{cases}
\end{equation}
From this, it follows that $\R P^{4n+3}$ ($n \geq 0$) and $\R P^1$
are spin and hence are $\tau$-relatively spin
by Example \ref{Remark44.18}.
On the other hand we prove:
\begin{prop}\label{Proposition44.19}
The real projective space $\R P^{4n+1} \subset \C P^{4n+1}$ $(n
\geq 1)$ is relatively spin but not $\tau$-relatively spin.
\end{prop}
\begin{proof}
The homomorphism
$$
H^2(\C P^{4n+1};\Z_2) \longrightarrow H^2(\R P^{4n+1};\Z_2)
$$
is an isomorphism. We can construct an isomorphism explicitly as
follows: Let $\eta$ be the tautological complex line bundle on $\C
P^{4n+1}$ such that
$$
c_1(\eta) = y \in H^2(\C P^{4n+1};\Z)
$$
is a generator. We can easily see that
$$
\eta\vert_{\R P^{4n+1}} = \xi \oplus \xi
$$
where $\xi$ is the real line bundle over $\R P^{4n+1}$ chosen as above.
Since $c_1(\eta)$ is the Euler class which reduces to the second
Stiefel-Whitney class under $\Z_2$-reduction, $y \mapsto x^2$ under
the above isomorphism. But \eqref{eq:w2} shows that $x^2 = w_2(\R P^{4n+1})$.
This proves that $\R P^{4n+1}$ is relatively spin: For $st$, we take $st = y$.
\par
Now we examine the relative spin structures of $\R P^{4n+1}$. It is easy to
check that $H^2(\C P^{4n+1}, \R P^{4n+1};\Z_2) \cong \Z_2$ and so there are
$2$ inequivalent relative spin structures by Proposition \ref{prop:8.1.6}.
Let $st = y$ and take
$$
V = \eta^{\oplus 2n+1} \oplus \R
$$
for the vector bundle $V$ noting $w_2(V) \equiv c_1(V) = (2n+1)y = y = st \mod 2$.
\par
Next we have the isomorphism
$$
\tilde{\sigma} : T\R P^{4n+1} \oplus \R^2 \cong \xi^{\oplus
(4n+2)} \oplus \R \cong \eta^{\oplus (2n+1)}\vert_{\R P^{4n+1}}
\oplus \R
$$
and so it induces a trivialization of
$$
\aligned (T\R P^{4n+1} \oplus V) \oplus \R^2 & \cong T\R
P^{4n+1} \oplus \R^2 \oplus V \\
& \cong (\eta^{\oplus (2n+1)}\vert_{\R P^{4n+1}} \oplus \R) \oplus
(\eta^{\oplus (2n+1)}\vert_{\R P^{4n+1}} \oplus \R).
\endaligned
$$
We note that on a $2$ dimensional CW complex, any stable isomorphism
between two oriented real vector bundles $V_1$, $V_2$ induces a
stable trivialization of $V_1 \oplus V_2$.
In particular,
$(\eta^{\oplus (2n+1)}\vert_{\R P^{4n+1}} \oplus \R) \oplus
(\eta^{\oplus (2n+1)}\vert_{\R P^{4n+1}} \oplus \R)$ has a canonical
stable trivialization on the $2$-skeleton of $\R P^{4n+1}$,
which in turn
provides a spin structure on $T\R P^{4n+1} \oplus V$
denoted by $\sigma$. This provides a relative spin structure on
$\R P^{4n+1} \subset \C P^{4n+1}$.
\par
Next we study the question on the $\tau$-relatively spin property.
By the definition of the tautological line bundle $\eta$ on $\C P^{4n+1}$,
the involution $\tau$ lifts to an anti-complex linear isomorphism
of $\eta$ which we denote
$$
c ~:~ \eta \longrightarrow \eta.
$$
Then
$$
c^{\oplus (2n+1)} \oplus (-1) ~:~ \eta^{\oplus (2n+1)} \oplus \R
\longrightarrow
\eta^{\oplus (2n+1)} \oplus \R
$$
is an isomorphism which covers $\tau$. Therefore
we may identify
$$
\tau^*V = V = \eta^{\oplus (2n+1)} \oplus \R
$$
on $\R P^{4n+1}$ and also
$$
\tau^* = c^{\oplus (2n+1)} \oplus (-1).
$$
Then we have
$$
\tau^*(V,\sigma,st) = (V,\sigma',st)
$$
where the spin structure $\sigma'$ corresponds to the isomorphism
$$
(c^{\oplus (2n+1)} \oplus (-1)) \circ \tilde{\sigma}.
$$
Therefore to complete the proof of Proposition
\ref{Proposition44.19} it suffices to
show that the restriction of $c^{\oplus (2n+1)} \oplus (-1)$
to $(\R P^{4n+1})_{[2]}$ is not stably homotopic to the identity map as
a bundle isomorphism.
\par
Note that the $2$-skeleton $(\R P^{4n+1})_{[2]}$ is $\R P^2$.
We have $\pi_1(SO(m)) \cong \Z_2$ and $\pi_2(SO(m)) = 1$ (for $m > 2$).
Hence an oriented isomorphism of real vector bundles on $(\R P^{4n+1})_{[2]}$ is
stably homotopic to identity if it is so on the $1$-skeleton $S^1 = (\R P^{4n+1})_{[1]}$.
\par
It is easy to see that $c \oplus c$ is homotopic to identity.
So it remains to consider $c \oplus -1 : \eta \oplus \R \to \eta \oplus \R$ on $S^1$.
Note that $\eta\vert_{S^1} = \xi \oplus \xi$ and this bundle is trivial.
The splitting corresponds to the basis
$(\cos t/2,\sin t/2)$, $(-\sin t/2,\cos t/2)$. (Here $t\in S^1 = \R/2\pi\Z$.)
The map $c$ is given by
$c = (1,-1) : \xi \oplus \xi \to \xi \oplus \xi$.
So when we identify $\eta \oplus \R \cong \R^3$ on $S^1$,
the isomorphism $c \oplus -1$ is represented by the matrix
$$\aligned
&\left(
\begin{matrix}
\cos t/2  & \sin t/2 & 0 \\
- \sin t/2 & \cos t/2 & 0\\
0 & 0 & 1
\end{matrix}
\right)
\left(
\begin{matrix}
1  & 0 & 0 \\
0 & - 1 & 0\\
0 & 0 & -1
\end{matrix}
\right)
\left(
\begin{matrix}
\cos t/2  & -\sin t/2 & 0 \\
\sin t/2 & \cos t/2 & 0\\
0 & 0 & 1
\end{matrix}
\right)
\\
&=
\left(
\begin{matrix}
\cos t  & -\sin t & 0 \\
-\sin t & - \cos t & 0\\
0 & 0 & -1
\end{matrix}
\right)
\\
&
=
\left(
\begin{matrix}
\cos t  & \sin t & 0 \\
-\sin t & \cos t & 0\\
0 & 0 & 1
\end{matrix}
\right)
\left(
\begin{matrix}
1  & 0 & 0 \\
0 & -1 & 0\\
0 & 0 & -1
\end{matrix}
\right).
\endaligned$$
This loop represents the nontrivial homotopy class in $\pi_1(SO(3)) \cong \Z_2$.
This proves that the involution
$\tau^*: \text{Spin}(\C P^{4n+1},\R P^{4n+1}) \to \text{Spin}(\C P^{4n+1},\R P^{4n+1})$ is
non-trivial. Since $\text{Spin}(\C P^{4n+1},\R P^{4n+1}) \cong \Z_2$,
the proof of Proposition \ref{Proposition44.19} is complete.
\end{proof}

Using the results in this section, we calculate Floer cohomology
of $\R P^{2n+1}$ over $\Lambda_{0,{\rm nov}}^{\Z}$
(see (\ref{eq:nov0}) for the definition of $\Lambda_{0,{\rm nov}}^{\Z}$)
in Subsection \ref{subsec:Appl2}, which
provides an example of Floer cohomology that is {\it not}
isomorphic to the ordinary cohomology.

\subsection{Orientations on
$\CM_{k+1}^{\rm {main}}(J;\beta;\vec{P})$
and $\CM_{k+1, m}^{\rm {main}}(J; \beta; \vec{Q}, \vec{P})$
}
\label{subsec:orimain}
In this subsection we recall the definitions of the orientations
of $\CM_{k+1}^{\rm {main}}(J;\beta;\vec{P})$ and $\CM_{k+1, m}^{\rm {main}}(J; \beta; \vec{Q}, \vec{P})$
from Section 8.4 and Subsection 8.10.2 in \cite{fooobook2}.
Here $L$ is not necessarily the fixed point set of an anti-symplectic
involution $\tau$.

When we discuss the orientation problem, it suffices to
consider the regular parts of the moduli spaces.
See Remark \ref{rem:thmori} (2).
By Theorem \ref{thm:ori}, we have an orientation on $\widetilde{\CM}^{\rm reg}(J;\beta)$ for
each stable conjugacy class of relative spin structure.
Including marked points,
we define an orientation on $\CM^{\rm reg}_{k+1,m}(J;\beta)$ by
$$
\aligned
& \CM^{\rm reg}_{k+1,m}(J;\beta) \\
= & \left( \widetilde{\CM}^{\rm reg}(J;\beta) \times \partial D^2_0
\times D^2_1\times \cdots \times D^2_m \times \partial D^2_{m+1}  \times \cdots \times \partial D^2_{m+k} \right)/
PSL(2:\R).
\endaligned
$$
Here the sub-indices in $\partial D^2_0$ and $\partial D^2_{m+i}$ (resp. $D^2_j$) stand for
the positions of the marked points $z_0$ and $z_i$ (resp. $z_j^+$).
(In Subsection 8.10.2 in \cite{fooobook2} we write the above space
as $\CM_{(1,k),m}(\beta)$.)
Strictly speaking, since the marked points are required to be distinct,
the left hand side above is not exactly equal to the right hand side but is
an open subset. However, when we discuss orientation problem,
we sometimes do not distinguish them when no confusion can occur.

In (\ref{withsimplicies}), (\ref{withP})
we define $\CM_{k+1}(J;\beta;\vec{P})$
and $\CM_{k+1, m}(J; \beta; \vec{Q}, \vec{P})$ by fiber products.
Now we equip the right hand sides in (\ref{withsimplicies}) and (\ref{withP})
with the fiber product orientations using
Convention 8.2.1 (3) \cite{fooobook2}.
However, we do not use the fiber product orientation themselves as the orientations on
$\CM_{k+1}(J;\beta;\vec{P})$
and $\CM_{k+1, m}(J; \beta; \vec{Q}, \vec{P})$,
but we use the following orientations twisted from the fiber product orientation:
We put $\deg P_i =n-\dim P_i$, $\deg Q_j =2n-\dim Q_j$
for smooth singular simplicies $f_i : P_i \to L$ and
$g_j : Q_j \to M$.
\begin{defn}[Definition 8.4.1 \cite{fooobook2}]\label{Definition8.4.1}
For given smooth singular simplicies $f_i : P_i  \to L$,
we define an orientation on
$
\CM _{k +1}(J;\beta ;P_1 ,\dots ,P_{k})$ by
$$
\CM _{k +1}(J;\beta ;P_1 ,\dots ,P_{k})
:=
(-1)^{\epsilon (P) }
\CM _{k +1}(J;\beta ) {}_{(ev_1 ,\ldots ,ev_{k})}
\times _{f_1 \times \dots \times f_{k}}
\left(\prod _{i=1}^{k} P_{i}\right),
$$
where
$$
\epsilon (P) =
(n+1)\sum _{j=1}^{k -1}
\sum _{i=1}^{j} \deg P_i .
$$
\end{defn}
\begin{defn}[Definition 8.10.2 \cite{fooobook2}]\label{Definition8.10.2}
For given smooth singular simplicies $f_i:P_i \to L$ in $L$ and $g_j:Q_j \to M$ in $M$, we define
$$
\aligned
 & {\CM }_{k+1,m}(J;\beta ;
Q_1,  \dots , Q_{m}; P_1 , \dots ,P_{k}) \\
:= &
(-1)^{\epsilon (P,Q)}
{\CM }_{k+1, m}(J;\beta ) {}_{(ev_1^{+} , \dots ,ev_{m}^{+},ev_{1},  \dots ,ev_{k})}
\times _{g_1 \times \dots \times f_k}
\left(\prod _{j=1}^{m} Q_{j} \times \prod_{i=1}^{k} P_i\right),
\endaligned
$$
where
\begin{equation}\label{epsilonPQ}
\epsilon (P,Q) =
(n+1)\sum _{j=1}^{k-1}
\sum _{i=1}^{j} \deg P_i + \left( (k+1)(n+1) + 1 \right) \sum_{j=1}^{m} \deg Q_j.
\end{equation}
\end{defn}

Replacing $\CM_{k+1}(J;\beta)$ and
${\CM }_{k+1, m}(J;\beta )$ on the right hand sides of above definitions
by
$\CM_{k+1}^{\rm main}(J;\beta)$ and
${\CM }_{k+1, m}^{\rm main}(J;\beta )$ respectively,
we define orientations
on the main components
$\CM _{k +1}^{\rm main}(J;\beta ;P_1 ,\dots ,P_{k})$ and
${\CM }_{k+1,m}^{\rm main}(J;\beta ;
Q_1,  \dots , Q_{m}; P_1 , \dots ,P_{k})$
in the same way.

When we do not consider the fiber product with
$g_j : Q_j \to M$,
we drop the second term in $(\ref{epsilonPQ})$.
Thus when $m=0$, the moduli space
in Definition \ref{Definition8.10.2} is nothing but
${\CM }_{k+1}(J;\beta ;P_1,\ldots ,P_k)$ equipped with the orientation given by Definition \ref{Definition8.4.1}.
\par
When we study the map $\tau^{\rm main}_{\ast}$
in (\ref{eq:taumain}),
we have to change the ordering of boundary marked points.
Later we use the following lemma which describes the behavior of orientations
under the change of ordering of boundary marked points:
\begin{lem}[Lemma 8.4.3 \cite{fooobook2}]\label{Lemma8.4.3}
Let $\sigma$ be the transposition element $(i,i+1)$ in the $k$-th
symmetric group $\mathfrak S_{k}$.
$(i=1,\ldots ,k -1)$.
Then the action of $\sigma$ on
the moduli space
$\CM_{k +1}(J;\beta ;P_1 ,\ldots ,P_i ,P_{i+1}, \ldots ,P_{k})$
changing the order of marked points
induces an orientation preserving isomorphism
$$
\aligned
& \sigma ~:~ \CM_{k +1}(J;\beta ;P_1 ,\ldots ,P_i ,P_{i+1}, \ldots ,P_{k}) \\
& \longrightarrow (-1)^{(\deg P_i +1)(\deg P_{i+1}+1)}
\CM_{k +1}(J;\beta ;P_1 ,\ldots ,P_{i+1},P_i ,\ldots ,P_{k}).
\endaligned
$$
\end{lem}

\section{The induced maps $\tau_{\ast}$ and $\tau_{\ast}^{\rm main}$}
\label{sec:inducedtau}

Let $(M,\omega)$ be a compact, or tame, symplectic manifold and let $\tau: M \to M$
be an anti-symplectic involution, i.e., a map satisfying $\tau^2 = id$ and
$
\tau^*(\omega) = -\omega.
$
We also assume that the fixed point set $L=\text{ Fix }\tau$ is non-empty, oriented and compact.
\par
Let $\CJ_{\omega}$ be the set of all $\omega$ compatible almost
complex structures
and $\CJ_\omega^\tau$ its subset consisting of
$\tau$-anti-invariant almost complex structures $J$ satisfying
$
\tau_* J = -J.
$
\begin{lem}[Lemma 11.3 \cite{fooo00}. See also Proposition 1.1 \cite{Wel}.]\label{Lemma38.3}
The space $\CJ_{\omega}^\tau$ is non-empty and
contractible. It becomes an infinite dimensional
(Fr\'echet) manifold.
\end{lem}
\begin{proof} For given $J\in \CJ^\tau_\omega$, its tangent space
$T_J\CJ^\tau_\omega$
consists of sections $Y$ of the bundle $\operatorname{End}(TM)$ whose fiber at $p \in M$ is
the
space of linear maps $Y:T_pM \to T_pM$ such that
$$
Y J + J Y=0, \quad \omega(Y v,w) + \omega(v,Y w)=0,
\quad \tau^*Y= - Y.
$$
Note that the second condition means that $JY$ is a symmetric endomorphism
with respect to the metric $g_J= \omega(\cdot, J\cdot)$.
It immediately follows that $\CJ^\tau_\omega$ becomes a manifold.
The fact that $\CJ^\tau_\omega$ is non-empty
(and contractible) follows from the polar decomposition theorem
by choosing a $\tau$-invariant
Riemannian metric on $M$.
\end{proof}

\subsection{The map $\tau_{\ast}$ and orientation}
\label{subsec:tau}
We recall the definition of $\Pi(L)=\pi_2(M,L)/\sim$ where the equivalence relation is defined by
$\beta \sim \beta' \in \pi_2(M,L)$ if and only if
$
\omega(\beta)=\omega(\beta')$
and
$
\mu_L(\beta) =\mu_L(\beta')$.
(See (\ref{eq:Pi}).)
We notice that
for each $\beta \in \Pi (L)$, we defined the moduli space
$\CM(J;\beta)$ as the union
\begin{equation}\label{remarkmoduli}
\CM(J;\beta)=
\bigcup_{B \in \pi_2(M,L);[B]=\beta \in \Pi(L)} \CM(J;B).
\end{equation}
We put $D^2 =\{ z \in \C ~\vert~ \vert z\vert \le 1 \}$
and $\overline{z}$ denotes the complex conjugate.
\begin{defn}\label{def:inducedtau}
Let $J \in \CJ_{\omega}^\tau$.
For $J$ holomorphic curves $w : (D^2,\partial D^2) \to (M,L)$ and
$u : S^2 \to M$,
we define $\widetilde w$, $\widetilde u$ by
\begin{equation}\label{38.4}
\widetilde w(z) = (\tau\circ w)(\overline z),
\qquad
\widetilde u(z) = (\tau\circ u)(\overline z).
\end{equation}
For $(D^2,w) \in \CM^{\text{reg}}(J;{\beta})$,
$((D^2,\vec z,{\vec z}\,^+),w) \in \CM^{\text{reg}}_{k+1,m}(J;{\beta})$
we define
\begin{equation}\label{38.5}
\tau_*((D^2,w)) = (D^2,\widetilde w),
\qquad
\tau_*(((D^2,\vec z,\vec z\,^+),w)) = ((D^2,\vec{\overline z},
\vec{\overline z}\,^{+}),\widetilde w),
\end{equation}
where
$$
\vec{\overline z} = (\overline z_0,\dots,\overline z_k),
\qquad
\vec{\overline z}\,^+ = (\overline z^+_0,\dots,\overline z^+_m).
$$
\end{defn}
\begin{rem}\label{rem:tau}
For  $\beta = [w]$, we put $\tau_*\beta=[\widetilde{w}]$.
Note if $\tau_{\sharp} :  \pi_2(M,L) \to \pi_2(M,L)$ is the
natural homomorphism induced by $\tau$, then
$$\tau_*\beta = -\tau_{\sharp}\beta.$$
This is because $z \mapsto \overline z$ is of degree $-1$.
In fact, we have
$$
\tau_*(\beta) = \beta
$$
in $\Pi(L)$,
since $\tau_*$ preserves both the symplectic area and the Maslov index.
\end{rem}
\begin{lem}\label{Lemma38.6}
The definition
$(\ref{38.5})$ induces the maps
$$
\tau_* : \CM^{\operatorname{reg}}(J;\beta) \to
\CM^{\operatorname{reg}}(J;\beta),
\quad
\tau_* : \CM^{\operatorname{reg}}_{k+1,m}(J;\beta) \to
\CM^{\operatorname{reg}}_{k+1,m}(J;\beta),
$$
which satisfy
$\tau_{\ast} \circ \tau_{\ast} = \operatorname{id}$.
\end{lem}
\begin{proof}
If $(w,(z_0,\dots,z_k),(z_1^+,\dots,z_m^+)) \sim (w',(z'_0,\dots,z'_k),
(z_1^{\prime +},\dots,z_m^{\prime +}))$, we have
$\varphi \in PSL(2,\Bbb R) = \text{Aut}(D^2)$
such that $w' = w\circ \varphi^{-1}$, $z'_i = \varphi(z_i)$,
$z^{\prime +}_i = \varphi(z^+_i)$ by definition.
We define $\overline \varphi: D^2 \to D^2$ by
\begin{equation}
\label{eq:barPSL2R}
\overline{\varphi}(z) = \overline{(\varphi(\overline z))}.
\end{equation}
Then $\overline{\varphi} \in PSL(2,\Bbb R)$ and
$\widetilde w' = \widetilde w \circ \overline{\varphi}^{-1}$,
$\overline z'_i = \overline{\varphi}(\overline z_i)$,
$\overline z^{\prime +}_i = \overline\varphi(\overline z^+_i)$.
The property
$\tau_{\ast} \circ \tau_{\ast} = \operatorname{id}$
is straightforward.
\end{proof}
We note that the mapping $\varphi \mapsto \overline\varphi$, $PSL(2,\R)
\to PSL(2,\R)$ is  orientation preserving.

We have the following
\begin{prop}\label{regtau}
The involution $\tau_*$ is induced by an automorphism of
$\CM^{\operatorname{reg}}(J;\beta)$ as a space with Kuranishi structure.
\end{prop}
Proposition \ref{regtau} is a special case of Theorem \ref{Proposition38.11} (1).
The proof of  Theorem \ref{Proposition38.11} (1)
is similar to the proof of Theorem \ref{Lemma38.14}
which will be proved in Section \ref{sec:Proofth}.
See Definition \ref{def:auto}
for the definition of an automorphism of a space with Kuranishi structure and
Definition \ref{def:oripres} for the definition for an automorphism to be {\it orientation preserving} in the sense of Kuranishi structure.
In this paper, we use the terminology {\it orientation preserving} only in the sense of Kuranishi structure.
We refer Section A1.3 \cite{fooobook2}
for more detailed explanation of the group action on
a space with Kuranishi structure.

In Section \ref{sec:relspin}, we
explained that a choice of stable conjugacy class
$[(V,\sigma)] \in
\operatorname{Spin} (M,L)$ of relative spin structure on $L$ induces an
orientation on
$\CM_{k+1,m}(J;\beta)$ for any given $\beta\in \Pi(L)$. Hereafter we equip
$\CM_{k+1,m}(J;\beta)$ with this orientation when we regard it as a space
with oriented Kuranishi structure.
We write it as $\CM_{k+1,m}(J;\beta)^{[(V,\sigma)]}$ when we specify the
stable conjugacy class of relative spin structure.
\par
For an anti-symplectic involution $\tau$ of $(M,\omega)$,
we have the pull back $\tau^{\ast}[(V,\sigma)]$ of
the stable conjugacy class of relative spin structure $[(V,\sigma)]$.
See (\ref{pullback}).
Then from the definition of the map $\tau_{\ast}$ in Lemma \ref{Lemma38.6}
we obtain the maps
$$
\aligned
\tau_* & : \CM^{\operatorname{reg}}(J;\beta)^{\tau^{\ast}[(V,\sigma)]} \longrightarrow
\CM^{\operatorname{reg}}(J;\beta)^{[(V,\sigma)]},
\\
\tau_* & : \CM^{\operatorname{reg}}_{k+1,m}(J;\beta)^{\tau^{\ast}[(V,\sigma)]}
\longrightarrow
\CM^{\operatorname{reg}}_{k+1,m}(J;\beta)^{[(V,\sigma)]}.
\endaligned
$$
Here we note that $\tau_{\ast}J =-J$
and we use the same $\tau$-anti-symmetric
almost complex structure $J$ in both the source and the target
spaces of the map $\tau_{\ast}$.
If $[(V,\sigma)]$ is
$\tau$-relatively spin (i.e.,
$\tau^{\ast}[(V,\sigma)]=[(V,\sigma)]$),
$\tau_{\ast}$ defines involutions
of $\CM^{\operatorname{reg}}(J;\beta)^{[(V,\sigma)]}$ and $\CM^{\operatorname{reg}}_{k+1,m}(J;\beta)^{[(V,\sigma)]}$ with Kuranishi structures.
\begin{thm}\label{Proposition38.7}
Let $L$ be a fixed point set of an anti-symplectic involution $\tau$ and
$J \in \CJ^{\tau}_{\omega}$.
Suppose that $L$ is oriented and carries a relative spin structure $(V,\sigma)$.
Then the map
$
\tau_*: {\CM}^{\operatorname{reg}}(J;\beta)^{\tau^{\ast}[(V,\sigma)]} \to
{\CM}^{\operatorname{reg}}(J;\beta)^{[(V,\sigma)]}
$
is orientation preserving if $\mu_L(\beta) \equiv 0
\mod 4$ and is orientation reversing if
$\mu_L(\beta) \equiv 2
\mod 4$.
\end{thm}
\begin{cor}\label{corProposition38.7}
Let $L$ be as in Theorem \ref{Proposition38.7}.
In addition, if $L$ carries a $\tau$-relative spin structure
$[(V,\sigma)]$, then the map $\tau_{\ast} : {\CM}^{\operatorname{reg}}(J;\beta)^{[(V,\sigma)]} \to
{\CM}^{\operatorname{reg}}(J;\beta)^{[(V,\sigma)]}
$ is orientation preserving if $\mu_L(\beta) \equiv 0
\mod 4$ and is orientation reversing if
$\mu_L(\beta) \equiv 2
\mod 4$.
\end{cor}
We prove Theorem \ref{Proposition38.7} in Section \ref{sec:Proofth}.
Here we give a couple of examples.

\begin{exm}\label{Example38.8.}
(1) Consider the case of $M=\C P^n$, $L=\R P^n$.
In this case, each Maslov
index $\mu_L(\beta)$ has the form
$$
\mu_L(\beta) = \ell_\beta  (n+1)
$$
where $\beta = \ell_\beta$ times the generator.
We know that when $n$ is even $L$ is not orientable, and so we
consider only the case where $n$ is odd.  On the other hand,
when $n$ is odd, $L$ is relatively spin.
Moreover, we have proved in Proposition \ref{Proposition44.19}
that
$\R P^{4n+3}$ $(n\ge 0)$ is $\tau$-relatively spin, (indeed, $\R P^{4n+3}$
is spin), but
$\R P^{4n+1}$ $(n\ge 1)$ is {\it not} $\tau$-relatively spin.
Then using the above formula for the Maslov index,
we can conclude from
Theorem \ref{Proposition38.7} that the map
$\tau_*: {\CM}^{\text{reg}}(J;\beta)^{[(V,\sigma)]} \to
{\CM}^{\text{reg}}(J;\beta)^{[(V,\sigma)]}
$ is always an orientation preserving involution
for any $\tau$-relative spin structure $[(V,\sigma)]$ of $\R P^{4n+3}$.
\par
Of course, $\R P^1$ is spin and so $\tau$-relatively spin.
The map $\tau_*$ is an orientation preserving involution if $\ell_{\beta}$
is even, and an orientation reversing involution if $\ell_{\beta}$
is odd.
\par
(2) Let $M$ be a Calabi-Yau 3-fold and let $L \subset M$ be the set of real points
(i.e., the fixed point set of an anti-holomorphic involutive isometry).
In this case, $L$ is orientable (because it is a special
Lagrangian) and spin (because any orientable 3-manifold is spin).
Furthermore $\mu_L(\beta) = 0$ for any $\beta \in \pi_2(M,L)$.
Therefore Theorem \ref{Proposition38.7} implies that the map $
\tau_*: {\CM}^{\text{reg}}(J;\beta)^{[(V,\sigma)]}
\to {\CM}^{\text{reg}}(J;\beta)^{[(V,\sigma)]}
$ is
orientation preserving for any $\tau$-relative spin structure
$[(V,\sigma)]$.
\end{exm}
\par
We next include marked points.
We consider the moduli space $\CM^{\text{reg}}_{k+1,m}(J;\beta)$.
\par
\begin{prop}\label{Lemma38.9}
The map
$\tau_* : \CM_{k+1,m}^{\operatorname{reg}}(J;\beta)^{\tau^{\ast}[(V,\sigma)]} \to
\CM_{k+1,m}^{\operatorname{reg}}(J;\beta)^{[(V,\sigma)]}
$ is orientation preserving if and only if $\mu_L(\beta)/2 + k + 1 + m$
is even.
\end{prop}
\begin{proof}
Assuming Theorem \ref{Proposition38.7}, we prove
Proposition \ref{Lemma38.9}.
Let us consider the diagram:
$$
\xymatrix{
(S^1)^{k+1} \times (D^2)^m \ar[r]^c & (S^1)^{k+1}\times (D^2)^m\\
((S^1)^{k+1} \times (D^2)^m)_0 \ar[u]^{\text{\footnotesize inclusion}} \ar[r]^c
\ar[d] & ((S^1)^{k+1}\times (D^2)^m)_0 \ar[u]^{\text{\footnotesize inclusion}} \ar[d]
\ar[u]^{\text{\footnotesize inclusion}}\\
\widetilde\CM_{k+1,m}^{\operatorname{reg}}(J;\beta)^{\tau^{\ast}[(V,\sigma)]}
\ar[r]^{\text{Prop \ref{Lemma38.9}}} \ar[d]_{\mathfrak{forget}} &
\widetilde\CM_{k+1,m}^{\operatorname{reg}}(J;\beta)^{[(V,\sigma)]}
\ar[d]_{\mathfrak{forget}} \\
\widetilde\CM^{\operatorname{reg}}(J;\beta)^{\tau^{\ast}[(V,\sigma)]}
\ar[r]^{\text{Thm \ref{Proposition38.7}}} & \widetilde\CM^{\operatorname{reg}}(J;\beta)^{[(V,\sigma)]}
}
$$
\centerline{\bf Diagram \ref{sec:inducedtau}.1}
\par\medskip
\noindent
Here $c$ is defined by
$$c(z_0, z_1, \dots, z_k,z_1^{+},\dots,
z_m^{+}) =
(\overline z_0, \overline z_1,\dots, \overline z_k,
\overline {z}_1^{+},\dots,
\overline {z}_m^{+})$$
and ${\mathfrak{forget}}$ are the forgetful maps of marked points. Here we denote by
$((S^1)^{k+1} \times (D^2)^m)_0$ the set of all
$c(z_0, z_1, \dots, z_k,z_1^{+},\dots,
z_m^{+})$ such that $z_i \ne z_j$, $z_i^+ \ne z_j^+$ for $i \ne j$.
\par
Proposition \ref{Lemma38.9} then follows from Theorem \ref{Proposition38.7} and the
fact that the $\Z_2$-action $\varphi \mapsto \overline\varphi$
on $PSL(2,\R)$ given by \eqref{eq:barPSL2R}
is orientation preserving.
\end{proof}
We next extend $\tau_*$ to the compactification $\CM_{k+1,m}(J;\beta)$ of
$\CM_{k+1,m}^{\rm reg}(J;\beta)$
and define
a continuous map
$$
\tau_* :  \CM_{k+1,m}(J;\beta)^{\tau^{\ast}[(V,\sigma)]}
\longrightarrow
\CM_{k+1,m}(J;\beta)^{[(V,\sigma)]}.
$$
\begin{thm}\label{Proposition38.11}
\begin{enumerate}
\item The map
$\tau_* : \CM_{k+1,m}^{\operatorname{reg}}(J;\beta)^{\tau^{\ast}[(V,\sigma)]}
\longrightarrow
\CM_{k+1,m}^{\operatorname{reg}}(J;\beta)^{[(V,\sigma)]}$
extends to an automorphism
$\tau_*$, denoted by the same symbol:
\begin{equation}\label{38.10}
\tau_* :  \CM_{k+1,m}(J;\beta)^{\tau^{\ast}[(V,\sigma)]} \longrightarrow
\CM_{k+1,m}(J;\beta)^{[(V,\sigma)]}
\end{equation} \\
between spaces with Kuranishi structures.
\item
It preserves orientation if and only if
$\mu_L(\beta)/2 + k + 1 + m$ is even.
In particular, if $[(V,\sigma)]$ is a $\tau$-relative spin structure, it can be
regarded as an involution on the space
$\CM_{k+1,m}(J;\beta)^{[(V,\sigma)]}$ with Kuranishi structure.
\end{enumerate}
\end{thm}
\begin{proof}
(1) The proof of (1) is given right after the proof of Theorem \ref{Lemma38.14}
in Section \ref{sec:Proofth}.
\par
(2) The statement follows from the corresponding
statement on $\CM_{k+1,m}^{\text{reg}}(J;\beta)$ in Proposition \ref{Lemma38.9}.
For $((\Sigma, \vec{z}, \vec{z}\,^+), w) \in \mathcal{M}_{k+1,m}(J,\beta)$, we denote by
$$((\Sigma_i,\vec{z}\,^{(i)}, \vec{z}\,^{+(i)}),w_{(i)}) \in
\mathcal{M}^{\operatorname{reg}}_{k_i+1, m_i}(J,\beta_{(i)})$$ the irreducible disc components and by
$((\Sigma_j^S, \vec{z}\,^{+(j)_S}),u_{(j)}) \in \mathcal{M}^{\operatorname{sph}}_{\ell_j}(\alpha)$ the irreducible sphere components.
\begin{itemize}
\item
By Proposition \ref{Lemma38.9}, we find that $\tau_*$ respects the orientation of $\mathcal{M}^{\operatorname{reg}}_{k_i+1, m_i}(J,\beta_{(i)})$
if and only if $\mu(\beta_{i})/2 + k_i + 1 + m_i$ is even.
\item
In the same way, we find that $\tau_*$ respects the orientation of $ \mathcal{M}^{\operatorname{sph}}_{\ell_j}(\alpha)$ if and only if
$n + c_1(M)[\alpha ] + \ell_j - 3$ is even.
\item
$m \equiv \sum_i m_i + \sum_j \ell_j \mod 2$ and $k+1 \equiv \sum_i (k_i + 1) \mod 2$.
\item The number of interior nodes is equal to the number of sphere components,
since $\Sigma$ is a bordered stable curve of genus $0$ such that $\partial \Sigma$ is connected.
\item The involution $\tau_*$ acts on the space of parameters for smoothing interior nodes with orientation preserving if and only if
the number of interior nodes is even.
\item
The fiber product is taken over either $L$ or $M$.  The involution $\tau$ respects the orientation of $M$ if and only if $n$ is even.
\end{itemize}
\par
Combining these with Lemma 8.2.3 (4) in \cite{fooobook2}, we obtain that
$\tau_*$ respects the orientation on $\mathcal{M}_{k+1,m}(J,\beta)$ if and only if $\mu(\beta)/2 + k+1 + m$ is even.
Hence we obtain the second statement of the theorem.
\end{proof}
\subsection{The map $\tau_{\ast}^{\rm main}$ and orientation}
\label{subsec:taumain}
We next restrict our maps to the main component of $\CM_{k+1,m}(J;\beta)$.
As we mentioned before,
we observe that the induced map $\tau_* : \CM_{k+1,m}(J;\beta)
\to \CM_{k+1,m}(J;\beta)$ does {\it not} preserve the main component
for $k > 1$. On the other hand the assignment given by
\begin{equation}\label{38.13}
\aligned
& \quad (w, \vec{z}, \vec{z}\,^+)=(w,(z_0, z_1, z_2,\dots, z_{k-1},z_k),(z_1^+, \dots,z_m^+)) \\
& \longmapsto
(\widetilde {w}, \vec{\overline z}^{\rm ~ rev}, \vec{\overline z}\,^+)=(\widetilde{w}, (\overline{z}_0, \overline{z}_k, \overline{z}_{k-1},\dots,
\overline{z}_2,\overline{z}_1),(\overline{z}_1^+,
\dots,\overline{z}_m^+))
\endaligned
\end{equation}
respects the counter-clockwise cyclic order of $S^1 = \partial D^2$ and so
preserves the main component, where $\widetilde{w}$ is as in
(\ref{38.4}).
Therefore we consider this map instead which we denote by
\begin{equation}\label{taumain}
\tau_*^{\operatorname{main}} :
\CM_{k+1,m}^{\operatorname{main}}(J;\beta)^{\tau^{\ast}[(V,\sigma)]}
\longrightarrow
\CM_{k+1,m}^{\operatorname{main}}(J;\beta)^{[(V,\sigma)]}.
\end{equation}
We note that for $k=0,1$ we have
\begin{equation}\label{taumaink=0,1}
\tau^{\rm main}_{\ast} = \tau_{\ast}.
\end{equation}
\begin{thm}\label{Lemma38.14}
The map $\tau_*^{\operatorname{main}}$ is induced by an automorphism
between the spaces with Kuranishi structures and satisfies $\tau_*^{\operatorname{main}}\circ
\tau_*^{\operatorname{main}} = \operatorname{id}$. In particular, if $[(V,\sigma)]$ is
$\tau$-relatively spin, it defines an involution of the space
$
\CM_{k+1,m}^{\operatorname{main}}(J;\beta)^{[(V,\sigma)]}
$ with Kuranishi
structure.
\end{thm}
The proof will be given in Section \ref{sec:Proofth}.
\par
We now have the following commutative diagram:
$$
\xymatrix{
(S^1)^{k+1}\times (D^2)^m \ar[r]^{c'} & (S^1)^{k+1}\times (D^2)^m\\
((S^1)^{k+1}\times (D^2)^m)_{00} \ar[u]^{\text{\footnotesize inclusion}}\ar[r]^{c'} \ar[d]
& ((S^1)^{k+1}\times (D^2)^m)_{00} \ar[u]^{\text{\footnotesize inclusion}} \ar[d]\\
\widetilde\CM_{k+1,m}^{\operatorname{main}}(J;\beta)^{\tau^{\ast}[(V,\sigma)]}
\ar[r]^{\tau_*^{\operatorname{main}}} \ar[d]_{\mathfrak{forget}} &
\widetilde\CM_{k+1,m}^{\operatorname{main}}(J;\beta)^{[(V,\sigma)]}
\ar[d]_{\mathfrak{forget}} \\
\widetilde\CM(J;\beta)^{\tau^{\ast}[(V,\sigma)]}
\ar[r]^{\tau_*} & \widetilde\CM(J;\beta)^{[(V,\sigma)]}
}
$$
\centerline{\bf Diagram \ref{sec:inducedtau}.2}
\par\medskip
\noindent
Here $c'$ is defined by
$$c'(z_0, z_1 \dots, z_k,
z_1^{+},\dots,
z_m^{+}) =
(\overline z_0, \overline z_k,\dots, \overline z_1,
\overline{z}_1^{+},\dots,
\overline{z}_m^{+})$$
and ${\mathfrak{forget}}$ are the forgetful maps of marked points.
In the diagram, $((S^1)^{k+1} \times (D^2)^m)_{00}$ is the open subset of
$(S^1)^{k+1} \times (D^2)^m$ consisting of the points
such that all $z_i$'s and $z_j^+$'s are distinct respectively.

Let $\text{\rm Rev}_k : L^{k+1} \to L^{k+1}$ be the map defined by
$$
\text{\rm Rev}_k(x_0, x_1,\dots,x_k) = (x_0,x_k,\dots,x_1).
$$
It is easy to see that
\begin{equation}\label{38.15}
ev\circ \tau_*^{\text{main}} = \text{\rm Rev}_k \circ ev.
\end{equation}
We note again that $\text{\rm Rev}_k = \operatorname{id}$
and $\tau_*^{\operatorname{main}}= \tau_*$ for $k=0,1$.
\par
Let $P_1,\dots,P_k$
be smooth singular simplices on $L$. By taking the fiber product and using
(\ref{38.13}), we obtain a map
\begin{equation}\label{38.16}
\tau_*^{\operatorname{main}} : \CM^{\operatorname{main}}_{k+1,m}(J;\beta ;
P_1,\dots,P_k)^{\tau^{\ast}[(V,\sigma)]}
\to \CM^{\operatorname{main}}_{k+1,m}(J; \beta ;
P_k,\dots,P_1)^{[(V,\sigma)]}
\end{equation}
which satisfies
$\tau_*^{\operatorname{main}} \circ \tau_*^{\operatorname{main}}
= \operatorname{id}$.
We put
\begin{equation}\label{epsilonmain}
\epsilon = \frac{\mu_L(\beta)}{2} + k + 1 + m + \sum_{1 \le i < j \le k}
\deg'P_i\deg'P_j.
\end{equation}
\begin{thm}\label{Lemma38.17}
The map $(\ref{38.16})$ preserves orientation if $\epsilon$ is even, and
reverses orientation if $\epsilon$ is odd.
\end{thm}
The proof of Theorem \ref{Lemma38.17} is given in Section \ref{sec:Proofth}.

\section{Proofs of Theorems \ref{Proposition38.7},
\ref{Proposition38.11} (1), \ref{Lemma38.14} and  \ref{Lemma38.17}}
\label{sec:Proofth}
In this section we prove Theorems \ref{Proposition38.7}
(= Theorem \ref{thm:fund}), \ref{Proposition38.11} (1), \ref{Lemma38.14}  and
 \ref{Lemma38.17}
(= Theorem \ref{withsimplex}) stated in the previous sections.
\begin{proof}[Proof of Theorem \ref{Proposition38.7}.]
Pick
$J \in \CJ_\omega^\tau$, a $\tau$-anti-invariant almost complex structure compatible with $\omega$.
For a $J$ holomorphic curve $w :
(D^2,\partial D^2) \to (M,L)$, we recall that we define $\widetilde w$ by
$$
\widetilde w(z) = (\tau\circ w)(\overline z).
$$
Moreover
for $(D^2,w) \in \CM^{\text{reg}}(J;{\beta})$,
$((D^2,\vec z,{\vec z}\,^+),w) \in \CM^{\text{reg}}_{k+1,m}(J;{\beta})$
we define
$$
\tau_*((D^2,w)) = (D^2,\widetilde w),
\qquad
\tau_*(((D^2,\vec z,\vec z\,^+),w)) = ((D^2,\vec{\overline z},
\vec{\overline z}\,^{+}),\widetilde w),
$$
where
$$
\vec{\overline z} = (\overline z_0,\dots,\overline z_k),
\qquad
\vec{\overline z}\,^+ = (\overline z^+_0,\dots,\overline z^+_m).
$$
Let $[D^2,w]\in \CM^{\rm reg}(J;\beta)$. We consider the
deformation complex
\begin{equation}\label{47.6.1}
D_w\overline\partial: \Gamma(D^2,\partial D^2: w^*TM,
w\vert_{\partial D^2}^*TL) \longrightarrow \Gamma(D^2;\Lambda^{0,1}\otimes w^*TM)
\end{equation}
and
\begin{equation}\label{47.6.2}
D_{\widetilde w}\overline\partial: \Gamma(D^2,\partial D^2: {\widetilde w}^*TM,
{\widetilde w}\vert_{\partial D^2}^*TL) \longrightarrow \Gamma(D^2;\Lambda^{0,1}\otimes {\widetilde w}^*TM),
\end{equation}
where $D_{w}\overline{\partial}$ is the linearized operator
of pseudo-holomorphic curve equation as (\ref{linearizedeq}).
(Here and hereafter, $\Lambda^1 = \Lambda^{1,0} \oplus \Lambda^{0,1}$ is the decomposition of the
complexified cotangent bundle of the {\it domain} of
pseudo-holomorphic curves.)
\par
We have the commutative diagram
\vskip0.2cm
$$
\xymatrix{
(w^*TM, w\vert_{\partial D^2}^*TL) \ar[r]^{T\tau} \ar[d] &
(\widetilde w^*TM, \widetilde w\vert_{\partial D^2}^*TL) \ar[d]\\
(D^2, \partial D^2) \ar[r]^c & (D^2, \partial D^2)
}
$$
\centerline{\bf Diagram \ref{sec:Proofth}.1}
\par\bigskip
\noindent
where $c(z) = \overline z$ and we denote by $T\tau$ the differential of $\tau$.
It induces a bundle map
\begin{equation}\label{47.7}
\operatorname{Hom}_{\R}(TD^2,w^*TM) \longrightarrow
\operatorname{Hom}_{\R}(TD^2,{\widetilde w}^*TM),
\end{equation}
which covers $z \mapsto \overline z$. The bundle map
(\ref{47.7}) is fiberwise anti-complex linear i.e.,
$$
\operatorname{Hom}_{\R}(T_zD^2,T_{w(z)}M) \longrightarrow
\operatorname{Hom}_{\R}(T_{\overline z} D^2,T_{\tau(w(z))}M),
$$
is anti-complex linear at each $z \in D^2$
with respect to both of the complex structures
$a \mapsto J \circ a$
and
$a \mapsto a\circ j$.
Therefore it preserves the decomposition
\begin{equation}\label{47.8}
\operatorname{Hom}_{\R}(TD^2,w^*TM) \otimes \C = (\Lambda^{1,0} \otimes w^*TM)
\oplus  (\Lambda^{0,1} \otimes w^*TM),
\end{equation}
since (\ref{47.8}) is the decomposition to the complex and anti-complex
linear parts. Hence we obtain a map
\begin{equation}\label{Ttau1}
(T_{w,1}\tau)_* : \Gamma(D^2;\Lambda^{0,1}\otimes w^*TM)
\longrightarrow \Gamma(D^2;\Lambda^{0,1} \otimes {\widetilde w}^*TM)
\end{equation}
which is anti-complex linear. In the similar way, we obtain an anti-complex linear map:
$$
(T_{w,0}\tau)_* :  \Gamma(D^2,\partial D^2: w^*TM,
w\vert_{\partial D^2}^*TL) \longrightarrow \Gamma(D^2,\partial D^2: {\widetilde w}^*TM,
{\widetilde w}\vert_{\partial D^2}^*TL).
$$
Since $\tau$ is an isometry, it commutes with the
covariant derivative. This gives rise to the following commutative diagram.
$$
\xymatrix{
\Gamma(D^2,\partial D^2: w^*TM,
w\vert_{\partial D^2}^*TL) \ar[d]_{(T_{w,0}\tau)_*} \ar[r]^-{D_w\overline\partial}
& \Gamma(D^2;\Lambda^{0,1} \otimes w^*TM) \ar[d]_{(T_{w,1}\tau)_*}\\
\Gamma(D^2,\partial D^2: {\widetilde w}^*TM,
{\widetilde w}\vert_{\partial D^2}^*TL) \ar[r]^-{D_{\widetilde w}\overline\partial}
&  \Gamma(D^2;\Lambda^{0,1} \otimes {\widetilde w}^*TM)
}
$$
\vskip0.2cm
\centerline{\bf Diagram \ref{sec:Proofth}.2}
\par\medskip
We study the orientation. Let $w \in \widetilde\CM^{\rm reg} (J;\beta)$ and
consider $\widetilde w \in \widetilde\CM^{\rm reg} (J;\beta)$.
We consider the commutative Diagram \ref{sec:Proofth}.1.
A trivialization
$$
\Phi: (w^*TM, w\vert_{\partial D^2}^*TL) \longrightarrow (D^2, \partial D^2; \C^n, \Lambda)
$$
naturally induces a trivialization
$$
\widetilde \Phi: (\widetilde w^*TM, \widetilde w\vert_{\partial D^2}^*TL)
\longrightarrow (D^2, \partial D^2; {\C}^n, \widetilde \Lambda),
$$
where $\Lambda: S^1 \simeq \partial D^2 \to \Lambda (\C^n)$
is a loop of Lagrangian subspaces given by
$\Lambda(z) : = T_{w(z)}L$ in the trivialization
and $\widetilde \Lambda$ is defined by 
\begin{equation}\label{widetildeLambda}
\widetilde \Lambda(z) = \overline{\Lambda(\overline z)}.  
\end{equation}
With respect to these trivializations, we have the commutative diagram
$$
\xymatrix{
(D^2,\partial D^2; \C^n, \Lambda) \quad
\ar[r]^{\widetilde \Phi\circ T\tau
\circ \Phi^{-1}} \ar[d]  & \quad
(D^2,\partial D^2; {\C}^n, \widetilde\Lambda) \ar[d] \\
(D^2, \partial D^2) \ar[r]^c & (D^2, \partial D^2)
}
$$
\vskip0.2cm
\centerline{\bf Diagram \ref{sec:Proofth}.3}
\par\bigskip
\noindent and the elliptic complex (\ref{47.6.1}) and (\ref{47.6.2}) are identified with 
$\overline{\partial}_{(D^2,\partial D^2; \C^n, \Lambda)}$ and 
$\overline{\partial}_{(D^2,\partial D^2; {\C}^n, \widetilde{\Lambda})}$, respectively.  
The relative spin structure $\tau^*[(V,\sigma)]$ (resp. $[(V,\sigma)]$) determines a trivialization $\Lambda \cong \partial D^2 \times \R^n$ 
(resp. $\widetilde{\Lambda} \cong \partial D^2 \times \R^n$) unique up to homotopy. 
These trivializations are compatible with $\widetilde \Phi \circ T\tau \circ \Phi^{-1}$ in Diagram 5.3.  

We recall the argument explained in Remark \ref{rem:prop8.1.4}.
We have the complex vector bundle $E'$ over the nodal curve
$\Sigma = D^2 \cup \C P^1$ with a nodal point $D^2 \ni O=p \in \C P^1$.  
The topology of the bundle $E'\vert_{\C P^1} \to \C P^1$ is determined by
the loop $\Lambda$ of Lagrangian subspaces
defined by $\Lambda(z) = T_{w(z)}L$ in the trivialization.
The Cauchy-Riemann operator 
$\overline{\partial}_{(D^2 \times \C ^n, \partial D^2 \times \R ^n)}$ 
is surjective and 
the Cauchy-Riemann operator 
$$
\overline{\partial}_{E'\vert_{\C P^1}}: \Gamma ({\C P^1}; E'\vert _{\C P^1}) \to \Gamma ({\C P^1}; \Lambda^{0,1} \otimes E'\vert_{\C P^1})
$$ 
is 
approximated by a finite dimensional model which is $ 0{\text{-map}} : H^0({\C P^1}; E'\vert_{\C P^1}) \to H^1({\C P^1}; E'\vert_{\C P^1})$, 
where $H^1({\C P^1}; E'\vert_{\C P^1})$
is regarded as the obstruction bundle.   
For a later purpose, we take a {\it stabilization} of this finite dimensional model so that the evaluation at $p$ is surjective to $E'_p$.  
Namely, we take a finite dimensional complex linear subspace 
${\mathbb V}^+ \subset \Gamma ({\mathbb C}P^1; \Lambda^{0,1} \otimes E'\vert_{{\mathbb C}P^1})$ such that 
the Cauchy-Riemann operator $\overline{\partial}_{E'\vert_{{\mathbb C}P^1}}$ is surjective modulo ${\mathbb V}^+$ and 
the evaluation at $p$ $ev_p:(\overline{\partial}_{E'\vert_{{\mathbb C}P^1}})^{-1}({\mathbb V}^+) \to E'\vert_p$  is surjective.  
Set $\mathbb V = {\mathbb V}^+ \cap \text{\rm Im~} \overline{\partial}_{E'\vert_{{\mathbb C}P^1}} \subset \Gamma ({\mathbb C}P^1; \Lambda^{0,1} \otimes E'\vert_{{\mathbb C}P^1})$.   
Then we have an isomorphism ${\mathbb V}^+ \cong {\mathbb V} \oplus H^1({\C P^1}; E'\vert_{\C P^1})$ and 
$(\overline{\partial}_{E'\vert_{{\mathbb C}P^1}})^{-1}({\mathbb V}^+) =(\overline{\partial}_{E'\vert_{{\mathbb C}P^1}})^{-1}({\mathbb V})$.   
The Cauchy-Riemann operator $\overline{\partial}_{E'\vert_{{\mathbb C}P^1}}$  has a finite dimensional approximation by 
$$s:(\overline{\partial}_{E'\vert_{{\mathbb C}P^1}})^{-1}({\mathbb V}) \to {\mathbb V} \oplus H^1({\C P^1}; E'\vert_{\C P^1}), 
\ \  s(\xi)=(\overline{\partial}_{E'\vert_{{\mathbb C}P^1}}(\xi), 0).$$  
We use the notation in Convention 8.2.1 (3) (4) in \cite{fooobook2} 
to describe 
the kernel of the operator $\overline{\partial}_{(E', \lambda')}$ as  the zero set of $s$ in 
the fiber product of the kernel of $\overline{\partial}_{(D^2 \times \C ^n, \partial D^2 \times \R ^n)}$ and 
$(\overline{\partial}_{E'\vert_{{\mathbb C}P^1}})^{-1}({\mathbb V})$.  
We decompose it as follows.  
$$
(\overline{\partial}_{E'\vert_{{\mathbb C}P^1}})^{-1}({\mathbb V})= E'_p \times ~^{\circ} 
(\overline{\partial}_{E'\vert_{{\mathbb C}P^1}})^{-1}({\mathbb V}).
$$ 
Here $~^{\circ} (\overline{\partial}_{E'\vert_{{\mathbb C}P^1}})^{-1}({\mathbb V})$ is the space of sections 
in $(\overline{\partial}_{E'\vert_{{\mathbb C}P^1}})^{-1}({\mathbb V})$, which 
vanish at $p$.  (See (8.2.1.6) in \cite{fooobook2} for the notation used here.)  

Since $\operatorname{ker} \overline{\partial}_{(D^2 \times \C ^n, \partial D^2 \times \R ^n)}
 \cong \R ^{n}$ by (\ref{eq:isokernel}), the complex conjugate induces 
the trivial action on 
$\operatorname{ker} \overline{\partial}_{(D^2 \times \C ^n, \partial D^2 \times \R ^n)}$.  
Therefore $(T\tau)_*:\det D_w \overline{\partial} \to \det D_{\widetilde w} \overline{\partial}$ is orientation preserving or not if and only if so is the complex conjugation action on 
$$ 
\det({\mathbb V} \oplus H^1({\C P^1};E'\vert_{\C P^1}))^* \otimes \det ~^{\circ}\bigl((\overline{\partial}_{E'\vert_{{\mathbb C}P^1}})^{-1}({\mathbb V})\bigr),  
$$
which is isomorphic to 
$$ 
\det({\mathbb V} \oplus H^1({\C P^1};E'\vert_{\C P^1}))^* \otimes \det (E'\vert_p^*) \otimes 
\det \bigl((\overline{\partial}_{E'\vert_{{\mathbb C}P^1}})^{-1}({\mathbb V})\bigr).  
$$
On the other hand, we observe that 
$$ 
\det({\mathbb V} \oplus H^1({\C P^1};E'\vert_{\C P^1}))^*  \otimes 
\det \bigl((\overline{\partial}_{E'\vert_{{\mathbb C}P^1}})^{-1}({\mathbb V})\bigr) 
$$
is isomorphic to 
$$ 
\det (H^1({\C P^1};E'\vert_{\C P^1}))^* \otimes \det \text H^0(\C P^1;E'\vert_{\C P^1}),   
$$  
on which the complex conjugation acts by 
the multiplication by $(-1)^{\mu(\Lambda)/2 + n}$, 
since 
$$\dim_{\mathbb C} H^0(\C P^1;E'\vert_{\C P^1}) - \dim_{\mathbb C} H^1({\C P^1};E'\vert_{\C P^1}) = \frac12 \mu (\Lambda)+ n .$$
Here $n$ is the rank of $E'$ as a complex vector bundle.  
Note also that the complex conjugation acts on $E'\vert_p$ by the multiplication by $(-1)^n$.  
Combining these, we find that the complex conjugation acts on 
$$ 
\det({\mathbb V} \oplus H^1({\C P^1};E'\vert_{\C P^1}))^* \otimes \det ~^{\circ}\bigl((\overline{\partial}_{E'\vert_{{\mathbb C}P^1}})^{-1}({\mathbb V})\bigr),
$$
by the multiplication by $(-1)^{\mu(\Lambda)/2 }$.  
Note that the action of $T_{w} \tau$ on the determinant bundle of $D_w\overline{\partial}$ is isomorphic to 
the conjugation action explained above.  
Therefore this map is orientation preserving if and only
if $\frac1{2}\mu(\Lambda) \equiv 0 \mod 2$, i.e., $\mu(\Lambda) \equiv 0
\mod 4$. We note, by definition,
that $\mu_L(\beta)=\mu (\Lambda)$.
This finishes the proof of Theorem \ref{Proposition38.7}. 
\end{proof}

The argument above is an adaptation of Lemma 8.3.2 (4) in \cite{fooobook2} 
with $X_1=\text{\rm index}~ \overline{\partial}_{(D^2 \times {\mathbb C}^n, \partial D^2 \times {\mathbb R}^n)}$, 
$X_2=\text{\rm index}~ (\overline{\partial}_{E'\vert_{{\mathbb C}P^1}})^{-1}({\mathbb V})$ and $Y=E'\vert_p$.  
The complex conjugation action on $X_1$ (resp. $X_2$, $Y$) is a $+1$ (resp. $(-1)^{\mu(\Lambda)/2+n}$, $(-1)^n$)-oriented 
isomorphism.  Hence the action on $X_1 \times_Y X_2$ is a $(-1)^{\mu(\Lambda)/2}$-oriented isomoprhism.

\par\medskip
\begin{proof}[Proof of Theorem \ref{Lemma38.14}]
We will extend the map $\tau_{\ast}^{\rm main}$
(See (\ref{taumain}).) to
an automorphism of Kuranishi structure by a triple induction over
$\omega(\beta)= \int_{\beta}\omega$, $k$ and $m$.
Namely we define an order on the set of triples
$(\beta, k, m)$ by the relation
\begin{subequations}\label{tripleorder}
\begin{equation}
\omega(\beta')  <  \omega(\beta).
\end{equation}
\begin{equation}
\omega(\beta')  =  \omega(\beta), \quad k' < k.
\end{equation}
\begin{equation}
\omega(\beta')  =  \omega(\beta), \quad k' = k, \quad m' <m.
\end{equation}
\end{subequations}

We will define the extension of (\ref{taumain}) for $(\beta, k, m)$ to
an automorphism of Kuranishi structure
under the assumption that such extension is already defined for
all $(\beta',k',m')$ smaller than $(\beta, k, m)$ with respect to this order.

Firstly, we consider the case that the domain is irreducible, i.e.,
a pseudo-holomorphic map $w:(D^2, \partial D^2) \to (M,L)$.
Let $((D^2,\vec z,\vec z\,^+),w)$ be an element of
$\CM^{\operatorname{main}}_{k+1,m}(J;\beta)^{[(V,\sigma)]}$.

We consider $((D^2,\vec z\,',\vec z\,^{+\prime}),w')$ where
$(D^2,\vec z\,',\vec z\,^{+\prime})$ is close to $(D^2,\vec z,\vec z\,^+)$ in the
moduli space of discs with $k+1$ boundary and $m$ interior marked points.
We use local trivialization of this moduli space
to take a diffeomorphism  $(D^2,\vec z) \cong (D^2,\vec z\,')$.
(In case $2m + k <2$ we take additional interior marked points
to stabilize the domain. See \cite{FO} appendix and
\cite{foootech} Definition 18.9.)
We assume $w' : (D^2,\partial D^2) \to (M,L)$ is
$C^1$-close to $w$, using this identification.
To define a Kuranishi chart in a neighborhood of $[((D^2,\vec z,\vec z\,^+),w)]$ we
take a family of finite dimensional subspaces
$E_{[((D^2,\vec z,\vec z\,^+),w)]}((D^2,\vec z\,',\vec z\,^{+\prime}),w')$ of $\Gamma(D^2;\Lambda^{0,1} \otimes (w')^*TM)$
such that
$$
\operatorname{Im}D_{w'}\overline\partial + E_{[((D^2,\vec z,\vec z\,^+),w)]}((D^2,\vec z\,',\vec z\,^{+\prime}),w')
= \Gamma(D^2;\Lambda^{0,1}\otimes (w')^*TM).
$$
The Kuranishi neighborhood $V_{[((D^2,\vec z,\vec z\,^+),w)]}$ is
constructed in Section 7.1 in \cite{fooobook2} and
is given by the set of solutions of the
equation
\begin{equation}\label{4-16}
\overline\partial w' \equiv 0  \mod E_{[((D^2,\vec z,\vec z\,^+),w)]}((D^2,\vec z\,',\vec z\,^{+\prime}),w').
\end{equation}
Moreover we will take it so that the evaluation maps
$$
V_{[((D^2,\vec z,\vec z\,^+),w)]} \to L^{k+1}
$$
defined by $((D^2,\vec z\,',\vec z\,^{+\prime}),w')
\mapsto (w'(z'_1),\dots,w'(z'_k),w'(z'_0))$ is a submersion.
\par
We choose $E_{[((D^2,\vec z,\vec z\,^+),w)]}((D^2,\vec z\,',\vec z\,^{+\prime}),w')$ so that it is invariant under $\tau^{\rm main}$,
in the following sense.
We define $\widetilde w$, $\widetilde w'$, and $\vec{\overline z}\,^+$,
$\vec{\overline z}\,^{+\prime}$ as in
(\ref{38.5}).
We also define $\vec{\overline z}^{\rm ~rev}$,
$\vec{\overline z}\,^{\prime {\rm ~rev}}$ as in  (\ref{38.13}).
(So $\tau^{\rm main}_*([((D^2,\vec z,\vec z\,^+),w)])
= [((D^2,\vec{\overline z}^{\rm ~rev},\vec{\overline z}\,^+),\widetilde w)]$.)
Then we require:
\begin{equation}\label{4-17}
\aligned
&E_{[((D^2,\vec{\overline z}^{\rm ~rev},\vec{\overline z}\,^+),\widetilde w)]}
([((D^2,\vec{\overline z}\,^{\prime \rm rev},\vec{\overline z}\,^{+\prime}),\widetilde w')])
\\
&= (T_{w',1}\tau)_*(E_{[((D^2,\vec z,\vec z\,^+),w)]}((D^2,\vec z\,',\vec z\,^{+\prime}),w')).
\endaligned
\end{equation}
Here $(T_{w',1}\tau)_*$ is as in (\ref{Ttau1}).
If (\ref{4-16}) is satisfied then
it is easy to see the following.
\begin{enumerate}
\item[(*)]
If $w'$ satisfies
(\ref{4-16}) then
$
\overline\partial \widetilde w' \equiv 0  \mod E_{[((D^2,\vec{\overline z}^{\rm ~rev},\vec{\overline z}\,^+),\widetilde w)]}.
$
\end{enumerate}
$(*)$ implies that the map $[(D^2,\vec z\,',\vec z\,^{+\prime}),w')] \mapsto
[((D^2,\vec{\overline z}\,^{\prime \rm ~rev},\vec{\overline z}\,^{+\prime}),\widetilde w')]$ defines a diffeomorphism
between Kuranishi neighborhoods of
$[((D^2,\vec z,\vec z\,^+),w)]$ and of
$[((D^2,\vec{\overline z}^{\rm ~rev},\vec{\overline z}\,^+),\widetilde w)]$.
Moreover the Kuranishi map
$
[(D^2,\vec z\,',\vec z\,^{+\prime}),w')] \mapsto s([(D^2,\vec z\,',\vec z\,^{+\prime}),w')]) = \overline\partial w'
$
commutes with $\tau^{\rm main}_*$.
Therefore, $\tau^{\rm main}_*$ induces an isomorphism of
our Kuranishi structure on
$\CM_{k+1,m}^{\operatorname{main}}(J;\beta)$.
\par\smallskip
In order to find $E_{[((D^2,\vec z,\vec z\,^+),w)]}((D^2,\vec z\,',\vec z\,^{+\prime}),w')$
satisfying (\ref{4-17}) in addition, we proceed as follows.
We first recall briefly the way how we defined it in \cite{FO} appendix, \cite{fooobook2} pages 423-424,
\cite{foootech} Section 18.
We take a sufficiently dense subset $\{((D^2_{\frak a},\vec z_{\frak a},\vec z\,^+_{\frak a}),w_{\frak a}) \mid \frak a \in \frak A\}
\subset \CM_{k+1,m}^{\operatorname{main}}(J;\beta)$.
We choose
$$
E_{\frak a}
\subset
\Gamma(D^2;\Lambda^{0,1}\otimes (w_{\frak a})^*TM)
$$
that is a finite dimensional vector space of smooth sections
whose supports are away from node or marked points and
satisfy
$$
\operatorname{Im}D_{w_{\frak a}}\overline\partial + E_{\frak a}
= \Gamma(D^2;\Lambda^{0,1}\otimes (w_{\frak a})^*TM).
$$
Moreover we assume
$$
\bigoplus dev_{z_i} :
(\operatorname{Im}D_{w_{\frak a}}\overline\partial)^{-1}(E_{\frak a})
\to \bigoplus_{i=0,\dots,k}T_{w_{\frak a}(z_i)}L
$$
is surjective. See \cite{fooobook2} Lemma 7.1.18.
\par
We take a sufficiently small neighborhood $W_{\frak a}$ of $w_{\frak a}$
such that
$$
\bigcup_{\frak a \in \frak A} W_{\frak a}
=
\CM_{k+1,m}^{\operatorname{main}}(J;\beta).
$$
Let $[((D^2,\vec z,\vec z\,^+),w)] \in W_{\frak a}$
and
$((D^2,\vec z\,',\vec z\,^{+\prime}),w')$ be as above.
We take an isomorphism
\begin{equation}
I_{w',\frak a} :
\Gamma(D^2,\partial D^2: w_{\frak a}^*TM,
w_{\frak a}\vert_{\partial D^2}^*TL)
\cong
\Gamma(D^2,\partial D^2: (w')^*TM,
w'\vert_{\partial D^2}^*TL).
\end{equation}
Then we define
\begin{equation}
E_{[((D^2,\vec z,\vec z\,^+),w)]}((D^2,\vec z\,',\vec z\,^{+\prime}),w')
=
\bigoplus_{\frak a \in \frak A, \atop w' \in W_{\frak a}}
I_{w',\frak a}(E_{\frak a}).
\end{equation}
(It is easy to see that we can perturb $E_{\mathfrak a}, {\mathfrak a} \in {\mathfrak A}$, a bit so that the right hand side is a 
direct sum. See \cite{foootech} Section 27 for example.)
\par
We now explain the way how we take
$E_{[((D^2,\vec z,\vec z\,^+),w)]}((D^2,\vec z\,',\vec z\,^{+\prime}),w')$
so that (\ref{4-17}) is satisfied.
\par
We first require that
\begin{enumerate}
\item[(i)]
The set
$\{((D^2_{\frak a},\vec z_{\frak a},\vec z\,^+_{\frak a}),w_{\frak a}) \mid \frak a \in \frak A\}$
is invariant under the $\tau^{\rm main}_*$ action.
\item[(ii)]
If
$((D^2_{\frak b},\vec z_{\frak b},\vec z\,^+_{\frak b}),w_{\frak b})
= \tau^{\rm main}_*((D^2_{\frak a},\vec z_{\frak a},\vec z\,^+_{\frak a}),w_{\frak a})$,
$\frak a \ne \frak b$ then
$
E_{\frak b} =
(T_{w_{\frak a},1}\tau)_*(E_{\frak a}).
$
Moreover $\tau^{\rm main}_*(W_{\frak a}) = W_{\frak b}$.
\item[(iii)]
If
$((D^2_{\frak a},\vec z_{\frak a},\vec z\,^+_{\frak a}),w_{\frak a})
= \tau^{\rm main}_*((D^2_{\frak a},\vec z_{\frak a},\vec z\,^+_{\frak a}),w_{\frak a})$
then
$
E_{\frak a} =
(T_{w_{\frak a},1}\tau)_*(E_{\frak a}).
$
Moreover $\tau^{\rm main}_*(W_{\frak a}) = W_{\frak a}$.
\end{enumerate}
It is easy to see that such a choice exists.
\par
We next choose $I_{w',\frak a}$ such that
the following holds.
\begin{enumerate}
\item[(I)]
If $((D^2_{\frak b},\vec z_{\frak b},\vec z\,^+_{\frak b}),w_{\frak b})
= \tau^{\rm main}_*((D^2_{\frak a},\vec z_{\frak a},\vec z\,^+_{\frak a}),w_{\frak a})$,
$\frak a \ne \frak b$ then
\begin{equation}\label{47.13}
(T_{w',1}\tau)_* \circ I_{w',\frak a} = I_{\widetilde w',\frak b} \circ (T_{w_{\frak a},1}\tau)_*.
\end{equation}
\item[(II)]
If
$((D^2_{\frak a},\vec z_{\frak a},\vec z\,^+_{\frak a}),w_{\frak a})
= \tau^{\rm main}_*((D^2_{\frak a},\vec z_{\frak a},\vec z\,^+_{\frak a}),w_{\frak a})$
then
\begin{equation}\label{47.132}
(T_{w',1}\tau)_* \circ I_{w',\frak a} = I_{\widetilde w',\frak a} \circ (T_{w_{\frak a},1}\tau)_*.
\end{equation}
\end{enumerate}
We can find such $I_{w',\frak a}$ by taking various data
to define this isomorphism (See \cite{foootech} Definition 17.7) to be invariant under the $\tau$ action.
\par
It is easy to see that (i)(ii)(iii) and (I)(II) imply (\ref{4-17}).
\par
We have thus constructed the Kuranishi neighborhood
of the point $ \CM^{\operatorname{main}}_{k+1,m}(J;\beta)$ corresponding to a pseudoholomorphic map from
a disc (without disc or sphere bubbles).
\par
Let us consider an element
$((\Sigma,\vec z,\vec z\,^{ +}),w) \in \CM^{\operatorname{main}}_{k+1,m}(J;\beta)$,
such that $\Sigma$ is not irreducible. (Namely $\Sigma$ is not $D^2$.)
\begin{lem-def}\label{lem51}
 $((\Sigma,\vec z,\vec z\,^{ +}),w)$
is obtained by gluing  elements of
$\CM^{\operatorname{main}}_{k'+1,m'}(J;\beta')$ with
$(\beta',k',m') < (\beta,k,m)$ and
sphere bubbles.
$\tau^{\text{\rm main}}_* ((\Sigma,\vec z,\vec z\,^{ +}),w)$ is defined using
$\tau^{\text{\rm main}}_*$ on $\CM^{\operatorname{main}}_{k'+1,m'}(J;\beta')$ with
$(\beta',k',m') < (\beta,k,m)$.
\end{lem-def}
\begin{proof}
First we suppose that $\Sigma$ has a sphere bubble $S^2 \subset \Sigma$.
We remove it from $\Sigma$ to obtain $\Sigma_0$.
We add one more marked point to $\Sigma_0$ at the location
where the sphere bubble used to be attached.
Then we obtain an element
$$
((\Sigma_0,\vec z,\vec z\,^{(0)}),w_0) \in \CM^{\operatorname{main}}_{k+1,m+1-\ell}(J;\beta').
$$
Here $\ell$ is the number of marked points on $S^2$.
By the induction hypothesis, $\tau^{\text{\rm main}}_*$ is already defined on
$\CM_{k+1,m+1-\ell}^{\text{\rm main}}(J;\beta')$,  since
$\omega(\beta) \geq \omega(\beta')$ and if $\omega(\beta) = \omega(\beta')$, $\ell \geq 2$ .
We denote
$$
 ((\Sigma'_0,\vec z\,',\vec z\,^{(0) \prime}),w_0^\prime)
:=
\tau_*(((\Sigma_0,\vec z,\vec z\,^{(0)}),w_0)).
$$
We define $v: S^2 \to M$ by
$$
v(z) = \tau \circ w\vert_{S^2}(\overline z).
$$
We assume that the nodal point
in $\Sigma_0 \cap S^2$ corresponds to $0 \in \C \cup \{\infty\} \cong
S^2$. We also map $\ell$ marked points on $S^2$ by $z \mapsto \overline z$
whose images we denote by $\vec z\,^{(1)} \in S^2$.
We then glue $((S^2,\vec z\,^{(1)}), v)$ to
$((\Sigma'_0,\vec z\,',\vec z\,^{(0) \prime}), w_0')$ at the
point $0 \in S^2$ and at the
last marked point of $(\Sigma'_0,\vec z\,',\vec z\,^{(0) \prime})$ and obtain
a curve which is to be the definition of
$\tau^{\text{\rm main}}_*(((\Sigma,\vec z,\vec z\,^+),w))$.
\par
Next suppose that there is no sphere bubble on
$\Sigma$.
Let $\Sigma_0$ be the component containing the 0-th marked point.
If there is only one irreducible component
of $\Sigma$, then $\tau^{\text{\rm main}}_*$ is already defined there.
So we assume that there is at least one other disc component. Then $\Sigma$ is a union of $\Sigma_0$ and
$\Sigma_i$ for $i = 1,\dots,m$ ($m \ge 1$).
We regard the unique point in $\Sigma_0 \cap \Sigma_i$
as a marked point of $\Sigma_0$ for $i=1,\dots,m$.
Here each  $\Sigma_i$ itself is a union of disc components and is connected.
We also regard the point in $\Sigma_0 \cap \Sigma_i$ as
$1 \in D^2 \cong \H \cup \{\infty\}$ where $D^2$ is the irreducible component
of $\Sigma_i$ joined to $\Sigma_0$, and also as
one of the marked points of $\Sigma_i$. This defines an element
$((\Sigma_i,\vec z\,^{(i)},\vec z\,^{(i) +}),w_{(i)})$
for each $i = 0,\dots, m$.
By an easy combinatorics and the induction hypothesis, we can show that
$\tau^{\text{\rm main}}_*$ is already constructed on them.
Now we define $\tau^{\text{\rm main}}_*(((\Sigma,\vec z,\vec z\,^+),w))$
by gluing $\tau^{\text{\rm main}}_*(((\Sigma_i,\vec z\,^{(i)},\vec z\,^{(i) +}),w_{(i)}))$ to
$\tau^{\text{\rm main}}_*((\Sigma_0,\vec{z}\,^{(0)},{\vec z}\,^{(0) +}),w_{(0)})$.
\end{proof}
\par
Thus we proved that,
if $\Sigma$ is not irreducible, then $((\Sigma,\vec z,\vec z\,^{ +}),w)$
is obtained by gluing some elements corresponding to
$(\beta',k',m') < (\beta,k,m)$ and
sphere bubbles.
\par
We define
the map $u \mapsto \widetilde u$
by the same formula as (\ref{38.4}) on the moduli space of {\it spheres}.
Then we can regard this map as an involution
on the space with Kuranishi structure in the same way
as the case of discs.
(In other words, we construct $\tau$ invariant Kuranishi structures
on the moduli space of spheres before starting the construction of the  Kuranishi structures
on the moduli space of discs.)
\par
Let  us consider an element
$((\Sigma,\vec z,\vec z\,^{ +}),w) \in \CM^{\operatorname{main}}_{k+1,m}(J;\beta)$,
such that $\Sigma$ is not irreducible.
By Lemma \ref{lem51} and the above remark, the
involution of its Kuranishi neighborhood is constructed
by the induction hypothesis, on each irreducible component (which is either a disc or a sphere).
A Kuranishi neighborhood of
$((\Sigma,\vec z,\vec z\,^{ +}),w)$
is a fiber product of the Kuranishi neighborhoods of
the gluing pieces and the space of the
smoothing parameters of the singular points.
By definition, our involution obviously commutes with the process to take
the fiber product.  For the parameter space of smoothing the interior singularities,
the action of the involution is the complex
conjugation.  For the parameter space of smoothing the boundary
singularities, the action of involution is trivial.
The fiber product corresponding to a boundary node is taken over the Lagrangian submanifold $L$ and
the fiber product corresponding to an interior node is taken over the ambient symplectic manifold $M$.
Hence the fiber product construction can be carried out in a $\tau$-invariant way.
It is easy to see that the analysis we worked out in Section 7.1 \cite{fooobook2} (see also \cite{foootech}
Parts 2 and 3 for more detail) of the gluing is
compatible with the involution.
Thus $\tau^{\text{\rm main}}_*$ defines an involution on $\CM^{\operatorname{main}}_{k+1,m}(J;\beta)$
with Kuranishi structure.
The proof  of Theorem \ref{Lemma38.14} is complete.
\end{proof}
\begin{proof}[Proof of Theorem \ref{Proposition38.11} (1).]
The proof is the same as the proof of Theorem \ref{Lemma38.14}
except the following point.
Instead of $\CM^{\operatorname{main}}_{k+1,m}(J;\beta)$ we
will construct Kuranishi structure on $\CM_{k+1,m}(J;\beta)$.
We want it invariant under $\tau_*$ instead of $\tau^{\rm main}_*$.
(Note  $\CM^{\operatorname{main}}_{k+1,m}(J;\beta)$
is not invariant under  $\tau_*$.)
Taking this point into account the proof of Theorem \ref{Proposition38.11} (1)
goes in the same way as the proof of Theorem \ref{Lemma38.14}.
\end{proof}
Note we actually do not use Theorem \ref{Proposition38.11} (1)
to prove our main results.
(It is Theorem  \ref{Lemma38.14} that we actually use.)
So we do not discuss its proof in more detail.
\begin{proof}[Proof of Theorem \ref{Lemma38.17}.]
To prove the assertion on orientation,
it is enough to consider the orientation on
the regular part
$\CM^{\text{\rm main, reg}}_{k+1,m}(J;\beta;P_1, \dots , P_k)$.
See Remark \ref{rem:thmori} (2).
By Theorem \ref{Proposition38.7},
$\tau_*:\CM^{\text{\rm reg}}(J;\beta)^{\tau^*[(V,\sigma)]}
\to \CM^{\text{\rm reg}}(J;\beta)^{[(V,\sigma)]}$
is orientation preserving if and only if $\mu_L(\beta )/2$ is even.
Recall when we consider the main component
$\CM^{\rm main, reg}_{k+1, m}(J;\beta)$, the boundary marked points is in counter-clockwise
cyclic ordering. However,
by the involution $\tau_*$ in Theorem \ref{Lemma38.17}, each boundary marked point $z_i$ is mapped to
$\overline{z}_i$ and each interior marked point $z^+_j$ is mapped to
$\overline{z^+_j}$.
Thus the order of the boundary marked points changes to
clockwise ordering.
Denote by
$\CM^{\text{\rm clock, reg}}_{k+1,m}(J;\beta)^{[(V,\sigma)]}$
the moduli space with the boundary
marked points $(z_0,z_1, \dots ,z_k)$ respect the {\it clock-wise} orientation
and interior marked points $z^+_1, \dots, z^+_m$.
Since $z \mapsto \overline{z}$ reverses the orientation on the boundary
and $z^+ \mapsto \overline{z^+}$ reverses the orientation on the interior, the argument in the proof of Proposition \ref{Lemma38.9}
shows that
$\tau_*:\CM^{\text{\rm main, reg}}_{k+1, m}(J;\beta)^{\tau^*[(V,\sigma)]}
\to \CM^{\text{\rm clock, reg}}_{k+1,m}(J;\beta)^{[(V,\sigma)]}$
respects the orientation if and only if $\mu_L(\beta)/2 + k+1 + m$ is even.
Thus we have
$$
\aligned
& \CM^{\text{\rm main, reg}}_{k+1,m}(J;\beta;P_1, \dots , P_k)^{\tau^{\ast}[(V,\sigma)]} \\
= & (-1)^{\mu_L(\beta)/2 + k+1+m}
\CM^{\text{\rm clock, reg}}_{k+1,m}(J;\beta;P_1, \dots ,P_k)^{[(V,\sigma)]}.
\endaligned
$$
Recall that Lemma \ref{Lemma8.4.3} describes how the orientation of $\CM_{k+1,m}(J;\beta;P_1,\dots ,P_k)$
changes by changing ordering of boundary marked points.
Thus, using Lemma \ref{Lemma8.4.3}, we obtain Theorem \ref{Lemma38.17} immediately.
\end{proof}
Since the map $\tau_{\ast}^{\rm main}$ preserves the
ordering of interior marked points, we also obtain the following.
When we study bulk deformations (\cite{fooo06}, \cite{fooobook1}) of $A_{\infty}$ algebra 
for $L=\text{Fix } \tau$, 
which we do not discuss in this article, we need the next theorem.
\begin{thm}\label{Lemma38.17withQ}
Let $Q_1,\dots ,Q_m$ be smooth singular simplicies of
$M$. Then the map
$$
\aligned
\tau_{\ast}^{\rm main}~:~
& ~\CM^{\text{\rm main}}_{k+1,m}(J;\beta;Q_1,\dots ,
Q_m;P_1, \dots , P_k)^{\tau^{\ast}[(V,\sigma)]} \\
\longrightarrow
& ~\CM^{\text{\rm main}}_{k+1,m}(J;\beta;\tau (Q_1),\dots ,\tau (Q_m);P_k, \dots ,P_1)^{[(V,\sigma)]}
\endaligned
$$
preserves orientation if and only if
$$
\epsilon = \frac{\mu_L(\beta)}{2} + k + 1 + m + nm + \sum_{1 \le i < j \le k} \deg'P_i\deg'P_j
$$
is even.
\end{thm}
\begin{proof}
Note that $\tau$ preserves the orientation on $M$ if and only if 
$\frac{1}{2}\dim_{\R}M=n$ is even. Taking this into account,
Theorem \ref{Lemma38.17withQ} follows from Theorem \ref{Lemma38.17}.
See also Lemma \ref{oripreversing}.
\end{proof}

\section{Applications}
\label{sec:Appl}
Using the results obtained in the previous sections, we prove Theorem \ref{Theorem34.20},
Corollary \ref{Corollary34.22},
Corollary \ref{TheoremN} in Subsection \ref{subsec:Appl1},
Theorem \ref{Proposition34.25}
and Corollary \ref{qMassey} in Subsection \ref{proof1.9} and
calculate Floer cohomology of $\R P^{2n+1}$
over $\Lambda_{0,{\rm nov}}^{\Z}$ in Subsection
\ref{subsec:Appl2}.

\subsection{Filtered $A_{\infty}$ algebra
and Lagrangian Floer cohomology}
\label{subsec:Ainfty}
In order to explain how our study of
orientations can be used for the applications to Lagrangian Floer theory,
we briefly recall the construction of the filtered $A_\infty$ algebra for the
relatively spin Lagrangian submanifold and its obstruction/deformation theory developed in our books
\cite{fooobook1}, \cite{fooobook2}.

See a survey paper \cite{Oht} for a more detailed review.

Let $R$ be a commutative ring with unit.
Let $e$ and $T$ be formal variables of degree $2$ and $0$, respectively.
We use the {\it universal Novikov ring} over $R$
as our coefficient ring:
\begin{align}\label{eq:nov}
 \Lambda^{R}_{{\rm nov}} & = \left.\left\{ \sum_{i=0}^{\infty} a_i T^{\lambda_i}e^{\mu_i}  \
\right\vert \ a_i \in R,
\ \mu_i \in \Z, \ \lambda_i \in \R, \
\lim_{i \to \infty}\lambda_i = +\infty \right\},
\\
 \Lambda^{R}_{0,{\rm nov}} & =
\left.\left\{ \sum_{i=0}^{\infty} a_i T^{\lambda_i}e^{\mu_i} \in \Lambda_{{\rm nov}}
\ \right\vert \ \lambda_i \geq 0 \right\}.
\label{eq:nov0}
\end{align}
We define a filtration $F^{\lambda}\Lambda_{0,{\rm nov}}^R=
T^{\lambda}\Lambda_{0,{\rm nov}}^R$ ($\lambda \in \R_{\ge 0}$) on $\Lambda_{0,{\rm nov}}^R$
which induces a filtration $F^{\lambda}\Lambda_{{\rm nov}}^R$
($\lambda \in \R$) on $\Lambda_{{\rm nov}}^R$.
We call this filtration the {\it energy filtration}.
Given these filtrations, both $\Lambda_{0,{\rm nov}}^R$ and $\Lambda_{{\rm nov}}^R$
become graded filtered commutative rings.
In the rest of this subsection and the next, we take $R=\Q$.
We use the case $R=\Z$ in Subsection \ref{subsec:Appl2}.

In Section \ref{sec:pre} we define
$\CM^{\rm main}_{k+1}(J;\beta;P_1,\dots ,P_k)$
for smooth singular simplicies $(P_i,f_i)$ of $L$.
By the result of Section 7.1 \cite{fooobook2} it has a Kuranishi structure.
Here we use the same notations for the Kuranishi structure as the ones used in Appendix
of the present paper.
The space $\CM^{\rm main}_{k+1}(J;\beta;P_1,\dots ,P_k)$
is locally described by
$s^{-1}_p(0)/\Gamma_p$.
If the Kuranishi map $s_p$ is transverse to the zero
section, it is locally an orbifold.
However, if $\Gamma_p$ is non trivial, we can not perturb $s_p$ to a
$\Gamma_p$-equivariant section transverse to the zero
section, in general. Instead of single valued sections,
we take a $\Gamma_p$-equivariant {\it multi-valued} section
(multi-section)
$\mathfrak s_p$ of $E_p \to V_p$
so that each branch of the multi-section is transverse to the zero section
and ${\mathfrak s}_p^{-1}(0)/\Gamma_p$ and sufficiently close
to $s_p^{-1}(0)/\Gamma_p$.
(See Sections 7.1 and 7.2 in \cite{fooobook2} for the precise
statement.)
We denote the perturbed zero locus (divided by $\Gamma_p$) by
$\CM^{\rm main}_{k+1}(J;\beta;P_1,\dots ,P_k)^{\mathfrak s}$.
We have the evaluation map at the zero-th marked point
for the perturbed moduli space:
$$
ev_0 :
\CM^{\rm main}_{k+1}(J;\beta;P_1,\dots ,P_k)^{\mathfrak s}
\longrightarrow L.
$$
Then
such a system $\mathfrak s =\{ \mathfrak s_p \}$ of muti-valued sections gives
rise to the virtual fundamental chain over $\Q$ as follows:
By Lemma 6.9 in \cite{FO} and Lemma A1.26 in \cite{fooobook2},
$\CM^{\rm main}_{k+1}(J;\beta;P_1,\dots ,P_k)^{\mathfrak s}$ carries a smooth triangulation. We take a smooth triangulation on it.
Each simplex $\Delta^d_a$ of dimension $d=\dim \mathfrak s_p^{-1}(0)$ in the triangulation comes with multiplicity
$\text{mul}_{\Delta^d_a} \in \Q$.
(See Definition A1.27 in \cite{fooobook2} for the definition of multiplicity.)
Restricting $ev_0$ to $\Delta_a^d$, we have a
singular simplex of $L$ denoted by $(\Delta^d_a, ev_0)$.
Then the virtual fundamental chain over $\Q$ which we denote by
$(\CM^{\rm main}_{k+1}(J;\beta;P_1,\dots ,P_k)^{\mathfrak s},ev_0)$
is defined by
$$
(\CM^{\rm main}_{k+1}(J;\beta;P_1,\dots ,P_k)^{\mathfrak s},ev_0)
=
\sum_a \text{mul}_{\Delta_a^d} \cdot (\Delta^d_a, ev_0).
$$
When the virtual dimension is zero, i.e. when $d=0$, we denote
$$
\#\CM^{\rm main}_{k+1}(J;\beta;P_1,\dots ,P_k)^{\mathfrak s}
= \sum_a \text{mul}_{\Delta_a^0} \in \Q.
$$
Now
put $\Pi(L)_0 = \{ (\omega(\beta), \mu_L(\beta)) ~\vert~\beta \in \Pi(L),
~\CM(J;\beta) \ne \emptyset \}$, see (\ref{eq:Pi})
for $\Pi(L)$.
Let $G(L)$ be a submonoid of
$\R_{\ge 0} \times 2\Z$ generated by $\Pi(L)_0$.
We put $\beta_0 =(0,0) \in G(L)$.

\begin{defn}\label{def:mk}
For smooth singular simplicies $P_i$ of $L$ and $\beta \in G(L)$,
we define a series of maps
$
\mathfrak m_{k, \beta}
$
by
$$
\aligned
\mathfrak m_{0,\beta}(1) & =
\begin{cases}
(\CM_1(J;\beta)^{\mathfrak s},ev_0) & \text {for $\beta \ne \beta _0$} \\
0 & \text {for $\beta =\beta _0$,}
\end{cases} \\
\mathfrak m_{1,\beta}(P) & =
\begin{cases}
({\CM}_{2}^{\text {main}}(J;\beta;P)^{\mathfrak s},ev_0) & {\text {for $\beta \ne
\beta_0$}}\\
(-1)^{n}\partial P & {\text {for $\beta = \beta_0$}},
\end{cases}
\endaligned
$$
and
$$
\mathfrak m_{k,\beta}(P_1, \ldots ,P_k) =
 (\CM_{k+1}^{\text {main}}(J;\beta;P_1,\ldots, P_k)^{\mathfrak s},ev_0), \quad k\ge 2.
$$
\end{defn}
Here $\partial$ is the usual boundary operator and
$n=\dim L$.
Then,
one of main results proved in \cite{fooobook1} and \cite{fooobook2} is
as follows:
For a smooth singular chain $P$ on $L$
we put the cohomological grading
$\deg P = n -\dim P$ and regard
a smooth singular chain complex $S_{\ast}(L;\Q)$
as a smooth singular cochain complex $S^{n-\ast}(L;\Q)$.
For a subcomplex $C(L;\Q)$ of $S(L;\Q)$
we denote by
$C(L;\Lambda_{0,{\rm nov}}^{\Q})$
a completion of $C(L;\Q) \otimes \Lambda_{0,{\rm nov}}^{\Q}$
with respect to the
filtration induced from one on $\Lambda_{0,{\rm nov}}^{\Q}$
introduced above. We shift the degree by 1, i.e., define
$$
C(L;\Lambda_{0,{\rm nov}}^{\Q})[1]^{\bullet}=C(L;\Lambda_{0,{\rm nov}}^{\Q})^{\bullet +1},
$$
where we define $\deg (P T^{\lambda}e^{\mu})=\deg P +2\mu$ for
$P T^{\lambda}e^{\mu} \in C(L;\Lambda_{0,{\rm nov}}^{\Q})$.
We put
\begin{equation}\label{eq:defmk}
\mathfrak m_k =
\sum_{\beta \in G(L)}
\mathfrak m_{k,\beta} \otimes T^{\omega (\beta )}
e^{{\mu_L (\beta)}/{2}}, \quad
k=0,1,\ldots .
\end{equation}
To simplify notation we write $C=C(L;\Lambda_{0,{\rm nov}}^{\Q})$. Put
\begin{equation}\label{Bk}
B_k(C[1]) = \underbrace{C[1]\otimes \cdots \otimes C[1]}_{k}
\end{equation}
and take its completion with respect to the energy filtration.
By an abuse of notation, we denote the completion by the same symbol.
We define the {\it bar complex} $B(C[1])= \bigoplus_{k=0}^\infty B_k(C[1])$ and extend $\mathfrak m_k$ to the graded coderivation
$\widehat{\mathfrak m}_k$ on $B(C[1])$ by
\begin{equation}\label{eq:hatmk}
\widehat{\mathfrak m}_k(x_1 \otimes \cdots \otimes x_n) =
\sum_{i=1}^{n-k+1} (-1)^{*}
x_1 \otimes \cdots \otimes \mathfrak m_k(x_i,\dots,x_{i+k-1})
\otimes \cdots \otimes x_n
\end{equation}
where $* = \deg x_1+\cdots+\deg x_{i-1}+i-1$.
We put
\begin{equation}\label{eq:d}
\widehat{d}=\sum _{k=0}^{\infty} \widehat{\mathfrak m}_k : B(C[1]) \longrightarrow B(C[1]).
\end{equation}
\begin{thm}[Theorem 3.5.11 in \cite{fooobook1}]\label{thm:Ainfty}
For any closed relatively spin Lagrangian submanifold $L$ of $M$,
there exist a countably generated subcomplex $C(L;\Q)$
of smooth singular cochain complex of $L$ whose cohomology
is isomorphic to $H(L;\Q)$ and
a system of multi-sections $\mathfrak s$
of $\CM_{k+1}^{{\text {\rm main}}}(J;\beta;P_1,\ldots ,P_k)$
($\mathfrak s$ are chosen depending on $P_i \in C(L;\Q)$)
such that
$$
\mathfrak m_k
 :
\underbrace{C(L;\Lambda_{0,{\rm nov}}^{\Q})[1] \otimes \dots \otimes
C(L;\Lambda_{0,{\rm nov}}^{\Q})[1]}_{k} \to
C(L;\Lambda_{0,{\rm nov}}^{\Q})[1], \quad k=0,1,\ldots
$$
are defined and satisfy
$\widehat{d} \circ \widehat{d} =0$.
\end{thm}
The equation $\widehat{d}\circ \widehat{d} = 0$ is equivalent to
$$
\sum_{k_1+k_2=k+1}
\sum_i
(-1)^{\deg' x_1 + \cdots + \deg' x_{i-1}}
\mathfrak m_{k_1}(x_1,\ldots,
\mathfrak m_{k_2}
(x_i,\ldots,
x_{i+k_2-1}),\ldots,x_k)=0
$$
which we call the {\it $A_{\infty}$ formulas} or {\it relations}.
Here $\deg'x_i = \deg x_i -1$, the shifted degree.
In particular, the $A_{\infty}$ formulas imply
an equality
$$\mathfrak m_2 (\mathfrak m_0(1), x) +(-1)^{\deg' x}\mathfrak m_2(x,\mathfrak m_0(1)) + \mathfrak m_1 \mathfrak m_1(x)=0,
$$
which shows $\mathfrak m_1 \circ \mathfrak m_1 = 0$ may not hold unless
if $\mathfrak m_0(1) = 0$, in general.
So $\mathfrak m_0$ gives an obstruction to define
$\mathfrak m_1$-cohomology.
\begin{defn}\label{def:boundingcochain}
An element $b \in C(L;\Lambda_{0,{\rm nov}}^{\Q})[1]^0$
with $b \equiv 0 \mod \Lambda_{0,{\rm nov}}^+$
is called a {\it solution of the Maurer-Cartan equation}
or {\it bounding cochain}, if it satisfies
the Maurer-Cartan equation:
$$
\mathfrak m_0 (1) + \mathfrak m_{1}(b) + \mathfrak m_2(b,b)
+ \mathfrak m_3(b,b,b) + \dots =0.
$$
Here $\Lambda_{0,{\rm nov}}^+ =
\{ \sum a_i T^{\lambda_i}e^{\mu_i} \in \Lambda_{0,{\rm nov}} ~\vert ~
\lambda_i >0 \}$.
We denote by $\CM (L;\Lambda_{0,{\rm nov}}^{\Q})$
the set of bounding cochains.
We say $L$ is {\it unobstructed} if $\CM (L;\Lambda_{0,{\rm nov}}^{\Q})
\ne \emptyset$.
\end{defn}
\begin{rem}\label{rem:boundingcochain}
We do not introduce the notion of
gauge equivalence of bounding cochains
(Definition 4.3.1 in \cite{fooobook1}), because we do not use it in this paper.
\end{rem}
If $\CM (L;\Lambda_{0,{\rm nov}}^{\Q})
\ne \emptyset$, then by using any $b \in \CM (L;\Lambda_{0,{\rm nov}}^{\Q})$ we can deform the $A_{\infty}$ structure
$\mathfrak m$ to $\mathfrak m^b$
by
$$
\mathfrak m^b_k(x_1,\dots,x_k)
= \sum_{\ell_0,\dots,\ell_k}
\mathfrak m_{k+\sum \ell_i}(\underbrace{b,\dots,b}_{\ell_0},x_1,
\underbrace{b,\dots,b}_{\ell_1},\dots,
\underbrace{b,\dots,b}_{\ell_{k-1}},x_k,
\underbrace{b,\dots,b}_{\ell_k})
$$
so that $\mathfrak m_1^b \circ \mathfrak m_1^b =0$
(Proposition 3.6.10 in \cite{fooobook1}).
Then we can define
$$
HF((L,b);\Lambda_{0,{\rm nov}}^{\Q}) :=
H(C(L;\Lambda_{0,{\rm nov}}^{\Q}), \mathfrak m_1^b)
$$
which we call {\it Floer cohomology of $L$} (deformed by $b$).
\par\smallskip
In the actual proof of Theorem \ref{thm:Ainfty} given in
Section 7.2 \cite{fooobook2},
we do not construct the filtered $A_{\infty}$ structure at once,
but we first construct a {\it filtered $A_{n,K}$ structure} for any non negative integers $n,K$ and
promote it to a filtered $A_{\infty}$ structure by developing certain
obstruction theory (Subsections 7.2.6--7.2.10 \cite{fooobook2}).
Here we recall the notion of filtered $A_{n,K}$ structure
from Subsection 7.2.6 \cite{fooobook2}, which is mentioned later in the proof of Theorem \ref{Theorem34.20}.
See also Subsection 2.6 and Section 4 \cite{foooYash} for a quick review.
We briefly summarize the obstruction theory in Subsection \ref{subsec:promotion}.
Let $G\subset \R_{\ge 0}\times 2\Z$ be a monoid such that
$\text {pr}_1^{-1}([0,c])$ is finite for any $c\ge 0$ and
$\text {pr}_1^{-1}(0)=\{\beta_0 =(0,0)\}$, where $\text {pr}_i$ denotes
the projection to the $i$-th factor.
We note that in the geometric situation we take $G=G(L)$ introduced above.
For $\beta \in G$, we define
\begin{equation}\label{def:betanorm}
\Vert \beta \Vert = \left\{
\begin{array}{ll}
\sup \{ n ~\vert~ \exists \beta_i \in G \setminus \{\beta_0\}, \
\sum_{i=1}^n \beta_i = \beta \} + [{\rm pr}_1(\beta)] -1 &
\text{ if } \beta \neq \beta_0 \\
-1 & \text{ if } \beta = \beta_0,
\end{array}
\right.
\end{equation}
where $[{\rm pr}_1(\beta)]$ stands for the largest integer less than
or equal to ${\rm pr}_1(\beta)$.
Using this, we introduce a partial order on $(G \times \Z_{\geq 0})
\setminus \{(\beta_0,0)\}$ denoted by
$(\beta_1,k_1) < (\beta_2,k_2)$ if and only if
either
$$
\Vert \beta_1 \Vert + k_1 <
\Vert \beta_2 \Vert + k_2
$$ or
$$
\Vert \beta_1 \Vert + k_1 =
\Vert \beta_2 \Vert + k_2 \text{ and }
\Vert \beta_1 \Vert < \Vert \beta_2 \Vert.
$$
We write $(\beta_1,k_1) \sim (\beta_2,k_2)$, when
$$
\Vert \beta_1 \Vert + k_1 =
\Vert \beta_2 \Vert + k_2 \text{ and }
\Vert \beta_1 \Vert = \Vert \beta_2 \Vert.
$$
We denote $(\beta_1,k_1) \lesssim (\beta_2,k_2)$ if
either $(\beta_1,k_1) < (\beta_2,k_2)$ or $(\beta_1,k_1) \sim
(\beta_2,k_2)$.
For non-negative integers $n, n', k, k'$, we also use
the notation $(\beta,k)<(n,k)$, $(n,k)<(n',k')$,
$(\beta,k)\lesssim (n,k)$, $(n,k)\lesssim (n',k')$ etc. in a similar way.
\par
Let $\overline{C}$ be a cochain complex over $R$ and
$C=\overline{C}\otimes \Lambda_{0,\text{nov}}^R$.
Suppose that there is a sequence of $R$ linear maps
$$
{\mathfrak m}_{k,\beta}:B_k(\overline{C}[1]) \to \overline{C}[1]
$$
for $(\beta,k) \in (G \times \Z) \setminus \{(\beta_0,0)\}$
with $(\beta,k) < (n,K)$.
We also suppose that ${\mathfrak m}_{1,\beta_0}$ is
the boundary operator of the cochain complex $\overline{C}$.
\begin{defn}\label{def:AnK}
We call $(C,\{{\mathfrak m}_{k,\beta}\})$
a ($G$-gapped) {\it filtered $A_{n,K}$ algebra}, if the identity
\begin{equation}\label{AnKformula}
\sum_{\beta_1+\beta_2=\beta, k_1+k_2=k+1} \sum_i
(-1)^{\deg'{\mathbf x}_i^{(1)}} {\mathfrak m}_{k_2,\beta_2}
\bigl({\mathbf x}_i^{(1)},
{\mathfrak m}_{k_1,\beta_1}({\mathbf x}_i^{(2)}),
{\mathbf x}_i^{(3)} \bigr) = 0
\end{equation}
holds for all $(\beta,k) <  (n,K)$,
where
$$
\Delta^2({\mathbf x})=\sum_i{\mathbf x}_i^{(1)} \otimes
{\mathbf x}_i^{(2)} \otimes {\mathbf x}_i^{(3)}.
$$
Here $\Delta$ is the coproduct of the tensor coalgebra.
\end{defn}

Let $C$, $C'$ be filtered $A_{n,K}$ algebras. We
consider a sequence of $R$ linear maps of degree zero
$$
\mathfrak f_{k,\beta}: B_k(\overline C[1]) \to \overline C'[1]
$$
satisfying $\mathfrak f_{0,\beta_0} = 0$.
\begin{defn}\label{def:AnKhom}
We call $\{\mathfrak f_{k,\beta}\}$ a {\it filtered
$A_{n,K}$ homomorphism}, if the identity
$$\aligned
&\sum_{m,i}\sum_{\beta'+\beta_1+\cdots+\beta_m = \beta}
\sum_{k_1+\cdots+k_m = k}
\mathfrak m_{m,\beta'}\left(
\mathfrak f_{k_1,\beta_1}(\text{\bf x}_i^{(1)}),
\cdots,
\mathfrak f_{k_m,\beta_m}(\text{\bf x}_i^{(m)})
\right)\\
& =\sum_{\beta_1+\beta_2 = \beta, k_1+k_2 = k+1} \sum_i
(-1)^{\deg'\text{\bf x}_i^{(1)}
}\mathfrak f_{k_2,\beta_2}\left(\text{\bf x}_i^{(1)} ,
\mathfrak m_{k_1,\beta_1}(\text{\bf x}_i^{(2)}),
\text{\bf x}_i^{(3)}\right)
\endaligned$$
holds for $(\beta,k) \lesssim (n,K)$.
\end{defn}

We also have the notion of
filtered $A_{n,K}$ homotopy equivalences in a natural way.
(See Subsection 7.2.6 \cite{fooobook2}.)
In \cite{fooobook2}, we proved the following:

\begin{thm}[Theorem 7.2.72 in \cite{fooobook2}]\label{ext(n,K)}
Let $C_1$ be a filtered $A_{n,K}$ algebra and
$C_2$ a filtered $A_{n',K'}$ algebra such that
$(n,K) < (n',K')$.
Let ${\mathfrak h}:C_1 \to C_2$
be a filtered $A_{n,K}$ homomorphism.
Suppose that ${\mathfrak h}$ is a filtered $A_{n,K}$ homotopy
equivalence.
Then there exist a filtered $A_{n',K'}$ algebra structure on
$C_1$
extending the given filtered $A_{n,K}$ algebra structure
and a filtered $A_{n',K'}$ homotopy equivalence
$C_1 \to C_2$ extending the given
filtered $A_{n,K}$ homotopy equivalence ${\mathfrak h}$.
\end{thm}

\par\smallskip
Next, let $(L^{(1)}, L^{(0)})$ be a relatively spin pair of closed Lagrangian submanifolds.
 We first assume that $L^{(0)}$ is transverse to $L^{(1)}$.
Let $C(L^{(1)},L^{(0)};\Lambda_{0,{\rm nov}}^{\Q})$ be the free $\Lambda_{0,{\rm nov}}^{\Q}$ module
generated by the intersection points $L^{(1)} \cap L^{(0)}$.
Then we can construct a filtered $A_{\infty}$ bimodule
structure $\{\mathfrak n_{k_1, k_0}\}_{k_1, k_0 =0, 1, \dots}$
on $C(L^{(1)},L^{(0)};\Lambda_{0,{\rm nov}}^{\Q})$ over the pair
$(C(L^{(1)};\Lambda_{0,{\rm nov}}^{\Q}), C(L^{(0)};\Lambda_{0,{\rm nov}}^{\Q}))$
of $A_\infty$ algebras as follows. Here we briefly describe
the map
$$
\aligned
\mathfrak n_{k_1,k_0} :
B_{k_1} (C(L^{(1)};\Lambda_{0,{\rm nov}}^{\Q})[1])
\otimes C(L^{(1)},L^{(0)};& \Lambda_{0,{\rm nov}}^{\Q})
\otimes B_{k_0}(C(L^{(0)};\Lambda_{0,{\rm nov}}^{\Q})[1]) \\
& \longrightarrow
C(L^{(1)},L^{(0)};\Lambda_{0,{\rm nov}}^{\Q}).
\endaligned
$$
A typical element of the tensor product above
is written as
$$
\left(P_{1}^{(1)}T^{\lambda_{1}^{(1)}}e^{\mu_{1}^{(1)}} \otimes \dots
\otimes P_{k_1}^{(1)}T^{\lambda_{k_1}^{(1)}}e^{\mu_{k_1}^{(1)}}\right)
\otimes T^{\lambda}e^{\mu}\langle p \rangle \otimes
\left(P_{1}^{(0)}T^{\lambda_{1}^{(0)}}e^{\mu_{1}^{(0)}} \otimes \dots
\otimes P_{k_0}^{(0)}T^{\lambda_{k_0}^{(0)}}e^{\mu_{k_0}^{(0)}}\right)
$$
for $p \in L^{(1)}\cap L^{(0)}$. Then $\mathfrak n_{k_1,k_0}$ maps
it to
$$
\sum_{q, B}
\# \left(\CM(p,q;B;P_{1}^{(1)},\dots,P_{k_1}^{(1)};
P_{1}^{(0)},\dots,P_{k_0}^{(0)})\right)T^{\lambda'}e^{\mu'} \langle q \rangle
$$
with $\lambda'= \omega (B) +\sum \lambda_i^{(1)} + \lambda + \sum \lambda_i^{(0)}$ and
$\mu' = \mu_{L}(B) +\sum \mu_i^{(1)} + \mu + \sum \mu_i^{(0)}$.
Here $B$ is the homotopy class of Floer trajectories connecting
$p$ and $q$, and the sum is taken over all $(q,B)$ such that
the virtual dimension of the moduli space
$\CM(p,q;B;P_{1}^{(1)},\dots,P_{k_1}^{(1)};
P_{1}^{(0)},\dots,P_{k_0}^{(0)})$ of Floer trajectories is zero.
See Subsection 3.7.4 of \cite{fooobook1} for the precise definition of
$$
\CM(p,q;B;P_{1}^{(1)},\dots,P_{k_1}^{(1)};
P_{1}^{(0)},\dots,P_{k_0}^{(0)}).
$$
Strictly speaking, we also need to take a suitable system of multi-sections on this
moduli space to obtain the virtual fundamental chain that enters in the construction of
the operators $\mathfrak n_{k_1,k_0}$ defining the desired $A_{\infty}$ bimodule structure.
Because of the usage of multi-sections, the counting number with sign
$$
\# \left(\CM(p,q;B;P_{1}^{(1)},\dots,P_{k_1}^{(1)};
P_{1}^{(0)},\dots,P_{k_0}^{(0)})\right)
$$ is
a rational number, in general.

Now let $B(C(L^{(1)};\Lambda_{0,{\rm nov}}^{\Q})[1])
\otimes C(L^{(1)},L^{(0)};\Lambda_{0,{\rm nov}}^{\Q})
\otimes B(C(L^{(0)};\Lambda_{0,{\rm nov}}^{\Q})[1])$ be
the completion of
$$\bigoplus_{k_0 \ge 0, k_1 \ge 0} B_{k_1} (C(L^{(1)};\Lambda_{0,{\rm nov}}^{\Q})[1])
\otimes C(L^{(1)},L^{(0)};\Lambda_{0,{\rm nov}}^{\Q})
\otimes B_{k_0}(C(L^{(0)};\Lambda_{0,{\rm nov}}^{\Q})[1])
$$
with respect to the induced energy filtration.
We extend $\mathfrak n_{k_1,k_0}$ to a bi-coderivation
on $B(C(L^{(1)};\Lambda_{0,{\rm nov}}^{\Q})[1])
\otimes C(L^{(1)},L^{(0)};\Lambda_{0,{\rm nov}}^{\Q})
\otimes B(C(L^{(0)};\Lambda_{0,{\rm nov}}^{\Q})[1])$ which is given by
the formula
\begin{equation}\label{eq:dbimodule}
\aligned
&
\widehat d(x_{1}^{(1)} \otimes \cdots \otimes x_{k_1}^{(1)} \otimes y \otimes
x_{1}^{(0)}
\otimes \cdots \otimes x_{k_0}^{(0)}) \\
= & \sum_{k^{\prime}_1\le k_1,k^{\prime}_0\le k_0}
(-1)^{\deg' x_{1}^{(1)}+\cdots+\deg' x_{k_1-k'_1}^{(1)}}\\
& \quad x_{1}^{(1)} \otimes \cdots \otimes
x_{k_1-k'_1}^{(1)}\otimes
\mathfrak n_{k'_1,k'_0}(x_{k_1-k'_1+1}^{(1)},\cdots,
y,\cdots,x_{k'_0}^{(0)})\otimes
x_{k'_0+1}^{(0)} \otimes \cdots \otimes x_{k_0}^{(0)} \\
& + \widehat {d}^{(1)} (x_{1}^{(1)} \otimes \cdots \otimes x_{k_1}^{(1)}) \otimes
y \otimes x_{1}^{(0)} \otimes \cdots \otimes x_{k_0}^{(0)} \\
& + (-1)^{\Sigma \deg' x_{i}^{(1)} + \deg' y}
x_{1}^{(1)} \otimes \cdots \otimes x_{k_1}^{(1)}
\otimes y \otimes \widehat {d}^{(0)} (x_{1}^{(0)}
\otimes \cdots \otimes x_{k_0}^{(0)}).
\endaligned
\end{equation}
Here $\widehat {d}^{(i)}$ is defined by
(\ref{eq:hatmk}) and (\ref{eq:d}), using
the filtered $A_{\infty}$
structure $\mathfrak {m}^{(i)}$ of $(C(L^{(i)};\Lambda_{0,{\rm nov}}^{\Q}), \mathfrak m^{(i)})$ ($i=0,1$).

\begin{thm}[Theorem 3.7.21 in \cite{fooobook1}]
\label{thm:Ainftybimodule}
For any relatively spin pair $(L^{(1)}, L^{(0)})$
of closed Lagrangian submanifolds, the family of maps $\{ \mathfrak n_{k_1,k_0}\}_{k_1,k_0}$ defines
a filtered $A_{\infty}$ bimodule structure on $C(L^{(1)},L^{(0)};\Lambda_{0,{\rm nov}}^{\Q})$
over  $(C(L^{(1)};\Lambda_{0,{\rm nov}}^{\Q}),C(L^{(0)};\Lambda_{0,{\rm nov}}^{\Q}))$.
Namely, $\widehat{d}$ in (\ref{eq:dbimodule}) satisfies
$\widehat{d} \circ \widehat{d} =0$.
\end{thm}
Since the equation $\widehat{d} \circ \widehat{d} =0$ implies, in particular,
$$
\mathfrak n_{0,0}\circ \mathfrak n_{0,0} (y) + \mathfrak n_{1,0}(\mathfrak m_{0}^{(1)}(1),y) + (-1)^{\deg'y}\mathfrak {n}_{0,1}(y,\mathfrak m_{0}^{(0)}(1))
=0,
$$
we have $\mathfrak n_{0,0} \circ \mathfrak n_{0,0} \ne 0$, in general.
However,
if both of $L^{(0)}$ and $L^{(1)}$ are unobstructed
in the sense of Definition \ref{def:boundingcochain},
we can deform the filtered $A_{\infty}$ bimodule
structure $\mathfrak n$ by $b_i \in \CM (L^{(i)};\Lambda_{0,{\rm nov}}^{\Q})$
so that
$$
{}^{b_1}\mathfrak n_{0,0}^{b_0} (y):=
\sum_{k_1, k_0} \mathfrak n_{k_1,k_0}
(\underbrace{b_1,\ldots , b_1}_{k_1}, y,
\underbrace{b_0, \ldots , b_0}_{k_0})
$$
satisfies ${}^{b_1}\mathfrak n_{0,0}^{b_0} \circ {}^{b_1}\mathfrak n_{0,0}^{b_0}=0$ (Lemma 3.7.14 in \cite{fooobook1}).
Then we can define
$$
HF((L^{(1)},b_1),(L^{(0)},b_0);\Lambda_{0,{\rm nov}}^{\Q})
:= H(C(L^{(1)},L^{(0)};\Lambda_{0,{\rm nov}}^{\Q}), {}^{b_1}\mathfrak n_{0,0}^{b_0})
$$
which we call {\it Floer cohomology of a pair $(L^{(1)},L^{(0)})$}
(deformed by $b_1,b_0$).
\par

So far we assume that $L^{(0)}$ is transverse to $L^{(1)}$.
But we can generalize the story to the Bott-Morse case,
that is, each component of $L^{(0)}\cap L^{(1)}$ is a smooth
manifold.
Especially, for the case $L^{(1)}=L^{(0)}$, we have
$\mathfrak n_{k_1,k_0}=\mathfrak m_{k_1+k_0+1}$ (see Example 3.7.6
in \cite{fooobook1}) and an isomorphism
\begin{equation}\label{isoFloer}
HF((L,b),(L,b);\Lambda_{0,{\rm nov}}^{\Q}) \cong HF((L,b);\Lambda_{0,{\rm nov}}^{\Q})
\end{equation}
for any $b \in \CM(L;\Lambda_{0,{\rm nov}}^{\Q})$ by Theorem G (G.1) in \cite{fooobook1}.
Moreover, if we extend the coefficient ring $\Lambda_{0,{\rm nov}}^{\Q}$
to $\Lambda_{{\rm nov}}^{\Q}$, we can find that
Hamiltonian isotopies $\Psi_i^s : M \to M$ ($i=0,1, s\in [0,1]$) with $\Psi_i^0={\rm id}$ and $\Psi_i^1=\Psi_i$ induce an isomorphism
\begin{equation}\label{invariance}
\aligned
& HF((L^{(1)},b_1),(L^{(0)},b_0);\Lambda_{{\rm nov}}^{\Q}) \\
\cong~
& HF((\Psi_1(L^{(1)}),\Psi_{1\ast}b_1),(\Psi_0(L^{(0)}),\Psi_{0\ast}b_0);\Lambda_{{\rm nov}}^{\Q})
\endaligned
\end{equation}
by Theorem G (G.4) in \cite{fooobook1}.
This shows invariance of Floer cohomology
of a pair $(L^{(1)},L^{(0)})$ over $\Lambda_{{\rm nov}}^{\Q}$ under Hamiltonian isotopies.

\subsection{Proofs of Theorem \ref{Theorem34.20}, Corollary \ref{Corollary34.22}
and Corollary \ref{TheoremN} }
\label{subsec:Appl1}

\begin{proof}[Proof of Theorem \ref{Theorem34.20}]
We consider
the map (\ref{38.16}) for the case $m=0$.
It is an automorphism of order $2$.
We first take its quotient by Lemma \ref{lem:quot}
(Lemma A1.49 \cite{fooobook2}) in the sense
of Kuranishi structure, and
take a perturbed multi-section of the quotient space,
which is transverse to zero section.
After that we lift the perturbed multi-section.
Then
we can obtain a system of multi-sections
$\mathfrak s$ on the moduli space
$\CM^{{\text{\rm {main}}}}_{k+1}(J;\beta ;P_1,\dots,P_k)$
which is preserved by (\ref{38.16}).
Then Definition \ref{def:mk} yields
the operators $\{\mathfrak m_{k,\beta}\}_{k,\beta}$
which satisfy the filtered $A_{n,K}$ relations \eqref{AnKformula}
together with
\eqref{34.21}.
The sign in \eqref{34.21} follows from Theorem \ref{Lemma38.17}.
To complete the proof
we need to promote the filtered $A_{n,K}$ structure to a filtered $A_{\infty}$ structure keeping the symmetry \eqref{34.21}.
This follows from an involution invariant version
(see Theorem \ref{invariantext(n,K)}) of
Theorem \ref{ext(n,K)} (=Theorem 7.2.72 \cite{fooobook2}).
Although the proof of the invariant version is a straightforward modification of
that of Theorem 7.2.72 \cite{fooobook2},
we give the outline of the argument in Appendix, Subsection
\ref{subsec:promotion}, for readers' convenience.
\end{proof}
\par
\begin{rem}\label{remonmult}
In general, it is not possible to perturb a section of the obstruction bundle transversal to the zero section by a single-valued
perturbation.
Using multi-valued perturbation, we can take the perturbation, which is transversal to the zero section
and invariant under the action of stabilizer as well as other finite group action coming from the symmetry of the problem.
In our case, there may be a fixed point of (\ref{38.16}).
But this does not cause any problem as far as we work with {\it multi}-sections and
study virtual fundamental chain over $\Q$
as in the proof of Theorem \ref{Theorem34.20}.
\end{rem}

\begin{proof}[Proof of Corollary \ref{Corollary34.22}]
For $w : (D^2,\partial D^2) \to (M,L)=(M, \text{\rm Fix}~\tau)$ we define its double
$v : S^2 \to M$ by
$$
v(z) = \begin{cases}  w(z) & \quad \text{for $z \in \H$} \\
\tau \circ w(\overline{z}) & \quad \text{for $z\in \C
\setminus \H$,}
\end{cases}
$$
where $(D^2, (-1,1,\sqrt{-1}))$ is identified with
the upper half plane $(\H ,(0,1,\infty))$ and
$S^2 = \C \cup \{ \infty \}$.
Then it is easy to see that
$c_1(TM)[v]=\mu_{L}([w])$.
(See Example \ref{Example38.8.} (1).)
Then the assumption (1) implies that $\mu_L \equiv 0 \mod 4$.
Next
we note the following general lemma.
\begin{lem}\label{c1maslov}
Let $L$ be an oriented Lagrangian submanifold of $M$.
Then the composition
$$
\pi_2 (M) \longrightarrow \pi_2(M,L)
\overset{\mu_L}\longrightarrow \Z
$$
is equal to $2c_1(TM)[\alpha]$ for $[\alpha]\in \pi_2(M)$.
\end{lem}
The proof is easy and so it is omitted.
Then by this lemma the assumption (2) also implies that
the Maslov index of $L$ modulo $4$ is trivial.
Therefore in either case of (1) and (2),
Theorem \ref{Theorem34.20} implies $\mathfrak m_{0,\tau_*\beta}(1) = - \mathfrak m_{0,\beta}(1)$.
On the other hand we have
$$
\mathfrak m_0 (1)=\sum _{\beta \in \pi_2(M,L)} \mathfrak m_{0,\beta}(1)
T^{\omega(\beta)}e^{\mu(\beta)/2}
$$
by Definition \ref{def:mk} and (\ref{eq:defmk}) which is also the same as
$$
\sum_{\beta \in \pi_2(M,L)}
\mathfrak m_{0,\tau_*\beta}(1)T^{\omega(\tau_*\beta)}e^{\mu(\tau_*\beta)/2}
$$
because $\tau_*^2 = id$ and $\tau_*:\pi_2(M,L) \to \pi_2(M,L)$ is a one-one correspondence.
Therefore since $\omega(\beta) = \omega(\tau_*\beta)$ and $\mu(\beta) = \mu(\tau_*\beta)$, we can rewrite
$\mathfrak m_0 (1)$ into
$$
\mathfrak m_0 (1) = \frac{1}{2} \sum_{\beta}\left(\mathfrak m_{0,\beta}(1) + \mathfrak m_{0,\tau_*\beta}(1)\right)T^{\omega(\beta)}e^{\mu(\beta)/2}
$$
which becomes 0 by the above parity consideration. Hence $L$ is unobstructed.
Actually, we find that $0$ is a bounding cochain; $0 \in \CM(L;\Lambda_{0,{\rm nov}}^{\Q})$.
Furthermore, (\ref{34.21}) implies
\begin{equation}\label{34.23}
\mathfrak m_{2,\beta}(P_1,P_2) = (-1)^{1+\deg'P_1\deg'P_2}
\mathfrak m_{2,\tau_*\beta}(P_2,P_1).
\end{equation}
We denote
\begin{equation}\label{qproduct}
P_1 \cup_Q P_2 : = (-1)^{\deg P_1(\deg P_2+1)} \sum_{\beta}
\mathfrak m_{2,\beta}(P_1,P_2)T^{\omega(\beta)}e^{\mu(\beta)/2}.
\end{equation}
Then a simple calculation shows that (\ref{34.23}) gives
rise to
$$
P_1 \cup_Q P_2 = (-1)^{\deg P_1\deg P_2} P_2 \cup_Q P_1.
$$
Hence $\cup_Q$ is graded commutative.
\end{proof}

\begin{proof}[Proof of Corollary \ref{TheoremN}]
Let $L$ be as in Corollary \ref{TheoremN}.
By Corollary \ref{Corollary34.22}, $L$ is unobstructed.
Since $L = \text{Fix } \tau$, we find that $c_1(TM)\vert{\pi_2(M)} = 0$ implies $\mu_L = 0$.
Then Theorem E and Theorem 6.1.9 in \cite{fooobook1} show that the Floer cohomology of $L$ over $\Lambda_{{\rm nov}}^{\Q}$
does not vanish for any $b \in \CM (L;\Lambda_{0,{\rm nov}}^{\Q})$:
$$
HF((L,b);\Lambda_{{\rm nov}}^{\Q}) \ne 0.
$$
(Note that Theorem E holds not only over
$\Lambda_{0,{\rm nov}}^{\Q}$ but also over $\Lambda_{{\rm nov}}^{\Q}$. See
Theorem 6.1.9.)
By extending the isomorphism
(\ref{isoFloer}) to $\Lambda_{{\rm nov}}^{\Q}$ coefficients
(by taking the tensor product with $\Lambda_{{\rm nov}}^{\Q}$ over
$\Lambda_{0,{\rm nov}}^{\Q}$),
we also have $HF((L,b),(L,b);\Lambda_{{\rm nov}}^{\Q}) \ne 0$.
Therefore by (\ref{invariance}) we obtain
$$
HF((\psi(L),\psi_{\ast}b),(L,b);\Lambda_{{\rm nov}}^{\Q}) \ne 0
$$
which implies $\psi(L) \cap L \ne \emptyset$.
\end{proof}

\subsection{Proofs of Theorem  \ref{Proposition34.25} and Corollary \ref{qMassey}}\label{proof1.9}
In this subsection, we prove Theorem \ref{Proposition34.25} and Corollary \ref{qMassey}.

\subsubsection{Proof of Theorem \ref{Proposition34.25} (1)}\label{1.9(1)}
\begin{proof}[Proof of Theorem \ref{Proposition34.25} (1) ]
Let $(N,\omega)$ be a symplectic manifold, $M = N\times N$, and
$\omega_M = -{\rm pr}_1^*\omega_N + {\rm pr}_2^*\omega_N$.
We consider an anti-symplectic involution $\tau : M \to M$ defined by $\tau(x,y) = (y,x)$.
Then $L = \text{\rm Fix}\,\tau \cong N$. Let $J_N$ be a compatible almost structure on
$N$, and $J_M = -J_N \otimes 1 +1 \otimes J_N$. The almost
complex structure $J_M$ is compatible with $\omega_M$.
Note that $w_2(T(N\times N))={\rm pr}_1^{\ast}w_2(TN) + {\rm pr}_2^{\ast}w_2(TN)$.
\par
If $N$ is spin, then $L= \text{\rm Fix}\,\tau \cong N$ is $\tau$-relatively spin
by Example \ref{Remark44.18} and $c_1(T(N\times N)) \equiv
w_2(T(N\times N)) \equiv 0 \mod 2$.
Since $\pi_1(L) \to \pi_1(N \times N)$ is injective,
Corollary \ref{Corollary34.22} shows that $L$ is unobstructed
and $\mathfrak m_2$ defines a graded commutative product structure $\cup_Q$ by (\ref{qproduct}).
\par
Suppose that $N$ is not spin. We take a relative spin structure $(V,\sigma)$ on $L= \text{\rm Fix}\,\tau \cong N$ such that
$V={\rm pr}_1^{\ast}(TN)$ and $\sigma$ is the following spin structure on
$(TL \oplus V)\vert_{L} \cong (TL \oplus TL)\vert_{L}$.
Since the composition of the diagonal embedding $SO(n) \to SO(n) \times SO(n)$
and the inclusion $SO(n) \times SO(n) \to SO(2n)$ admits a unique lifting
$SO(n) \to Spin(2n)$, we can equip the bundle $TL \oplus TL$
with a canonical spin structure.
It determines the spin structure $\sigma$ on $(TL \oplus V)\vert_L$.
(In this case, we have $st={\rm pr}_1^{\ast}w_2(TN)$.)
Then clearly we find that $\tau^{\ast} V={\rm pr}_2^{\ast}TN$.
Note that 
${\rm pr}_1^*TN$ and ${\rm pr}_2^*TN$ are canonically isomorphic to $TL$ by
the differentials of the projections ${\rm pr}_1$ and ${\rm pr}_2$, respectively.
On the other hand, since $(TL \oplus \tau^{\ast}V)\vert_{L}  \cong (TL \oplus TL)\vert_{L} \cong
(TL \oplus V)\vert_{L}$, the spin structure
$\sigma$ is preserved by $\tau$.
Therefore the difference of the conjugacy classes of two relative spin structures
$[(V,\sigma)]$ and $\tau^{\ast}[(V,\sigma)]$ is measured
by
$w_2(V\oplus \tau^{\ast}V)= w_2({\rm pr}_1^{\ast}TN \oplus
{\rm pr}_2^{\ast}TN)$.
Using the canonical spin structure on $TL\oplus TL$
mentioned above,
we can give a trivialization of $V \oplus \tau^*V$
over the 2-skeleton of $L$.
Hence $w_2(V \oplus \tau^*V)$ is regarded as a class in
$H^2(N \times N,L;\Z_2)$.
Since $w_2({\rm pr}_1^{\ast}TN \oplus
{\rm pr}_2^{\ast}TN)=w_2(T(N\times N))\equiv c_1(T(N\times N))
\mod 2$ and $\pi_2(N\times N) \to \pi_2(N\times N, L)$ is surjective, Lemma \ref{c1maslov} 
shows that the class is equal to $\mu_L /2$.
Hence by Proposition \ref{Proposition44.16} we obtain the following:
\begin{lem}\label{changespin}
In the above situation,
the identity map
$$
\CM(J;\beta)^{[(V,\sigma)]} \longrightarrow
\CM(J;\beta)^{\tau^{\ast}[(V,\sigma)]}
$$
is orientation preserving if and only if
$\mu_L(\beta)/2$ is even.
\end{lem}
Combining Theorem \ref{withsimplex}, we find that
the composition
$$
\aligned
\CM ^{\rm main}_{k+1} (J;\beta;P_1,\ldots, P_k)^{[(V,\sigma)]}
& \longrightarrow
\CM ^{\rm main}_{k+1} (J;\beta;P_1,\ldots, P_k)^{\tau^{\ast}[(V,\sigma)]} \\
& \longrightarrow
\CM ^{\rm main}_{k+1} (J;\beta;P_k,\ldots, P_1)^{[(V,\sigma)]}
\endaligned$$
is orientation preserving if and only if
$$
k+1 +\sum_{1 \le i < j \le k} \deg'P_i\deg'P_j
$$
is even.
It follows that $\mathfrak m_{0,\tau_{\ast}\beta}(1)=-\mathfrak m_{0,\beta}(1)$ and hence we find that
$L= \text{\rm Fix}~\tau \cong N$ is unobstructed.
This finishes the proof of the assertion (1).
Moreover, we also find that $\mathfrak m_2$ satisfies Theorem \ref{Proposition34.25} (\ref{34.23}), which
induces the graded commutative product $\cup_Q$ as well for the non-spin case.
\end{proof}

\subsubsection{Proof of Theorem \ref{Proposition34.25} (2), I: preliminaries}
Before starting the proof of Theorem \ref{Proposition34.25} (2)
we clarify the choice of the bounding cochain $b$ for which this
statement holds.
Note we constructed a filtered $A_{\infty}$ structure on 
$C(L;\Lambda_{0,{\rm nov}}^{\Q})$
using $\tau$-invariant Kuranishi structure and $\tau$-invariant perturbation.
As we proved in Subsection \ref{1.9(1)}, $b=0$ is the bounding cochain of this
filtered $A_{\infty}$ structure.
In fact $\frak m_0$ becomes $0$ in the chain level by the cancelation.
This choice $b=0$ is one for which the conclusion of Theorem \ref{Proposition34.25} (2) holds.
\begin{rem}\label{b}
This particular choice $b=0$ does not make sense unless we specify the
particular way to construct our filtered $A_{\infty}$ structure.
Suppose we define a filtered $A_{\infty}$ structure on $C(L;\Lambda_{0,{\rm nov}}^{\Q})$
using a different perturbation.
The set of the gauge equivalence classes of the bounding cochains are independent of
the choice up to isomorphism, so there exists certain
bounding cochain which corresponds to the $0$ of the filtered $A_{\infty}$ structure
defined by the $\tau$-invariant perturbation.
The conclusion of Theorem \ref{Proposition34.25} (2) holds for that $b$ and that
filtered $A_{\infty}$ structure. However $b = 0$ may not hold in this different filtered $A_{\infty}$ structure.
\end{rem}
We now start the proof of Theorem \ref{Proposition34.25} (2).
\par
Firstly we explain the proof of Theorem \ref{Proposition34.25} (2) under the hypothesis that there do not appear
holomorphic disc bubbles.
\par
Let $v : S^2 \to N$ be a $J_N$-holomorphic map. We fix 3 marked points
$0,1,\infty \in S^2 = \Bbb C \cup \{\infty\}$. Then we consider the
upper half plane $\H \cup\{\infty\} \subset \C\cup \{\infty\}$ and
define a map $I(v) : \Bbb H \to M$ by
$$
I(v)(z) = (v(\overline z),v(z)).
$$
Identifying $(\Bbb H,(0,1,\infty))$ with $(D^2,(-1,1,\sqrt{-1}))$
where $(-1,1,\sqrt{-1}) \in \partial D^2$,
we obtain a map from $(D^2,\partial D^2)$ to $(M,L)$
which we also denote by $I(v)$. One can easily check the converse:
For any given $J_M$-holomorphic map
$w: (D^2,\partial D^2) \cong
(\H, \R\cup \{\infty \}) \to (M,L) = (N \times N, \Delta_N)$,
the assignment
$$
v(z) = \begin{cases} {\rm pr}_2 \circ w(z) & \quad \text{for $z \in \H$} \\
{\rm pr}_1 \circ w(\overline z) & \quad \text{for $z\in \C
\setminus \H$}
\end{cases}
$$
defines a $J_N$-holomorphic sphere on $N$.
Therefore the map
$v \mapsto I(v)$ gives an isomorphism between the moduli spaces of
$J_N$-holomorphic spheres and $J_M$-holomorphic discs with boundary in
$N$.
We can easily check that this map is induced by the isomorphism of
Kuranishi structures.
\par
We remark however that this construction works only at the interior of the
moduli spaces of pseudo-holomorphic discs and of pseudo-holomorphic spheres,
that is the moduli spaces of those without bubble.
To study the relationship between compactifications of them
we need some extra argument, which will be explained later in Subsection \ref{proof1.9}.
\par
We next compare the orientations on these moduli spaces.
The moduli spaces of holomorphic spheres have canonical orientation, see, e.g.,
Section 16 in \cite{FO}.
In Chapter 8 \cite{fooobook2}, we proved that a relative spin structure determines
a system of orientations on the moduli spaces of bordered stable maps of genus $0$.
We briefly review a crucial step for comparing orientations in our setting. See p. 677 of \cite{fooobook2}.
\par
Let $w:(D^2, \partial D^2) \to (M,L)$ be a $J_M$-holomorphic map.  Denote by $\ell$
the restriction of $w$ to $\partial D^2$.
Consider the Dolbeault operator
\par
$$\overline{\partial}_{(w^*TM, \ell^*TL)}:
W^{1,p}(D^2,\partial D^2;w^*TM, \ell^*TL) \to L^p(D^2;w^*TM \otimes \Lambda_{D^2}^{0,1})
$$
with $p>2$.
We deform this operator to an operator on the union $\Sigma$ of $D^2$ and $\C P^1$ with
the origin $O$ of $D^2$ and the ``south pole'' $S$ of $\C P^1$ identified.
The spin structure $\sigma$ on $TL \oplus V\vert_L$ gives a trivialization of
$\ell^*(TL \oplus V\vert_L)$.
Since $w^*V$ is a vector bundle on the disc, it has a unique trivialization up to homotopy.
Hence $\ell^* V$ inherits a trivialization, which is again unique up to homotopy.
Using this trivialization, we can descend the vector bundle
$E=w^*TM$ to $E'$ on $\Sigma$.
The index problem is reduced to the one for the Dolbeault operator on $\Sigma$.
Namely, the restriction of the direct sum of the following two operators to the fiber product
of the domains with respect to the evaluation maps at $O$ and $S$.
On $D^2$, we have the Dolbeault operator for the trivial vector bundle pair $(\underline{\C^n},
\underline{\R^n})$.
On $\C P^1$, we have the Dolbeault operator for the vector bundle $E'\vert_{\C P^1}$.
The former operator is surjective and its kernel is the space of constant sections in
$\underline{\R^n}$.
The latter has a natural orientation, since it is Dolbeault operator twisted
by $E'\vert_{\C P^1}$ on a closed Riemann surface.
Since the fiber product of kernels is taken on a complex vector space,
the orientation of the index is determined by the orientations of the two operators.
\par
Now we go back to our situation.
Pick a $1$-parameter family $\{\phi_t\}$ of dilations on $\C P^1=\C \cup \{ \infty\}$ such that
$\lim_{t \to +\infty} \phi_t(z) = -\sqrt{-1}$ for $z \in \C \cup \{\infty\} \setminus \{\sqrt{-1}\}$.
Here $\sqrt{-1}$ in the upper half plane and $-\sqrt{-1}$ in the lower half plane correspond to
the north pole and the south pole of $\C P^1$, respectively.
As $t \to + \infty$, the boundary of the second factor of the disc
$I(v \circ \phi_t)$ contracts to the point $v(-\sqrt{-1})$ and its image
exhausts the whole image of the sphere $v$, while the whole image of the first factor
contracts to $v(-\sqrt{-1})$.
Therefore as $t \to \infty$  the images of the map $z \mapsto I(v \circ \phi_t)(z)$ converge
to the constant disc at $(v(-\sqrt{-1}),v(-\sqrt{-1})$ with a sphere
$$
z \in S^2 \mapsto (v(-\sqrt{-1}), v(z))
$$
attached to the point. If we denote $w_t = I(v\circ \phi_t)$, it follows from our choice
$V = {\rm pr}_1^*TN$ that the trivialization of $\ell_t^*V$, which
is obtained by restricting the trivialization of $w_t^*V = ({\rm pr}_1\circ I(v\circ \phi_t))^*TN$,
coincides with the one induced by the frame of the fiber $V$ at $v(-\sqrt{-1})$ for
a sufficiently large $t$. Therefore considering the linearized index of the family
$w_t$ for a large $t$, it follows from the explanation given in the above
paragraph that the map  $v \mapsto I(v\circ \phi_t)=w_t$ induces an isomorphism
$$
\det (\text{Index} D\delbar_{J_N}(v)) \cong \det(\text{Index} D\delbar_{J_M}(w_t))
$$
as an oriented vector space.
By flowing the orientation
to $t=0$ by the deformation $\phi_t$, we have proven that the map
$v \mapsto I(v)$ respects the orientations of the moduli spaces.
\par
Now we compare the product $\cup_Q$ in  (\ref{qproduct}) and the product on the quantum
cohomology, presuming, for a while, that
they can be calculated by the contribution from the interior of the moduli spaces only.
\par
\begin{defn}\label{equivonpi2}
We define the equivalence relation $\sim$ on $\pi_2(N)$ by
$\alpha \sim \alpha'$ if and only if $c_1(N)[\alpha] = c_1(N) [\alpha']$ and
$\omega (\alpha ) = \omega (\alpha')$.
\end{defn}
For $\beta=[w:(D^2, \partial D^2) \to (N \times N, \Delta_N)] \in \Pi (\Delta_N)$,
we set
$\widetilde{\beta}=[({\rm pr}_2 \circ w) \# ({\rm pr}_1 \circ w): D^2 \cup \overline{D}^2 \to N] \in \pi_2(N)/\sim$,
where $\overline D^2$ is the unit disc with the complex structure reversed and
$D^2 \cup \overline{D}^2$ is the union of discs glued along boundaries.
This defines a homomorphism
\begin{equation}\label{picorresp}
\psi : \Pi(\Delta_N) \to \pi_2(N)/\sim.
\end{equation}
For $\alpha \in \pi_2(N)$ let $\mathcal M_3^{\text{\rm sph,reg}}(J_{N};\alpha)$
be the moduli space of pseudo-holomorphic map
$v : S^2 \to N$ of homotopy class $\alpha$ with three marked points,
(without bubble). For $\rho \in \pi_2(N)/\sim$, we put
$$
\mathcal M_3^{\text{\rm sph,reg}}(J_{N};\rho)
= \bigcup_{\alpha \in \rho}
\mathcal M_3^{\text{\rm sph,reg}}(J_{N};\alpha).
$$
For $\beta \in \Pi(\Delta_N)$, let $\mathcal M_3^{\text{\rm reg}}(J_{N\times N};\beta)$
be the moduli space of pseudo-holomorphic map
$u : (D^2,\partial D^2) \to (N\times N,\Delta_N)$ of class $\beta$ with
three boundary marked points (without disc or sphere bubble).
We denote by $\mathcal M_3^{\text{\rm main,reg}}(\beta)$
the subset that consists of elements in the main component.
We put
$$
\mathcal M_3^{\text{\rm main,reg}}(J_{N\times N};\rho)
= \bigcup_{\psi(\beta) = \rho}
\mathcal M_3^{\text{\rm main,reg}}(J_{N\times N};\beta).
$$
Summing up the above construction, we have the following proposition.
For a later purpose, we define the map
$\overset{\circ}{\mathfrak I}$ by the inverse of $I$.
\begin{prop}\label{regoripres}
\begin{equation}\label{mapI}
\overset{\circ}{\mathfrak I} :  \mathcal M_3^{\text{\rm main,reg}}(J_{N\times N};\rho)  \to \mathcal M_3^{\text{\rm sph,reg}}(J_{N};\rho)
\end{equation}
is an isomorphism as spaces with Kuranishi structure.  Moreover, $\overset{\circ}{\mathfrak I}$ respects the orientations in the sense of
Kuranishi structure.
\end{prop}
Denote by $*$ the quantum cup product on the quantum cohomology 
$QH^*(N; \Lambda_{0,{\rm nov}}^{\Q})$.  
\par
For cycles $P_0, P_1, P_2$ in $N$ such that 
$\dim {\mathcal M}_3^{\text{\rm sph}}(J_N;\rho) = \deg P_0 + \deg P_1 + \deg P_2$, 
we can take a multi-valued perturbation, {\it multi-section}, 
of ${\mathcal M}_3^{\text{\rm sph}}(J_N;\rho)$ such that the intersection of its zero set and 
$(ev_0 \times ev_1 \times ev_2 )^{-1}(P_0 \times P_1 \times P_2)$ is a finite subset in 
${\mathcal M}_3^{\text{\rm sph, reg}}(J_N;\rho)$, i.e., 
it does not contain elements with domains with at least two irreducible components.  
Counting these zeros with weights, we obtain the intersection number 
$$(P_0 \times P_1 \times P_2) \cdot [{\mathcal M}_3^{\text{\rm sph, reg}}(J_N;\rho)].$$  
In other words, 
for homology classes $[P_0], [P_1], [P_2] \in H_*(N)$, we take cocycles $a_0, a_1, a_2$ 
which represent the Poincar\'e dual of $[P_0], [P_1], [P_2]$, respectively.  
We can take a multi-section of $\mathcal M_3^{\text{\rm sph}}(J_{N};\rho)$ such that 
the intersection of 
its zero set and the support of $ev_0^* a_0 \cup ev_1^* a_1 \cup ev_2^* a_2$  
is contained in ${\mathcal M}_3^{\text{\rm sph, reg}}(J_N;\rho)$.     
Since the zero set of the multi-valued perturbation of ${\mathcal M}_3^{\text{\rm sph}}(J_N;\rho)$ 
is compact, $ev_0^* a_0 \cup ev_1^* a_1 \cup ev_2^* a_2$ is regarded as 
a cocycle with a compact support.  
Using such a multi-valued perturbation, we obtain 
$[\CM^{\text{\rm sph,reg}}_3(J_{N};\rho)]$, which ia locally finite fundamental cycle.  
Thus we find that  $(ev_0^*a_0 \cup ev_1^* a_1\cup ev_2^* a_2) [\CM^{\text{\rm sph,reg}}(J_N;\rho)]$ 
makes sense.  
\par
The Poincar\'e pairing on cohomology is given by 
\begin{equation}\label{cohompairing}
\langle a, b \rangle = (a \cup b) [N].
\end{equation}
By definition we have 
\begin{eqnarray}
\langle  a_0, a_1*a_2 \rangle 
& 
=  &\sum_{\rho \in \pi_2(N)/\sim}
( ev_0^* a_0 \cup ev_1^* a_1 \cup ev_2^* a_2)
 [\CM^{\text{\rm sph}}_3(J_{N};\rho)]T^{\omega (\rho)}e^{c_1(N)[\rho]}  \nonumber \\
& = & \sum_{\rho \in \pi_2(N)/\sim}
( ev_0^* a_0 \cup ev_1^* a_1 \cup ev_2^* a_2)
 [\CM^{\text{\rm sph,reg}}_3(J_{N};\rho)]T^{\omega (\rho)}e^{c_1(N)[\rho]}. \nonumber 
\end{eqnarray}
From the assumption we made at the beginning of this subsection,  the map $\mathfrak{m}_2$ is given by
\begin{equation}\label{comefrommain}
\sum_{\beta; \psi(\beta) = \rho}
\mathfrak{m}_{2,\beta} (P_1,P_2)T^{\omega(\beta)}e^{\mu(\beta)/2}
=(\CM^{\text{main,reg}}_3(J_{N\times N};\rho;P_1,P_2),ev_0)T^{\omega(\beta)}e^{\mu(\beta)/2}.
\end{equation}
Here $P_1$ and $P_2$ are cycles,
and
$$
\CM^{\text{main,reg}}_3(J_{N\times N};\rho;P_1,P_2)
= (-1)^{\epsilon} \CM^{\text{main,reg}}_3(J_{N\times N};\rho) \times_{N^2} (P_1\times P_2), 
$$
where $\epsilon = (\dim \Delta_N +1)\deg P_1= \deg P_1$, see \eqref{epsilonPQ}.
(We also assume that
the right hand side becomes a cycle.)
These assumptions are removed later in Subsections \ref{6.3}, \ref{6.4} and \ref{6.5}.
\par
Taking a homological intersection number with another cycle $P_0$, we have
$$
\aligned
& P_0 \cdot (\CM^{\text{main,reg}}_3(J_{N\times N};\rho;P_1,P_2),ev_0) \\
= & (-1)^{\epsilon} P_0 \cdot \left( \CM^{\text {main,reg}}_3(J_{N\times N};\rho) \times _
{(ev_1,ev_2)}(P_1\times P_2), ev_0  \right) \\
= & (-1)^{\epsilon} (ev_0^* PD[P_0] \cup
(ev_1,ev_2)^* PD[P_1 \times P_2] )[\CM^{\text{main,reg}}_3(J_{N\times N};\rho)],
\endaligned
$$
where $PD[P_i]$, resp. $PD[P_j \times P_k]$ is the Poincar\'e dual of
$P_i$ in $N$, resp. $P_j \times P_k$ in $N \times N$.  
We adopt the convention that 
\begin{equation}\label{PDconvention}
(PD [P] \cup a) [N] = a [P] \ \ \text{~for~} a \in H^*(N).
\end{equation}
Since $\dim \Delta_N$ is even, \eqref{PDconvention} implies that  
$$
ev_1^*PD[P_1] \cup ev_2^*PD[P_2]=(-1)^{\deg P_1 \cdot \deg P_2} (ev_1,ev_2)^*PD[P_1 \times P_2].$$
By identifying $\CM_{3}^{\text{\rm sph,reg}}(J_{N};\rho)$ and
$\CM^{\text{main,reg}}_{3}(J_{N\times N};\rho)$ as spaces with oriented Kuranishi structures, we find that
$$
\langle PD[P_0],  PD[P_1]  * PD[P_2] \rangle
 = (-1)^{\deg P_1 (\deg P_2 +1)}  \langle P_0, \mathfrak{m}_2(P_1, P_2) \rangle, 
$$
or equivalently, 
\begin{equation}
\langle PD[P_1]  * PD[P_2], PD[P_0] \rangle
= (-1)^{\deg P_1 (\deg P_2 +1)}  \langle  \mathfrak{m}_2(P_1, P_2), P_0 \rangle. \label{comparisonprod}
\end{equation}

Here the right hand side is the intersection product of $P_0$ and $\mathfrak{m}_2(P_1, P_2)$.  
Namely, we put 
\begin{equation}\label{signpairing}
\langle P, Q \rangle = P \cdot Q = (-1)^{\deg P \deg Q} \# (P \times_N Q) = \# (Q \times_N P).
\end{equation}
Note that we use a different convention of the pairing on cycles from \cite{fooobook1}, \cite{fooobook2}, 
cf. Definition 8.4.6 in \cite{fooobook2}, but the same as one in Definition 3.10.4, Subsection 3.10.1 in \cite{fooomirror1}.  
Therefore we observe the following consistency between pairings on homology \eqref{signpairing} and cohomology \eqref{cohompairing} 
$$\langle P, Q \rangle = \langle PD [P], PD [Q] \rangle.$$

\subsubsection{Proof of Theorem  \ref{Proposition34.25} (2), II: the isomorphism as modules}\label{6.3}
To complete the proof of Theorem \ref{Proposition34.25} (2), we need to
remove the assumption (\ref{comefrommain}), that is, the product ${\mathfrak m}_2(P_1,P_2)$ is
determined only on the part of the moduli space
where there is no bubble.
We study
how our identification of the moduli spaces of pseudo-holomorphic discs (attached to the diagonal
$\Delta_N$) and of pseudo-holomorphic spheres (in $N$) extends to their compactifications
for this purpose.
To study this point, we define the isomorphism in Theorem \ref{Proposition34.25} (2)
as {\it $\Lambda_{0,{\rm nov}}$ modules} more explicitly.
\par
As discussed in the introduction, this isomorphism follows
from the degeneration at $E_2$-level of the spectral sequence of
Theorem D \cite{fooobook1}. The proof of this degeneration is based on the
fact that the image of the differential is contained in the
Poincar\'e dual to the kernel of the inclusion induced homomorphism
$H(\Delta_N;\Lambda_{0,{\rm nov}}) \to H(N \times N;\Lambda_{0,{\rm nov}})$, which
is actually injective in our case.
This fact (Theorem D (D.3) \cite{fooobook1}) is proved by using
the operator $\mathfrak p$ introduced in \cite{fooobook1} Section 3.8.
Therefore to describe this isomorphism we recall
a part of the construction of this operator below.
\par
Let $\beta \in \Pi(\Delta_N) =\pi_2(N\times N, \Delta_N) /\sim$.
We consider $\mathcal M_{1;1}(J_{N\times N};\beta)$,
the moduli space of bordered stable maps of genus zero
with one interior and one boundary marked point in homotopy class $\beta$.
Let $ev_0 : \mathcal M_{1;1}(J_{N\times N};\beta) \to \Delta_N$ be the
evaluation map at the boundary marked point and
$ev_{\text{\rm int}} : \mathcal M_{1;1}(J_{N\times N};\beta) \to N \times N$ be
the evaluation map at the interior marked point.
Let $(P,f)$ be a smooth singular chain in $\Delta_N$. We put
$$
\mathcal M_{1;1}(J_{N\times N};\beta;P)
=  \mathcal M_{1;1}(J_{N\times N};\beta) \, {}_{ev_0}\times_f
P.
$$
It has a Kuranishi structure. We take its multisection $\mathfrak s$ and
a triangulation of its zero set $\mathcal M_{1;1}(J_{N\times N};\beta;P)^{\mathfrak s}$.
Then $(\mathcal M_{1;1}(J_{N\times N};\beta;P)^{\mathfrak s},ev_{\text{\rm int}})$ is a  singular
chain in $N \times N$, which is by definition
$\mathfrak p_{1,\beta}(P)$.
(See \cite{fooobook1} Definition 3.8.23.)
In our situation, where $\mathfrak m_{0}(1) = 0$ for
$\Delta_N$ in the chain level by Theorem \ref{Proposition34.25} (1), we have:
\par
\begin{lem}\label{pmainformula} We identify $N$ with $\Delta_N$. Then
for any singular chain $P \subset N$ we have
\begin{equation}\label{frakpfrakm}
(-1)^{n+1}\partial_{N \times N} (\mathfrak p_{1,\beta}(P) )
+ \sum_{\beta_1+\beta_2=\beta}
\mathfrak p_{1,\beta_1}(\mathfrak m_{1,\beta_2}(P)) = 0.
\end{equation}
Here $\partial_{N \times N}$ is the boundary operator in the singular chain complex of $N\times N$.
\end{lem}
\begin{rem}
When $\beta = 0$, ${\mathfrak p}_{1,0}$ is the identity map and
the second term in \eqref{frakpfrakm} is equal to ${\mathfrak m}_{1,0}(P)$.
Recalling ${\mathfrak m}_{1,0}(P)=(-1)^n \partial_{\Delta_N}(P)$
in Definition \ref{def:mk}, \eqref{frakpfrakm} turns out to be
\begin{equation}\label{frakpfrakmclasical}
(-1)^{n+1}\partial_{N \times N} (\mathfrak p_{1,0}(P) )
+ (-1)^n
\mathfrak p_{1,0}(\partial_{\Delta_N}(P)) = 0.
\end{equation}
When $\beta \neq 0$, \eqref{frakpfrakm} is equal to
$$
(-1)^{n+1}\partial_{N \times N} (\mathfrak p_{1,\beta}(P) )
+\mathfrak p_{1,\beta}(\mathfrak m_{1,0} (P))
+ \mathfrak p_{1,0}(\mathfrak m_{1,\beta}(P))
+ \sum_{\genfrac{}{}{0pt}{}{\beta_1+\beta_2=\beta,}
{\beta_1, \beta_2 \ne 0}}
\mathfrak p_{1,\beta_1}(\mathfrak m_{1,\beta_2}(P)) = 0.
$$
\end{rem}
Lemma \ref{pmainformula} is a particular case of \cite{fooobook1}
Theorem 3.8.9 (3.8.10.2).
See Remark \ref{rem:signNN} for the sign.
We also note that $\mathfrak p_{1,0}(P) = P$.
(\cite{fooobook1} (3.8.10.1).)
We remark that even in the case when $P$ is a singular cycle $\mathfrak m_1(P)$
may not be zero. In other words the identity map
\begin{equation}\label{notchainmap}
(C(\Delta_N;\Lambda_{0,{\rm nov}}),\partial) \to (C(\Delta_N;\Lambda_{0,{\rm nov}}),\mathfrak m_1)
\end{equation}
is {\it not} a chain map. We use the operator
$$
\mathfrak p_{1,\beta}: C(\Delta_N;\Lambda_{0,{\rm nov}}) \to C(N \times N; \Lambda_{0,{\rm nov}})
$$
to modify the identity map to obtain a chain map (\ref{notchainmap}).
Using the projection to the second factor, we define 
$p_2 : N \times N \ni (x,y) \mapsto (y,y) \in \Delta_N$.
We put
$$
\overline{\mathfrak p}_{1,\beta} = p_{2 *} \circ \mathfrak p_{1,\beta}.
$$
Then by applying $p_{2 *}$ to the equation \eqref{frakpfrakm},
we obtain, for $\beta \neq 0$,
\begin{equation}
\label{pbarformula}
-\mathfrak m_{1,0} (\overline{\mathfrak p}_{1,\beta}(P) )
+\overline{\mathfrak p}_{1,\beta}(\mathfrak m_{1,0} (P))
+ \mathfrak m_{1,\beta}(P)
+ \sum_{\genfrac{}{}{0pt}{}{\beta_1+\beta_2=\beta,}
{\beta_1, \beta_2 \ne 0}}
\overline{\mathfrak p}_{1,\beta_1}(\mathfrak m_{1,\beta_2}(P)) = 0.
\end{equation}
\begin{rem}\label{rem:signNN}
The sign in Formula (\ref{frakpfrakm})
looks slightly different from one in 
Formula (3.8.10.2) in \cite{fooobook1}. 
The sign for $\delta_M$ in \cite{fooobook1}, \cite{fooobook2} was not
specified, since it was not necessary there. 
Here we specify it as $\delta_M=(-1)^{n+1}\partial_M$. 
This sign is determined by considering the case
when $\beta=0$, which is nothing but (\ref{frakpfrakmclasical}).
Another way to determine this sign is as follows. In the proof of  
(3.8.10.2) given in Subsection 3.8.3 \cite{fooobook1}, we use the same argument in the proof of Theorem 3.5.11 (=Theorem \ref{thm:Ainfty} in this article) 
where we define $\mathfrak m_{1,0}=(-1)^n\partial_L$
to get the $A_{\infty}$ formula. 
The proof of Theorem 3.5.11 uses 
Proposition 8.5.1 \cite{fooobook2} which is the case $\ell = \ell_1 =\ell_2 = 0$ in the formulas 
in Proposition 8.10.5  \cite{fooobook2}. 
On the other hand, in the case of (3.8.10.2) 
we use the $0$-th {\it interior} marked point as the output evaluation point instead of the $0$-th {\it boundary} marked point. 
Then the proof of (3.8.10.2) uses  
Proposition 8.10.4 instead of Proposition 8.10.5  \cite{fooobook2}.
We can see that the difference of every corresponding sign appearing 
in Proposition 8.10.4 and Proposition 8.10.5  \cite{fooobook2} is exactly $-1$.
Thus we find that 
$\delta_{M}=(-1)^{n+1}\partial_M$ in  
the formula (3.8.10.2) (and also (3.8.10.3)) of \cite{fooobook1}.
This difference arises from 
the positions of the factors corresponding to the $0$-th {\it boundary} marked point and
the $0$-th {\it interior} marked point
in the definitions of  
orientations on $\mathcal{M}_{(1,k),\ell}(\beta)$ and 
$\mathcal{M}_{k,(1,\ell)}(\beta)$ respectively. See the formulas given just before Definition 8.10.1 and Definition 8.10.2  \cite{fooobook2} for the 
notations $\mathcal{M}_{(1,k),\ell}(\beta)$,  
$\mathcal{M}_{k,(1,\ell)}(\beta)$ respectively.
\end{rem}
\begin{defn}\label{P_beta}
For each given singular chain $P$ in $N$,
we put
$$
P(\beta) = \sum_{k=1}^{\infty} \sum_{\genfrac{}{}{0pt}{}{\beta_1+\dots+\beta_k = \beta}
{\beta_i  \ne 0}}
(-1)^{k}\left( \overline{\mathfrak p}_{1,\beta_1} \circ \cdots \circ
\overline{\mathfrak p}_{1,\beta_k} \right)(P)
$$
regarding $P$ as a chain in $\Delta_N$. Then we define
a chain $\mathcal I(P) \in C(N;\Lambda_{0,{\rm nov}})$ by
$$
\mathcal I(P) = P + \sum_{\beta\ne 0} P(\beta)T^{\omega(\beta)}e^{\mu(\beta)/2}.
$$
\end{defn}
\begin{lem}
$$
\mathcal I : (C(\Delta_N;\Lambda_{0,{\rm nov}}),\mathfrak m_{1,0}) \to (C(\Delta_N;\Lambda_{0,{\rm nov}}),\mathfrak m_1)
$$
is a chain homotopy equivalence.
\end{lem}
\begin{proof}
We can use \eqref{pbarformula} to show that $\mathcal I$ is a chain map
as follows.
We prove that $\mathfrak m_1 \circ \mathcal I -   \mathcal I \circ\mathfrak m_{1,0} \equiv 0
\mod T^{\omega(\beta)}$ by induction on $\omega(\beta)$.
We assume that it holds modulo  $T^{\omega(\beta)}$ and
will study the terms of order  $T^{\omega(\beta)}$.
Those terms are sum of
\begin{equation}\label{ordbetaterms1}
\mathfrak m_{1,0}(P(\beta))  + {\mathfrak m}_{1, \beta} (P)+
\sum_{\genfrac{}{}{0pt}{}{\beta_1+\beta_2=\beta}{\beta_i \ne 0}}
\mathfrak m_{1,\beta_1}(P(\beta_2)) - ({\mathfrak m}_{1,0} P)(\beta),
\end{equation}
for given $\omega(\beta)$'s.
By definition, (\ref{ordbetaterms1}) becomes:
\begin{equation}\label{ordbetaterms2}
\aligned
&\sum_{k=1,2,\dots}\sum_{\genfrac{}{}{0pt}{}{\beta_1+\dots+\beta_k = \beta}
{\beta_i  \ne 0}} (-1)^k \left(\mathfrak m_{1,0} \circ
 \overline{\mathfrak p}_{1,\beta_1} \circ \cdots \circ
\overline{\mathfrak p}_{1,\beta_k}\right)(P) \\
&+\sum_{k=1,2,\dots}
\sum_{\genfrac{}{}{0pt}{}{\beta_1+\dots+\beta_k = \beta}
{\beta_i  \ne 0}} (-1)^{k-1} \left(\mathfrak m_{1,\beta_1} \circ
 \overline{\mathfrak p}_{1,\beta_2} \circ \cdots \circ
\overline{\mathfrak p}_{1,\beta_k}\right)(P) \\
&  - ({\mathfrak m}_{1,0} P)(\beta).
\endaligned
\end{equation}
Using \eqref{pbarformula}
we can show that \eqref{ordbetaterms2} is equal to
\begin{equation}\label{ordbetaterms3}
\aligned
&\sum_{k=1,2,\dots}
\sum_{\genfrac{}{}{0pt}{}{\beta_1+\dots+\beta_k = \beta}
{\beta_i  \ne 0}} (-1)^k \left(  \overline{\mathfrak p}_{1,\beta_1}\circ
\mathfrak m_{1,0} \circ \overline{\mathfrak p}_{1,\beta_2}\circ\cdots \circ
\overline{\mathfrak p}_{1,\beta_k}\right)(P) \\
&+\sum_{k=1,2,\dots}
\sum_{\genfrac{}{}{0pt}{}{\beta_1+\dots+\beta_k = \beta}
{\beta_i  \ne 0}} (-1)^k \left( \overline{\mathfrak p}_{1,\beta_1} \circ
\mathfrak m_{1,\beta_2} \circ  \overline{\mathfrak p}_{1,\beta_3} \circ \cdots \circ
\overline{\mathfrak p}_{1,\beta_k}\right)(P) \\
& - \sum_{k=1,2,\dots}
\sum_{\genfrac{}{}{0pt}{}{\beta_1+\dots+\beta_k = \beta}
{\beta_i  \ne 0}} (-1)^k \left( \overline{\mathfrak p}_{1,\beta_1} \circ
\overline{\mathfrak p}_{1,\beta_2} \circ  \cdots \circ \overline{\mathfrak p}_{1,\beta_{k}} \circ {\mathfrak m}_{1,0} \right)(P) \\
= & \sum_{k=1,2,\dots}
\sum_{\beta_1 \neq 0}  \overline{\mathfrak p}_{1,\beta_1} \circ
\sum_{\genfrac{}{}{0pt}{}{\beta_1+\dots+\beta_k = \beta}
{\beta_i  \ne 0}} (-1)^k \bigl(
\mathfrak m_{1,0} \circ \overline{\mathfrak p}_{1,\beta_2}\circ\cdots \circ
\overline{\mathfrak p}_{1,\beta_k}(P) \\
& + \mathfrak m_{1,\beta_2} \circ  \overline{\mathfrak p}_{1,\beta_3} \circ \cdots \circ
\overline{\mathfrak p}_{1,\beta_k}(P)
- \overline{\mathfrak p}_{1,\beta_2} \circ  \cdots \circ \overline{\mathfrak p}_{1,\beta_{k}} \circ {\mathfrak m}_{1,0}(P)
\bigr).
\endaligned
\end{equation}
(\ref{ordbetaterms3}) vanishes by induction hypothesis.
\par
On the other hand $\mathcal I$ is identity modulo $\Lambda_{0,{\rm nov}}^+$.
The lemma follows by the standard spectral sequence argument.
\end{proof}
Thus we obtain an isomorphism as $\Lambda_{0,{\rm nov}}$-modules
$$
\mathcal I_{\#} : H(N;\Lambda_{0,{\rm nov}}) \cong HF(\Delta_N,\Delta_N;\Lambda_{0,{\rm nov}}).
$$
In order to complete the proof of Theorem \ref{Proposition34.25} (2),
we need to prove
\begin{equation}\label{diagmaineq'}
\langle
\mathcal I(P_1) \cup_Q \mathcal I(P_2),\mathcal I(P_0)\rangle
=
\langle
PD[P_1] * PD[P_2],PD[P_0]\rangle, 
\end{equation}
i.e., 
\begin{equation}\label{diagmaineq}
(-1)^{\deg P_1 (\deg P_2 +1)} 
\langle
\mathfrak m_2(\mathcal I(P_1),\mathcal I(P_2)),\mathcal I(P_0)\rangle
=
\langle
PD[P_1] * PD[P_2],PD[P_0]\rangle.
\end{equation}
For the orientation of the moduli spaces which define the operations in \eqref{diagmaineq},
see  
Sublemma \ref{sublemma}, \eqref{M_{1,1}}-\eqref{def:Mhat} for the left hand side
and Definition \ref{notationmoduli} for the right hand side.
We will prove (\ref{diagmaineq}) in the next two subsections.

\subsubsection{Proof of Theorem \ref{Proposition34.25} (2), III: moduli space  of stable
maps with circle system}\label{6.4}
To define the left hand side of \eqref{diagmaineq} we use 
the moduli spaces and multisections used in Lagrangian Floer theory, while we use
other moduli spaces and multisections to define the right hand side of \eqref{diagmaineq}.
Recall that in 
Theorem \ref{thm:Ainfty} and Theorem \ref{Theorem34.20}
we took particular multisections to get the 
$A_{\infty}$ structure. The point of the whole proof we will give here 
is to show that two 
multisections, one used for Lagrangian Floer theory and
the other one more directly related to quantum cup product,  
can be homotoped in the {\it moduli space of stable maps 
with circle system}, which will be introduced in Definition \ref{StableMapsWithCircles}, so that those two multisections 
finally give the same answer.
The purpose of this subsection is to define the moduli space of stable maps 
with circle system and to describe the topology and the Kuranishi structure in detail. In the next subsection we will compare two moduli spaces and multisections to complete the proof by interpolating the moduli space of stable maps 
with circle system.
\par
We consider a class $\rho \in \pi_2(N)/\sim$.
Put
$$
\mathcal M_{3}^{\text{\rm sph}}(J_N;\rho) = \bigcup_{\alpha \in \rho}
\mathcal M_{3}^{\text{\rm sph}}(J_N;\alpha),
$$
where $\mathcal M_{3}^{\text{\rm sph}}(J_N;\alpha)$ is the moduli space
of stable maps of genus $0$ with $3$ marked points and of homotopy class
$\alpha$.
Let $((\Sigma,(z_0,z_1,z_2)),v)$ be a representative of its element.
We decompose $\Sigma = \bigcup \Sigma_a$ into irreducible components.
\begin{defn}
\begin{enumerate}
\item
Let $\Sigma_0 \subset \Sigma$ be the minimal connected union of
irreducible components containing three marked points $z_0,z_1,z_2$.
\item
An irreducible component $\Sigma_a$ is said to be
{\it Type I} if it is contained in $\Sigma_0$.
Otherwise it is said to be {\it Type II}.
\item
Let $\Sigma_a$ be an irreducible component of Type I.
Let $k_a$ be the number of its singular points in
$\Sigma_a$ which do not intersect irreducible components of $\Sigma \setminus \Sigma_0$.
Let $k'_a$ be the number of marked points on $\Sigma_a$.
It is easy to see that $k_a + k'_a$ is either $2$ or $3$.
(See Lemma \ref{TypeI} below or the proof.)
We say that $\Sigma_a$ is {\it Type I-1} if $k_a + k'_a = 3$ and
{\it Type I-2} if $k_a + k'_a = 2$.
We call these $k_a+k'_a$ points
{\it the interior special points}.
\end{enumerate}
\end{defn}

\begin{lem}\label{TypeI} Let $k_a$, $k_a'$ be the numbers defined above.
Then $k_a + k'_a$ is either $2$ or $3$. Moreover, there exists exactly one irreducible component of
Type I-1.
\end{lem}
\begin{proof} Consider the dual graph $T$ of the prestable curve $\Sigma_0$.
Note that $T$ is a tree with 3 exterior edges with a finite number of interior vertices.
Therefore it follows that there is a unique 3-valent vertex and all others are 2-valent.
Since the number $k_a + k'_a$ is precisely the valence of the vertex associated to
the component $\Sigma_a$, the first statement follows.
Furthermore since the dichotomy of $k_a + k'_a$ being $3$ and $2$ corresponds to
the component $\Sigma_a$ being of Type I-1 and of Type I-2 respectively, 
the second statement follows.
\end{proof}
\par
Our next task is to relate a bordered stable map to $(N \times N, \Delta_N)$ of genus zero
with three boundary marked points to a stable map to $N$ of genus 0 with three marked points.
For this purpose, we introduce the notion of stable maps of genus 0 with circle systems,
see Definition \ref{StableMapsWithCircles}.
\par
Firstly, we fix a terminology ``circle'' on the Riemann sphere.
We call a subset $C$ in $\C P^1$ a {\it circle}, if it is the image of  $\R \cup \{ \infty \}$
by a projective linear transformation $\Phi$ as in complex analysis, i.e., $\Phi(\R \cup \{ \infty \}) =C$.
From now on, we only consider  $C$ with an orientation.  The projective linear transformation $\Phi$ can be chosen
so that the orientation coincides with the one as the boundary of the upper half space $\H$.
\par
For a pseudoholomorphic disc $w$ in$(N \times N, J_{N \times N})$ with boundary on the diagonal $\Delta_N$,
the map $\overset{\circ}{\mathfrak I}$ in Proposition \ref{regoripres} gives a pseudoholomorphic sphere $v=\overset{\circ}{\mathfrak I}(w)$ in $(N,J_N)$.
For three boundary marked points $z_0, z_1, z_2$ of the domain of $w$, we have corresponding marked points,
which we also denote by $z_0, z_1, z_2$ by an abuse of notation, on the Riemann sphere, which is the domain of $v$.
We also note that the boundary of the disc corresponds to the circle passing through $z_0, z_1, z_2$ on the Riemann sphere.
Thus, when describing the pseudoholomorphic sphere corresponding to a pseudoholomorphic disc,
we regard it as a stable map of genus $0$ with the image of
the boundary of the disc, which is a circle on the Riemann sphere.
\begin{rem}\label{bubbletree}
\begin{enumerate}
\item For each genus $0$ bordered stable map to $(N \times N, \Delta_N)$, we construct a genus $0$ stable map to $N$ as follows.  
The construction goes componentwise.  For a disc component, we apply $\overset{\circ}{\mathfrak I}$ as we explained above.
For a bubble tree $w^{\text{\rm bt}}$ of pseudoholomorphic spheres attached to a disc component at $z^{\text{\rm int}}$,
we attach $({\rm pr}_1)_* w^{\text{\rm bt}}$ (resp. $({\rm pr}_2)_* w^{\text{\rm bt}}$) to the lower hemisphere at $\overline{z^{\text{\rm int}}}$
(resp. the upper hemisphere at $z^{\text{\rm int}}$).  Here $({\rm pr}_i)_* w^{\text{\rm bt}}$ is ${\rm pr}_i \circ w^{\text{\rm bt}}$ with each unstable
component shrunk to a point.  If ${\rm pr}_i \circ w^{\text{\rm bt}}$ is unstable, we do not attach it to the Riemann sphere.
\item  For $\beta \in \pi_2(N \times N, \Delta_N)$, pick a representative $w:(D^2, \partial D^2) \to (N\times N, \Delta_N)$.
Although $w$ is not necessarily pseudoholomorphic, we have $v: S^2 \to N$ in the same way as $\overset{\circ}{\mathfrak I}$.
We call the class $[v] \in \pi_2(N)$ the {\it double} of the class $[w] \in \pi_2(N \times N, \Delta_N)$.
\end{enumerate}
\end{rem}
In preparation for the definition of stable maps of genus 0 with circle systems, we define {\it admissible systems of circles}
in Definitions \ref{def:circlesystemI}, \ref{TypeII} and \ref{circlesonSigma}.
\par
Definition \ref{def:circlesystemI} includes the case of moduli space representing
various terms of \eqref{diagmaineq},
that is, a union of doubles of several discs.
We glue them at a boundary marked point of one
component and an interior or a boundary marked
point with the other component.
\par
Let $\Sigma_a$ be an irreducible component of $\Sigma$ where
$((\Sigma,(z_0,z_1,z_2)),v)$ is an element of $\mathcal M_{3}^{\text{\rm sph}}(J_N;\alpha)$.
A {\it domain which bounds $C_a$} is a disc in $\C P^1$ whose boundary (together with orientation)
is $C_a$.
We decompose $\Sigma_a$ into the union
of two discs $D_a^\pm$ so that $\C P^1 = D_a^+ \cup D_a^-$
where $\del D_a^+ = C_a$ as an oriented manifold. Then
$\del D_a^- = - C_a$ as an oriented manifold.
\par
Now consider a component $\Sigma_a$ of Type I-2.
We say an interior special point of $\Sigma_a$ is {\it inward} if it is contained
in the connected component of the closure of $\Sigma_0 \setminus \Sigma_a$
that contains the unique Type I-1 component.
Otherwise it is called {\it outward}.
(An inward interior special point must necessarily be a singular point.)
Note that each Type I-2 component contains a unique inward interior marked point and a unique outward interior marked point.
\begin{defn}\label{def:circlesystemI}
An {\it admissible system of circles of Type I} for $((\Sigma,(z_0,z_1,z_2)),v)$
is an assignment of $C_a$, which is a circle or an empty set (which may occur in the case (2) below, see Example \ref{example} (2)), to each irreducible component $\Sigma_a$ of Type I,
such that the following holds:
\begin{enumerate}
\item
If $\Sigma_a$ is Type I-1, $C_a$ contains all the three interior special points.
\item Let $\Sigma_a$ be Type I-2.
If $C_a$ is not empty, we require the following two conditions.
\begin{itemize}
\item $C_a$ contains the outward interior special point.
\item The inward interior special point lies on the disc $D_a^+$, namely either on $C_a$ or $\text{\rm Int }D_a^+$.
If the inward interior special point $p$ lies on $C_a$, the circle $C_{a'}$ on the adjacent component $\Sigma_{a'}$ contains
$p$.  Here $\Sigma_{a'}$ meets $\Sigma_a$ at the node $p$.
\end{itemize}
\item
Denote by $\Sigma_{a_0}$ the unique irreducible component of Type I-1 and
let $C$ be the maximal connected union of $C_a$'s containing $C_{a_0}$.
If $C$ contains all $z_i$, we require the orientation of $C$
to respect the cyclic order of $(z_0,z_1,z_2)$.
If some $z_i$ is not on $C$, we instead consider the
following point $z_i'$ on $C$ described below and require that the orientation of $C$
respects the cyclic order of $(z_0',z_1',z_2')$:
There exists a unique irreducible component $\Sigma_{a}$ such that
$C$ is contained in a connected component of the closure of $\Sigma_0 \setminus \Sigma_{a}$,
$z_i$ is contained in the other connected component or $\Sigma_a$,
and that $C$ intersects $\Sigma_a$ (at the inward interior special point of
$\Sigma_a$). Then the point $z_i'$ is this inward interior
special point of $\Sigma_a$.
\end{enumerate}
\end{defn}
\begin{exm}\label{example}
(1) Let us consider the admissible system of circles as in Figure 1
below. The left sphere is of Type I-2 and the right sphere is of Type I-1.
The circle in the right sphere is $C$ in Definition
\ref{def:circlesystemI} (3).
\par
\centerline{
\epsfbox{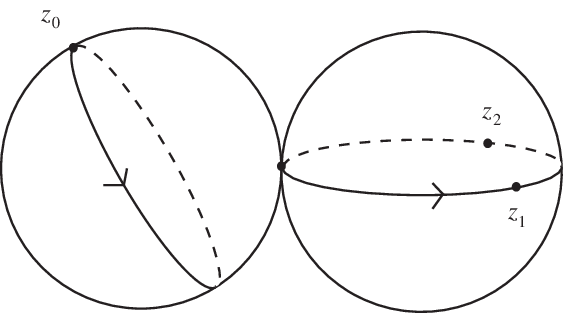}}
\par
\centerline{\bf Figure 1}
\par\smallskip
This is the double of the following configuration:
\par
\centerline{
\epsfbox{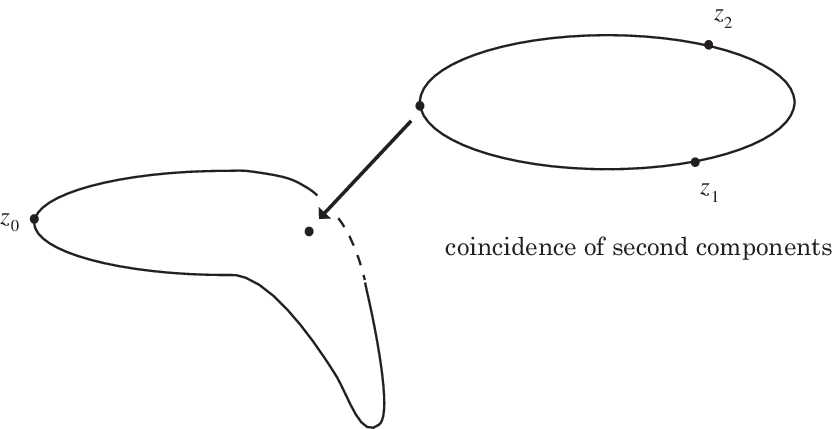}}
\par
\centerline{\bf Figure 2}
\par\smallskip
The moduli space of such configurations is
identified with the moduli space that is used to define
\begin{equation}\label{Fig12eq}
\langle \mathfrak m_{2,\beta_2}(P_1,P_2),
\overline{\mathfrak p}_{1,\beta_1}(P_0)\rangle.
\end{equation}
(2) Type I-2 components may not have circles. For example, see Figure 3.
\par\smallskip
\centerline{
\epsfbox{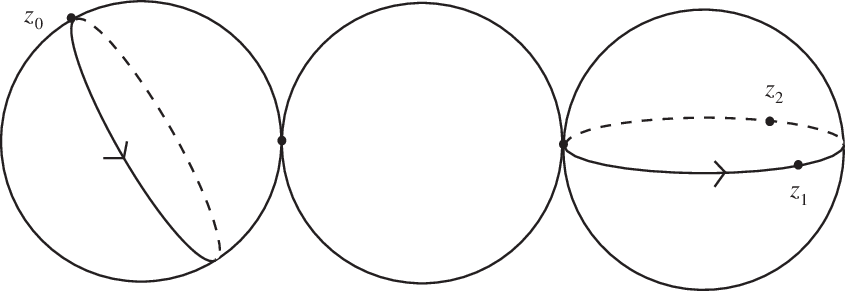}}
\par
\centerline{\bf Figure 3}
\par\smallskip
\end{exm}
We next discuss the admissible system of circles on the irreducible
components of Type II.
A {\it connected component of Type II} of $\Sigma$
is by definition the closure of a connected component of
$\Sigma \setminus \Sigma_0$.
Each connected component of Type II intersects
$\Sigma_0$ at one point.
We call this point the {\it root} of our connected component of Type II.
\par
We denote by $\Sigma_\rho$ a connected component of Type II
and decompose it into the irreducible components:
$$
\Sigma_{\rho} = \bigcup_{a \in I_{\rho}}\Sigma_a.
$$
Then we consider a Type II irreducible component $\Sigma_a$
contained in a $\Sigma_{\rho}$.
If $\Sigma_a$ does not contain the root of $\Sigma_{\rho}$,
we consider the connected component
of the closure of $\Sigma_{\rho} \setminus \Sigma_a$ that
contains the root of $\Sigma_{\rho}$.
Then, there is a unique singular point of $\Sigma_a$
contained therein. We call this singular point the {\it root} of $\Sigma_a$.
Note that if $\Sigma_a$ contains the root of $\Sigma_{\rho}$,
it is, by definition, the root of $\Sigma_a$.
\begin{defn}\label{TypeII}
Let an admissible system of circles of Type I on
$\Sigma$ be given. We define an {\it admissible system of circles of Type II} on $\Sigma_{\rho}$ to
be a union
$$
C_{\rho} = \bigcup_{a \in I_{\rho}} C_a
$$
in which $C_a$ is either a circle or an empty set and which
we require to satisfy the following:
\begin{enumerate}
\item
If the root of $\Sigma_{\rho}$ is not contained in our system
of circles of  Type I, then
all of $C_a$ are empty set.
\item
If $C_{\rho}$ is nonempty, then it is connected and
contains the root of $\Sigma_{\rho}$.
\item Let $\Sigma_a$ be a Type II irreducible component
contained in  $\Sigma_{\rho}$ and
let $\Sigma_b$ be the irreducible component
of $\Sigma$ that contains the root of $\Sigma_a$ and such that $a \ne b$.
If the root of $\Sigma_a$ is contained in $C_b$, we require
$C_a$ to be nonempty.
\end{enumerate}
\end{defn}
\begin{defn}\label{circlesonSigma}
An {\it admissible system of circles} on $\Sigma$ is, by definition,
an admissible system of circles of Type I together with that of Type II
on each connected component of Type II.
\end{defn}
\par

\begin{defn}\label{smoothable}
Let $\Sigma = \cup_a \Sigma_a$ be the decomposition into irreducible components and $\{C_a\}$ the admissible system of circles.   
($C_a$ is either a circle or an empty set.)
Let $p$ be a node joining components $\Sigma_a$ and $\Sigma_b$.
We call $p$ a {\it non-smoothable} node if  $p$ lies exactly one of  $C_a$ and $C_b$.  Otherwise, we call $p$ a {\it smoothable} node.
That is, a node $p$ is smoothable, if and only if one of the following conditions holds:  (1) $p \in C_a$ and $p \in C_b$, (2) $p \notin C_a$ and $p \notin C_b$.
\end{defn}
\begin{rem}\label{gluesmoothablenode}
For a smoothable node $p$ joining two components $\Sigma_a$ and $\Sigma_b$, we can glue them in the following way.
We call such a process the {\it smoothing} at the smoothable node $p$.
If neither $C_a \subset \Sigma_a$ nor $C_b \subset \Sigma_b$ contains the node $p$, we can perform the gluing of stable maps at the node $p$.
(In such a case, the admissibility of the circle system prohibits that both $C_a$ and $C_b$ are non-empty.)
If both $C_a$ and $C_b$ contain the node $p$, we choose a complex coordinate $z_a$, resp. $z_b$, of $\Sigma_a$, resp. $\Sigma_b$, around $p$
such that $C_a$, resp. $C_b$ is described as the real axis with the standard orientation.
Here we give an orientation on the real axis by the positive direction.
Gluing $\Sigma_a$ and $\Sigma_b$ by $z_a \cdot z_b = -t$, $t \in [0, \infty)$, we obtain the gluing of stable maps such that $C_a$ and $C_b$ are
glued to an oriented circle.
\end{rem}
\begin{defn}\label{StableMapsWithCircles}
Let $\Sigma=\bigcup_a \Sigma_a$ be a prestable curve, i.e., a singular Riemann surface
of genus $0$ at worst with nodal singularities,
$z_0,z_1,z_2$ marked points on the smooth part of $\Sigma$, $u:\Sigma \to N$ a holomorphic map
and let $C_a \subset \Sigma_a$ be either an oriented circle or an empty set.
We call ${\mathbf x}=(\Sigma, z_0,z_1,z_2, \{C_a\}, u:\Sigma \to N)$ a {\it stable map of genus $0$ with circle system}, if the following conditions are
satisfied:
\begin{enumerate}
\item $\{C_a\}$ is an admissible system of circles in the sense of Definition \ref{circlesonSigma}.\\
\item Let $P$ be the set of non-smoothable nodes on $(\Sigma, z_0,z_1,z_2, \{C_a\})$.
For the closure $\Sigma'$ of each connected component  of $\Sigma \setminus P$, one of  the following conditions holds.
\begin{enumerate}
\item With circles forgotten, $u\vert_{\Sigma'} :\Sigma' \to N$ is still a stable map.  Here we put marked points $(\{z_0,z_1,z_2\} \cup P) \cap \Sigma'$ on $\Sigma'$. \\
\item The map $u$ is non-constant on some irreducible component of $\Sigma'$.
\end{enumerate}
\item The automorphism group ${\rm Aut} ({\mathbf x})$ is finite.  Here we set
${\rm Aut} ({\mathbf x})$ the group of automorphisms $\phi$ of the singular Riemann surface $(\Sigma, z_0,z_1,z_2)$ such that $u \circ \phi = u$ and
$\phi (\cup C_a) = \cup C_a$.
\end{enumerate}
\end{defn}
Since we only consider stable maps of genus $0$, we omit ``of genus $0$'' from now on.
\begin{rem}\label{remcircle}
\begin{enumerate}
\item If $C_a \neq \emptyset$, $C_a$ must contain a node or a marked point. \\
\item Let $\Sigma_a$ be an irreducible component, which becomes unstable after the circle system forgotten, i.e.,
the map $u$ is constant on $\Sigma_a$ and the number of special points contained in $\Sigma_a$ is less than $3$.
Here a special point means a marked point or a node.
In such a case,  we find that the number of special points on $\Sigma_a$ is exactly $2$ and exactly one of them is
on $C_a$.
The Type I-1 component contains three special points, hence it cannot be such a component.
There are three possibilities, case (i), case (ii$+$), case (ii$-$) described in Definition \ref{unstwocirc} below.
\end{enumerate}
\end{rem}
By an abuse of terminology, we call the inward special point of a Type I-2 component $\Sigma_a$ the root of $\Sigma_a$.
(The definition of the root of a Type II component is given just before Definition \ref{TypeII}.)
\begin{defn}\label{unstwocirc}
Let $\Sigma_a$ be an irreducible component, which becomes unstable after the circle system forgotten
as we discussed in Remark \ref{remcircle} (2).
There are the following three cases (see Figure 4):
\begin{enumerate}
\item [(\text{\rm i})]
the root is in $\text{\rm Int }D_a^+$ and another special point is on $C_a$,
\item[(\text{\rm ii}$+$)]
the root is on $C_a$ and the other node is contained in $\text{\rm Int }D_a^+$,
\item[(\text{\rm ii}$-$)]
the root is on $C_a$ and  the other node is contained in $\text{\rm Int }D_a^-$.
\end{enumerate}
\end{defn}
\par
\centerline{
\epsfbox{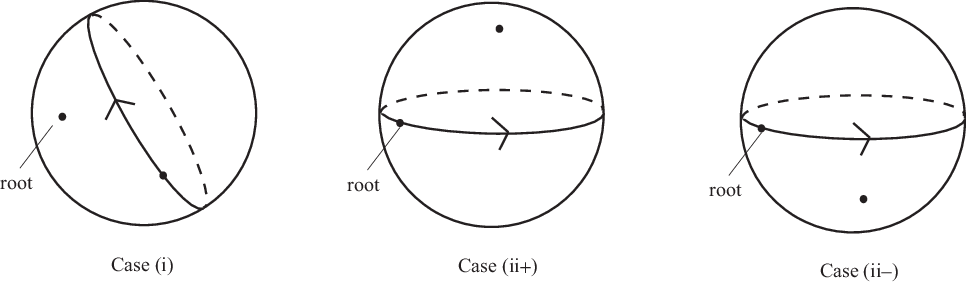}}
\par
\centerline{\bf Figure 4}
\par\smallskip
\begin{rem}\label{remcircle2}
Suppose that $\Sigma_a$ is a component of case (i).  Since the root is not on $C_a$, $\Sigma_a$ must be Type I.
Suppose that  $\Sigma_a$ is a component of case (ii).  Since the node, which is different from the root, is not on $C_a$, $\Sigma_a$ must be
Type II.
\end{rem}
\begin{lem}\label{adjacentleq2}
Let ${\mathbf x}=(\Sigma, z_0,z_1,z_2, \{C_a\}, u:\Sigma \to N)$ be a stable map of genus $0$ with circle system.
If two adjacent components become unstable when forgetting the circle system 
in the sense of Remark \ref{remcircle} (2), both of them are of case (i).
Moreover, there cannot appear more than two consecutive components, which become unstable after the circle system forgotten.
\end{lem}
\begin{proof}
Let $\Sigma_{a_1}, \Sigma_{a_2}$ be adjacent components which become unstable after the circle system forgotten.
Without loss of generality, we assume that the root of $\Sigma_{a_2}$ is attached to $\Sigma_{a_1}$.
Then there are three cases:
\begin{enumerate}
\item[{(A)}] $\Sigma_{a_1}, \Sigma_{a_2}$ are Type I.
\item[{(B)}] $\Sigma_{a_1}$ is Type I and $\Sigma_{a_2}$ is Type II.
\item[{(C)}]  $\Sigma_{a_1}, \Sigma_{a_2}$ are Type II.
\end{enumerate}
By Remark \ref{remcircle2}, we find that, in each case,
\begin{enumerate}
\item[{(A)}]  $\Sigma_{a_1}, \Sigma_{a_2}$ belong to case (i).
\item[{(B)}]  $\Sigma_{a_1}$ belongs to case (i) and $\Sigma_{a_2}$ belongs to case (ii).
\item[{(C)}]  $\Sigma_{a_1}, \Sigma_{a_2}$ belong to case (ii).
\end{enumerate}
Consider Case (B).  Since $\Sigma_{a_1}$ is Type I, $C_{a_1}$ must contain either one of $z_0, z_1, z_2$ or a root of a component of Type I.
But the node at $\Sigma_{a_1} \cap \Sigma_{a_2}$ is the only special point on $C_{a_1}$ and $\Sigma_{a_2}$ is Type II.  Hence Case (B) cannot occur.
Next consider Case (C).  In this case, $\Sigma_{a_2}$ is Type II, but its root is not on $C_{a_1}$.  Hence Case (C) cannot occur.  
Therefore the only remaning case is (A).  Namely, both $\Sigma_{a_1}$ and $\Sigma_{a_2}$ are Type I, hence case (i) in Definition \ref{unstwocirc}.
\par
Suppose that $\Sigma_{a_1}, \Sigma_{a_2}, \Sigma_{a_3}$ are consecutive components, which become unstable when forgetting the circle system.
As we just showed, all of them are case (i).  Note that the nodes $\Sigma_{a_1} \cap \Sigma_{a_2}$, $\Sigma_{a_2} \cap \Sigma_{a_3}$ are non-smoothable
nodes.  Thus the middle component $\Sigma_{a_2}$ does not satisfy Condition (2) in Definition \ref{StableMapsWithCircles}.  Hence the proof.
\end{proof}
\begin{defn}\label{notationmoduli}
We denote by $\mathcal M^{\text{sph}}_3(J_N;\alpha)$ (resp. $\mathcal M^{\text{sph}}_3(J_N;\alpha;\mathcal C)$)
the moduli space consisting of stable maps (resp. stable maps with circle system) with three marked points
$\vec z=(z_0,z_1,z_2)$ representing class $\alpha$.
We put 
$$\aligned
\mathcal M^{\text{\rm sph}}_3(J_N;\alpha;P_1,P_2,P_0)
&= 
 P_0 \times_{ev_0} \left(\mathcal M^{\text{\rm sph}}_3(J_N;\alpha) {}_{(ev_1,ev_2)}
\times (P_1 \times P_2) \right)\\
\mathcal M^{\text{\rm sph}}_3(J_N;\alpha;\mathcal C;P_1,P_2,P_0)
&= 
 P_0 \times_{ev_0} \left(\mathcal M^{\text{\rm sph}}_3(J_N;\alpha;\mathcal C) {}_{(ev_1,ev_2)}
\times (P_1 \times P_2)\right).
\endaligned$$
As a space with oriented Kuranishi structure, we define 
$$
\mathcal M^{\text{\rm sph}}_3(J_N;\alpha;P_1,P_2,P_0)
= (-1)^{\deg P_1 \deg P_2} P_0 \times_{ev_0} 
\left(\mathcal M^{\text{\rm sph}}_3(J_N;\alpha) {}_{(ev_1,ev_2)}
\times (P_1 \times P_2)\right).
$$
To define the orientation on $\mathcal M^{\text{\rm sph}}_3(J_N;\alpha;\mathcal C;P_1,P_2,P_0)$, we consider 
$$
\aligned
\mathcal M^{\text{\rm sph, reg}}_3(J_N;\alpha;P_1,P_2,P_0)
& = (-1)^{\deg P_1 \deg P_2} P_0 \times_{ev_0} 
\left(\mathcal M^{\text{\rm sph, reg}}_3(J_N;\alpha) {}_{(ev_1,ev_2)}
\times (P_1 \times P_2)\right) \\
& \subseteq \mathcal M^{\text{\rm sph}}_3(J_N;\alpha;P_1,P_2,P_0).
\endaligned
$$
(See the sentence after \eqref{picorresp} for the notation $\mathcal M^{\text{\rm sph, reg}}_3(J_N;\alpha)$.)
For any element in 
$\mathcal M^{\text{\rm sph, reg}}_3(J_N;\alpha;P_1,P_2,P_0)$, 
there is a unique circle passing through 
$z_0, z_1, z_2$ in this order.  
Hence $\mathcal M^{\text{\rm sph, reg}}_3(J_N;\alpha;P_1,P_2,P_0)$ can be identified with  
a subset of $\mathcal M^{\text{\rm sph}}_3(J_N;\alpha;\mathcal C;P_1,P_2,P_0)$.  
We denote this subset by 
$\mathcal M^{\text{\rm sph, reg}}_3(J_N;\alpha;\mathcal C;P_1,P_2,P_0)$.
We define an orientation on $\mathcal M^{\text{\rm sph, reg}}_3(J_N;\alpha;\mathcal C;P_1,P_2,P_0)$
in such a way that this identification respects the orientations. 
We will explain in Remark \ref{inversionTypeII} (1) how to equip the whole space  
$\mathcal M^{\text{\rm sph}}_3(J_N;\alpha;\mathcal C;P_1,P_2,P_0)$ 
with an orientation.  
\par
For $\rho \in \pi_2(N)/\sim$ we define 
$\mathcal M^{\text{\rm sph}}_3(J_N;\rho;P_1,P_2,P_0)$, 
$\mathcal M^{\text{\rm sph}}_3(J_N;\rho;\mathcal C;P_1,P_2,P_0)$
and $ \mathcal M^{\text{\rm sph}}_3(J_N;\rho;\mathcal C)$ in an
obvious way.
\end{defn}

Then we have  
$$
\aligned
& \langle PD [P_1] * PD [P_2], PD [P_0] \rangle \\
= & (-1)^{\epsilon_1 } \sum_{\rho} (ev_0^* PD [P_0] \cup ev_1^* PD [P_1] \cup ev_2^* PD [P_2]) [\mathcal M^{\text{\rm sph}}_3(J_N;\rho)] \\
= & (-1)^{\epsilon_1 + \epsilon_2} \sum_{\rho} (ev_0^* PD [P_0] \cup (ev_1 \times ev_2)^* PD [P_1 \times P_2]) [\mathcal M^{\text{\rm sph}}_3(J_N;\rho)] \\
= & (-1)^{\epsilon_2} \sum_{\rho} (ev_0^* PD [P_0]) [\mathcal M^{\text{\rm sph}}_3(J_N;\rho)_{(ev_1,ev_2)} \times (P_1 \times P_2)] \\
= &  \sum_{\rho} (ev_0^* PD [P_0]) [\mathcal M^{\text{\rm sph}}_3(J_N;\rho;P_1,P_2)] \\
= & \sum_{\rho} (-1)^{\epsilon_1} P_0 \cdot  [\mathcal M^{\text{\rm sph}}_3(J_N;\rho;P_1,P_2)] \\
= & \sum_{\rho} \# {\mathcal M}_3^{\text{\rm sph}}(J_N;\rho;P_1, P_2, P_0), 
\endaligned
$$
where $\epsilon_1= \deg P_0 (\deg P_1 + \deg P_2)$ and $\epsilon_2 = \deg P_1 \cdot \deg P_2$.  
For the second equality we use that $PD [P_1 ] \times PD [P_2] = (-1)^{\deg P_1 \cdot \deg P_2} PD [P_1 \times P_2]$ 
in $H^*(N \times N)$.  
The third and fifth equalities follow from our convention of the Poincar\'e dual in \eqref{PDconvention}.  
The fourth equality is due to the definition of $\mathcal M^{\text{\rm sph}}_3(J_N;\rho;P_1,P_2)$.  
The last equality follows from \eqref{signpairing}.  
In sum, we have 
\begin{equation}
\langle PD[P_1] * PD [P_2], PD [P_0] \rangle = 
\sum_{\rho} \# {\mathcal M}_3^{\text{\rm sph}}(J_N;\rho;P_1, P_2, P_0). 
\end{equation}
\par
Now we will put a topology on the moduli space $\mathcal M^{\text{\rm sph}}_3(J_N;\alpha;\mathcal C)$ derived from the topology on
the moduli spaces $\mathcal M^{\text{\rm sph}}_{3+L}(J_N;\alpha)$ of stable maps of genus $0$ with $3+L$ marked points in Definition \ref{top} and
Proposition \ref{nbhdtop}.
Here $L$ is a suitable positive integer explained later.
To relate $\mathcal M^{\text{\rm sph}}_3(J_N; \alpha; C)$ to $\mathcal M^{\text{\rm sph}}_{3+L}(J_N; \alpha)$,
we will put $L$ marked points on the source curve of elements of $\mathcal M^{\text{\rm sph}}_3(J_N;\alpha; C)$.
(This process of adding additional marked points is somewhat reminiscent of a similar process
in the definition of stable map topology given in \cite{FO}.)
\par
We start with an elementary fact.  That is,
for distinct three points $p, q, r$ on ${\C}P^1$, the circle passing through $p,q,r$ is characterized as
the set of points $w \in {\C}P^1$ such that the cross ratio of $p,q,r,w$ is either real number or the infinity.

The following lemma is clear.

\begin{lem}\label{crossratio}
Let $u^{(i)}, u: \Sigma^{(i)} \to N$ be pseudoholomorphic maps from prestable curves $\Sigma^{(i)}, \Sigma$ representing the class $\alpha$.
Let $w^{(i)}_1, w^{(i)}_2, w^{(i)}_3, w^{(i)}_4$ be distinct four points on $\Sigma^{(i)}$ and
$w_1, w_2, w_3, w_4$ distinct four points on an irreducible component $\Sigma_a$ of $\Sigma$
 such that a sequence
$(\Sigma^{(i)}, \vec{z}\,^{(i)} \cup \{w^{(i)}_1, \dots, w^{(i)}_4 \}, u^{(i)})$ converges to
$(\Sigma, \vec{z} \cup \{w_1, \dots, w_4\}, u)$ in $\mathcal M^{\text{\rm sph}}_{3+4}(J_N;\alpha)$.
Then  $w^{(i)}_1, \dots, w^{(i)}_4$ belong to an irreducible component of $\Sigma^{(i)}$ for sufficiently large $i$ and
the cross ratio $[w^{(i)}_1: \dots : w^{(i)}_4]$ converges to $[w_1: \dots :w_4]$ as $i$ tends to $+\infty$.
\end{lem}
We consider a stratification of $\mathcal M^{\text{\rm sph}}_3(J_N;\alpha;\mathcal C)$ by their combinatorial types as follows.
Recall that by circles we always mean oriented circles in $\C P^1$ as we stated in the second paragraph after the proof of Lemma \ref{TypeI}.
\begin{defn}\label{combtype}
The combinatorial type $\mathfrak{c}({\mathbf x})$ of ${\mathbf x} \in \mathcal M^{\text{\rm sph}}_3(J_N;\alpha;\mathcal C)$ is defined by the following data:
\begin{enumerate}
\item The dual graph, whose vertices (resp. edges) correspond to irreducible components of the domain $\Sigma$ (resp., nodes of $\Sigma$).
\item The data that tells the irreducible components which contain $z_0$, $z_1$, $z_2$, respectively.
\item The homology class represented by $u$ restricted to each irreducible component.
\item For each irreducible component $\Sigma_a$, whether $C_a$ is empty or not.  If $C_a \neq \emptyset$, 
we include the data of the list of all nodes contained in $\text{\rm Int }D_a^+$ bounded by the oriented circle $C_a$ 
and the list of all nodes on $C_a$.
For each node $p$ that is not the root of $\Sigma_a$, 
we include the data that determines
whether $p$ lies in the domain bounded by $C_a$ or not.
This data determines 
which side of $C_a$ $p$ lies on.  (Recall that $C_a$ is an oriented circle, 
and that by `domain $D_a^+$ bounded by $C_a$' we involve the orientation together as we mentioned in the third paragraph after Remark \ref{bubbletree}. So the orientation of $C_a$ is a part of data of the combinatorial type.)
\end{enumerate}
For a node $p$ of $\Sigma$, we denote by $\Sigma_{\text{\rm in}, p}$ (resp. $\Sigma_{\text{\rm out}, p}$) the component of $\Sigma$, which
contains $p$ as an outward node (resp. the root node).
The combinatorial type of $p$ is defined by the following data:
\begin{enumerate}
\item $C_{\text{in},p}$ contains $p$ or not.
\item $C_{\text{out},p}$ contains $p$ or not.
\end{enumerate}
\end{defn}

\begin{rem}
The combinatorial type of a node $p$ of $\Sigma$ only depends on the components which contain $p$.
The combinatorial data ${\mathfrak c}({\mathbf x})$ determine the combinatorial type of each node $p$, in particular, whether the node $p$ is smoothable or not.
\end{rem}

\begin{lem}\label{finite-combi-type} There are only finitely many combinatorial types in
$\mathcal M^{\text{sph}}_3(J_N;\alpha;\mathcal C)$.
\end{lem}
\begin{proof} This follows from Lemma \ref{adjacentleq2} and the finiteness of combinatorial types of
$\mathcal M^{\text{sph}}_3(J_N;\alpha)$.
\end{proof}

\begin{defn}
Let $\mathfrak{c}_1, \mathfrak{c}_2$ be the combinatorial types of some elements in $\mathcal M^{\text{\rm sph}}_3(J_N;\alpha;\mathcal C)$.
We define the partial order $\mathfrak{c}_1 \preceq  \mathfrak{c}_2$ if and only if $\mathfrak{c}_2$ is obtained from $\mathfrak{c}_1$
by smoothing some of smoothable nodes.
See Remark \ref{gluesmoothablenode} for smoothing of a smoothable node.

We set
$${\mathcal M}_3^{{\text{\rm sph}}, \succeq{\mathfrak{c}}({\mathbf x})} (J_N;\alpha;{\mathcal C}) =
\{ {\mathbf x}' \in {\mathcal M}_3^{\text{\rm sph}}(J_N;\alpha;{\mathcal C}) \ \vert \  {\mathfrak c}({\mathbf x}') \succeq {\mathfrak c}({\mathbf x}) \}$$
and
$${\mathcal M}_3^{{\text{\rm sph}}, ={\mathfrak{c}}({\mathbf x})} (J_N;\alpha;{\mathcal C}) =
\{ {\mathbf x}' \in {\mathcal M}_3^{\text{\rm sph}}(J_N;\alpha;{\mathcal C}) \ \vert \  {\mathfrak c}({\mathbf x}') = {\mathfrak c}({\mathbf x}) \}. $$
\end{defn}
\begin{rem}\label{dualgraph}
\begin{enumerate}
\item 
A combinatorial type $\mathfrak c$ determines the intersection pattern of the circle systems.
For $\mathbf x \in {\mathcal M}_3^{\text{\rm sph}} (J_N;\alpha;{\mathcal C})$, consider
the dual graph $\Gamma ({\mathbf x})$ of $\cup C_a$.
Note that $\Gamma ({\mathbf x})$ is a subgraph of the dual graph of the domain of ${\mathbf x}
\in {\mathcal M}_3^{\text{\rm sph}}(J_N; \alpha; {\mathcal C})$.
Since the genus of the domain of ${\mathbf x}$ is $0$, each connected
component of $\Gamma({\mathbf x})$ is a  tree.
Namely, we assign a vertex $v_a({\mathbf x})$ to each non-empty circle $C_a$ and an edge joining the
vertices $v_{a_1}({\mathbf x})$ and $v_{a_2}({\mathbf x})$ corresponding to $C_{a_1}$ and $C_{a_2}$
respectively if they intersect at a node of $\Sigma({\mathbf x})$.
The graph $\Gamma({\mathbf x})$ is determined by the combinatorial type ${\mathfrak c}({\mathbf x})$
and we also denote it by $\Gamma({\mathfrak c})$.
The smoothing of a node, which is an intersection point of circles $C_{a_1}$ and $C_{a_2}$,
corresponds to the process of contracting the edge joining $v_{a_1}({\mathbf x})$ and
$v_{a_2}({\mathbf x})$.
Hence, if ${\mathfrak c}_1 \preceq {\mathfrak c}_2$, $\Gamma({\mathfrak c}_2)$ is obtained from
$\Gamma({\mathfrak c}_1)$ by contracting some of edges.
Therefore we have a canonical one-to-one correspondence between
connected components of $\Gamma({\mathfrak c}_1)$ and $\Gamma({\mathfrak c}_2)$.
In particular, we find that
$$\# \pi_0(\Gamma({\mathfrak c}_1)) = \# \pi_0(\Gamma({\mathfrak c}_2)),$$
where $\#\pi_0(\Gamma({\mathfrak c}))$ denotes the number of connected components of
$\Gamma({\mathfrak c})$.
\item
Each circle $C_a$ in the admissible system of circles on $\mathbf x$
is oriented.  If $C_a$ intersects any other $C_b$ in
the circle system, we cut $C_a$ at these intersection points to get a collection of oriented arcs.
Recall that each connected component $J$ of $\Gamma({\mathbf x})$ is a tree.
Hence the union of oriented circles corresponding to the vertices in $J$ is regarded as
an oriented Eulerian circuit, i.e.,  an oriented loop $\ell_J({\mathbf x})$ which is a concatenation of
the oriented arcs arising from $C_a$ ($a$ is a vertex in $J$).
The oriented loop $\ell_J({\mathbf x})$ is determined up to orientation preserving reparametrization. 
\end{enumerate}
\end{rem}
For a stable map with circle system, we can put appropriate additional marked points on the circles in such a way that the circles can be recovered from the additional
marked points.
More precisely, we consider the following conditions for the elements in 
$\mathcal{M}_{3+L}^{\text{\rm sph}}(J_N;\alpha)$:
\par
Let $\tilde{\mathbf y} \in \mathcal{M}_{3+L}^{\text{\rm sph}}(J_N;\alpha)$ and let $\vec{z}\,^+=\vec{z} \cup \{ z_3, \dots, z_{2+L} \}$ be the marked points.
Namely, $\tilde{\mathbf y}=(\Sigma, \vec{z}\,^+, u: \Sigma \to N)$ is a stable map representing the class $\alpha$.
We consider the following.
\begin{conds}\label{cond}
For each irreducible component $\Sigma_a$ of the domain $\Sigma$, 
$\Sigma_a \cap \vec{z}\,^+$ is either empty or consisting of at least three points.  In the latter case,
$\Sigma_a \cap \vec{z}\,^+$ lies on a unique circle $C_a$ on $\Sigma_a$.  
\end{conds} 
Note, however, the orientation of circles is not directly determined by Condition 6.41.
For the circles $\{ C_a \}$ in Condition \ref{cond}, we can associate its dual graph 
$\Gamma(\tilde{\mathbf y})$ in the same manner. 
Let $E(\tilde{\mathbf y})$ be the set of edges of $\Gamma(\tilde{\mathbf y})$.   
\par
For a fixed $\mathbf x \in {\mathcal M}_3^{\text{\rm sph}}(J_N;\alpha;{\mathcal C})$, we put 
additional marked points on the union of circles in the admissible system of circles to 
obtain $\tilde{\mathbf x} \in {\mathcal M}_{3+L}^{\text{\rm sph}}(J_N;\alpha)$ such that 
Condition \ref{cond} is satisfied.  
We can find a neighborhood $V(\tilde{\mathbf x})$ of $\tilde{\mathbf x}$ in 
$$\{ \tilde{\mathbf y} \in {\mathcal M}_{3+L}^{\text{\rm sph}}(J_N;\alpha) ~\vert~ 
\text{$\tilde{\mathbf y}$ satisfies Condition \ref{cond}}. \},$$
so that 
if $\tilde{\mathbf y} \in V(\tilde{\mathbf x})$, 
we can obtain 
$\Gamma(\tilde{\mathbf y})$ from $\Gamma(\tilde{\mathbf x})$ by contracting edges 
in $E'$ and removing edges in $E''$.  
Here $E', E''$ are disjoint subsets of $E(\tilde{\mathbf x})$, which may possibly be empty.  
In particular, we have 
$$\# \pi_0 (\Gamma(\tilde{\mathbf x})) \leq \# \pi_0 (\Gamma(\tilde{\mathbf y})).
$$
\begin{conds}\label{cond2}  $\tilde{\mathbf y} \in V(\tilde{\mathbf x})$ and 
$\# \pi_0 (\Gamma(\tilde{\mathbf x})) = \# \pi_0 (\Gamma(\tilde{\mathbf y})).$
\end{conds}
Under Condition \ref{cond2}, we have a canonical one-to-one correspondence between 
connected components of $\Gamma(\tilde{\mathbf x})$ and those of $\Gamma(\tilde{\mathbf y})$.  
To prove this it suffices to describe the way to determine the orientation of circles $\{ C_a\}$.
We consider Condition \ref{cond3} below for this purpose.
Let $J(\tilde{\mathbf x})$ be a connected component of $\Gamma(\tilde{\mathbf x})$ 
and $J(\tilde{\mathbf y})$ be the corresponding 
connected component of $\Gamma(\tilde{\mathbf y})$, i.e., 
$J(\tilde{\mathbf y})$ obtained from $J(\tilde{\mathbf x})$ by contracting some edges.  
\par
To make clear that $\vec{z}^+=\vec{z} \cup \{ z_3, \dots, z_{2+L}\}$ are marked points 
on $\Sigma(\tilde{\mathbf x})$, we denote it by 
$\vec{z}^+(\tilde{\mathbf x})=(z_0(\tilde{\mathbf x}), \dots, z_{2+L}(\tilde{\mathbf x})))$.  
Let $\Sigma_{J(\tilde{\mathbf x})}(\tilde{\mathbf x})$ be the union of irreducible components, 
which contain circles $C_a$ corresponding to the vertices in $J({\mathbf x})$.  
We can assign a cyclic order on 
$$I_{J(\tilde{\mathbf x})}=\{i ~\vert~ z_i (\tilde{\mathbf x}) \in \Sigma_{J(\tilde{\mathbf x})}(\tilde{\mathbf x}), ~i=0,1,2,3,\dots,2+L\}$$ 
using the oriented loop $\ell_{J(\tilde{\mathbf x})}(\tilde{\mathbf x})$
defined as in Remark \ref{dualgraph} (2).  
Namely, the oriented loop $\ell_{J(\tilde{\mathbf x})}(\tilde{\mathbf x})$ passes at $z_i$, 
$i \in I_{J(\tilde{\mathbf x})}$ compatible with the cyclic order on $I_{J(\tilde{\mathbf x})}$.   
\par
For $\tilde{\mathbf y}$ satisfying Condition \ref{cond2}, 
we find that $I_{J(\tilde{\mathbf x})} = I_{J(\tilde{\mathbf y})}$, which we denote 
by $I_J$.  
\begin{conds}\label{cond3}
Each circle $C_{a'}$ on the domain of $\tilde{\mathbf y}$ given in Condition \ref{cond} is equipped with an orientation with the following property.  
For each connected component $J$ of $\Gamma(\tilde{\mathbf x})$, 
the cyclic order on $I_J$ coming from 
the oriented loop $\ell_J(\tilde{\mathbf x})$
coincides with the one coming from the oriented loop $\ell_J (\tilde{\mathbf y})$
defined as in Remark \ref{dualgraph} (2) using the orientation on $C_{a'}$ for 
any vertex $a'$ in $J(\tilde{\mathbf y})$.
\end{conds}
Note that such an orientation on $C_{a'}$ is unique, if there exists one.   
Under Condition \ref{cond3}, each circle $C_a$ in Condition \ref{cond} is oriented 
in such a manner. 
This is the way to define the 
canonical one-to-one correspondence between 
connected components of $\Gamma(\tilde{\mathbf x})$ and those of $\Gamma(\tilde{\mathbf y})$.  
\begin{conds}\label{cond4}
The quadruple $(\Sigma, \vec{z}, \{ C_a \}, u)$ defines a stable map with circle system 
${\mathcal C}=\{C_a\}$.  
\end{conds}
\begin{defn}
We set
$$
 {\mathcal U}(\tilde{\mathbf x})=\{ \tilde{\mathbf y} \in V(\tilde{\mathbf x})  ~\vert~  \tilde{\mathbf y} \text{  satisfies Conditions  \ref{cond}, \ref{cond2}, \ref{cond3}, \ref{cond4}.}  \}.
$$
We denote the natural map by
\begin{align}
\pi^L_{\tilde{\mathbf x}} : {\mathcal U}(\tilde{\mathbf x}) \to {\mathcal M}_{3}^{\text{\rm sph}}(J_N;\alpha;{\mathcal C}) \nonumber \\
(\Sigma, \vec{z}\,^+, u) \mapsto (\Sigma, \vec{z}, \{ C_a\}, u). \nonumber
\end{align}
\end{defn}
\begin{rem}\label{choiceV(x)}Note that we can take the above set $V(\tilde{\mathbf x})$ 
for $\tilde{\mathbf x}$ in such a way that 
${\mathcal M}_3^{\text{\rm sph}, ={\mathfrak c}(\mathbf x)} (J_N;\alpha;{\mathcal C})$ is contained in the image of 
$\pi^L_{\tilde{\mathbf x}} : {\mathcal U}(\tilde{\mathbf x}) \to {\mathcal M}_{3}^{\text{\rm sph}}(J_N;\alpha;{\mathcal C})$.  
Note also that we have 
$$\pi^L_{\tilde{\mathbf x}} ({\mathcal U}(\tilde{\mathbf x})) \subset {\mathcal M}_3^{{\text{\rm sph}}, \succeq{\mathfrak{c}}({\mathbf x})} (J_N;\alpha;{\mathcal C}).$$
\end{rem}
Summarizing the construction above, we have 
\begin{lem}\label{add}
For any ${\mathbf x} \in {\mathcal M}_3^{\text{\rm sph}}(J_N;\alpha; {\mathcal C})$, there exist a positive integer $L$ and $\tilde{\mathbf x} \in {\mathcal M}_{3+L}^{\text{\rm sph}}(J_N;\alpha)$ 
such that the above naturally defined map $\pi^L_{\tilde{\mathbf x}} : {\mathcal U}(\tilde{\mathbf x}) \to {\mathcal M}_{3}^{\text{\rm sph}}(J_N;\alpha;{\mathcal C})$ satisfies 
$\pi^L_{\tilde{\mathbf x}}(\tilde{\mathbf x}) = {\mathbf x}$.
\end{lem}
We equip ${\mathcal U}(\tilde{\mathbf x})$ with the subspace topology of ${\mathcal M}_{3+L}^{\text{\rm sph}}(J_N;\alpha)$.
\begin{defn}\label{top}
For $U \subset {\mathcal M}_3^{\text{\rm sph}}(J_N;\alpha;{\mathcal C})$, $U$ is defined to be a neighborhood of ${\mathbf x}$ if and only if
there exist a positive number $L$, 
$\tilde{\mathbf x} \in {\mathcal M}_{3+L}^{\text{\rm sph}}(J_N;\alpha)$
as in Lemma \ref{add}
such that there exists a neighborhood $\widetilde{U} \subset {\mathcal U}(\tilde{\mathbf x})$ of $\tilde{\mathbf x}$ satisfying
$\pi^L_{\tilde{\mathbf x}}(\widetilde{U}) \subset U$.  Let ${\mathfrak N}(\mathbf x)$ be the collection of all neighborhoods of $\mathbf x$.
\end{defn}
We can show the following
\begin{prop}\label{nbhdtop}
The collection $\{ {\mathfrak N}(\mathbf x) ~ \vert  ~ \mathbf x \in {\mathcal M}_3^{\text{\rm sph}}(J_N;\alpha;{\mathcal C}) \}$ 
satisfies the axiom of the system of neighborhoods.
Thus it defines a topology on ${\CM}_3^{\text{\rm sph}}(J_N;\alpha;{\mathcal C})$.
\end{prop}
\begin{proof}
Let $U$ be a neighborhood of $\mathbf x$ and 
$\widetilde{U}$ a neighborhood of $\tilde{\mathbf x}$ in ${\mathcal U}(\tilde{\mathbf x})$ as in Definition \ref{top}.  
Take an open neighborhood ${\widetilde U}^{\circ}$ of $\tilde{\mathbf x}$ in $\widetilde U$.   
For any $\tilde{\mathbf y} \in {\widetilde U}^{\circ}$, there is a neighborhood $\widetilde W$ of 
$\tilde{\mathbf y}$ in ${\mathcal U}(\tilde{\mathbf x}) \subset {\mathcal M}_{3+L}^{\text{\rm sph}}(J_N;\alpha)$ such that 
$\widetilde W$ is contained in ${\mathcal U}(\tilde{\mathbf y})$.  
Thus $\widetilde W \subset {\mathcal U}(\tilde {\mathbf x}) \cap {\mathcal U}(\tilde{\mathbf y})$.  
Hence $\pi^L_{\tilde{\mathbf y}}(\widetilde W)=\pi^L_{\tilde{\mathbf x}}(\widetilde W)$ is a neighborhood of $\mathbf y$.  
It remains to show that the definition of the neighborhood in Definition \ref{top} is independent of 
the choice of $\tilde{\mathbf x}$.
 This follows from the next lemma.  
\end{proof}
\begin{lem}\label{lemaddpoints}
Let $\tilde{\mathbf x} \in {\mathcal M}_{3+L}^{\text{\rm sph}}(J_N;\alpha)$ and 
$\tilde{\mathbf x}' \in {\mathcal M}_{3+L'}^{\text{\rm sph}}(J_N;\alpha)$ as in Lemma \ref{add} such that 
$\pi^L_{\tilde{\mathbf x}}(\tilde{\mathbf x})=\pi^{L'}_{\tilde{\mathbf x}'}(\tilde{\mathbf x}')={\mathbf x}$.
Then for any neighborhood $\widetilde{U} \subset {\mathcal U}(\tilde{\mathbf x})$ of $\tilde{\mathbf x}$,
there exists a neighborhood $\widetilde{U}' \subset {\mathcal U}(\tilde{\mathbf x}')$ of $\tilde{\mathbf x}'$
such that $\pi^{L'}_{\tilde{\mathbf x}'} (\widetilde{U}') \subset \pi^L_{\tilde{\mathbf x}} (\widetilde{U})$.
\end{lem}
\begin{proof}
Let $\widetilde{W}' \subset {\mathcal U}(\tilde{\mathbf x}')$ be a sufficiently small neighborhood 
of $\tilde{\mathbf x}'$, which will be specified later.  
We will define a continuous mapping $\Phi_{\tilde{\mathbf x}'}: \widetilde{W}' \to {\mathcal U}(\tilde{\mathbf x}) \subset 
{\mathcal M}_{3+L}^{\text{\rm sph}}(J_N;\alpha)$ such that 
$\pi^L_{\tilde{\mathbf x}} \circ \Phi_{\tilde{\mathbf x}'} = \pi^{L'}_{\tilde{\mathbf x}'}$. 
In other words, for $\tilde{\mathbf y}' \in {\widetilde W}'$, 
we will find $\tilde{\mathbf y} = \Phi_{{\mathbf x}'}(\widetilde{\mathbf y}') \in {\mathcal U}(\tilde{\mathbf x})$ with
$\pi^L_{\tilde{\mathbf x}}(\tilde{\mathbf y})={\pi^{L'}_{\tilde{\mathbf x}'}(\tilde{\mathbf y}')=\mathbf y'}$.  
Namely, we find additional $L$ marked points on the circles on the domain 
$\Sigma({\mathbf y}')$ of ${\mathbf y}'$.  
We define $\tilde{\mathbf y}$ in the following steps.  
Firstly, for $\tilde{\mathbf y}' \in {\mathcal U}(\tilde{\mathbf x}')$, we define $L$ mutually distinct marked points 
$w_j(\tilde{\mathbf y}')$, $j=3, \dots,  2+L$, on the union of circles $\cup C({\mathbf y}')$ 
in the admissible circle systems of ${\mathbf y}'$.  
By the construction below, any $w_j(\tilde{\mathbf y}')$ does not coincide with 
$z_0({\mathbf y}'), z_1({\mathbf y}'), z_2({\mathbf y}')$.  
Then there exists a neighborhood of $\tilde{\mathbf x}'$ in ${\mathcal U}(\tilde{\mathbf x}')$ 
such that $w_j(\tilde{\mathbf y}')$ does not coincide with any nodes of ${\mathbf y}'$ and 
defines an element in ${\mathcal M}_{3+L}^{\text{\rm sph}}(J_N;\alpha)$.  
Finally, we find a neighborhood $\widetilde{W}'$ of $\tilde{\mathbf x}'$
 in ${\mathcal U}(\tilde{\mathbf x}')$ 
such that $\tilde{\mathbf y}' \in {\widetilde W}'$ and $\tilde{\mathbf y} \in {\mathcal U}(\tilde{\mathbf x})$.  
\par
Denote by $\vec{z}\,^+ (\tilde{\mathbf x})$, (resp. $\vec{z}\,^{+} (\tilde{\mathbf x}')$),  the set of marked points in 
$\tilde{\mathbf x}$, (resp. $\tilde{\mathbf x}')$.
For $j=3, \dots, 2 + L$, we need to find $w_j(\tilde{\mathbf y'})$ on the admissible system of circles on  
$\Sigma({\mathbf y}')$.  Here recall that ${\mathbf y}' = \pi^{L'}_{\tilde{\mathbf x}'}(\tilde{\mathbf y}')$.  
\par
Firstly, we pick three distinct points $z_{i_1(a)} (\tilde{\mathbf x}'), z_{i_2(a)}(\tilde{\mathbf x}'), z_{i_3(a)}(\tilde{\mathbf x}')$ 
on each circle $C({\mathbf x})_a$ in the admissible circle system of ${\mathbf x}$.  
Let $C({\mathbf x})_{a(j)} \subset \Sigma({\mathbf x})_{a(j)}$ be a circle on an irreducible
component $\Sigma({\mathbf x})_{a(j)}$ of the domain of $\mathbf x$
such that $z_j(\tilde{\mathbf x}') \in \vec{z}\,^{+}(\tilde{\mathbf x}')$ is on $C({\mathbf x})_{a(j)}$.
Note that the graph $\Gamma(\tilde{\mathbf y}')$ is obtained by contracting some edges of $\Gamma(\tilde{\mathbf x}')$.  
Denote by $a'(j)$ the vertex of $\Gamma(\tilde{\mathbf y}')$, which is the image of the vertex $a(j)$ of $\Gamma(\tilde{\mathbf x}')$ 
by the contracting map.  
Then the component $\Sigma({\mathbf y}')_{a'(j)}$ of the domain of ${\mathbf y}'$ contains  
$z_{i_1(a(j)))}(\tilde{\mathbf y}'), z_{i_2(a(j)))}(\tilde{\mathbf y}'), z_{i_3(a(j))}(\tilde{\mathbf y}')$.  
There is a unique circle passing through these three points, which is nothing but 
$C({\mathbf y}')_{a'(j)} \subset \Sigma({\mathbf y}')_{a'(j)}$, since 
${\mathbf y}' = \pi^{L'}_{\tilde{\mathbf x}'}(\tilde{\mathbf y}')$.  
Then we define $w_j(\tilde{\mathbf y}')$ on
$C({\mathbf y}')_{a'(j)} \subset \Sigma({\mathbf y}')_{a'(j)}$ so that the cross ratio of
$z_{i_1(a(j))}(\tilde{\mathbf x}'), z_{i_2(a(j))}(\tilde{\mathbf x}'), z_{i_3(a(j))}(\tilde{\mathbf x}'), z_j(\tilde{\mathbf x})$ 
is equal to that of $z_{i_1(a(j))}(\tilde{\mathbf y}'), z_{i_2(a(j))}(\tilde{\mathbf y}'), z_{i_3(a(j))}(\tilde{\mathbf y}'),
w_j(\tilde{\mathbf y}')$.  
Here the cross ratio is taken on $\Sigma({\mathbf x})_{a(j)}$ and $\Sigma({\mathbf y}')_{a'(j)}$, which are biholomorphic to ${\mathbb C}P^1$, respectively.  
Note that $w_j(\tilde{\mathbf y}')$ depends continuously on $\tilde{\mathbf y}'$ in the following sense.  
Take a real projectively linear isomorphism $\psi_{\tilde{\mathbf y}'}$ from $C({\mathbf y}')_{a'(j)}$ to ${\mathbb R}P^1$ 
in ${\mathbb C}P^1$ such that 
$z_{i_1(a(j))}(\tilde{\mathbf y}')$, $z_{i_2(a(j))}(\tilde{\mathbf y}'), z_{i_3(a(j))}(\tilde{\mathbf y}')$ are sent to 
$0,1, \infty$.  
Then $\psi_{\tilde{\mathbf y}'}(w_j(\tilde{\mathbf y}'))$ is continuous with respect to $\tilde{\mathbf y}'$.  
Since $w_j(\tilde{\mathbf x}')=z_j(\tilde{\mathbf x})$ by the construction and $w_j(\tilde{\mathbf y}')$ depends continuously 
on $\tilde{\mathbf y}'$, 
if $\tilde{\mathbf y'}$ is sufficiently close to $\tilde{\mathbf x}'$ in ${\mathcal U}(\tilde{\mathbf x}')$,
the 0-th, 1-st, 2-nd marked points on $\mathbf y'$ and
$w_j(\tilde{\mathbf y}')$, $j=3, \dots, 2+L$, are mutually distinct and do not coincide nodes of $\Sigma({\mathbf y}')$ on 
$\Sigma({\mathbf y}')_{a'}$.  
Hence ${\mathbf y}'$ with marked points $z_0({\mathbf y}'), z_1({\mathbf y}'), z_2({\mathbf y}'), w_3(\tilde{\mathbf y}'), 
\dots, w_{2+L}(\tilde{\mathbf y}')$ defines an element in ${\mathcal M}_{3+L}^{\text{\rm sph}}(J_N;\alpha)$, 
which we denote by $\tilde{\mathbf y}$.   
We take a neighborhood $\widetilde{W}'$ of $\tilde{\mathbf x}'$ in ${\mathcal U}(\tilde{\mathbf x}')$ in order that 
$\tilde{\mathbf y} \in V(\tilde{\mathbf x})$ for $\tilde{\mathbf y}' \in \widetilde{W}'$.  
Note that the number of marked points of $\tilde{\mathbf y}$ is either zero or at least three by the construction.  
Since $\pi^{L'}_{\tilde{\mathbf x}'}(\tilde{\mathbf y}')={\mathbf y}'$ is a stable map with admissible system of circles, 
any $\tilde{\mathbf y} \in \widetilde{W}'$ satisfies Conditions  \ref{cond}, \ref{cond2}, \ref{cond3}, \ref{cond4}.
We define $\Phi_{\tilde{\mathbf x}'}:\widetilde{W}' \to {\mathcal U}(\tilde{\mathbf x})$ 
by $\Phi_{\tilde{\mathbf x}'}(\tilde{\mathbf y}')=\tilde{\mathbf y}$ constructed above.   
It is continuous and $\Phi_{\tilde{\mathbf x}'}(\tilde{\mathbf x}') 
= \tilde{\mathbf x}$.  Hence, for any neighborhood $\widetilde{U}$ of $\tilde{\mathbf x}$ in 
${\mathcal U}(\tilde{\mathbf x})$, there exists a neighborhood $\widetilde{U}'$ of $\tilde{\mathbf x}'$ 
in ${\mathcal U}(\tilde{\mathbf x}')$ such that 
$\Phi_{\tilde{\mathbf x}'}(\widetilde{U}') \subset {\mathcal U}(\tilde{\mathbf x})$.  
It implies that $\pi^{L'}_{\tilde{\mathbf x}'}(\widetilde{U}') \subset \pi^L_{\tilde{\mathbf x}}(\widetilde{U})$.   
\end{proof}
\par

When the combinatorial type $\mathfrak{c}=\mathfrak{c}({\mathbf x})$ is fixed, 
we have mentioned in Remark \ref{choiceV(x)} that ${\mathcal U}(\tilde{\mathbf x})$ enjoys the following property.

\begin{equation} 
{\mathcal M}_3^{{\text{\rm sph}}, ={\mathfrak{c}}({\mathbf x})} (J_N;\alpha;{\mathcal C}) \subset
\pi^L_{\tilde{\mathbf x}}({\mathcal U}(\tilde{\mathbf x})). \label{samecomb}
\end{equation}

We observe that ${\mathcal U}(\tilde{\mathbf x})$ satisfies the second axiom
of countability, since so does the moduli space
${\mathcal M}_{3+L}^{\text{\rm sph}}(J_N;\alpha)$.  
We recall from Lemma \ref{finite-combi-type} that
${\mathcal M}_3^{\text{\rm sph}}(J_N;\alpha;{\mathcal C})$
carries only a finite number of combinatorial types.
Combining these observations with  \eqref{samecomb}, we obtain
the following
\begin{prop}
${\mathcal M}_3^{\text{\rm sph}}(J_N;\alpha;{\mathcal C})$ satisfies the second axiom of countability.
\end{prop}
Hence compactness is equivalent to sequential compactness and
Hausdorff property is equivalent to the uniqueness of the limit of convergent sequences
for moduli spaces of stable maps with circle system.
\begin{prop}\label{seqcpt}
The moduli space ${\mathcal M}_3^{\text{\rm sph}}(J_N;\alpha;{\mathcal C})$ is sequentially compact.
\end{prop}
\begin{proof}
Let  ${\mathbf x}^{(j)}= \bigl( \Sigma({\mathbf x}^{(j)}), z_0^{(j)},z_1^{(j)}, z_2^{(j)}, \{C_a^{(j)}\}, u^{(j)}:\Sigma({\mathbf x}^{(j)}) \to N \bigr)$ be a sequence in
${\mathcal M}_3^{\text{\rm sph}}(J_N;\alpha;{\mathcal C})$.
Because there are a finite number of combinatorial types, we may assume that ${\mathfrak{c}}({\mathbf x}^{(j)})$
are independent of $j$.
First of all, we find a candidate of the limit of a subsequence of ${\mathbf x}^{(j)}$.
In Step 1, we consider irreducible components in $\Sigma({\mathbf x}^{(j)})$ explained in Remark \ref{remcircle} (2).
Since these components are not stable after forgetting the admissible system of circles, we put a point on each component
so that these components become stable.  We obtain ${\mathbf x}^{(j)+}$ at this stage.
Because the moduli space of stable maps with marked points is compact, we can take a convergent subsequence.
We denote its limit  by ${\mathbf x}_{\infty}^+$.
We will find an admissible system of circles on ${\mathbf x}_{\infty}^+$ and insert  irreducible components explained in
Remark \ref{remcircle} (2), if necessary, to obtain a candidate of the limit of ${\mathbf x}^{(j)}$ in Steps 2 and 3.
There are three cases, case (i), case (ii$+$), case (ii$-$) in Definition \ref{unstwocirc}.
In Step 2, we deal with sequences $C_a^{(j)}$ of circles on $\Sigma({\mathbf x}^{(j)+})_a$, which do not collapse
to any node of  $\Sigma({\mathbf x}_{\infty}^+)$.
In Step 3, we discuss when insertions of such  irreducible components are necessary and explain how to
perform insertions to obtain the candidate of the limit of ${\mathbf x}^{(j)}$.  The detail follows.

{\bf Step 1.} \ \ If ${\mathbf x}^{(j)}$ has irreducible components which become unstable
after the circle system forgotten, we put additional marked point $z_a^{(j),++}$ for each such component
$\Sigma({\mathbf x}^{(j)})_a$.
Such an irreducible component contains its root $p_{\text{\rm root}, a}^{(j)}$, another special point $p_{\text{\rm out}, a}^{(j)}$
and the circle $C_a^{(j)}$.
We can take $z_a^{(j),++}$ for each $j$ with the following property.
For any $j, j'$,  there exists a biholomorphic map $\phi_{j',j} :\Sigma({\mathbf x}^{(j)})_a \to \Sigma({\mathbf x}^{(j')})_a$ such that
$\phi_{j',j}$ sends $p_{\text{\rm root}, a}^{(j)}$, $p_{\text{\rm out}, a}^{(j)}$, $z_a^{(j), ++}$ and $C_a^{(j)}$ to
$p_{\text{\rm root}, a}^{(j')}$, $p_{\text{\rm out}, a}^{(j')}$, $z_a^{(j'), ++}$ and $C_a^{(j')}$, respectively.
Such components are determined by the combinatorial data, hence the number of these components are independent of $j$.
If there are $\ell$ such components, we obtain a sequence
${\mathbf x}^{(j)+} \in {\mathcal M}_{3+\ell}^{\text{\rm sph}}(J_N;\alpha)$.

{\bf Step 2.} \ \ Since ${\mathcal M}_{3+ \ell}^{\text{\rm sph}}(J_N;\alpha)$ is compact,
there is a convergent subsequence of
${\mathbf x}^{(j)+}$.  We may assume that ${\mathbf x}^{(j)+}$ is convergent.
Denote its limit by
${\mathbf x}_{\infty}^+ \in {\mathcal M}_{3+ \ell}^{\text{\rm sph}}(J_N;\alpha)$.
Forgetting the marked points $z_3, \dots , z_{2+ \ell}$,   $\Sigma({\mathbf x}_{\infty}^+)$ is a prestable curve of genus 0 with three marked points
$z_0, z_1, z_2$.  Hence we can define a unique irreducible component of Type I-1 in the same way as in Definition \ref{def:circlesystemI} (1).

Let $\{V_k\}_k$, $V_{k+1} \subset V_k$, be a sequence of open neighborhoods of the set of nodes on the domain $\Sigma({\mathbf x}_{\infty}^+)$
such that $\cap_k V_k$ is the set of nodes.  There exists a positive integer $N(k)$ such that
if $j \geq N(k)$, there is a holomorphic embedding
$$\phi^{(j)}_k: \Sigma({\mathbf x}_{\infty}^+) \setminus V_k \to \Sigma({\mathbf x}^{(j)+}).$$
For each component $\Sigma({\mathbf x}_{\infty}^+)_a$,
there is an irreducible component $\Sigma({\mathbf x}^{(j)+})_{a'}$ such that $\phi^{(j)}_k(\Sigma({\mathbf x}_{\infty}^+)_a
\setminus V_k) \subset \Sigma({\mathbf x}^{(j)+})_{a'}$.

From now on, we will take subsequences of $\{ j \}$ successively and rename it $\{ j \}$.

{\bf Step 2-1. } \ \ If $C_{a'}^{(j)} \subset \Sigma({\mathbf x}^{(j)+})_{a'}$ is empty for any $j$,
we set $C_a \subset \Sigma({\mathbf x}_{\infty}^+)_a$ to be an empty set.

{\bf Step 2-2.} \ \  Consider the case where  there is $k$ such that $\phi^{(j)}_k (\Sigma({\mathbf x}_{\infty}^+)_a
\setminus V_k)$ intersects the circle $C_{a'}^{(j)}$ on $\Sigma({\mathbf x}^{(j)+})_{a'}$ for any $j$.

We treat the following two cases separately.

{\bf Case 1:} \ \ For any point $p \in \Sigma({\mathbf x}_{\infty}^+)_a$, which is not a node, there is a neighborhood $U(p)$ of $p$ such that
$C^{(j)}_{a'}$ is not contained in $\phi^{(j)}_k (U(p))$ for any $j$.

Since $\Sigma({\mathbf x}_{\infty}^+)_a \setminus V_k$ is compact,
after taking a subsequence of $j$, we may assume that
there are three distinct points $p_1, p_2, p_3 \in \Sigma({\mathbf x}_{\infty}^+)_a
\setminus V_k$ such that there are mutually disjoint neighborhoods $U(p_1), U(p_2), U(p_3)$
and $\phi^{(j)}_k(U(p_i)) \cap C^{(j)}_{a'} \neq \emptyset$. 
We pick $p_i^{(j)} \in \phi^{(j)}_k(U(p_i)) \cap C^{(j)}_{a'}$.  
After taking a suitable subsequence of $j$, we may assume that $p_i^{(j)}$ converges to $p_i^{(\infty)}$ 
for $i=1,2,3$.  
Then we take the circle passing through $p_1^{(\infty)}, p_2^{(\infty)}, p_3^{(\infty)}$, which we denote by
 $C_a \subset\Sigma({\mathbf x}_{\infty}^+)_a$.  Since $C^{(j)}_{a'}$ are oriented and  
all ${\mathbf x}^{(j)}$ have the same
combinatorial type, $C_a$ is also canonically oriented.  
It is clear that the circle $C_a$ is uniquely determined once the subsequence of ${j}$ is taken as above.  
(For example, consider the cross ratios.)  

Note that  Case 1 is applied to each irreducible component treated in Step 1, hence an oriented circle is
put on each of such components.

{\bf Case 2:} \ \ There is a point $p$ on $\Sigma({\mathbf x}_{\infty}^+)_a$  such that $p$ is not a node and,
for any neighborhood $U(p)$ of $p$, there exists a positive integer $N$ such that
if $j \geq N$, $C_{a'}^{(j)}$ is contained in $\phi^{(j)}_k(U(p))$.

Since $p$ as above is not a node, $C_{a'}^{(j)}$ must contain a unique special point by the admissibility
of the circle system.  (If there are at least two special points, then these points get closer as $j \to \infty$.
Hence there should appear a new component attached at $p$ in
$\Sigma({\mathbf x}_{\infty}^+)$.)
In this case, we attach a new irreducible component at $p$ as in case (i) in Definition \ref{unstwocirc}.
Note that the attached component contains a circle of Type I.

{\bf Step 3.} \ \ Let $p$ be a node on $\Sigma({\mathbf x}_{\infty}^+)$.
Then either $p$ is the limit of nodes $p^{(j)}$ on $\Sigma({\mathbf x}^{(j)+})$ or
$p$ appears as a degeneration of $\Sigma({\mathbf x}^{(j)+})$.
We insert an irreducible component explained in Remark \ref{remcircle} (2) as discussed
in Case 1-2 and Case 2-2 below.

{\bf Case 1:} \ \ $p$ is the limit of nodes $p^{(j)}$.

{\bf Case 1-1:}
If $p$ and $p^{(j)}$ are either both smoothable or both non-smoothable,
we keep the node $p$ as it is.

{\bf Case 1-2:}
If one of them is smoothable and the other is non-smoothable,
we insert a new irreducible component  as in Remark \ref{remcircle} (2) in the following way.

Firstly, we consider the case that $p^{(j)}$ are non-smoothable but $p$ is smoothable.
In this case, we find that both component containing $p^{(j)}$ are of Type I.  We insert
the component of case (i) in Definition \ref{unstwocirc} at $p$.

Next,  we consider the case that $p^{(j)}$ are smoothable but $p$ is non-smoothable.
If  $p^{(j)}$ are smoothable nodes joining two components of Type I in  $\Sigma({\mathbf x}^{(j)+})$,
we insert the component of case (i) in Definition \ref{unstwocirc}.
Suppose that at least one of irreducible components containing $p^{(j)}$ is of Type II.
There are two possibilities, which we discuss separately.
The first possibility is  that both components contain non-empty circles in the admissible system.
Let $\Sigma({\mathbf x}^{(j)+})_a$ be the component such that $p^{(j)}$ is its outward node and
$\Sigma({\mathbf x}^{(j)+})_b$ the component such that $p^{(j)}$ is its root node.
In this case, $C_a^{(j)}$ cannot collapse to $p$, since $C_a^{(j)}$ must contain at least one special points
other than $p^{(j)}$.  Namely, if $\Sigma({\mathbf x}^{(j)})_a$ is of Type I, then $C_a^{(j)}$ must contain one of $z_0, z_1, z_2$ 
or a node of a tree of components of Type I.  If  $\Sigma({\mathbf x}^{(j)})_a$ is of Type II, then $C_a^{(j)}$ must contain
the root node of $\Sigma({\mathbf x}^{(j)})_a$.
Hence we consider the case that $C_b^{(j)}$ collapse to $p$.
Recall that $D_b^{(j)+}$ denotes the domain in $\Sigma({\mathbf x}^{(j)+})_b$  bounded by the oriented circle $C_b^{(j)}$.
If $D_b^{(j)+}$ collapse to $p$, then we insert the component of case (ii$-$) in Definition \ref{unstwocirc}.
Otherwise, we insert the component of case (ii$+$) in Definition \ref{unstwocirc}.
The second possibility is that only one of $\Sigma({\mathbf x}^{(j)+})_a$ or $\Sigma({\mathbf x}^{(j)+})_b$ contains a non-empty circle in the admissible system.
Note that the component $\Sigma({\mathbf x}^{(j)+})_a$ must contain $p^{(j)}$ as an outward node.
It is impossible that $C_a^{(j)} = \emptyset$ and $C_b^{(j)} \neq \emptyset$.
Hence $C_a^{(j)}$ is non-empty.  Since $p^{(j)}$ is a smoothable node, $C_a^{(j)}$ does not contain $p^{(j)}$.
Denote by $D_a^{(j)+}$ the domain in $\Sigma({\mathbf x}^{(j)+})_a$ bounded by the oriented circle $C_a^{(j)}$.
If $D_a^{(j)+}$ contains the node $p^{(j)}$, we insert the component of case (ii$+$) in Definition \ref{unstwocirc} at the node $p$.
Otherwise, we insert the component of case (ii$-$) in Definition \ref{unstwocirc} at the node $p$.

{\bf Case 2:} \ \ $p$ appears as a degeneration of $\Sigma({\mathbf x}^{(j)+})$.

{\bf Case 2-1:}  If $p$ is smoothable, we keep the node $p$ as it is.

{\bf Case 2-2:}  .  If $p$ is non-smoothable, we insert a new irreducible component
explained in Remark \ref{remcircle} (2) in a similar way to Case 1-2 above as follows.

Suppose that a sequence $\Sigma({\mathbf x}^{(j)+})_a$ degenerates to a nodal curve
with a node $p$, which is non-smoothable.
Let $\Sigma({\mathbf x}_{\infty}^+)_{a_1}$ and $\Sigma({\mathbf x}_{\infty}^+)_{a_2}$ be components containing the node $p$
such that $\Sigma({\mathbf x}_{\infty}^+)_{a_2}$ is farther from the component of Type I-1 than $\Sigma({\mathbf x}_{\infty}^+)_{a_1}$.
Here the component of Type I-1 of $\Sigma({\mathbf x}_{\infty}^+)$ is defined in the beginning of Step 2.

There are two possibilities.
The first possibility is that  there exists a positive integer $k$ such that
$C_a^{(j)} \cap \phi_k^{(j)} (\Sigma({\mathbf x}_{\infty}^+)_{a_1} \setminus V_k) = \emptyset$
for any sufficiently large $j > N(k)$.
In this case, we find that $\Sigma({\mathbf x}^{(j)+})_a$ is of Type I.
We insert the component of case (i) in Definition \ref{unstwocirc} at the node $p$.
The second possibility is that there exists a positive integer $k$ such that
$C_a^{(j)} \cap \phi_k^{(j)} (\Sigma({\mathbf x}_{\infty}^+)_{a_2} \setminus V_k) = \emptyset$
for any sufficiently large $j > N(k)$.
If there exists a positive integer $k'$ such that $D_a^{(j)+} \cap \phi_{k'}^{(j)} (\Sigma({\mathbf x}_{\infty}^+)_{a_2} \setminus V_{k'}) = \emptyset$,
we insert the component of case (ii$-$) in Definition \ref{unstwocirc} at the node $p$.
Here $D_{a}^{(j)+}$ be the domain in $\Sigma({\mathbf x}^{(j)+})_a$ bounded by the oriented circle $C_a^{(j)}$.
Otherwise, we insert the component of case (ii$+$) in Definition \ref{unstwocirc} at the node $p$.

\begin{rem}
Suppose that the adjacent components $\Sigma({\mathbf x}_{\infty}^+)_{a_1}$
and $\Sigma({\mathbf x}_{\infty}^+)_{a_2}$
contain circles $C_{a_1}$ and $C_{a_2}$, respectively, which are put in Step 2-2, and suppose that
$\Sigma({\mathbf x}^{(j)})_{a_1'}=\Sigma({\mathbf x}^{(j)})_{a_2'}$,
i.e., this component  degenerates to a nodal curve including
$\Sigma({\mathbf x}_{\infty}^+)_{a_1}$ and $\Sigma({\mathbf x}_{\infty}^+)_{a_2}$.
Since $C_{a'_1}^{(j)}$ is a circle which is connected, it passes through the neck region corresponding to
the node between $\Sigma({\mathbf x}_{\infty}^+)_{a_1}$ and
$\Sigma({\mathbf x}_{\infty}^+)_{a_2}$.
Hence $C_{a_1}$ and $C_{a_2}$ pass through the node, which is smoothable in the sense of
Definition \ref{smoothable}.
Case 2-2 is the case that one of $C_{a_1}$ and $C_{a_2}$ is a circle passing through
the node and the other is empty.
\end{rem}

After these processes, we obtain ${\mathbf x}_{\infty}^{\prime \prime}$ as the candidate of the limit of
${\mathbf x}^{(j)}$.
By the construction, ${\mathbf x}_{\infty}^{\prime \prime}$ is equipped with
an admissible system of circles.

\begin{rem}
By our construction, in particular, Step 3 Case 1-2 and Case 2-2, we find
$${\mathfrak{c}}({\mathbf x}_{\infty}^{\prime \prime}) \preceq {\mathfrak{c}}({\mathbf x}^{(j)}).$$
\end{rem}

Now we show the following

\begin{lem}\label{lemcomp}
There exists a subsequence of ${\mathbf x}^{(j)}$, which  converges to ${\mathbf x}_{\infty}^{\prime \prime}$
in ${\mathcal M}_3^{\text{\rm sph}}(J_N;\alpha;{\mathcal C})$.
\end{lem}

\begin{proof}
We will add suitable marked points on ${\mathbf x}^{(j)}$ to obtain $\tilde{\mathbf x}^{(j)}$.
For an irreducible component in ${\mathbf x}^{(j)}$, which becomes unstable when forgetting the circle, we added the marked point $z_a^{(j),++}$ in Step 1.
Since such an irreducible component contains three special points, i.e., nodes or marked points and the holomorphic map is constant on the irreducible component,
the limit  ${\mathbf x}_{\infty}^+$ must contain an irreducible component of the same type.
We add one more marked point on the circle on the component in ${\mathbf x}^{(j)}$ and ${\mathbf x}_{\infty}^{\prime \prime}$, respectively,
so that the four special  points on the component have the fixed cross ratio.  (In total, we add two marked points on the circle in this case.)

We put additional marked points on other irreducible components as follows.
If an irreducible component does not contain a circle in the admissible system, we do not put additional marked points.
Then the remaining components are either those discussed in Step 2 or those discussed in Step 3.
We deal with them separately.

For an irreducible component in Case 1 of Step 2-2, we have three marked points $p_1, p_2, p_3$ on the circle $C_a$.
We pick $p_i^{(j)} \in \phi_k^{(j)}(U(p_i)) \cap C_{a'}^{(j)}$ , $i=1,2,3$.  (Here we added three marked points on the circle.)

For an irreducible component in Case 2 of Step 2-2, we have a special point $q$ on the new attached irreducible component $\Sigma({\mathbf x}_{\infty}^{\prime \prime})_b$.
For sufficiently large $j$, the circle $C^{(j)}_{a'}$ is contained in $\phi^{(j)}_k(U(p))$.
By the admissibility of the circle system, there should be a special point $q^{(j)}$ on $C^{(j)}_{a'}$.
We put two more marked points $q_1 ,q_2$ on the circle $C_b$ on $\Sigma({\mathbf x}_{\infty}^{\prime \prime})_b$.
We choose $p^{(j)} \notin \phi_k^{(j)}(U(p))$ which converges to some $p' \in \Sigma({\mathbf x}_{\infty}^{\prime \prime})_a$.
We choose $q_1^{(j)}, q_2^{(j)} \in C^{(j)}_{a'}$ in such a way that the cross ratio of $p^{(j)}, q^{(j)}, q_1^{(j)}, q_2^{(j)}$ is equal to
the one of $\overline{p}, q, q_1, q_2$ for any $j$.
Here $\overline{p}$ is the node, where we attach
$\Sigma({\mathbf x}_{\infty}^{\prime \prime})_b$
to $\Sigma({\mathbf x}_{\infty}^{\prime \prime})_a$ in Case 2 of Step 2-2.
(In this case, we add two marked points.)

For each newly inserted component $\Sigma({\mathbf x}_{\infty}^{\prime \prime})_c$
in Step 3, we add additional two marked points $q_1, q_2 \in C_c$  as follows.
Firstly, we consider Case 1.
Let $\Sigma({\mathbf x}_{\infty}^+)_a$ and $\Sigma({\mathbf x}_{\infty}^+)_b$ be the irreducible components of
$\Sigma({\mathbf x}_{\infty}^+)$, which intersect at $p$
and let $\Sigma({\mathbf x}^{(j)+})_a$ and $\Sigma({\mathbf x}^{(j)+})_b$ be the irreducible components of
$\Sigma({\mathbf x}^{(j)+})$, which intersect at $p^{(j)}$.
Here we arrange that $p$ (resp. $p^{(j)}$) is the root node of $\Sigma({\mathbf x}_{\infty}^+)_b$ (resp. $\Sigma({\mathbf x}^{(j)+})_b$).
The new component $\Sigma({\mathbf x}_{\infty}^{\prime \prime})_c$ is inserted between $\Sigma({\mathbf x}_{\infty}^+)_a$ and $\Sigma({\mathbf x}_{\infty}^+)_b$.
We denote by $\tilde{p}$  the node of $\Sigma({\mathbf x}_{\infty}^{\prime \prime})_c$, which has the same combinatorial type as the node $p^{(j)}$ (see
Definition \ref{combtype} for the definition of the combinatorial type of a node) and by ${\tilde p}'$ the other node of $\Sigma({\mathbf x}_{\infty}^{\prime \prime})_c$.
Let $d=a \text{\rm ~or~} b$ such that $\tilde p$ is (resp. is not) the root of $\Sigma({\mathbf x}_{\infty}^{\prime \prime})_c$  if and only if
$p^{(j)}$ is (resp. is not) the root of $\Sigma({\mathbf x}^{(j)+})_d$.
We discuss the following two cases separately:

\begin{enumerate}
\item[({\rm a})] $\tilde p \in C_c$ and ${\tilde p}' \notin C_c$,
\item[({\rm b})] $\tilde p \notin C_c$ and ${\tilde p}' \in C_c$.
\end{enumerate}

Pick and fix $k$.  Then we have $V_k$ and $\phi^{(j)}_k: \Sigma({\mathbf x}_{\infty}^+) \setminus V_k \to \Sigma({\mathbf x}^{(j)+})$ as in the beginning of
Step 2.
In Case (a), we find that $p^{(j)} \in C_d^{(j)}$.  Pick $q_1, q_2 \in C_c \setminus \{\tilde p\}$.
We choose points $q_1^{(j)}, q_2^{(j)}$ on $C_{d}^{(j)}$ as follows.
Pick $\tilde{p}^{\prime (j)}$ on $\Sigma({\mathbf x}^{(j)})_{d} \cap \phi_k^{(j)}(\Sigma({\mathbf x}_{\infty}^+) \setminus V_k)$.  
Then we take $q_1^{(j)}, q_2^{(j)} \in C_d^{(j)}$ such that
the cross ratios of $p^{(j)}, q_1^{(j)}, q_2^{(j)}, \tilde{p}^{\prime (j)}$ are
equal to the one of $\tilde{p}, q_1, q_2, \tilde{p}'$.

In Case (b), we find that $p^{(j)} \notin C_d^{(j)}$.  Pick $q_1, q_2 \in C_c \setminus \{\tilde p' \}$.
We choose points $q_1^{(j)}, q_2^{(j)}$ on $C_{d}^{(j)}$ as follows.
Pick $\tilde{p}^{\prime (j)}$ on $C_{d}^{(j)} \cap \phi_k^{(j)}(\Sigma({\mathbf x}_{\infty}^+) \setminus V_k)$.
Then we take $q_1^{(j)}, q_2^{(j)} \in C_d^{(j)}$ such that
the cross ratios of $p^{(j)}, q_1^{(j)}, q_2^{(j)}, \tilde{p}^{\prime (j)}$ are
equal to the one of $\tilde{p}, q_1, q_2, \tilde{p}'$.

Next we consider Case 2.
Let $\Sigma({\bf x}_{\infty}^+)_{a_1}$ and $\Sigma({\bf x}_{\infty}^+)_{a_2}$ be the irreducible components,
which share the node $p$.
Pick distinct three points $q_1,q_2,q_3$ on the newly inserted components so that none of them are nodes.
We note that the circle $C_{a'}^{(j)}$ intersects at least one of
$\phi^{(j)}_k(\Sigma({\mathbf x}_{\infty}^+)_{a_i} \setminus V_k)$, $i=1,2$.
From now on, we fix such a $k$ and denote it by $k_0$.
We may assume that  $\phi^{(j)}_{k_0}(\Sigma({\bf x}_{\infty}^+)_{a_1} \setminus V_{k_0})$ intersects $C_{a'}^{(j)}$.
(The other case is similar.)
Then  Case 2 in Step 2-2 is applied to  $\Sigma({\mathbf x}_{\infty}^+)_{a_2}$.
Pick $p' \in C_a \cap \Sigma({\bf x}_{\infty}^+)_{a_1} \setminus V_{k_0}$.
For all sufficiently large $j$, we arrange the neck region $V_{\text{\rm neck}, p}^{(j)}$,
which is a connected component of the complement of
$\phi^{(j)}_{k_0}(\Sigma({\bf x}_{\infty}^+) \setminus V_{k_0})$
and degenerates to a neighborhood of the node $p$, as follows.
Pick and fix a suitable biholomorphic map
$\varphi^{(j)}: \Sigma({\bf x}^{(j)})_{a'} \to {\bf C}P^1$  such that
\begin{itemize}
\item $V_{\text{\rm neck}, p}^{(j)}$ is mapped to an annulus
$\{z \in \C ~  \vert ~ r^{(j)} < \vert z \vert < R^{(j)}\}$ for some $0 < r^{(j)} < 1/2, R^{(j)} >1$,
\item
$(\varphi^{(j)})^{-1}(\{z \in \C ~  \vert ~  \vert z \vert < r^{(j)} \})$ contains
 $\phi^{(j)}_k(\Sigma({\bf x}_{\infty}^+)_{a_1} \setminus V_k)$.
\end{itemize}
By applying a  dilation fixing $0, \infty$, we may assume that the circle $C^{(j)}_{a'}$ is tangent to the unit circle
$\{z \in \C  ~ \vert ~ \vert z \vert =1\}$.
Since $C^{(j)}_{a'} \cap \phi^{(j)}_k(\Sigma({\bf x}^+_{\infty})_{a_1} \setminus V_k)$ is not empty for any $k$,
we find that $r^{(j)}$ tends to $0$.
Similarly, since for each given $k$,  $C^{(j)}_{a'}$ does not intersect
$\phi^{(j)}_k(\Sigma({\bf x}^+_{\infty})_{a_2} \setminus V_k)$ for all sufficiently large $j$,
the number $R^{(j)}$ tends to $+\infty$.
Pick $p^{\prime (j)} \in C^{(j)}_{a'}$ such that $\vert \varphi^{(j)}(p^{\prime (j)}) \vert < r^{(j)}$ and
${\bf x}_j^+$ with $p^{\prime (j)}$ added converges to ${\bf x}_{\infty}^+$ with $p'$ added.
We pick $q^{(j)}_1, q^{(j)}_2, q^{(j)}_3$ on $C^{(j)}_{a'}$ such that
$\vert \varphi^{(j)}(q^{(j)}_1) \vert =1$, $\vert \varphi^{(j)}(q^{(j)}_2) \vert =1/2$
and the cross ratios of $p^{\prime (j)}, q^{(j)}_1, q^{(j)}_2, q^{(j)}_3$ are the same as the cross ratio of
$p', q_1, q_2, q_3$.

After adding those new marked points, we obtain $\tilde{\mathbf x}^{(j)}$ for all sufficiently large $j$ and $\tilde{\mathbf x}_{\infty}^{\prime \prime}$
in ${\mathcal M}_{3+L}^{\text{\rm sph}}(J_N;\alpha)$ for some $L$.
By the choice of those points, we find that $\tilde{\mathbf x}^{(j)}$ converges to $\tilde{\mathbf x}_{\infty}^{\prime \prime}$ in 
${\mathcal M}_{3+L}^{\text{\rm sph}}(J_N;\alpha)$.
By Definition \ref{top}, ${\mathbf x}^{(j)}$ converges to
${\mathbf x}^{\prime \prime}_{\infty}$ in
${\mathcal M}_{3}^{\text{\rm sph}}(J_N;\alpha; {\mathcal C})$.
\end{proof}
This finishes the proof of Proposition \ref{seqcpt}.
\end{proof}

Now we have

\begin{thm}
The moduli space ${\mathcal M}_3^{\text{\rm sph}}(J_N;\alpha;{\mathcal C})$ is compact and Hausdorff.
\end{thm}

\begin{proof}
Since ${\mathcal M}_3^{\text{\rm sph}}(J_N;\alpha;{\mathcal C})$ satisfies the second axiom of countability,
Proposition \ref{seqcpt} implies the compactness.
Then the Hausdorff property follows in the same way as in Lemma 10.4 in \cite{FO}.
\end{proof}

\begin{thm}\label{KuraonMsph}
The moduli space $\mathcal M^{\text{\rm sph}}_3(J_N;\alpha;\mathcal C)$ carries a Kuranishi structure.
\end{thm}

\begin{proof}
We construct a Kuranishi structure on
$\mathcal M^{\text{\rm sph}}_3(J_N;\alpha;\mathcal C)$ in the same way as  the case of moduli space of stable maps, \cite{FO}
see \cite{foootech} Part 3 and Part 4 for more details.

Our strategy to construct Kuranishi structures on the moduli space of bordered stable maps with admissible system of
circles is to reduce the construction to the one for moduli spaces of bordered stable maps with marked points
by putting a suitable number of points on circles.
The only points, which we have to take care of, are the following two points.

The first point to be taken care of is the way how to deal with the admissible system of circles in terms of additional marked points on the domain curve.
An element in $\mathcal M^{\text{\rm sph}}_3(J_N;\alpha;\mathcal C)$ is a stable map with an admissible system of circles.  
For $\mathbf x \in {\mathcal M}^{\text{\rm sph}}_3(J_N;\alpha;\mathcal C)$, we put suitable marked points on circles in the admissible system of 
circles to obtain an element $\tilde{\mathbf x} \in {\mathcal M}_{3+L}^{\text{\rm sph}}(J_N;\alpha)$ and 
\begin{eqnarray}
\pi^L_{\tilde{\mathbf x}} : {\mathcal U}(\tilde{\mathbf x})  \to {\mathcal M}_{3}^{\text{\rm sph}}(J_N;\alpha;{\mathcal C}) \nonumber \\
(\Sigma, \vec{z}\,^+, u) \mapsto (\Sigma, \vec{z}, \{ C_a\}, u). \nonumber
\end{eqnarray}
(See Lemma \ref{add}.)
Note that  $\pi^L_{\tilde{\mathbf x}}$ is not injective.
For ${\mathbf x} \in  {\mathcal M}_{3}^{\text{\rm sph}}(J_N;\alpha;{\mathcal C})$,
we take a subspace of a Kuranishi neighborhood of $\tilde {\mathbf x} \in {\mathcal U}(\tilde{\mathbf x})$ as we explain in the next paragraph
to obtain a Kuranishi neighborhood of ${\mathbf x}$.

If an irreducible component $\Sigma_a \subset \Sigma$ contains a circle $C_a \neq \emptyset$, $C_a$ must contain
at least one special point, i.e., a node or a marked point.  
Note that $C_a$ is an oriented circle on $\Sigma_a$.  
If a holomorphic automorphism $\varphi$ of $\Sigma_a$ is of finite order preserving $C_a$ and 
its orientation and if $\varphi$ fixes a point on $C_a$, then $\varphi$ must be the identity.  
Hence the stabilizer of this component must be trivial.
If the number of special points on $C_a$ is less than $3$,  we take the minimal number of marked points on $C_a$
in such a way that the total number of special points  is $3$.
Let $p_1, \dots, p_c$ be nodes on $C_a$ and $w_1, \dots, w_k$ marked points on $C_a$.
Then, for each marked point  $w_j$ on $C_a$, we choose a short embedded arc  $A_{w_j}$ on $\Sigma_a$ which is transversal to $C_a$ at  $w_j$.  
We allow to move the marked point $w_j$ to  $w'_j$ on $A_{w_j}$ such that
$p_1, \dots, p_c, w'_1, \dots, w'_k$ lie on a common circle.
This last condition is expressed using the cross ratio and
these constraints cut out the set of such $w'_1, \dots, w'_k$ transversally.
Thus if we restrict $\pi^L_{\tilde{\mathbf x}}$ to the subset of ${\mathcal U}(\tilde{\mathbf x})$
such that the extra $L$ marked points hit
$A_{w_j}$, the restricted map is injective.
(More precisely a similar map on the Kuranishi neighborhood is
injective.)
Therefore we can use this subset of Kuranishi neighborhood of
${\mathcal U}(\tilde{\mathbf x})$
as a Kuranishi neighborhood of ${\mathcal M}^{\text{\rm sph}}_3(J_N;\alpha;\mathcal C)$.

The second point to be taken care of is about the gluing construction.
We use gluing construction following Part 3 and Part 4 of \cite{foootech} at smoothable nodes.
Let $p$ be a smoothable node.
If no circle in the admissible system of circles  passes through $p$,  we use the gluing construction as in the case of stable maps.
Otherwise, we proceed as follows.
Let $\Sigma_a$ and $\Sigma_b$ intersect at the node $p$.  Then the circles $C_a$ and $C_b$ contain $p$.
In order to perform gluing, we need {\it coordinate at infinity}, Definition 16.2 in \cite{foootech}.
For a stable map with an admissible system of circles, we use a coordinate at infinity adapted to the circle system as follows.   
We pick a complex local coordinate $\xi$, $\eta$ on $\Sigma_a$, resp., $\Sigma_b$, around $p$ such that
$C_a$, resp., $C_b$, with the given orientation corresponds to the real line oriented from $-\infty$ to $+\infty$
in the $\xi$-plane, resp. the $\eta$-plane.
For the gluing construction,
we use the gluing parameter $T \in [T_0, \infty]$ for a sufficiently large $T_0  >0$ such that the $\xi$-plane and $\eta$-plane are glued by
$\xi \cdot \eta = - e^{-T}$.
In this case the parameter to smooth this node is $[T_0,\infty)$.
Therefore such a point corresponds to a boundary. (Or corner if there are more such points.)
We remark that in the case when there is no circle on the node the
parameter to smooth this node is $[T_0,\infty) \times S^1$.
Then the construction of a Kuranishi structure goes through as in Part 3 and Part 4 in \cite{foootech}.
\end{proof}

Once we have Kuranishi structures on ${\mathcal M}_3^{\text{\rm sph}}(J_N;\alpha)$ and ${\mathcal M}_3^{\text{\rm sph}}(J_N;\alpha;{\mathcal C})$,
we also have Kuranishi structures on their fiber products with singular chains $P_i$'s.

For $\mathbf x = (\Sigma({\mathbf x}), \vec{z}, \{ C({\mathbf x})_a \}, u:\Sigma({\mathbf x}) \to N)
\in \mathcal M^{\text{\rm sph}}_3(J_N;\alpha;\mathcal C)$, we set
\begin{eqnarray}
 n({\mathbf x})&= &\text{the number of irreducible components of $\Sigma({\mathbf x})$},  \nonumber \\
 n_{I,\mathfrak p}({\mathbf x})&= &\text{the number of irreducible components $\Sigma({\mathbf x})_a$ of Type I-2 in $\Sigma({\mathbf x})$,}
\nonumber \\
& & \text{such that $C({\mathbf x})_a \neq \emptyset$  does not contain the root node},  \nonumber \\
n_{I, \text{\rm bubble}}({\mathbf x})&= &\text{the number of irreducible components $\Sigma({\mathbf x})_a$ of Type I-2 in $\Sigma({\mathbf x})$,}
\nonumber \\
& & \text{such that $C({\mathbf x})_a$ contains the root node}, \nonumber \\
n_{I, \emptyset}({\mathbf x}) &= &\text{the number of irreducible components $\Sigma({\mathbf x})_a$ of Type I-2 in $\Sigma({\mathbf x})$,}
\nonumber \\
& & \text{such that $C({\mathbf x})_a = \emptyset $}, \nonumber \\
n_{II,\text{\rm circ}}({\mathbf x})&= &\text{the number of irreducible components $\Sigma({\mathbf x})_a$ of Type II in
$\Sigma({\mathbf x})$,}
\nonumber\\
& & \text{such that $C({\mathbf x})_a$ is non-empty}. \nonumber \\
 n_{II, \emptyset}({\mathbf x}) & = & \text{the number of irreducible components $\Sigma({\mathbf x})_a$ of Type II,} \nonumber \\
 & & \text{such that $C({\mathbf x})_a$ is empty.}  \nonumber
\end{eqnarray}
Since these numbers depend only on the combinatorial type $\mathfrak c$, we also denote them by
$n({\mathfrak c})$, $n_{I,\mathfrak p}({\mathfrak c})$, $n_{I,\text{\rm bubble}}({\mathfrak c})$, $n_{I, \emptyset}({\mathfrak c})$,
$n_{II,\text{\rm circ}}({\mathfrak c})$, $n_{II, \emptyset}({\mathfrak c})$.
Note that
$$n({\mathfrak c}) =1+n_{I, \mathfrak p}({\mathfrak c}) + n_{I, \text{\rm bubble}}({\mathfrak c}) +
n_{I, \emptyset}({\mathfrak c}) + n_{II, \text{\rm circ}}({\mathfrak c}) + n_{II, \emptyset}({\mathfrak c}),$$
because there always exists a unique irreducible component of Type I-1.
Then we find the following proposition.
The proof is easy so omitted.
\begin{prop}\label{virdimcombtype}
\begin{enumerate}
\item The virtual codimension $\text{\rm vcd}({\mathfrak c})$ of the stratum with the combinatorial type $\mathfrak c$ is equal to
\begin{eqnarray}
& &   2\left( n({\mathfrak c} \right)-1) - 2 n_{I,\mathfrak p}({\mathfrak c}) - n_{I,\text{\rm bubble}}({\mathfrak c}) -
n_{II,\text{\rm circ}}({\mathfrak c})  \nonumber \\
& = &  n_{I,\text{\rm bubble}}({\mathfrak c}) +  2 n_{I, \emptyset}({\mathfrak c}) + n_{II, \text{\rm circ}}({\mathfrak c}) + 2 n_{II, \emptyset}({\mathfrak c}) ,
\nonumber
\end{eqnarray}
which is non-negative.
\item $\text{\rm vcd}({\mathfrak c})=0$ if and only if $n({\mathfrak c})=n_{I,\mathfrak p}({\mathfrak c})+1$.
Namely, all irreducible components are of Type I with non-empty circles and  the circle on each component of Type I-2 does not contain
the root node.

\item $\text{\rm vcd}({\mathfrak c})=1$ if and only if either
Case (A) $n({\mathfrak c})=n_{I,\mathfrak p}({\mathfrak c})+2$ and
$n_{I,\text{\rm bubble}}({\mathfrak c})=1$ or Case (B) $n({\mathfrak c})=n_{I,\mathfrak p}({\mathfrak c})+2$ and $n_{II,\text{\rm circ}}({\mathfrak c})=1$.
Namely, either all irreducible components are of Type I with non-empty circles and there is exactly one irreducible component $\Sigma_a$ of
Type I-2 such that the circle $C_a$ contains the root node of $\Sigma_a$, or
there is exactly one Type II component $\Sigma_a$ with $C_a \neq \emptyset$, all others are of Type I with non-empty circles and the circle
$C_a$ on each irreducible component of Type I-2 does not contain the root node.
\end{enumerate}
\end{prop}

Proposition \ref{virdimcombtype} describes combinatorial types $\mathfrak c$ such that the corresponding
strata in $\mathcal M^{\text{\rm sph}}_3(J_N;\alpha;\mathcal C)$ is codimension $1$.
There are two cases:

(A)  $n({\mathfrak c})=n_{I,\mathfrak p}({\mathfrak c})+2$ and $n_{I,\text{\rm bubble}}({\mathfrak c})=1$

(B) $n({\mathfrak c})=n_{I,\mathfrak p}({\mathfrak c})+2$ and $n_{II,\text{\rm circ}}({\mathfrak c})=1$.

Case (A) and Case (B) are treated in a different way.  Firstly, we consider Case (A).

Note that the stable map is constant on the irreducible components explained in Remark \ref{remcircle} (2).
By our convention, we do not put obstruction bundles on these components.
Therefore we can identify the following two codimension 1 boundary components equipped with Kuranishi structures:
\begin{enumerate}
\item A Type I component splits into two irreducible components.
\item A Type I circle $C_a$ meets an inward interior marked point and an irreducible component of case (i) in Definition \ref{unstwocirc}
is inserted at the node of the two irreducible components.
\end{enumerate}
See Figure 5 which illustrates an example with $n_{I,\mathfrak p}({\mathfrak c})=0$.
\par
\centerline{
\epsfbox{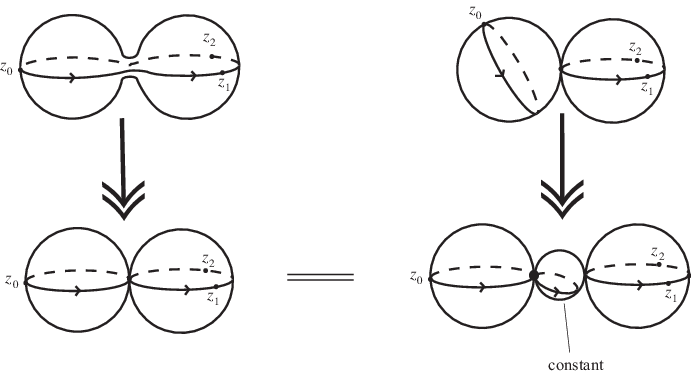}}
\par
\centerline{\bf Figure 5}
\par\smallskip
These two strata are glued to cancel codimension 1 boundaries.
See Remark \ref{inversionTypeII} (2) for the cancellation with orientation.
This is a key geometric idea to see the equality in Lemma \ref{pmainformula}.
From now on, we denote by $${\mathcal M}_3^{\text{\rm sph}}(J_N;\alpha;{\mathcal C})$$
the moduli space with the codimension $1$ boundaries of (1) and (2)  identified
as above.

The remaining codimension 1 boundary components are
Case (B), i.e., those with non-empty Type II circles in the admissible system of circles,
which correspond to codimension 1 disc bubbling phenomenon in
 ${\mathcal M}^{\text{\rm main}}_3(J_{N\times N};\beta)$.
We will study these codimension 1 boundary components in Lemma \ref{forgettypeII}.
See Remark \ref{inversionTypeII}  (3), (4) and Subsection 6.3.1 for the cancellation with orientation in Case (B).

\subsubsection{Proof of Theorem  \ref{Proposition34.25} (2), IV: completion of the proof}\label{6.5}
In this subsection we prove \eqref{diagmaineq}.
First of all, we recall
the following lemma, which is a well-known fact on the moduli space of pseudo-holomorphic spheres
which is used in the definition of quantum cup product \cite{FO}.
Let $\rho \in \pi_2(N)/\sim$, where the equivalence relation $\sim$ was defined in Definition \ref{equivonpi2}.
For given $\rho$ and cycles  $P_0,P_1,P_2$ in $N$,  we defined $\mathcal M^{\text{\rm sph}}_3(J_N;\rho;P_1,P_2,P_0)$ in Definition \ref{notationmoduli}.
\begin{lem}\label{moduliqp}
The moduli space $\mathcal M^{\text{\rm sph}}_3(J_N;\rho;P_1,P_2,P_0) $
carries a Kuranishi structure ${\mathfrak K}_0$ and a multisection $\mathfrak s_0$ such that
$$
\sum_{\rho}
\#\left(\mathcal M^{\text{\rm sph}}_3(J_N;\rho;P_1,P_2,P_0)
\right)^{\mathfrak s_0} T^{\omega(\rho)}e^{c_1(N)[\rho]}
= \langle
PD[P_1] *  PD[P_2],PD[P_0]\rangle
$$
where the sum is taken over $\rho$ for which the virtual dimension of
$\mathcal M^{\text{\rm sph}}_3(J_N;\rho;P_1,P_2,P_0)$ is zero.
\end{lem}

Now we consider the moduli space used to define
the left hand side of (\ref{diagmaineq}).
Let $\vec \beta = (\beta_{1},\ldots,\beta_{{k}}), \beta_{j} \neq 0 \in \Pi (\Delta_N)$.  
Set $\text{length}(\vec{\beta)}=k$.  
We define $\mathcal M_{1;1}(J_{N\times N};\vec \beta;P)$ by induction on $\text{length}(\vec{\beta})$. 
Firstly, we consider the moduli space 
$\mathcal M_{1,(1,0)}(J_{N\times N};\beta)$ of bordered stable maps 
representing the class $\beta$ attached to $(N\times N, \Delta_N)$ with one interior marked point and 
one boundary marked point.   
Here to specify the interior marked point as an output marked point
we use the notation $\mathcal M_{1,(1,0)}(J_{N\times N};\beta)$ 
used in Subsection 8.10.2 \cite{fooobook2}. 
See the line just before Definition 8.10.2 \cite{fooobook2} where  
the orientation on $\mathcal M_{1,(1,0)}(J_{N\times N};\beta)$ is given.
We denote by $z_1$ the first (only one) boundary marked point.
Then we define 
$$
\mathcal M_{1,{(1,0)}}(J_{N\times N};\beta;P)=\mathcal M_{1,(1,0)}(J_{N\times N};\beta)_{ev_1}\times P.  
$$
This is a special case of Definition 8.10.2 in \cite{fooobook2} with $k=1, \ell=0$ and 
the sign is $(-1)^{\epsilon}=+1$ in this case.
\par
When $\text{length}(\beta)=1$, i.e., $\vec{\beta}=(\beta_{1})$, we set 
\begin{equation} \label{M_{1,1}}
\mathcal M_{1;1}(J_{N\times N};\vec \beta;P)
= -\mathcal M_{1,(1,0)}(J_{N\times N}; \beta_1;P).
\end{equation}
Here we reversed the orientation of $\mathcal M_{1,(1,0)}(J_{N\times N}; \beta_1;P)$ so that 
it is compatible with Definition \ref{P_beta} in the case that $k=1$. 
\par
Suppose that the orientation of 
$\mathcal M_{1;1}(J_{N\times N};\vec \beta;P)$ is given for $\text{length}(\vec{\beta}) \leq k$.  
For $\vec{\beta}=(\beta_1, \dots, \beta_{k+1})$, we write 
$\vec{\beta}^-=(\beta_2, \dots, \beta_{k+1})$.  
Then we define 
\begin{equation}\label{def:M11}
\aligned
\mathcal M_{1;1}(J_{N\times N};\vec \beta;P)
=
& - \mathcal M_{1,1}(J_{N \times N};\beta_1)_{ev_1} \times_{p_2 \circ ev_{int}} 
\mathcal M_{1,1}(J_{N \times N};\vec{\beta}^-;P).
\endaligned\end{equation}
Namely we reversed the orientation so that it is consistent with Definition \ref{P_beta} 
for each positive integer $k= \text{length}(\vec{\beta})$.  
We also denote by $\mathcal M_{1;1}(J_{N\times N};\vec \beta;P)$ the chain 
$$
p_2 \circ ev_{int}:\mathcal M_{1;1}(J_{N\times N};\vec \beta;P) \to \Delta_N,
$$ 
where $p_2(x,y)=(y,y)$.  
By an abuse of notation, we set 
$\mathcal M_{1,1}(J_{N \times N};\emptyset;P)=P$.  
From now on, $\vec{\beta}$ is either $\emptyset$ or $(\beta_1, \dots, \beta_k)$ with 
$\beta_j \neq 0$ for each $j=1, \dots, k$.  
For each $P_0,, P_1, P_2$ and $\vec{\beta}_0, \vec{\beta}_1, \vec{\beta}_2$, we define 
\begin{equation}\label{def:Mhat}
\aligned
&\widehat{\mathcal M}(J_{N\times N};\beta';\vec \beta_1,\vec \beta_2,\vec \beta_0;P_1,P_2,P_0)
\\
&= \mathcal M_{1,1}(J_{N \times N};\vec{\beta}_0;P_0)_{ev_1} \times_{ev_0} 
\mathcal M_3^{\text{\rm main}}(J_{N \times N}; \beta';\mathcal M_{1,1}(J_{N \times N};\vec{\beta}_1;P_1), 
\mathcal M_{1,1}(J_{N \times N};\vec{\beta}_2;P_2)).
\endaligned\end{equation}
Taking our Convention 8.2.1 (4) in \cite{fooobook2} and the pairing \eqref{signpairing} into account, 
the following is immediate from definition.
\begin{lem}\label{multiondisc}
There exist a Kuranishi structure ${\mathfrak K}_1$ on
$\widehat{\mathcal M}(J_{N\times N};\beta';\vec \beta_1,\vec \beta_2,\vec \beta_0;P_1,P_2,P_0)$
and a multisection $\mathfrak s_1$ with the
following properties. We denote by $n_\beta$ the sum of
$$
\# \left(
\widehat{\mathcal M}(J_{N\times N};\beta';\vec \beta_1,\vec \beta_2,\vec \beta_0;P_1,P_2,P_0)
\right)^{\mathfrak s_1}
$$
over $(\beta';\vec \beta_1,\vec \beta_2,\vec \beta_0)$ whose total sum
is $\beta$.  Then we have
\begin{equation}\label{eq627}
\langle\mathfrak m_2(\mathcal I(P_1),\mathcal I(P_2)),\mathcal I(P_0)\rangle
= \sum_{\beta} n_{\beta}T^{\omega(\beta)}e^{\mu(\beta)/2}.
\end{equation}
Moreover, The multi-section ${\mathfrak s}_1$ is invariant under the involution $\tau_*$ on
each disc component with bubble trees of spheres.
\end{lem}
The last statement follows from the fact that
the multi-section ${\mathfrak s}_1$  is constructed by induction on the energy of
the bordered stable maps keeping the invariance under $\tau_*$, see the explanation in Remark \ref{remonmult}.

Consider the union of
$\widehat{\mathcal M}(J_{N\times N};\beta';\vec \beta_1,\vec \beta_2,\vec \beta_0;P_1,P_2,P_0)$
over $(\beta';\vec \beta_1,\vec \beta_2,\vec \beta_0)$ such that
the total sum of $(\beta';\vec \beta_1,\vec \beta_2,\vec \beta_0)$ is $\beta$
whose double belongs to class $\rho \in\pi_2(N)/\sim$.
(See Remark \ref{bubbletree} (2) for the double of $\beta$.)
We glue them along virtual codimension one strata appearing in Case (A) in Proposition \ref{virdimcombtype} and denote it by 
\begin{equation}
\widehat{\mathcal M}(J_{N\times N};\rho;P_1,P_2,P_0).  \label{hatmoduli}
\end{equation}
See Sublemma \ref{sublemma} for the description of codimension one strata which we identify.  
Each $\widehat{\mathcal M}(J_{N\times N};\beta';\vec \beta_1,\vec \beta_2,\vec \beta_0;P_1,P_2,P_0)$ has Kuranishi structure in such a way that
we can glue them to obtain a Kuranishi structure on $\widehat{\mathcal M}(J_{N\times N};\rho;P_1,P_2,P_0)$ (see also Lemma \ref{frak s_2}). 
Namely we have
\begin{sublem}\label{sublemma}
The orientations of $\widehat{\mathcal M}(J_{N\times N};\beta';\vec \beta_1,\vec \beta_2,\vec \beta_0;P_1,P_2,P_0)$ are compatible and $\widehat{\mathcal M}(J_{N\times N};\rho;P_1,P_2,P_0)$ 
has an oriented Kuranishi structure.  
\end{sublem} 
\begin{proof} 
It is sufficient to see that two top dimensional strata adjacent along a stratum of codimension 1 induce 
opposite orientations on the stratum of codimension 1.  
\par
Firstly, we consider the case that a transition of strata occurs in one of $\mathcal M_{1,1}(J_{N \times N};\vec{\beta}_i;P_i)$, $i=0,1,2$.   
It suffices to check the compatibility of orientations inside $\mathcal M_{1,1}(J_{N \times N};\vec{\beta};P)$ 
for a given $\vec{\beta}$.  
Let $\vec{\beta}=(\beta_1, \dots, \beta_i, \beta_{i+1}, \dots, \beta_k)$, 
$\vec{\beta}_{(1)}=(\beta_1, \dots, \beta_i)$,  $\vec{\beta}_{(2)}=(\beta_{i+1}, \dots, \beta_k)$ 
and $\vec{\beta}'=(\beta_1, \dots, \beta_{i-1}, \beta_i + \beta_{i+1}, \beta_{i+2}, \dots, \beta_k)$.  
We define 
$$\mathcal M_2(J_{N \times N}; \vec{\beta}_{(2)};P)=\mathcal M_2(J_{N \times N}; \beta_{i+1}; 
\mathcal M_{1,1}(J_{N\times N}; (\beta_{i+2}, \dots, \beta_k);P)).$$  
By Proposition 8.10.4 \cite{fooobook2} with $\beta_1=0$, 
$k=k_2=1$ and $\ell_1=\ell_2=0$, \eqref{M_{1,1}} 
and the proof that $\mathfrak p_{1,0} \equiv i_! \mod \Lambda^+_{0, \text{nov}}$ in page 739 thereof,  we find that 
$$
\mathcal M_2(J_{N \times N};\vec{\beta}_{(2)};P) \subset (-1)^{\dim \Delta_N+1} \partial \mathcal M_{1,1}(J_{N \times N};\vec{\beta}_{(2)};P).
$$  
By Proposition 8.10.4 (2) \cite{fooobook2}, we obtain 
$$
\mathcal M_{1,1}(J_{N \times N};\vec{\beta}_{(1)};\mathcal M_2(J_{N \times N};\vec{\beta}_{(2)};P)) 
\subset (-1)^{\dim \Delta_N +1} \partial \mathcal M_{1,1}(J_{N \times N}; \vec{\beta};P).
$$ 
On the other hand, Proposition 8.10.4 (1) \cite{fooobook2} also implies that 
$$
\mathcal M_{1,1}(J_{N \times N};\vec{\beta}_{(1)};\mathcal M_2(J_{N \times N};\vec{\beta}_{(2)};P))
\subset (-1)^{\dim \Delta_N} \partial \mathcal M_{1,1}(J_{N \times N}; \vec{\beta}';P).
$$ 
Hence the orientations of $ \mathcal M_{1,1}(J_{N \times N}; \vec{\beta};P)$ and 
$\mathcal M_{1,1}(J_{N \times N}; \vec{\beta}';P)$ are compatible along 
$\mathcal M_{1,1}(J_{N \times N};\vec{\beta}_{(1)};\mathcal M_2(J_{N \times N};\vec{\beta}_{(2)};P))$.  
\par
Next we consider the remaining case, i.e., a transition of strata involving $\mathcal M_3(J_{N \times N};\beta')$.  
For $\vec{\beta}_1=(\beta_{1,1}, \dots, \beta_{1,k})$, we write $\vec{\beta}_1^-=(\beta_{1,2}, \dots, 
\beta_{1,k})$.  
The moduli spaces 
$\widehat{\mathcal M}(J_{N\times N};\beta';\vec \beta_1,\vec \beta_2,\vec \beta_0;P_1,P_2,P_0)$ 
and 
$\widehat{\mathcal M}(J_{N\times N};\beta' + \beta_{1,1};\vec \beta^-_1,\vec \beta_2,\vec \beta_0;P_1,P_2,P_0)$ 
are adjacent along a stratum of codimension 1 
\begin{equation}
\mathcal M_{1,1}(J_{N \times N};\vec{\beta}_0;P_0)_{ev_1} \times_{ev_0} 
\mathcal M_3(J_{N \times N}; \beta';\mathcal M_2(J_{N \times N};\vec{\beta}_1;P_1), 
\mathcal M_{1,1}(J_{N \times N};\vec{\beta}_2;P_2)). \label{cod1stratum}
\end{equation}
Instead of Proposition 8.10.4 (1), (2), we use Proposition 8.5.1 in \cite{fooobook2} and find that 
the orientations of $\widehat{\mathcal M}(J_{N\times N};\beta';\vec \beta_1,\vec \beta_2,\vec \beta_0;P_1,P_2,P_0)$ and $\widehat{\mathcal M}(J_{N\times N};\beta' + \beta_{1,1};\vec \beta^-_1,\vec \beta_2,\vec \beta_0;P_1,P_2,P_0)$ are compatible along the stratum given in \eqref{cod1stratum}.  
\par
The same argument applies to $\mathcal M_{1,1}(J_{N \times N}; \vec{\beta}_i;P_i)$, $i=0,2$.  
Hence the orientations of the moduli spaces $\widehat{\mathcal M}(J_{N\times N};\beta';\vec \beta_1,\vec \beta_2,\vec \beta_0;P_1,P_2,P_0)$ are compatible with one another and so defines 
$\widehat{\mathcal M}(J_{N\times N};\rho;P_1,P_2,P_0)$.  
\par
Hence we can glue oriented Kuranishi structures on  
$\widehat{\mathcal M}(J_{N\times N};\beta';\vec \beta_1,\vec \beta_2,\vec \beta_0;P_1,P_2,P_0)$ 
to obtain an oriented Kuranishi structure on $\widehat{\mathcal M}(J_{N\times N};\rho;P_1,P_2,P_0)$. 
\end{proof}

The map $\overset{\circ}{\mathfrak I}$ in \eqref{mapI} induces
\begin{equation}\label{Ireg}
{\mathfrak I}^{\text{\rm reg}}: {\mathcal M}_3^{\text{\rm main, reg}}(J_{N \times N}; \rho;P_1, P_2, P_0) \to
{\mathcal M}_3^{\text{\rm sph, reg}}(J_N;\rho; P_1, P_2, P_0),
\end{equation}
where we recall the definition 
$$
{\mathcal M}_3^{\text{\rm sph, reg}}(J_N;\rho ;P_1, P_2, P_0)= 
(-1)^{\deg P_1 \cdot \deg P_2} P_0 \times_{ev_0} \left({\mathcal M}_3^{\text{\rm sph, reg}}(J_N;\rho )_{(ev_1, ev_2)} \times_{N^2}(P_1 \times P_2)\right)
$$
from Definition \ref{notationmoduli}.  
We extend 
$\mathfrak J^{\text{\rm reg}}$ to a map 
\begin{equation}
\label{doublingmap}
\mathfrak I: \widehat{\mathcal M}(J_{N\times N};\rho;P_1,P_2,P_0)
\to
\mathcal M^{\text{\rm sph}}_3(J_N;\rho;\mathcal C;P_1,P_2,P_0)
\end{equation}  
defined on the full moduli space as follows.
Let $(p_i,(S_{i,j},(z_{i,j;0},z_{i,j;\text{\rm int}}),u_{i,j})_{j=1}^{k_i})$ be an element of
$\mathcal M_{1;1}(J_{N\times N};\vec \beta_i;P_i)$. Here
$p_i \in \vert P_i \vert$, $(S_{i,j},(z_{i,j;0},z_{i,j;\text{\rm int}}),u_{i,j}) \in \mathcal M_{1,1}(J_{N\times N};\beta_{i,j})$
such that
$$
f(p_i) = u_{i,1}(z_{i,1;0}), \,\,
u_{i,1}(z_{i,1;\text{\rm int}}) = u_{i,2}(z_{i,2;0}), \,\,
\ldots,\,\,
u_{i,k_i-1}(z_{i,k_i-1;\text{\rm int}}) = u_{i,k_i}(z_{i,k_i;0}),
$$
where $P_i$ is $(\vert P_i\vert,f)$, $\vert P_i\vert$ is a simplex,
and $f : \vert P_i\vert \to N$ is a smooth map.
\par
Suppose that $S_{i,j}$ is a disc component.
Writing $u_{i,j} = (u_{i,j}^+,u_{i,j}^{-})$, we obtain
a map $\widehat u_{i,j} : \Sigma_{i,j} \to N$ with
$\Sigma_{i,j}$ a sphere, the double of $S_{i,j}$.
For bubble trees, we goes as in Remark \ref{bubbletree} (1).
\par
We denote by $C_{i,j}\subset \Sigma_{i,j}$
the circle along which we glued two copies of $S_{i,j}$.
Then $\widehat u_{i,j}$ is defined by gluing  $u_{i,j}^+$ and $u_{i,j}^{-}\circ c$
along $C_{i,j}$ in $\Sigma_{i,j}$ where $c: \Sigma_{i,j} \to \Sigma_{i,j}$
is the conjugation with $C_{i,j}$ as its fixed point set.
\par
We glue $(\Sigma_{i,j},u_{i,j})$ and $(\Sigma_{i,j+1},u_{i,j+1})$ at $z_{i,j;\text{\rm int}}$
and $z_{i,j+1;0}$. Here we identify $z_{i,j;\text{\rm int}} \in S_{i,j}$ as the corresponding point
in $\Sigma_{i,j}$ such that it is in the disc bounding $C_{i,j}$.
We thus obtain a configuration of tree of spheres and system of circles on it,
for each $i=1,2,0$. We glued them with the double of an element
$\mathcal M^{\text{\rm main}}_3(J_{N\times N};\beta')$ in an obvious way. Thus we obtain the map
(\ref{doublingmap}).
(In case some of the sphere component becomes unstable we need to
shrink it. See the proof of Lemma \ref{Inotiso} below.)
\par
It is easy to see that $\mathfrak I$ is surjective.
\begin{lem}\label{Inotiso}
There is a subset ${\mathcal D}({\mathfrak I}) \subset \widehat{\mathcal M}(J_{N\times N};\rho;P_1,P_2,P_0)$ of
codimension $\geq 2$ such that
the map $\mathfrak I$ is an isomorphism outside ${\mathcal D}({\mathfrak I})$.
\end{lem}
\begin{proof}
We can easily check that the map $\mathfrak I$ fails to be an isomorphism only by the
following reason.
Let $((\Sigma^{\text{\rm dis}},(z_1,z_2,z_0)),v)$ be an element of
$\widehat{\mathcal M}(J_{N\times N};\rho;P_1,P_2,P_0)$
and let $\Sigma^{\text{\rm dis}}_i$ be one of its irreducible sphere component.
Suppose that  $\Sigma^{\text{\rm dis}}_i$ is unstable.
(Namely we assume that it has one or two singular points.)
Then its automorphism group $\text{Aut}(\Sigma^{\text{\rm dis}}_i)$ will
have positive dimension by definition of stability.
(We require the elements of $\text{Aut}(\Sigma^{\text{\rm dis}}_i)$
to fix the singular point.) By restricting $v$ to $\Sigma^{\text{\rm dis}}_i$, we
obtain $v_i = (v_i^+,v_i^-)$ where $v_i^{\pm}:\Sigma^{\text{\rm dis}}_i \to N$
are maps from the sphere domain $\Sigma^{\text{\rm dis}}_i$.
On the double (which represents $\mathfrak I((\Sigma^{\text{\rm dis}},(z_1,z_2,z_0)),v)$)
the domain $\Sigma^{\text{\rm dis}}$ contains two sphere components $\Sigma_{i}^{+}$
and $\Sigma_{i}^{-}$ on which the maps $v_i^+$ and $v_i^-$ are defined respectively.
We have two alternatives:
\begin{enumerate}
\item If one of $v_i^+$ and $v_i^-$ is a constant map, then this double itself is not a stable
map. So we shrink the corresponding component $\Sigma_{i}^{+}$ or $\Sigma_{i}^{-}$
to obtain a stable map. (This is actually a part of the construction used in
the definition of $\mathfrak I$.)
\item Suppose both $v_i^+$ and $v_i^-$ are nonconstant and let $g \in \text{Aut}(\Sigma^{\text{\rm dis}}_i)$.
Then the map $v_i^g = (v_i^+, v_i^-\circ g)$ defines
an element of $\widehat{\mathcal M}(J_{N\times N};\rho;P_1,P_2,P_0)$
different from $v_i = (v_i^+,v_i^-)$ but is mapped to the same element under the map $\mathfrak I$.)
\end{enumerate}
We denote by ${\mathfrak D}({\mathfrak I})$ the subset of $\widehat{\mathcal M}(J_{N\times N};\rho;P_1,P_2,P_0)$
consisting of $((\Sigma^{\text{\rm dis}},(z_1,z_2,z_0)),v)$ with at least one unstable sphere component $\Sigma_i^{\text{\rm dis}}$.
This phenomenon occurs only at the stratum of codimension $\ge 2$ because
it occurs only when there exists a sphere bubble.

This finishes the proof.
\end{proof}

\begin{rem}\label{inversionTypeII}

\begin{enumerate}
\item 
We give the orientation on 
$\mathcal M^{\text{\rm sph}}_3(J_N;\rho;\mathcal C;P_1,P_2,P_0)$
as follows: 
We recall from Proposition \ref{regoripres} that the map
$$\overset{\circ}{\mathfrak I}:{\mathcal M}_3^{\text{\rm main, reg}}(J_{N\times N};\rho) \to {\mathcal M}_3^{\text{\rm sph, reg}}(J_N;\rho)$$
is an orientation preserving isomorphism between spaces with oriented Kuranishi structures.  
Taking Definitions \ref{Definition8.4.1} and \ref{notationmoduli} into account, 
we find that the map ${\mathfrak I}^{\text{\rm reg}}$ in \eqref{Ireg} induces an isomorphism from 
$${P_0 \times_{ev_0} \mathcal M}^{\text{\rm main, reg}}_3(J_{N\times N};\rho;P_1,P_2)$$ to 
$\mathcal M^{\text{\rm sph, reg}}_3(J_N;\rho;\mathcal C;P_1,P_2,P_0)$, which is orientation 
preserving if and only if $(-1)^{\deg P_1 (\deg P_2 +1)}=1$.   
Since $P_i = {\mathcal M}_{1,1}(J_{N \times N}; \emptyset; P_i)$, the orientation of the fiber product 
${P_0 \times_{ev_0} \mathcal M}^{\text{\rm main, reg}}_3(J_{N\times N};\rho;P_1,P_2)$ 
is the restriction of the orientation of $\widehat{\mathcal M}(J_{N \times N}; \rho; P_1, P_2, P_0)$.  
Recall that $\overset{\circ}{\mathfrak I}$ extends to ${\mathfrak I}$ in \eqref{doublingmap}, which is 
an isomorphism outside ${\mathcal D}({\mathfrak I})$ of codimension at least $2$ (Lemma \ref{Inotiso}).  
Hence we can use ${\mathfrak I}$ to equip the moduli space 
$\mathcal M^{\text{\rm sph}}_3(J_N;\rho;\mathcal C;P_1,P_2,P_0)$ with an orientation in such a way 
that 
${\mathfrak I}$ is orientation preserving if and only if  
$(-1)^{\deg P_1 (\deg P_2 +1)}=1$.  
\item
For strata of virtual codimension $1$, there are two cases, i.e., Case (A) and Case (B) in Proposition \ref{virdimcombtype} (3).    
We also explained that
each stratum in Case (A) arises in two ways of codimension $1$ boundary of top dimensional strata,
i.e., phenomena (1) and (2), see Figure 5.
Note that there is a canonical identification, i.e., inserting/forgetting the component of case (i) in Definition \ref{unstwocirc},
between those arising from the phenomenon (1) and those arising from the phenomenon (2).
The orientation of each stratum of top dimension in
${\mathcal M}_3^{\text{\rm sph}}(J_N;\rho;{\mathcal C};P_1, P_2, P_0)$ is defined using the orientation
of $\widehat{\mathcal M}(J_{N\times N};\rho;P_1, P_2, P_0)$ as we just mentioned in Remark \ref{inversionTypeII} (1).
Since the moduli space $\widehat{\mathcal M}(J_{N\times N};\rho;P_1,P_2,P_0)$ is oriented,
these two orientations are opposite under the above identification to give an orientation on the glued space
with Kuranishi structure.  As for the cancellation in Case (B), see Subsection \ref{1.9(1)}, Lemma \ref{forgettypeII} and the following items (3) and (4).
\item We consider the involution $\tau_*$
applied to one of the disc component $S$ of the fiber product factors
appearing in a stratum of
$\widehat{\mathcal M}(J_{N\times N};\rho;P_1,P_2,P_0)$,  such that
the double of $S$ is Type II.
Then the orientation of the circle $C=\del S$ embedded in the domain $\Sigma$ of the corresponding
sphere component is inverted under the operation on
$\mathcal M^{\text{\rm sph}}_3(J_N;\rho;\mathcal C;P_1,P_2,P_0)$ induced by
$\tau_*$ under the map $\mathfrak I$.
\item
The moduli space $\mathcal M^{\text{\rm sph}}_3(J_N;\rho;\mathcal C;P_1,P_2,P_0)$ is stratified accroding to
their combinatorial types $\mathfrak c$, see Definition \ref{combtype}.
When $\mathfrak c$ is fixed, there are finitely many Type II components.  There are involutions  acting on
these components by the reflection with respect to the circle of Type II.
Namely, the involution reverses the orientation of the circle of Type II.
We call a Kuranishi structure, resp. a multi-section, on $\mathcal M^{\text{\rm sph}}_3(J_N;\rho;\mathcal C;P_1,P_2,P_0)$
{\it invariant under the inversion of the orientaion of circles of Type II}, if the Kuranishi structure, resp. the multi-section, restricted to each stratum
corresponding to ${\mathfrak c}$ is invariant under the involution acting on each Type II component in $\mathfrak c$.
Note that these involutions are defined on the corresponding strata, not on the whole moduli space
$\mathcal M^{\text{\rm sph}}_3(J_N;\rho;\mathcal C;P_1,P_2,P_0)$.
%
\end{enumerate}
\end{rem}

\begin{lem}\label{frak s_2}
The moduli space $\mathcal M^{\text{\rm sph}}_3(J_N;\rho;\mathcal C;P_1,P_2,P_0)$
carries a Kuranishi structure ${\mathfrak K}_2$ invariant under the inversion of the orientation of circles of Type II.
The Kuranishi structure can be canonically pull-backed to a Kuranishi structure on  the space $\widehat{\mathcal M}(J_{N\times N};\rho;P_1,P_2,P_0)$.
Moreover, there is a multisection ${\mathfrak s}_2$ of the Kuranishi structure on $\mathcal M^{\text{\rm sph}}_3(J_N;\rho;\mathcal C;P_1,P_2,P_0)$ with the
following properties.
\begin{enumerate}
\item The multisection $\mathfrak s_2$ is transversal to the zero section.
\item The multisection $\mathfrak s_2$ is invariant under the inversion of the orientation of the circles of Type II.
\item The multisection $\mathfrak s_2$ does not vanish on ${\mathcal D}({\mathfrak I})$.
\end{enumerate}
\end{lem}
\begin{proof}
Lemma \ref{frak s_2} is clear from construction except the following points.

Firstly, we consider the point of the moduli space $\mathcal M^{\text{\rm sph}}_3(J_N;\rho;\mathcal C;P_1,P_2,P_0)$ such that
one of the following two conditions is satisfied.
\begin{enumerate}
\item A circle in Type II component $\Sigma_a$ hits the singular point of $\Sigma_a$
other than the root thereof.
\item A circle in Type I-2 component $\Sigma_a$ hits the
singular point other than its outward interior special point.
\end{enumerate}
We have to glue various different strata meeting at such a point in 
$\mathcal M^{\text{\rm sph}}_3(J_N;\rho;\mathcal C;P_1,P_2,P_0)$.  
We have already given such a construction during the proof of Theorem \ref{KuraonMsph}. 
By examining the way how the corresponding strata 
are glued in $\widehat{\mathcal M}(J_{N\times N};\rho;P_1,P_2,P_0)$, 
the gluing of corresponding strata of $\mathcal M^{\text{\rm sph}}_3(J_N;\rho;\mathcal C;P_1,P_2,P_0)$ 
are performed in the same way. 
We like to note that the phenomenon spelled out in the proof of Lemma \ref{Inotiso}
concerns {\it sphere} bubbles of the elements of $\widehat{\mathcal M}(J_{N\times N};\rho;P_1,P_2,P_0)$,
while the phenomenon we concern here arises from  {\it disc} bubbles.
Therefore they do not interfere with each other.
\par
Secondly, we need to make the choice of the
obstruction bundle of the Kuranishi structure of
$\mathcal M^{\text{\rm sph}}_3(J_N;\rho;\mathcal C;P_1,P_2,P_0)$
in such a way that  it is compatible with one of
$\widehat{\mathcal M}(J_{N\times N};\rho;P_1,P_2,P_0)$.
Lemma \ref{Inotiso} describes the locus ${\mathcal D}({\mathfrak I})$ where the map $\mathfrak I$ fails to be an isomorphism.
Let $\nu_i=(\nu_i^+,\nu_i^-)$ be the sphere bubble as in the proof of Lemma \ref{Inotiso}.
By ${\mathfrak I}$, $\nu_i^+$ (resp. $\nu_i^-$) corresponds to a sphere bubble attached to a pseudoholomorphic sphere at a point
in the lower hemisphere (resp. the upper hemisphere).
In the construction of a Kuranishi structure on the moduli space of holomorphic spheres (or stable maps
of genus $0$), we take obstruction bundles in order to construct Kuranishi neighborhoods.
Let $E(\nu_i^{\pm})$ be a finite dimensional subspace in $\Omega^{0,1}((\nu_i^{\pm})^*TN)$ such that
the linearization operator of the holomorphic curve equation at $\nu_i^{\pm}$ becomes surjective modulo $E(\nu_i^{\pm})$.
In order to extend $E(\nu_i^{\pm})$ to a neighborhood of $\nu_i^{\pm}$, we used {\it obstruction bundle data}, introduced in \cite{foootech} Definition 17.7,
in particular, additional marked points $w_{i,j}^{\pm}$ and local transversals ${\mathcal D}_{i,j}^{\pm}$ to the image of $\nu_i^{\pm}$ at  $\nu_i^{\pm}(w_{i,j}^{\pm})$.
For $\nu_i:{\mathbb C}P^1 \to N \times N$, we regard $E(\nu_i^{+})$ (resp. $E(\nu_i^{-})$) as a subspace of $\Omega^{0,1}((\nu_i)^*(TN \oplus 0)) \cong
\Omega^{0,1}((\nu_i^+)^*TN \oplus 0)$ (resp. $\Omega^{0,1}((\nu_i)^*(0 \oplus TN)) \cong \Omega^{0,1}(0\oplus (\nu_i^-)^*TN)$).   
Note that the linearization operator of the pseudoholomorphic curve equation at $\nu_i$ is surjective
modulo $E(\nu_i^+) \oplus E(\nu_i^-) \subset \Omega^{0,1}((\nu_i)^*(T (N \times N))$.
When we extend $E(\nu_i^{+})$ (resp. $E(\nu_i^{-})$)  to a neighborhood of $\nu_i$, we use $w_{i,j}^+$ and ${\mathcal D}_{i,j}^+ \times N$
(resp. $w_{i,j}^-$ and $N \times {\mathcal D}_{i,j}^-$).
Namely, we use the data $w_{i,j}^+$ and ${\mathcal D}_{i,j}^+$ and the data $w_{i,j}^-$ and ${\mathcal D}_{i,j}^-$ separately, not simultaneously.

The Kuranishi structure can be taken invariant under stratawise involutions, since the Kuranishi structure is constructed
by induction on the energy and we can keep the finite symmetries as in the explanation in Remark \ref{remonmult}.

The existence of a multisection in the statement follows from general theory of Kuranishi structures once the following point is taken into account.
By Lemma \ref{Inotiso}, there is a subset ${\mathcal D}({\mathfrak I})$ of codimension at least $2$ such that $\mathfrak I$ is an isomorphism outside
${\mathcal D}({\mathfrak I})$.
Since the expected dimension of $\mathcal M^{\text{\rm sph}}_3(J_N;\rho;\mathcal C;P_1,P_2,P_0)$ is $0$, $\mathfrak s_2$ can be chosen such that
$\mathfrak s_2$ does not vanish on ${\mathcal D}({\mathfrak I})$.
\end{proof}

\begin{rem}\label{reasonforC}
Generally, note that we can pull back a Kuranishi structure ${\mathfrak K}$ on 
$\mathcal M^{\text{\rm sph}}_3(J_N;\rho;\mathcal C;P_1,P_2,P_0)$ 
to a Kuranishi structure on  the space 
$\widehat{\mathcal M}(J_{N\times N};\rho;P_1,P_2,P_0)$ in a canonical way, 
if the next condition (*) is satisfied.
By construction of the Kuranishi structure 
we take a sufficiently dense finite subset $\frak P \subset 
\mathcal M^{\text{\rm sph}}_3(J_N;\rho;\mathcal C;P_1,P_2,P_0)$
and for each $\frak x  \in \frak P$ we take a finite 
dimensional subspace $E_0(\frak x)$ of 
$
\Omega^{0,1}(\Sigma_{\mathfrak x}, u_{\mathfrak x}^*TN)=
C^{\infty}(\Sigma_{\frak x}, \Lambda^{0,1} \otimes u_{\mathfrak x}^*TN)
$,
where $(\Sigma_{\frak x}, \vec{z}_{\frak x}^+, u_{\frak x}: \Sigma_{\frak x} \to N)$
is a stable map appearing in $\frak x$.
The subspace $E_0(\frak x)$ consists of smooth sections of compact 
support away from nodes.
Moreover the union of $E_0(\frak x)$ and the image of linearized 
operator of the Cauchy-Riemann equation spans the 
space $\Omega^{0,1}(\Sigma_{\mathfrak x}, u_{\mathfrak x}^*TN)$.
(See section 12 in \cite{FO}.)
Now we require
\begin{enumerate}
\item[(*)]
The support of any element of $E_0(\frak x)$ does not intersect
with the circles consisting $\frak x 
\in \mathcal M^{\text{\rm sph}}_3(J_N;\rho;\mathcal C;P_1,P_2,P_0)$.
\end{enumerate}
We show that the condition (*) implies that the Kuranishi structure $\mathfrak K$ can be pulled back 
to  $\widehat{\mathcal M}(J_{N\times N};\rho;P_1, P_2, P_0)$ below.  
\par
We recall the construction of Kuranishi structure on 
$\mathcal M^{\text{\rm sph}}_3(J_N;\rho;\mathcal C;P_1,P_2,P_0)$ in Theorem \ref{KuraonMsph}
a bit more.
For each $\frak x 
\in \frak P$ we take a sufficiently small closed neighborhood $U(\frak x)$ of
$\frak x$ in $\mathcal M^{\text{\rm sph}}_3(J_N;\rho;\mathcal C;P_1,P_2,P_0)$.  
Let $\frak y \in \mathcal M^{\text{\rm sph}}_3(J_N;\rho;\mathcal C;P_1,P_2,P_0)$. We consider 
$\frak P(\frak y) 
= \{ \frak x \in \frak P \mid \frak y \in U(\frak x)\}$.  
Using the complex linear part of parallel transport  along minimal geodesics 
as in Definition 17.15, Lemma 18.6, Definition 18.7 in \cite{foootech}, 
we transform a subspace $E_0(\frak x)$ with $\frak x \in \frak P(\frak y)$
to a subspace of $\Omega^{0,1} (\Sigma_{\mathfrak y}, u_{\mathfrak y}^*TN)$,  
where $(\Sigma_{\frak y}, \vec{z}_{\frak y}^+, u_{\frak y}: \Sigma_{\frak y} \to N)$ 
is a stable map appearing in $\frak y$.
We fix various data such as obstruction bundle data on $\frak x$ for our construction. (See \cite[Definition 17.7]{foootech}.)
We define 
$E(\frak y) \subset \Omega^{0,1} (\Sigma_{\mathfrak y}, u_{\mathfrak y}^*TN)$
as the sum of those subspaces for various $\frak x \in \frak P(\frak y)$.
(We remark that this sum can be taken to be a direct sum \cite[Lemma 18.8]{foootech}.)    

\par 
By taking $U(\frak x)$ small we may and will require 
that the supports of elements of $E(\frak y)$ are disjoint from the 
circles consisting $\frak y$.  
\par
Now let $\tilde{\frak y} \in \widehat{\mathcal M}(J_{N\times N};\rho;P_1,P_2,P_0)$ with 
${\mathfrak I}(\tilde{\frak y}) = {\frak y}$.
Using the fact that 
the supports of elements of $E(\frak y)$ are disjoint from the 
circles consisting $\frak y$ we can lift  $E(\frak y)$
to a subspace  $E(\tilde{\frak y})$ of 
$\Omega^{0,1}(\Sigma_{\tilde{\mathfrak y}}, u_{\tilde{\frak y}}^{ *}TN)$, where 
$(\Sigma_{\tilde{\frak y}}, \vec{z}_{\tilde{\frak y}}^+, u_{\tilde{\frak y}}: \Sigma_{\tilde{\frak y}} \to N)$
is a stable map appearing in $\tilde{\frak y}$.
We use $E(\tilde{\frak y})$ as the obstruction bundle to define the 
lift of our Kuranishi structure.  
See also Remark \ref{reasonforC2}.   
\end{rem}

Let $Z$ be a compact metrizable space and
$\mathfrak K_{0}^Z$, $\mathfrak K_{1}^Z$ its Kuranishi structures with orientation.
Let $\mathfrak s_0^Z$ and $\mathfrak s_1^Z$ be multisections of the Kuranishi structures
$\mathfrak K_{0}^Z$ and $\mathfrak K_1^Z$, respectively.
We say that $(\mathfrak K_{0}^Z,\mathfrak s_0^Z)$ is {\it homotopic} to
$(\mathfrak K_{1}^Z,\mathfrak s_1^Z)$  if there exists an
oriented Kuranishi structure ${\mathfrak K}^{Z \times [0,1]}$ on $Z\times [0,1]$ and its multisection
$\mathfrak s^{Z \times [0,1]}$ which restricts to $(\mathfrak K_{0}^Z,\mathfrak s_0^Z)$ and
$(\mathfrak K_{1}^Z,\mathfrak s_1^Z)$ at $Z\times \{0\}$ and $Z\times \{1\}$,
respectively.  We call such $({\mathfrak K}^{Z \times [0,1]}, \mathfrak s^{Z \times [0,1]})$ a homotopy between
$(\mathfrak K_{0}^Z,\mathfrak s_0^Z)$ and $(\mathfrak K_{1}^Z,\mathfrak s_1^Z)$.
\par

The moduli space $\widehat{\mathcal M}(J_{N\times N};\rho;P_1,P_2,P_0)$ is stratified according to combinatorial types.
For ${\mathbf u} \in  \widehat{\mathcal M}(J_{N\times N};\rho;P_1,P_2,P_0)$, we decompose the domain of ${\mathbf u}$
into disc components and sphere components.
An {\it extended disc component} of $\mathbf u$ is, by definition, the union of a disc component $D_a({\mathbf u})$ and all trees of spheres rooted on
$D_a({\mathbf u})$.  We denote it by $\widehat{D}_a({\mathbf u})$.
An extended disc component is said to be of Type I (resp. Type II), if the corresponding component of 
${\mathfrak I}({\mathbf u})$, i.e., the double of $D_a({\mathbf u})$, 
is of Type I (resp. Type II).
The involution $\tau_*$ acts on each extended disc component.
In particular, $\tau_*$ acting on an extended disc component $\widehat{D}_a({\mathbf u})$ of Type II is compatible with the inversion of the orientation of the circle
on the component of ${\mathfrak I}({\mathfrak u})$ of Type II, which is the double of the disc component $D_a({\mathbf u})$, see
Remark \ref{inversionTypeII}.

\begin{lem}\label{htpy}
For the pull-back ${\mathfrak I}^*({\mathfrak K}_2,{\mathfrak s}_2)$ and $({\mathfrak K}_1,{\mathfrak s}_1)$,
there is a homotopy $({\mathbf K}, {\mathbf s})$  between them such that it is invariant under $\tau_*$ on each extended disc component  of Type II
acting on the first factor of $\widehat{\mathcal M}(J_{N\times N};\rho;P_1,P_2,P_0) \times [0,1]$.
\end{lem}
\begin{proof}
By Proposition \ref{frak s_2},
the moduli space $\widehat{\mathcal M}(J_{N\times N};\rho;P_1,P_2,P_0)$ has the pair ${\mathfrak I}^*({\mathfrak K}_2,{\mathfrak s}_2)$
of Kuranishi structure and multisection, which are invariant under the inversion of the orientation of circles.
We also have another such a pair $({\mathfrak K}_1,{\mathfrak s}_1)$.
Then the standard theory of Kuranishi structure shows the existence of the desired homotopy.
\end{proof}
\begin{lem}\label{calIcoincidence}
We have
$$
\sum_{\beta} n_{\beta} =
\# \left(\mathcal M^{\text{\rm sph}}_3(J_N;\rho;\mathcal C;P_1,P_2,P_0)\right)^{\mathfrak s_2}.
$$
Here the sum is taken over the class $\beta \in \Pi(\Delta_N) =\pi_2(N\times N, \Delta_N) /\sim$ whose double belongs to
class $\rho \in \pi_2(N)/\sim$ and the virtual dimension of 
$\mathcal M^{\text{\rm sph}}_3(J_N;\rho;\mathcal C;P_1,P_2,P_0)$ is zero.
\end{lem}
\begin{proof}
If the moduli space  $\widehat{\mathcal M}(J_{N\times N};\rho;P_1,P_2,P_0) \times [0,1]$ had no codimension $1$
boundary in the sense of Kuranishi structure,
the existence of the homotopy $({\mathbf K}, {\mathbf s})$ in Lemma \ref{htpy} would immediately imply the conclusion.
But in reality, there exists a codimension 1 boundary.
However the codimension $1$ boundary of $\widehat{\mathcal M}(J_{N\times N};\rho;P_1,P_2,P_0) \times [0,1]$
consists of elements with at least one component of
Type II.
Since ${\mathbf s}$ is invariant under the action $\tau_*$ on the disc component of Type II, the contribution from the boundary  cancels
as in the proof of unobstructedness of the diagonal in 
$(N\times N, - {\rm pr}_1^*\omega + {\rm pr}_2^* \omega)$ in Subsection \ref{1.9(1)}.
Hence the proof.
\end{proof}
\begin{rem}
Even though the dimension of a space $Z$ with Kuranishi structure is $0$, 
codimension $1$ boundary can be non-empty.  This is because the dimension 
of a space with Kuranishi structure is {\it virtual} dimension.  After taking a suitable 
multi-valued perturbation, its zero set does not meet codimension $1$ boundary.  
\end{rem}
We now consider the forgetful map:
\begin{eqnarray}
\mathfrak F :
\mathcal M^{\text{\rm sph}}_3(J_N;\rho;\mathcal C;P_1,P_2,P_0)
&\to& \mathcal M^{\text{\rm sph}}_3(J_N;\rho;P_1,P_2,P_0),
\end{eqnarray}
which is defined by forgetting all the circles in the admissible system of circles.
We recall from Lemmas \ref{moduliqp} and \ref{frak s_2} that both moduli spaces
$\mathcal M^{\text{\rm sph}}_3(J_N;\rho;P_1,P_2,P_0)$ and
$\mathcal M^{\text{\rm sph}}_3(J_N;\rho;\mathcal C;P_1,P_2,P_0)$
carry Kuranishi structures.
We have also used multisections on
$\mathcal M^{\text{\rm sph}}_3(J_N;\rho;P_1,P_2,P_0)$ and
on $\mathcal M^{\text{\rm sph}}_3(J_N;\rho;\mathcal C;P_1,P_2,P_0)$,
denoted by $\mathfrak s_0$ and $\mathfrak s_2$ respectively.
\begin{lem}\label{forgettypeII}
Let  $(\mathfrak K_{2},\mathfrak s_2)$  be a pair of Kuranishi structures and multisections on $\mathcal M^{\text{\rm sph}}_3(J_N;\rho;\mathcal C;P_1,P_2,P_0)$
in Lemma \ref{frak s_2}.
Then $(\mathfrak K_{2},\mathfrak s_2)$ is homotopic to the pull-back $\mathfrak F^*(\mathfrak K_{0},\mathfrak s_0)$.
Moreover, there is a homotopy between them, which is invariant under the inversion of the orientation of the Type II circles.
In particular we have
$$
\# \left(\mathcal M^{\text{\rm sph}}_3(J_N;\rho;\mathcal C;P_1,P_2,P_0)\right)^{\mathfrak s_2}
=
\# \left(\mathcal M^{\text{\rm sph}}_3(J_N;\rho;\;P_1,P_2,P_0)\right)^{\mathfrak s_0},
$$
if the virtual dimensions of the moduli spaces of the both hand sides are zero.
\end{lem}
\begin{proof}
The existence of a homotopy between Kuranishi structures ${\mathfrak K}_2$ and ${\mathfrak F}^*{\mathfrak K}_0$
is again a consequence of general theory, once the following point is taken into account.
Note that ${\mathfrak F}^*{\mathfrak s}_0$ and ${\mathfrak s}_2$ are invariant under the inversion of the orientation of circles of Type II.
Then we can take a homotopy, which is also invariant under the inversion of the orientation of circles of Type II.

We note that there are several components of
$\mathcal M^{\text{\rm sph}}_3(J_N;\rho;\mathcal C;P_1,P_2,P_0)$
which are of codimension $0$ or $1$ and contracted
by $\mathfrak F$.
Note that we took $\mathfrak s_0$ in such a way that its zero set does not contain elements with domains of at least two irreducible components,
see the paragraph right after Proposition \ref{regoripres}.
Hence the zeros of ${\mathfrak F}^*\mathfrak s_0$ is contained in the subset where ${\mathfrak F}$ gives an isomorphism and we can count them
with signed weights to obtain a rational number. 
Namely, we have
$$
\# \left(\mathcal M^{\text{\rm sph}}_3(J_N;\rho;\mathcal C;P_1,P_2,P_0)\right)^{{\mathfrak F}^*\mathfrak s_0}
=
\# \left(\mathcal M^{\text{\rm sph}}_3(J_N;\rho;\;P_1,P_2,P_0)\right)^{\mathfrak s_0}.
$$
If the moduli space $\mathcal M^{\text{\rm sph}}_3(J_N;\rho;\mathcal C;P_1,P_2,P_0)$ had no codimension 1 boundary in the sense of 
Kuranishi structure, the existence of a homotopy between $({\mathfrak K}_2, {\mathfrak s}_2)$ and ${\mathfrak F}^*({\mathfrak K}_0, {\mathfrak s}_0)$
would immediately imply that
$$
\# \left(\mathcal M^{\text{\rm sph}}_3(J_N;\rho;\mathcal C;P_1,P_2,P_0)\right)^{\mathfrak s_2}
=
\# \left(\mathcal M^{\text{\rm sph}}_3(J_N;\rho;\mathcal C;P_1,P_2,P_0)\right)^{{\mathfrak F}^*\mathfrak s_0},
$$
which would complete the proof.
However, $\mathcal M^{\text{\rm sph}}_3(J_N;\rho;\mathcal C;P_1,P_2,P_0)$ has a codimension $1$ boundary, which consists of
stable maps with admissible systems of circles containing at least one circle of Type II.
All the contribution from those components
cancel out by the involution, which inverts the orientation
of the circles of Type II. (This is a geometric way
to see the vanishing of $\mathfrak m_0(1)$ in the chain level.
We have already checked that it occurs {\it with sign} in Subsection 6.3.1.
See also Remark \ref{inversionTypeII} (3), (4).)
Hence the lemma.
\end{proof}
\begin{rem}\label{reasonforC2}
The pull-back Kuranishi structure $\mathfrak F^*(\mathfrak K_{0},\mathfrak s_0)$ does not satisfy the condition (*) 
appearing in Remark \ref{reasonforC}.
In fact the obstruction bundle of $\mathfrak F^*(\mathfrak K_{0},\mathfrak s_0)$ is 
independent of the position of the circles.
Therefore we may not pull back $\mathfrak F^*(\mathfrak K_{0},\mathfrak s_0)$ to a Kuranishi structure on  
$\widehat{\mathcal M}(J_{N\times N};\rho;P_1,P_2,P_0)$, while we can pull back the Kuranishi structure $\mathfrak K_{2}$.
\end{rem}
By Lemmas \ref{moduliqp}, \ref{multiondisc},  \ref{calIcoincidence} and  \ref{forgettypeII},
the proof of Theorem \ref{Proposition34.25} (2) is now complete.
\qed

\begin{proof}[Proof of Corollary \ref{qMassey}]
Viewing $N$ as a closed relatively spin Lagrangian submanifold
of $(N\times N, -{\rm pr}_1^* \omega_N + {\rm pr}_2^* \omega_N)$,
we can construct a filtered $A_{\infty}$ structure
on $H(N;\Lambda_{0,{\rm nov}}^{\Q})$ which is homotopy equivalent
to the filtered $A_{\infty}$ algebra given by Theorem \ref{thm:Ainfty}.
This is a consequence of Theorem W \cite{fooobook1}.
See also Theorem A \cite{fooobook1}.
Then (1) and (2) follow from
Theorem \ref{Proposition34.25}.
The assertion (3) follows from Theorem X \cite{fooobook1}.
\end{proof}

\subsection{Calculation of Floer cohomology of $\R P^{2n+1}$}
\label{subsec:Appl2}

In this subsection, we apply the results proved in the previous
sections to calculate Floer cohomology of real projective space of odd dimension.
Since the case $\R P^1 \subset \C P^1$ is already discussed in Subsection 3.7.6 \cite{fooobook1},
we consider $\R P^{2n+1}$ for $n>0$.
We note that $\R P^{2n+1} \subset \C P^{2n+1}$ is monotone with
minimal Maslov index $2n+2 > 2$ if $n>0$.
Therefore by \cite{Oh93} and Section 2.4 \cite{fooobook1} Floer cohomology
over $\Lambda_{0,{\rm nov}}^{\Z}$ is defined.
In this case we do not need to use the notion of Kuranishi structure and the technique of the virtual fundamental chain.
From the proof of Corollary \ref{Corollary34.22},
we can take $0$ as a bounding cochain.
Hereafter we omit the bounding
cochain $0$ from the notation.
By \cite{Oh96} and Theorem D in \cite{fooobook1},
we have a spectral sequence converging to the Floer cohomology.
Strictly speaking,
in \cite{Oh96} the spectral sequence is constructed
over $\Z_2$ coefficients. However, we can generalize his results
to ones over $\Lambda_{0,{\rm nov}}^{\Z}$ coefficients in a straightforward way,
as long as we take the
orientation problem, which is a new and crucial point of this calculation, into account.
Thus Oh's spectral sequence over $\Lambda_{0,{\rm nov}}^{\Z}$
is enough for our calculation of this example.
(See Chapters 8 and 6 in \cite{fooo06} for a
spectral sequence over $\Lambda_{0,{\rm nov}}^{\Z}$ in more
general setting.)
\par
We use a relative spin structure in Proposition \ref{Proposition44.19} when $n$ is even and
a spin structure when $n$ is odd. We already check that $\R P^{2n+1} \subset \C P^{2n+1}$
has two inequivalent relative spin structures. The next theorem applies
to both of them.
\begin{thm}\label{Theorem44.24}
Let $n$ be any positive integer. Then the spectral sequence calculating $HF(\R P^{2n+1},\R P^{2n+1};\Lambda_{0,{\rm nov}}^{\Z})$
has unique nonzero differential
$$
d^{2n+1} : H^{2n+1}(\R P^{2n+1};\Z) \cong \Z
\longrightarrow  H^{0}(\R P^{2n+1};\Z) \cong \Z
$$
which is multiplication by $\pm 2$. In particular, we have
$$
HF(\R P^{2n+1},\R P^{2n+1};\Lambda_{0,{\rm nov}}^{\Z})
\cong (\Lambda_{0,{\rm nov}}^{\Z}/2\Lambda_{0,{\rm nov}}^{\Z})^{\oplus(n+1)}.
$$
\end{thm}
\begin{rem}\label{Remark44.25}
(1)
Floer cohomology of $\R P^m$ over $\Z_2$ is calculated in
\cite{Oh93} and
is isomorphic to the ordinary cohomology.
This fact also follows from Theorem 34.16
in \cite{fooo06}, which implies that
Floer cohomology of $\R P^m$ over $\Lambda_{0,{\rm nov}}^{\Z_2}$
is isomorphic to the ordinary cohomology
over $\Lambda_{0,{\rm nov}}^{\Z_2}$.
\par
(2) Theorem \ref{Theorem44.24} gives an example where Floer cohomology
of the real point set is different from its ordinary cohomology.
Therefore it is necessary to use $\Z_2$ coefficient to study the
Arnold-Givental conjecture (see Chapter 8 \cite{fooo06}).
\par
\end{rem}
\begin{proof}[Proof of Theorem \ref{Theorem44.24}] If $n\ge 1$,
the set $\pi_2(\C P^{2n+1},\R P^{2n+1})$ has exactly one element $B_1$
satisfying $\mu_{\R P^{2n+1}}(B_1) = 2n+2$,
which is the minimal Maslov number of
$\R P^{2n+1}$.
By the monotonicity of $\R P^{2n+1} \subset \C P^{2n+1}$,
a degree counting argument shows that
only $\CM_2(J;B_1)$, among the moduli spaces
$\CM_2(J;B), \, B \in \pi_2(\C P^{2n+1},\R P^{2n+1})$,
contributes to the differential of the spectral sequence.
First of all, we note that $\tau$ induces an isomorphism
modulo orientations
\begin{equation}\label{44.27}
\tau_* : \CM_2(J;B_1) \longrightarrow  \CM_2(J;B_1).
\end{equation}
Later we examine whether $\tau_{\ast}$ preserves the orientation or not, after we specify relative spin structures.

Since $\omega [B_1]$ is the smallest positive
symplectic area,
$\CM_2(J;B_1)$ has codimension $1$ boundary corresponding to the strata
consisting of elements which have a constant disc component with two marked points and a disc bubble.
However when we consider the evaluation map
$ev=(ev_0, ev_1): \CM_2(J;B_1) \to \R P^{2n+1} \times \R P^{2n+1}$,
these strata are mapped to the diagonal set whose codimension is bigger than $2$.
Thus we can define
fundamental cycle over $\Z$ of $ev(\CM_2(J;B_1))$ which we denote by
$[ev(\CM_2(J;B_1))]$.
We also note
$$
\dim \CM_2(J;B_1) = 2n +2 + 2n + 1 + 2 - 3 = 2 \dim \R P^{2n+1}.
$$
\begin{lem}\label{Lemma44.28} Consider the evaluation map
$ev: \CM_2(J;B_1) \to \R P^{2n+1} \times \R P^{2n+1}$. Then we have
$$
[ev(\CM_2(J;B_1))] =  \pm 2[\R P^{2n+1} \times \R P^{2n+1}]
$$
where $[\R P^{2n+1} \times \R P^{2n+1}]$ is the fundamental cycle of $\R P^{2n+1} \times \R P^{2n+1}$.
\end{lem}
\begin{proof}
For any distinct two points $p,q \in \C P^{2n+1}$ there exists
a holomorphic map $w : S^2=\C \cup \{\infty \} \to \C P^{2n+1}$
of degree $1$ such that $w(0) = p$, $w(\infty) = q$, which is unique up to the action of $\C \setminus \{0\} \cong \text{\rm Aut}(\C P^1;0,\infty)$.
In case $p,q \in \R P^{2n+1}$ the uniqueness implies that
$w(\overline z) = \tau w(c z)$ for some $c \in \C \setminus \{0\}$.
Using this equality twice, we have $w(z) = w(\vert c \vert^2 z)$.   In particular, we find that $\vert c \vert =1$.
Let $a$ be a square root of $c$ and
set $w'(z) = w (a z)$.  Since $\overline{a} c/a = 1$,  we obtain $w' ({\overline z}) = \tau w' (z)$.
(Note that $w$ and $w'$ define the same element in the moduli space $\CM^{\text{\rm sph}}_2(J; [{\mathbb C}P^1])$.)
Thus the restriction of $w$ to the upper or lower half plane defines  elements $w_u$ or $w_l \in
\CM_2(J;B_1)$. Namely there exist $w_u, w_l \in \CM_2(J;B_1)$
such that $ev(w_u) = ev(w_l) = (p,q)$.
Conversely, any such elements determine a degree one curve by
the reflection principle.
\par
To complete the proof of Lemma \ref{Lemma44.28} we have to show that
the orientations of the evaluation map $ev$ at $w_u$ and $w_l$ coincide.
Note that
$\tau_{\ast}(w_u)=w_l$ and $\tau_{\ast}\circ ev=ev$.
Thus it suffices to show that $\tau_{\ast}$ in \eqref{44.27} preserves the orientation.
First, we consider the case of $\R P^{4n+3}$.
In this case, $\R P^{4n+3}$ is $\tau$-relatively spin. Therefore, by Theorem \ref{Proposition38.11},
the map (\ref{44.27}) is orientation preserving,
because
$$
\frac{1}{2}\mu_{\R P^{4n+3}}(B_1) + 2 = 2n+4
$$
is even. We next consider the case of $\R P^{4n+1}$. We pick its relative spin structure
$[(V,\sigma)]$. By Theorem \ref{Proposition38.11} again, the map
$$
\tau_* : \CM_2(J;B_1)^{\tau^*[(V,\sigma)]} \longrightarrow \CM_2(J;B_1)^{[(V,\sigma)]}
$$
is orientation reversing, because
$$
\frac{1}{2}\mu_{\R P^{4n+1}}(B_1) + 2 = 2n+3
$$
is odd. On the other hand, by Proposition \ref{Proposition44.19}
we have $\tau^*[(V,\sigma)] \ne [(V,\sigma)]$.
Let $\mathfrak x$ be the
unique nonzero element
of $H^2(\C P^{4n+1},\R P^{4n+1};\Z_2) \cong \Z_2$.
It is easy to see that
$\mathfrak x [B_1] \ne 0$. Then by Proposition \ref{Proposition44.16} the
identity induces an {\it orientation reversing} isomorphisms
$$
\CM_2(J;B_1)^{\tau^*[(V,\sigma)]} \longrightarrow \CM_2(J;B_1)^{[(V,\sigma)]}.
$$
Therefore we can find that
$$
\tau_* : \CM_2(J;B_1)^{[(V,\sigma)]} \longrightarrow \CM_2(J;B_1)^{[(V,\sigma)]}
$$
is orientation preserving.
This completes the proof of Lemma \ref{Lemma44.28}.
\end{proof}
Then
Lemma \ref{Lemma44.28} and the definition of
the differential $d$ imply
$$
d^{2n+1}(PD([p]))
= [ev_{0}\left( \CM_2(J;B_1) {}_{ev_1}\times [p] \right)]
= \pm 2PD[\R P^{2n+1}],
$$
which finishes the proof of Theorem \ref{Theorem44.24}.
\end{proof}

\begin{rem}
In Subsection 3.6.3 \cite{fooobook1}, we introduced
the notion of {\it weak unobstructedness} and {\it weak bounding cochains} using the homotopy unit of the filtered
$A_{\infty}$ algebra.
We denote by $\CM _{\rm weak} (L;\Lambda_{0,{\rm nov}})$ the
set of all weak bounding cochains.
We also defined the potential
function $\mathfrak {PO}: \CM _{\rm weak} (L;\Lambda_{0,{\rm nov}})
\to \Lambda_{0,{\rm nov}}^{+(0)}$,
where $\Lambda_{0,{\rm nov}}^{+(0)}$ is the degree zero part
of $\Lambda_{0,{\rm nov}}^+$.
Then the set of bounding cochains $\CM (L;\Lambda_{0,{\rm nov}})$
is characterized by $\CM (L;\Lambda_{0,{\rm nov}})=
\mathfrak {PO}^{-1}(0)$.
About the value of the potential function, we have the
following problem:

\begin{prob}\label{Problem44.26}
Let $L$ be a relatively spin Lagrangian submanifold
of a symplectic manifold $M$.
We assume that $L$ is weakly unobstructed
and that the Floer cohomology
$HF((L,b),(L,b);\Lambda_{0,{\rm nov}}^F)$
deformed by $b \in \CM_{\operatorname{weak}}(L)$
does not vanish for some field $F$.
In this situation, the question is whether $\mathfrak{PO}(b)$ is an eigenvalue of the operation
$$
c \mapsto c \cup_Q c_1(M)  : QH(M;\Lambda_{0,{\rm nov}}^F) \longrightarrow QH(M;\Lambda_{0,{\rm nov}}^F).
$$
Here $(QH(M;\Lambda_{0,{\rm nov}}^F), \cup_Q )$
is the quantum cohomology ring of $M$ over
$\Lambda_{0,{\rm nov}}^F$.
\end{prob}
Such statement was made by M. Kontsevich in 2006 during a conference
of homological mirror symmetry at Vienna. (According to some physicists this
had been known to them before.) See also \cite{Aur07}.
As we saw above,  $\R P^{2n+1} \subset \C P^{2n+1}$, $n \geq 1$, is unobstructed.
Since the minimal Maslov number is strictly greater than $2$, we find that any $b \in H^1(\R P^{2n+1}; F) \otimes \Lambda_{0, \text{\rm nov}}^F$ of
total degree $1$ is a bounding cochain, i.e.,
$\mathfrak{PO}(b)=0$, by the dimension counting argument.
On the other hand, Theorem \ref{Theorem44.24} shows that
the Floer cohomology does not vanish for $F=\Z_2$, and
the eigenvalue is zero in the field $F=\Z_2$ because
$c_1(\C P^{2n+1})\equiv 0 \mod 2$.
Thus this is consistent with the problem.
(If we take $F=\Q$, the eigenvalue is not zero in $\Q$.
But Theorem \ref{Theorem44.24} shows that
the Floer cohomology over $\Lambda_{0,{\rm nov}}^{\Q}$ vanishes.
So the assumption of the problem is not satisfied in this case.)
Besides this, we prove this statement for the
case of Lagrangian fibers of smooth toric manifolds
in \cite{fooomirror1}.
We do not have any counter example to this statement at the time of writing this paper.
\end{rem}

\subsection{Wall crossing term in \cite{fukaya;counting}}
\label{subsec:wall}

Let $M$ be a $6$-dimensional symplectic manifold and
let $L$ be
its relatively spin Lagrangian submanifold.
Suppose the Maslov index homomorphism $\mu_L
: H_2(M,L;\Z) \to \Z$ is zero.
In this situation the first named author \cite{fukaya;counting}
introduced an invariant
$$
\Psi_J : \mathcal M(L;\Lambda_{0,{\rm nov}}^{\C}) \to \Lambda_{0,{\rm nov}}^{+ \C}.
$$
In general, it depends on a compatible almost structure $J$ and
the difference $\Psi_J  - \Psi_{J'}$ is an element of $\Lambda_{0,{\rm nov}}^{+ \Q}$.
\par
Let us consider the case where $\tau : M  \to M$ is an
anti-symplectic involution and  $L = \operatorname{Fix}\,\tau$.
We take
the compatible almost complex structures $J_0$, $J_1$ such that
$\tau_* J_0 = - J_0$, $\tau_* J_1 = - J_1$.
Moreover, we assume that there exists a one-parameter family of compatible almost
complex structures
$\mathcal J = \{J_t\mid t\in [0,1]\}$  such that $\tau_* J_t = - J_t$.
We will study the difference
\begin{equation}\label{nowallcross}
\Psi_{J_1} - \Psi_{J_0}
\end{equation}
below.
Namely, we will study the wall crossing phenomenon by the method of this paper.
\par
Let $\alpha \in H_2(M;\Z)$. Denote by
$\mathcal M_1(\alpha;J)$ the moduli space of $J$-holomorphic {\it sphere}
with one interior marked point and of homology class $\alpha$.
We have an evaluation map $ev : \mathcal M_1(\alpha;J) \to M$.
We assume
\begin{equation}\label{nonintassum}
ev(\mathcal M_1(\alpha;J_0)) \cap L = ev(\mathcal M_1(\alpha;J_1)) \cap L = \emptyset
\end{equation}
for any $\alpha \ne 0$. Since the virtual dimension of $\mathcal M_1(\alpha;J)$ is
$2$, (\ref{nonintassum}) holds in generic cases.
The space
\begin{equation}\label{wallcrossingmoduli}
\CM_1(\alpha;\mathcal J;L)
= \bigcup_{t \in [0,1]} \{t\} \times \left(\CM_1(\alpha;J_t) {}_{ev}\times_M L\right)
\end{equation}
has a Kuranishi structure of dimension $0$, which is fibered on $[0,1]$.
The assumption (\ref{nonintassum})
implies that (\ref{wallcrossingmoduli}) has no boundary.
Therefore its virtual fundamental cycle is well-defined and gives a rational
number, which we denote by $\# \CM_1(\alpha;\mathcal J;L)$.
By Theorem 1.5 \cite{fukaya;counting} we have
$$
\Psi_{J_1} - \Psi_{J_0} = \sum_{\alpha}
\# \CM_1(\alpha;\mathcal J;L)T^{\omega(\alpha)}.
$$
The involution naturally induces a map
$
\tau : \CM_1(\alpha;\mathcal J;L)  \to \CM_1(\tau_*\alpha;\mathcal J;L)
$.
\begin{lem}\label{oripreversing}
The map
$\tau : \CM_1(\alpha;\mathcal J;L)  \to \CM_1(\tau_*\alpha;\mathcal J;L)$
is orientation preserving.
\end{lem}
\begin{proof}
In the same way as Proposition \ref{Lemma38.9}, we can
prove that $\tau : \CM_1(\alpha;J) \to \CM_1(\tau_*\alpha;J)$ is orientation
reversing. In fact, this case is similar to the case
$k=-1$, $\mu_L(\beta) = 2c_1(\alpha) = 0$ and $m=1$ of Proposition \ref{Lemma38.9}.
Note that $\tau$ also reverses the orientation on $M$ if $\frac{1}{2}\dim_{\R}M$ is odd.
Therefore for any $t\in [0,1]$, $\tau$ respects the orientation on
$\CM_1(\alpha;J_t){}_{ev}\times_M L$ and
that on $\CM_1(\tau_{\ast}\alpha;J_t){}_{ev}\times_M L$.
Hence the lemma.
\end{proof}
Lemma \ref{oripreversing} implies that, in the case of $\tau$ fixed point,
the cancelation of the wall crossing term via involution,
does {\it not} occur, because of the sign.
Namely Formula (8.1) in the first author's paper \cite{fukaya;counting} is wrong.

\section{Appendix: Review of Kuranishi structure -- orientation and
group action}
\label{sec:appendix}

In this appendix, we briefly review the orientation on a space with
Kuranishi structure and notion of group action on a space with Kuranishi structure for the readers convenience.
For more detailed explanation, we refer Sections A1.1 \cite{fooobook2}
and A1.3 \cite{fooobook2} for Subsections 7.1 and 7.2, respectively.

\subsection{Orientation}
\label{subsec:AOri}

To define orientation on a space with Kuranishi structure,
we first recall the notion of tangent bundle of it.
Let $\CM$ be a compact topological space and let $\CM$ have
a Kuranishi structure. That is, $\CM$ has a collection
of a finite number of Kuranishi neighborhoods
$(V_p,E_p,\Gamma_p,\psi_p,s_p), p\in \CM$
such that
\begin{enumerate}
\item[(k-1)]
$V_{p}$ is a finite dimensional smooth manifold which may have boundaries or corners;
\item[(k-2)]
$E_p$ is a finite dimensional real vector space and
$\dim V_p -\dim E_p$ is independent of $p$;
\item[(k-3)]
$\Gamma_{\alpha}$ is a finite group acting smoothly and effectively on $V_{p}$,
and on $E_p$ linearly;
\item[(k-4)]
$s_{p}$, which is called a {\it Kuranishi map}, is a $\Gamma_{p}$- equivariant smooth section of
the vector bundle $E_{p}\times V_p \to V_p$ called an {\it obstruction bundle};
\item[(k-5)]
$\psi_{p}: s_{p}^{-1}(0)/\Gamma_{p} \to \mathcal M$ is a homeomorphism to its image;
\item[(k-6)]
$\bigcup_p \psi_{p}(s_{p}^{-1}(0)/\Gamma_{p})=\CM$;
\item[(k-7)]
the collection $\{(V_p,E_p,\Gamma_p,\psi_p,s_p)\}_{p\in \CM}$ satisfies
certain compatibility conditions under coordinate change.
\end{enumerate}
See Definition A1.3 and Definition A1.5 in \cite{fooobook2} for the precise
definition and description of the {\it coordinate change} and the compatibility conditions in (k-7), respectively.
We denote by $\CP$ the finite set of $p \in \CM$ above.
By Lemma 6.3 \cite{FO}, we may assume that
$\{(V_p,E_p,\Gamma_p,\psi_p,s_p)\}_{p\in P}$
is a {\it good coordinate system} in the sense of
Definition 6.1 in \cite{FO}.
In other words, there is a partial
order $<$ on $\CP$ such that the following conditions hold:
Let $q < p, (p,q \in \CP)$
with
$
\psi_p(s_p^{-1}(0)/\Gamma_p) \cap \psi_q(s_q^{-1}(0)/\Gamma_q) \ne \emptyset
$.
Then there exist
\begin{enumerate}
\item[(gc-1)]
a $\Gamma_q$-invariant open subset $V_{pq}$ of $V_q$ such that
$$
\psi_q^{-1}(\psi_p(s_p^{-1}(0)/\Gamma_p) \cap \psi_q(s_q^{-1}(0)/\Gamma_q)) \subset V_{pq}/\Gamma_q,
$$
\item[(gc-2)]
an injective group homomorphism $h_{pq}:\Gamma_q \to \Gamma_p$,
\item[(gc-3)]
an
$h_{pq}$-equivariant smooth embedding $\phi_{pq} : V_{pq} \to V_p$
such that the induced map
$V_{pq}/\Gamma_{q} \to V_p/\Gamma_p$ is injective,
\item[(gc-4)]
an $h_{pq}$-
equivariant embedding $\widehat{\phi}_{pq}:E_q\times{V_{pq}} \to E_p\times V_{p}$ of vector bundles which covers $\phi_{pq}$ and
satisfies
$$
\widehat{\phi}_{pq} \circ s_q =s_p \circ \phi_{pq}, \quad
\psi_q= \psi_p \circ \underline{\phi}_{pq}.
$$
Here $\underline{\phi}_{pq}:V_{pq}/\Gamma_q \to V_p/\Gamma_p$
is the map induced by $\phi_{pq}$.
\end{enumerate}
Moreover, if $r<q<p$ and
$\psi_p(s_p^{-1}(0)/\Gamma_p) \cap \psi_q(s_q^{-1}(0)/\Gamma_q)
\cap \psi_r(s_r^{-1}(0)/\Gamma_r) \ne \emptyset$, then there exists
\begin{enumerate}
\item[(gc-5)]
$\gamma_{pqr} \in \Gamma_p$
such that
$$
h_{pq} \circ h_{qr} =
\gamma_{pqr} \cdot
h_{pr}  \cdot
\gamma^{-1}_{pqr}
, \quad
\phi_{pq} \circ \phi_{qr} = \gamma_{pqr} \cdot
\phi_{pr}, \quad
\widehat{\phi}_{pq} \circ \widehat{\phi}_{qr} = \gamma_{pqr} \cdot
\widehat{\phi}_{pr}.
$$
Here the second and third equalities
hold on $\phi_{qr}^{-1}(V_{pq}) \cap V_{qr}
\cap V_{pr}$ and on
$E_r\times({\phi_{qr}^{-1}(V_{pq}) \cap V_{qr}}\cap V_{pr})$),
respectively.
\end{enumerate}
\par\smallskip
Now we identify a neighborhood of
$\phi_{pq}(V_{pq})$ in $V_p$ with a neighborhood of the zero section
of the normal bundle $N_{V_{pq}}V_p \to V_{pq}$.
Then the differential of the Kuranishi map $s_p$ along the fiber
direction defines an $h_{pq}$-equivariant bundle homomorphism
$$
d_{\text{fiber}}s_p : N_{V_{pq}}V_p \to E_p \times V_{pq}.
$$
See Lemma A1.58 \cite{fooobook2} and
also Theorems 13.2, 19.5 \cite{foootech} for detail.
\begin{defn}\label{def:tangent}
We say that the space $\CM$ with Kuranishi structure
{\it has a tangent bundle} if $d_{\text{fiber}}s_p$
induces a bundle isomorphism
\begin{equation}\label{dfiber}
N_{V_{pq}}V_p \cong \frac{E_p \times V_{pq}}{\widehat{\phi}_{pq}(E_q\times {V_{pq})}}
\end{equation}
as $\Gamma_q$-equivariant bundles on a neighborhood of $V_{pq}\cap s_q^{-1}(0)$. (See also Chapter 2 \cite{foootech}.)
\end{defn}
\begin{defn}\label{def:oriKura}
Let $\CM$ be a space with Kuranishi structure with a tangent bundle.
We say that the Kuranishi structure on $\CM$ is {\it orientable} if
there is a trivialization of
$$
\Lambda^{\text{\rm top}}E^*_p \otimes \Lambda^{\text{\rm top}}TV_p
$$
which is compatible with the isomorphism (\ref{dfiber})
and whose homotopy class is preserved by
the $\Gamma_p$-action.
The {\it orientation} is the choice of the homotopy class of such a trivialization.
\end{defn}
Pick such a trivialization.
Suppose that $s_p$ is transverse to zero at $p$.
Then we define an orientation on the zero locus $s_p^{-1}(0)$
of the Kuranishi map $s_p$, which may be assumed so that $p \in s^{-1}_p(0)$, by
the following equation:
$$
E_{p} \times T_ps_p^{-1}(0) = T_p V.
$$
Since we pick a trivialization of
$\Lambda^{\text{\rm top}}E^*_p \otimes \Lambda^{\text{\rm top}}TV_p$ as in Definition \ref{def:oriKura}, the above equality
determines an orientation on $s_p^{-1}(0)$, and also on $s_p^{-1}(0)/\Gamma_p$.
See Section 8.2 \cite{fooobook2} for more detailed explanation of orientation on a space with
Kuranishi structure.

\subsection{Group action}
\label{subsec:Aaction}

We next recall the definitions of a finite group action on a space with Kuranishi structure
and its quotient space.
In this paper we used the $\Z_2$-action and its quotient space
(in the proof of Theorem \ref{Theorem34.20}).

\par
Let $\CM$ be a compact topological space with
Kuranishi structure.
We first define the notion of automorphism of Kuranishi structure.

\begin{defn}\label{def:auto}
Let $\varphi : \CM \to \CM$ be a homeomorphism of $\CM$.
We say that it induces an
{\it automorphism of Kuranishi structure} if the following holds:
Let $p \in \CM$ and $p' = \varphi(p)$. Then, for the
Kuranishi neighborhoods $(V_p,E_p,\Gamma_p,\psi_p,s_p)$ and
$(V_{p'},E_{p'},\Gamma_{p'},\psi_{p'},s_{p'})$ of $p$ and $p'$ respectively,
there exist $\rho_p : \Gamma_p \to \Gamma_{p'}$, $\varphi_p : V_p
\to V_{p'}$,  and
$\widehat{\varphi}_p : E_p \to E_{p'}$ such that
\begin{enumerate}
\item[(au-1)]
$\rho_p$ is an isomorphism of groups;
\item[(au-2)]
$\varphi_p$
is a $\rho_p$-equivariant diffeomorphism;
\item[(au-3)]
$\widehat{\varphi}_p$ is a $\rho_p$-equivariant
bundle isomorphism which covers $\varphi_p$;
\item[(au-4)]
$s_{p'} \circ \varphi_p = \widehat{\varphi}_p \circ s_p$;
\item[(au-5)]
$
\psi_{p'}\circ \underline{\varphi}_p = \varphi \circ\psi_p,
$
where $\underline{\varphi}_p : s_p^{-1}(0)/\Gamma_p \to s_{p'}^{-1}(0)
/\Gamma_{p'}$ is a homeomorphism induced by
$\varphi_p \vert_{s_p^{-1}(0)}$.
\end{enumerate}
We require that $\rho_p$, $\varphi_p$, $\widehat{\varphi}_p$ above satisfy
the following compatibility conditions with the coordinate changes of Kuranishi structure:
Let $q \in \psi_p(s_p^{-1}(0)/\Gamma_p)$ and
$q' \in \psi_{p'}(s_{p'}^{-1}(0)/\Gamma_{p'})$ such that
$\varphi(q) = q'$. Let
$(\widehat{\phi}_{pq},\phi_{pq},h_{pq})$,
$(\widehat{\phi}_{p'q'},\phi_{p'q'},h_{p'q'})$ be the coordinate changes.
Then there exists $\gamma_{pqp'q'} \in \Gamma_{p'}$ such that the following conditions hold:
\begin{enumerate}
\item[(auc-1)]
$\rho_p\circ h_{pq} =
\gamma_{pqp'q'} \cdot (h_{p'q'} \circ
\rho_q)  \cdot \gamma_{pqp'q'}^{-1}$;
\item[(auc-2)]
$\varphi_p\circ \phi_{pq} = \gamma_{pqp'q'} \cdot  (\phi_{p'q'} \circ
\varphi_q)$;
\item[(auc-3)]
$\widehat{\varphi}_p\circ \widehat{\phi}_{pq} = \gamma_{pqp'q'} \cdot  (\widehat{\phi}_{p'q'} \circ
\widehat{\varphi}_q)$.
\end{enumerate}
Then we call $((\rho_p,\varphi_p,\widehat{\varphi}_p)_p;\varphi)$ an {\it automorphism
of the Kuranishi structure}.
\begin{rem}
Here $(\rho_p, \phi_p, \widehat{\phi}_p)_p$ are included as data of an automorphism.
\end{rem}
\end{defn}
\begin{defn}\label{def:conjauto}
We say that an automorphism $((\rho_p,\varphi_p,\widehat{\varphi}_p)_p;\varphi)$
is {\it conjugate} to
$((\rho'_p,\varphi'_p,\widehat{\varphi}'_p)_p;\varphi')$,
if $\varphi = \varphi'$ and
if there exists $\gamma_p \in \Gamma_{\varphi(p)}$ for each $p$ such that
\begin{enumerate}
\item[(cj-1)]
$\rho'_p = \gamma_p \cdot \rho_p \cdot \gamma^{-1}_p$;
\item[(cj-2)]
$\varphi'_p = \gamma_p \cdot \varphi_p$;
\item[(cj-3)]
$\widehat{\varphi}'_p = \gamma_p \cdot \widehat{\varphi}_p$.
\end{enumerate}
\end{defn}
The {\it composition} of the two automorphisms is defined by
the following formula:
$$
\aligned
& ((\rho^1_p,\varphi^1_p,\widehat{\varphi}^1_p)_p;\varphi^1)\circ
((\rho_p^2,\varphi_p^2,\widehat{\varphi}_p^2)_p;\varphi^2) \\
&=
((\rho^1_{\varphi^2(p)} \circ \rho^2_p,\varphi^1_{\varphi^2(p)}\circ
\varphi^2_p,\widehat{\varphi}^1_{\varphi^2(p)}\circ\widehat{\varphi}^2_p)_p;\varphi^1\circ\varphi^2).
\endaligned
$$
Then we can easily check that the right hand side also satisfies the compatibility
conditions (auc-1)-(au-3).
Moreover, we can find that the composition induces the composition
of the conjugacy classes of automorphisms.

\par
\begin{defn}\label{def:oripres}
 An automorphism$((\rho_p,\varphi_p,\widehat{\varphi}_p)_p;\varphi)$ is {\it  orientation preserving}, if it is
compatible with the trivialization of $\Lambda^{\text{\rm top}}E^*_p \otimes \Lambda^{\text{\rm top}}TV_p$.
\end{defn}

\par
Let $\text{\rm Aut}(\CM)$ be the set of all conjugacy classes of the automorphisms of
$\CM$ and let $\text{\rm Aut}_0(\CM)$ be the set of all conjugacy classes of the orientation
preserving automorphisms of $\CM$.
Both of them become groups by composition.

\begin{defn}\label{def:action}
Let $G$ be a finite group which acts on a compact space $\CM$.
Assume that $\CM$ has a Kuranishi structure.
We say that $G$ {\it acts} on $\CM$ (as a space with Kuranishi structure) if, for
each $g\in G$, the homeomorphism $x \mapsto gx$, $\CM \to \CM$
induces an automorphism $g_*$ of the Kuranishi structure and
the composition of $g_*$ and $h_*$ is conjugate to $(gh)_*$.
In other words, an action of $G$ to $\CM$ is a group homomorphism
$G \to \text{\rm Aut}(\CM)$.
\par
An {\it involution} of a space with Kuranishi structure
is a $\Z_2$ action.
\end{defn}

Then we can show the following:
\begin{lem}[Lemma A1.49 \cite{fooobook2}]\label{lem:quot}
If a finite group $G$ acts on a space $\CM$ with Kuranishi structure, then
the quotient space $\CM/G$ has a Kuranishi structure.
\par
If $\CM$ has a tangent bundle and the action preserves it, then
the quotient space has a tangent bundle.
If $\CM$ is oriented and the action preserves the orientation, then the
quotient space has an orientation.
\end{lem}

\subsection{Invariant promotion}
\label{subsec:promotion}
As we promised in the proof of Theorem \ref{Theorem34.20},
we explain how we adopt the obstruction theory
developed in Subsections 7.2.6--7.2.10 \cite{fooobook2}
to promote
a filtered $A_{n,K}$ structure
to a filtered $A_{\infty}$ structure keeping
the symmetry \eqref{34.21}.
Since the modification is straightforward,
we give the outline for readers' convenience.
We use the same notation as in Subsection \ref{subsec:Ainfty}.
We first note that in our geometric setup the Lagrangian submanifold $L$
is the fixed point set of the involution $\tau$. So $\tau$ acts trivially  on $C(L;\Q)$ in Theorem \ref{thm:Ainfty}.
Moreover the induced map $\tau_{\ast}$ given by Definition \ref{def:inducedtau} is also trivial on the monoid $G(L)$. (See Remark \ref{rem:tau}.)
\par
Let $G\subset \R_{\ge 0}\times 2\Z$ be a monoid as in Subsection \ref{subsec:Ainfty}. We denote by $\beta$
an element of $G$.
Let $R$ be a field containing $\Q$ and
$\overline{C}$ a graded $R$-module.
Put $C=\overline{C}\otimes \Lambda_{0,\text{nov}}^R$.
Following Subsection 7.2.6 \cite{fooobook2}, we use the Hochschild cohomology to describe the obstruction to the promotion.
In this article we use
$\overline{C}^e =\overline{C} \otimes R[e,e^{-1}]$
instead of $\overline{C}$ to encode the data of the Maslov index
appearing in \eqref{34.21}.
Here
$e$ is the formal variable in $\Lambda_{0,\text{nov}}^R$.
Note that the promotion is made by induction on the partial order
$<$ on the set $G\times \Z_{\ge 0}$ which is absolutely independent of
the variable $e$. Thus the obstruction theory given in Subsections 7.2.6--7.2.10 \cite{fooobook2} also works on $\overline{C}^e$.
Let $\{\mathfrak m_{k,\beta}\}$ and $\{ \mathfrak f_{k,\beta}\}$ be a filtered $A_{n,K}$ algebra structure (Definition \ref{def:AnK}) and a filtered $A_{n,K}$ homomorphism (Definition \ref{def:AnKhom}). They are
$R$ linear maps from $B_k(\overline{C}[1])$ to $\overline{C}$.
We naturally extend them as $R[e,e^{-1}]$ module homomorphisms and denote the extensions by the same symbols.
We put
\begin{equation}\label{mkeandfke}
\aligned
\mathfrak{m}^e_{k,\beta} & = \mathfrak{m}_{k,\beta}e^{\text{pr}_2(\beta)/2} \quad : B_k(\overline{C}^e[1]) \to \overline{C}^e[1]\\
\mathfrak{f}^e_{k,\beta} & = \mathfrak{f}_{k,\beta}e^{\text{pr}_2(\beta)/2} \quad : B_k(\overline{C}^e[1]) \to \overline{C}^e[1],
\endaligned
\end{equation}
where $\text{pr}_2 : G\subset \R_{\ge 0}\times 2\Z \to 2\Z$ is the projection to the second factor.
In the geometric situation $\text{pr}_2$ is the Maslov index $\mu$.
(See \eqref{eq:defmk}.)
For each $K$ we have a map
\begin{equation}\label{OP}
\text{Op} ~:~ B_K\overline{C}^e[1] \to B_K\overline{C}^e[1]
\end{equation}
defined by
$$
\text{Op} ~(a_1x_1 \otimes \dots \otimes a_K x_K)=(-1)^{\ast}
(a_K^{\dag}x_K \otimes \dots \otimes a_1^{\dag}x_1)
$$
for $a_i=\sum_j c_j e^{\mu_j} \in R[e,e^{-1}]$ and $x_i \in \overline{C}$.
Here
\begin{equation}\label{opsign}
\ast = K+ 1 + \sum_{1\le i < j \le K}
\deg'x_i\deg'x_j
\end{equation}
and
\begin{equation}\label{adag}
a_i^{\dag} = \sum_j c_j (-e)^{\mu_j}.
\end{equation}
Obviously we have $\text{Op} \circ \text{Op}=id$.
\begin{defn}\label{def:Opinv}
An $R[e,e^{-1}]$ module homomorphism $\mathfrak g \in \text{Hom }B(\overline{C}^e_K[1]), \overline{C}^e[1])$ is called {\it $\text{\rm Op}$-invariant} if $\mathfrak g \circ \text{Op} = \text{Op} \circ \mathfrak g$.
\end{defn}
Then it is easy to check the following. Recall that $\tau_{\ast}\beta=\beta$ for $\beta \in G$.
\begin{lem}\label{lem:Opmkbeta}
A filtered $A_{n,K}$ structure $\{ \mathfrak m_{k,\beta} \}$ satisfies
\eqref{34.21} if and only if $\mathfrak m^e_{k,\beta}$ defined by
\eqref{mkeandfke} is $\text{Op}$-invariant.
\end{lem}
\begin{defn}\label{def:OpinvAnK}
A filtered $A_{n,K}$ algebra $(C, \{ \mathfrak m_{k,\beta} \})$ is called
an {\it $\text{Op}$-invariant filtered $A_{n,K}$ algebra} if
$\{ \mathfrak m_{k,\beta}^e \}$ is $\text{Op}$-invariant.
A filtered $A_{n,K}$ homomorphism $\{ \mathfrak f_{k,\beta} \}$
is called {\it $\text{Op}$-invariant} if $\{ \mathfrak f_{k,\beta}^e \}$
is $\text{Op}$-invariant.
We define an {\it \text{Op}-invariant filtered $A_{n,K}$ homotopy equivalence} in a similar way.
\end{defn}
The following is the precise statement of our invariant version of Theorem \ref{ext(n,K)} (=Theorem 7.2.72 \cite{fooobook2}) which is
used in the proof of Theorem \ref{Theorem34.20}.
\begin{thm}\label{invariantext(n,K)}
Let $C_1$ be an $\text{Op}$-invariant filtered $A_{n,K}$ algebra and
$C_2$ an $\text{Op}$-invariant filtered $A_{n',K'}$ algebra such that
$(n,K) < (n',K')$.
Let ${\mathfrak h}:C_1 \to C_2$
be an $\text{Op}$-invariant filtered $A_{n,K}$ homomorphism.
Suppose that ${\mathfrak h}$ is an $\text{Op}$-invariant
filtered $A_{n,K}$ homotopy equivalence.
Then there exist an $\text{Op}$-invariant filtered $A_{n',K'}$ algebra structure on $C_1$
extending the given $\text{Op}$-invariant filtered $A_{n,K}$ algebra structure and an $\text{Op}$-invariant filtered $A_{n',K'}$ homotopy equivalence
$C_1 \to C_2$ extending the given
$\text{Op}$-invariant filtered $A_{n,K}$ homotopy equivalence
${\mathfrak h}$.
\end{thm}
\begin{proof}
To prove this theorem,
we mimic the obstruction theory to promote
a filtered $A_{n,K}$ structure
to a filtered $A_{\infty}$ structure
given in Subsections 7.2.6--7.2.10 \cite{fooobook2}.
\par
Let $(C, \{ \mathfrak m_{k,\beta} \})$ is a filtered $A_{n,K}$ algebra. As mentioned before, we first rewrite the obstruction theory
by using $\overline{C}^e$ instead of $\overline{C}$.
This is done by extending the coefficient ring $R$ to
$R[e,e^{-1}]$ and replacing $\mathfrak m_{k,\beta}$ by $\mathfrak m^e_{k,\beta}$ in \eqref{mkeandfke} as follows:
We put $\overline{\mathfrak m}_{k}=
\mathfrak m_{k,\beta_0}$ for $\beta_0=(0,0)$.
(Note that $\overline{\mathfrak m}_{k}=\mathfrak m^e_{k, \beta_0}$.)
We naturally extend $\overline{\mathfrak m}_{k}$ to an
$R[e,e^{-1}]$ module homomorphism
$B_k(\overline{C}^e[1]) \to \overline{C}^e[1]$.
By abuse of notation we also write $\overline{\mathfrak m}_{k}~:~B_k(\overline{C}^e[1]) \to \overline{C}^e[1]$.
\par
As in the proof of Theorem 7.2.72 \cite{fooobook2}, we may assume that
$(n,K)<(n',K')=(n+1,K-1)$ or $(n,K)=(n,0)<(n',K')=(0,n+1)$ to consider
the promotion.
We consider an $R[e,e^{-1}]$-module
$
\text{Hom } (B_{K'}\overline{C}^e[1], \overline{C}^e[1])
$
and
define the coboundary operator $\delta_1$ on it by
\begin{equation}\label{def:Hcoboundary}
\delta_1(\varphi) = \overline{\mathfrak m}_1 \circ \varphi
+ (-1)^{\deg\varphi +1}\varphi \circ \widehat{\overline {\mathfrak m}}_1
\end{equation}
for $\varphi \in B_{K'}\overline{C}^e[1]$.
Here $\widehat{\overline {\mathfrak m}}_1: B\overline{C}^e[1] \to
B\overline{C}^e[1]$ is a coderivation induced by
$\overline{\mathfrak
m}_1$ on $\overline{C}^e$ as in \eqref{eq:hatmk}.
We denote the $\delta_1$-cohomology by
\begin{equation}\label{def:Hoch}
H(\text{Hom } (B_{K'}\overline{C}^e[1], \overline{C}^e[1]),\delta_1)
\end{equation}
and call it {\it Hochschild cohomology} of $\overline{C}^e$.
We modify the definition of the obstruction class as follows.
As in the proof of Lemma 7.2.74 \cite{fooobook2}
we put $(\Vert \beta \Vert, k)=(n',K')$. (See \eqref{def:betanorm} for the definition of $\Vert \beta \Vert$.)
Then, replacing $\mathfrak m_{k,\beta}$
by $\mathfrak m_{k,\beta}^e$,
(7.2.75) \cite{fooobook2} is modified so that
$$
\sum_{\beta_1+\beta_2 = \beta, k_1+k_2 = k+1, (k_i,\beta_i) \ne (k,
\beta)} \sum_i
(-1)^{\deg'\text{\bf x}_i^{(1)}
}\mathfrak m^e_{k_2,\beta_2}\left(\text{\bf x}_i^{(1)} ,
\mathfrak m^e_{k_1,\beta_1}(\text{\bf x}_i^{(2)}),
\text{\bf x}_i^{(3)}\right),
$$
where $\text{\bf x} \in B_{K'}\overline{C}^e[1]$.
Note that $e^{\text{pr}_2(\beta_1)/2}e^{\text{pr}_2(\beta_2)/2}=e^{\text{pr}_2(\beta)/2}$.
Then this defines an element
$$
o_{K',\beta}^e(C) \in \text{Hom } (B_{K'}\overline{C}^e[1],\overline{C}^e[1])
$$
which is a $\delta_1$-cocycle. Thus
we can define the {\it obstruction class}
\begin{equation}\label{def:obst}
[o_{K',\beta}^e(C)] \in H(\text{Hom } (B_{K'}\overline{C}^e[1], \overline{C}^e[1]),\delta_1)
\end{equation}
for each $\Vert \beta \Vert =n'$.
Under this modification
Lemma 7.2.74 \cite{fooobook2} holds for $\overline{C}^e$ as well.
\par
Now we consider the $\text{Op}$-invariant version.
The map $\text{Op}$ defined by \eqref{OP} acts on
$B_{K'}\overline{C}^e[1]$, $\overline{C}^e[1]$ and so
on
$\text{Hom } (B_{K'}\overline{C}^e[1],\overline{C}^e[1])$
as involution.
We decompose
$\text{Hom } (B_{K'}\overline{C}^e[1],\overline{C}^e[1])$
so that
\begin{equation}
\text{Hom } (B_{K'}\overline{C}^e[1],\overline{C}^e[1]) =
\text{Hom } (B_{K'}\overline{C}^e[1],\overline{C}^e[1])^{\text{Op}}
\oplus
\text{Hom } (B_{K'}\overline{C}^e[1],\overline{C}^e[1])^{-\text{Op}},
\end{equation}
where $\text{Hom } (B_{K'}\overline{C}^e[1],\overline{C}^e[1])^{\text{Op}}$ is the $\text{Op}$-invariant part and
$\text{Hom } (B_{K'}\overline{C}^e[1],\overline{C}^e[1])^{-\text{Op}}$
is the anti $\text{Op}$-invariant part.
\par
Suppose that $(C, \{ \mathfrak m_{k,\beta} \})$ is an $\text{Op}$-invariant  filtered $A_{n,K}$ algebra.
By Lemma \ref{lem:Opmkbeta} we have $\text{Op}$-invariant elements
$$
\mathfrak{m}^e_{k,\beta} = \mathfrak{m}_{k,\beta}e^{\text{pr}_2(\beta)/2}
\in \text{Hom } (B_{K'}\overline{C}^e[1],\overline{C}^e[1])^{\text{Op}}.
$$
Note that the map $\text{Op}$ and $\delta_1$ defined by
\eqref{def:Hcoboundary} commute. Therefore  if we
use $\mathfrak{m}^e_{k,\beta}$,
we can define
an {\it $\text{Op}$-invariant Hochschild cohomology}
\begin{equation}\label{def:OpHoch}
H(\text{Hom } (B_{K'}\overline{C}^e[1], \overline{C}^e[1]),\delta_1)^{\text{Op}}:=
H(\text{Hom } (B_{K'}\overline{C}^e[1], \overline{C}^e[1])^{\text{Op}}
,\delta_1).
\end{equation}
Moreover the construction of the obstruction class above
yields the $\text{Op}$-invariant obstruction class
\begin{equation}\label{def:Opobst}
[o_{K',\beta}^e(C)]^{\text{Op}} \in H(\text{Hom } (B_{K'}\overline{C}^e[1], \overline{C}^e[1]),\delta_1)^{\text{Op}}.
\end{equation}
Then the following lemma is the $\text{Op}$-invariant version of
Lemma 7.2.74 \cite{fooobook2} whose proof is straightforward.
\begin{lem}
Let $(n',K')$ be as above and $C$
an $\text{Op}$-invariant filtered $A_{n,K}$ algebra. Then the obstruction classes
$$
[o_{K',\beta}^e(C)]^{\text{\rm Op}} \in
H(\text{\rm Hom }(B_{K'}\overline{C}^e[1],\overline{C}^e[1]), \delta_{1})^{\text{\rm Op}}
$$
vanish
for all $\beta$ with $\Vert \beta\Vert = n'$
if and only if there exists an $\text{Op}$-invariant
filtered $A_{n',K'}$ structure extending the given
$\text{Op}$-invariant filtered $A_{n,K}$ structure.
\par
Moreover, if $C \to C'$ is an $\text{Op}$-invariant
filtered $A_{n,K}$ homotopy
equivalence, then $[o_{K',\beta}^e(C)]^{\text{\rm Op}}$ is mapped to
$[o_{K',\beta}^e(C')]^{\text{\rm Op}}$ by the isomorphism
$$H(\text{\rm Hom }(B_{K'}\overline{C}^e[1],\overline{C}^e[1]), \delta_{1})^{\text{\rm Op}}
\cong
H(\text{\rm Hom }(B_{K'}\overline{C'}^e[1],\overline{C'}^e[1]), \delta_{1})^{\text{\rm Op}}
$$
induced by the $\text{Op}$-invariant homotopy equivalence.
\end{lem}
Once $\text{Op}$-invariant obstruction theory is established,
the rest of the proof of Theorem \ref{invariantext(n,K)} is parallel to one in Subsection 7.2.6 \cite{fooobook2}.
\end{proof}
\par
Similarly, using the $\text{Op}$-invariant obstruction theory for
an $\text{Op}$-invariant filtered $A_{n,K}$ homomorphism, we can also show the $\text{Op}$-invariant version
of Lemma 7.2.129 \cite{fooobook2} in a straightforward way.
\begin{lem}\label{Opinvhompromot}
Let $(n,K) < (n',K')$ and
$C_1,C_2,C'_1,C'_2$ be $\text{Op}$-invariant filtered $A_{n',K'}$ algebras.
Let $\mathfrak h : C_1 \to C_2$, $\mathfrak h': C'_1 \to C'_2$ be
$\text{Op}$-invariant filtered $A_{n',K'}$
homotopy equivalences.
Let $\mathfrak g_{(1)}: C_1 \to C'_1$ be an $\text{Op}$-invariant
filtered $A_{n,K}$ homomorphism and $\mathfrak g_{(2)}: C_2 \to C'_2$ an
$\text{Op}$-invariant
filtered $A_{n',K'}$ homomorphism. We assume that $\mathfrak g_{(2)}
\circ \mathfrak h$ is an
$\text{Op}$-invariant $A_{n,K}$ homotopic to
$\mathfrak h' \circ \mathfrak
g_{(1)}$.
\par
Then there exists an $\text{Op}$-invariant filtered $A_{n',K'}$ homomorphism
$\mathfrak g_{(1)}^+: C_1 \to C'_1$ such that $\mathfrak g_{(1)}^+$ coincides to
$\mathfrak g_{(1)}$ as an $\text{Op}$-invariant filtered $A_{n,K} $ homomorphism and that
$\mathfrak g_{(2)} \circ \mathfrak h$ is $\text{Op}$-invariant filtered $A_{n',K'}$ homotopic to $\mathfrak h' \circ \mathfrak g_{(1)}^+$.
\end{lem}

\bibliographystyle{amsalpha}

\end{document}